\documentclass[a4paper]{article}
\usepackage[utf8]{inputenc}

\usepackage{amsmath}    
\usepackage{amsthm}
\usepackage{graphicx}   
\usepackage{verbatim}   
\usepackage{latexsym}
\usepackage{amsfonts}
\usepackage{amssymb}
\usepackage{stmaryrd}
\usepackage{enumitem}

\usepackage{afterpage}
\usepackage{array}
\usepackage{mathdots}
\usepackage{pgf}
\usepackage{pdflscape}
\usepackage{tikz}
\usetikzlibrary{calc,through,intersections,arrows, trees, positioning}
\usepackage{pgfplots}
\usepackage[english,ngerman]{babel}
\usepackage{dsfont}
\usepackage{makeidx}

\usepackage{relsize}
\usepackage{exscale}
\usepackage{float}
\usepackage{caption}
\usepackage[bookmarks]{hyperref}
\usepackage{hyperref}

\makeatletter
\newcommand\footnoteref[1]{\protected@xdef\@thefnmark{\ref{#1}}\@footnotemark}
\makeatother

\newcommand{\im}{\normalfont\text{im}}

\newcommand{\id}{\normalfont\text{id}}
\newcommand{\N}{\mathbb{N}}
\newcommand{\Z}{\mathbb{Z}}
\newcommand{\Q}{\mathbb{Q}}
\newcommand{\R}{\mathbb{R}}
\renewcommand{\L}{\mathbb{L}}
\renewcommand{\S}{\mathbb{S}}

\newcommand{\aA}{\mathfrak{A}}

\newcommand{\aB}{\mathfrak{B}}
\newcommand{\aC}{\mathfrak{C}}

\newcommand{\aS}{\mathfrak{S}}
\newcommand{\aF}{\mathfrak{F}}
\newcommand{\af}{\mathfrak{f}}
\newcommand{\ag}{\mathfrak{g}}
\newcommand{\aL}{\mathfrak{L}}

\newcommand{\I}{\mathbb{I}}
\newcommand{\cA}{\mathcal{A}}
\newcommand{\cB}{\mathcal{B}}
\newcommand{\cC}{\mathcal{C}}
\newcommand{\cD}{\mathcal{D}}
\newcommand{\cE}{\mathcal{E}}
\newcommand{\cF}{\mathcal{F}}

\newcommand{\cH}{\mathcal{H}}
\newcommand{\cP}{\mathcal{P}}
\newcommand{\cX}{\mathcal{X}}
\newcommand{\cN}{\mathcal{N}}

\newcommand{\cQ}{\mathcal{Q}}
\newcommand{\cU}{\mathcal{U}}
\newcommand{\cO}{\mathcal{O}}
\newcommand{\cV}{\mathcal{V}}
\newcommand{\cI}{\mathcal{I}}
\newcommand{\cJ}{\mathcal{J}}
\newcommand{\cL}{\mathcal{L}}
\newcommand{\cK}{\mathcal{K}}

\newcommand{\cS}{\mathcal{S}}
\newcommand{\cY}{\mathcal{Y}}
\newcommand{\cW}{\mathcal{W}}
\newcommand{\Tau}{\mathcal{T}}
\newcommand{\LP}{\Lambda}
\newcommand{\hPsi}{\hat{\Psi}}

\newcommand{\crit}{\normalfont\text{crit}}
\newcommand{\cl}{\normalfont\text{cl}}

\DeclareMathOperator{\medcup}{\mathsmaller{\bigcup}}

\newcommand{\invLim}{\varprojlim}
\newcommand{\iL}{\langle\mkern-1.4\thinmuskip\langle G \rangle\mkern-1.4\thinmuskip\rangle}
\newcommand{\ECO}{\langle\mkern-0.2\thinmuskip\vert G \vert\mkern-0.2\thinmuskip\rangle}
\newcommand{\Gq}{\tilde{\vartheta} G}
\newcommand{\Gqq}{\hat{\vartheta} G}
\newcommand{\TC}{\vartheta G}
\newcommand{\CL}{/\!\!/}
\newcommand{\HTC}{\mathcal{H}(\vartheta G)}
\newcommand{\GH}{\mathfrak{A}G}

\newcommand{\GG}{\mathfrak{G}}

\newcommand{\FG}{\mathfrak{F}G}
\newcommand{\LG}{\Lambda G}
\newcommand{\CG}{G}
\newcommand{\CTC}{\hat{C}_{\TC}}

\newcommand{\AuxE}{\mathfrak{E}}
\newcommand{\ATST}{$\mathfrak{A}$\textnormal{TST}}
\newcommand{\HDcomp}{Hausdorff compactification}
\newcommand{\ecomp}{$\Omega$-compactification}
\DeclareMathOperator{\bij}{\hookrightarrow\!\!\!\!\!\rightarrow}

\newcommand\A[1]{\mathfrak{#1}}
\mathchardef\arr="017E 
\newcommand\undervec[1]{\setbox0=\hbox{$#1$}\lower2ex\hbox to 0pt{\hbox to \wd0{\hss$\arr\,$\hss}\hss}\box0}

\newcommand{\utranscl}{\between\cdots\between}
\newcommand{\transcl}{\mathrel{\hbox{$\between$}\hskip-.1ex\hbox{$\cdot$}\hskip-.1ex\hbox{$\cdot$}\hskip-.1ex\hbox{$\cdot$}\hskip-.1ex\hbox{$\between$}}}
\newcommand*{\tsim}{%
  \mathrel{\vcenter{\offinterlineskip
  \hbox{$\sim$}\vskip-.35ex\hbox{$\sim$}\vskip-.35ex\hbox{$\sim$}}}}
\newcommand*{\tcsim}{%
  \mathrel{\vcenter{\offinterlineskip
  \hbox{$\sim$}\vskip-.85ex\hbox{$-$}\vskip-.55ex\hbox{$\sim$}}}}
\newcommand*{\tasymp}{%
  \mathrel{\vcenter{\offinterlineskip
  \hbox{$\asymp$}\vskip-1.35ex\hbox{$-$}}}}

\newcommand{\rest}{\!\!\upharpoonright\!}
\renewcommand\vec[1]{\overset{{}_{\;\rightarrow}}{#1}}
\newcommand\cev[1]{\overset{{}_{\,\leftarrow}}{#1}} 

\hypersetup{
pdftitle = {???},
pdfsubject = {Master Thesis},
pdfauthor = {Jan Kurkofka},
pdfkeywords = {},
colorlinks = false,
draft = false,
bookmarksopen=true}
\newtheorem{theorem}{Theorem}[subsection]
\newtheorem{proposition}[theorem]{Proposition}
\newtheorem{corollary}[theorem]{Corollary}
\newtheorem{lemma}[theorem]{Lemma}
\newtheorem{obs}[theorem]{Observation}
\newtheorem{example}[theorem]{Example}
\newtheorem{conjecture}[theorem]{Conjecture}
\newtheorem{claim_auxArc}{Claim}
\newtheorem{claim_bigConv}{Claim}
\newtheorem{claim_tcInvLim}{Claim}
\newtheorem{claim_subTSTs}{Claim}

\newtheorem{innercustomthm}{Theorem}
\newtheorem{innercustomprop}{Proposition}
\newtheorem{innercustomlem}{Lemma}
\newenvironment{customthm}[1]
  {\renewcommand\theinnercustomthm{#1}\innercustomthm}
  {\endinnercustomthm}
\newenvironment{customprop}[1]
  {\renewcommand\theinnercustomprop{#1}\innercustomprop}
  {\endinnercustomprop}
\newenvironment{customlem}[1]
  {\renewcommand\theinnercustomlem{#1}\innercustomlem}
  {\endinnercustomlem}

\setenumerate{label={\normalfont (\roman*)}}

\makeindex
\usepackage[nottoc]{tocbibind}

\begin{document}
\selectlanguage{english}


  \thispagestyle{empty}

  \begin{titlepage}

 \begin{center} 

    {\huge On the tangle compactification\\of infinite graphs}
    \vspace*{0.7cm}
    
    \normalsize von \LARGE
    \vspace*{0.7cm}

    Jan Kurkofka
    \vspace*{1.5cm}

	\LARGE\textsc{Masterarbeit}
    \vspace*{1.5cm}
    
    \normalsize vorgelegt der
    \vspace*{0.7cm}
    
    \Large Fakultät für Mathematik, Informatik und\\
    Naturwissenschaften der Universität Hamburg
    \vspace*{0.7cm}
    
    \normalsize im Juli 2017
    \vspace*{1.5cm}
    
    Gutachter:\\
    \large Dr. Max Pitz\\
    \large Prof. Dr. Reinhard Diestel
  \end{center}
\end{titlepage}
\newpage
\thispagestyle{empty}
\section*{Abstract of the arXiv version}

In finite graphs, finite-order tangles offer an abstract description of highly connected substructures.
In infinite graphs, infinite-order tangles compactify the graphs in the same way the ends compactify connected locally finite graphs.
Thus, the arising tangle compactification extends the well-known Freudenthal compactification from connected locally finite graphs to arbitrary infinite graphs.
This Master's thesis investigates the tangle compactifcation.

\section*{Preface of the arXiv version}

This is my Master's thesis from 2017.
I have fixed a critical typo in the definition of $\cO_{\GH}(X,\cC,\epsilon)$.
Section~\ref{ECOcompactification} is now a published paper~\cite{EndsTanglesCrit}.

My research on the tangle compactification and infinite-order tangles continues.
In~\cite{StoneCechTangles}, Pitz and I have shown that the tangle compactification is deeply linked to the Stone-\v{C}ech compactification, answering a question of Diestel.
And critical vertex sets, the key concept that drives Section~\ref{ECOcompactification} and~\cite{EndsTanglesCrit}, will be employed in two upcoming papers.

\newpage
\thispagestyle{empty}
\selectlanguage{ngerman}
\section*{Eidesstattliche Erklärung}
Die vorliegende Arbeit habe ich selbständig verfasst und keine anderen als die
angegebenen Hilfsmittel -- insbesondere keine im Quellenverzeichnis nicht benannten
Internet-Quellen -- benutzt. Die Arbeit habe ich vorher nicht in einem
anderen Prüfungsverfahren eingereicht. Die eingereichte schriftliche Fassung
entspricht genau der auf dem elektronischen Speichermedium.\\ \\
Seevetal, den 13. Juli 2017\\ \\ \\ \\ \\
Jan Kurkofka
\selectlanguage{english}

\newpage
\thispagestyle{empty}
\tableofcontents
\thispagestyle{empty}

\newpage
\setcounter{page}{1}

\section{Introduction}

Many theorems about finite graphs involving paths, cycles or spanning trees do not generalise to infinite graphs verbatim.
However, if we consider only infinite graphs which are locally finite,\footnote{A graph is \textit{locally finite} if each of its vertices has only finitely many neighbours.} then an elegant solution is known for generalising most of these theorems. 
To understand this solution, it is important to know that for connected locally finite graphs $G$ adding their ends\footnote{An \textit{end} of a graph is an equivalence class of rays, where a \textit{ray} simply is a 1-way infinite path and two rays are \textit{equivalent} whenever no finite set of vertices separates them in the graph.} yields a natural compactification $|G|$, their \textit{Freudenthal compactification} \cite{RDsBanffSurvey,Bible}.

Formally, to obtain $|G|$ we extend the 1-complex of $G$ (also denoted $G$) to a topological space $G\cup\Omega$ (where $\Omega$ is the set of ends of $G$) by declaring as open, for every finite set $X$ of vertices and each component $C$ of $G-X$ the set $\cO_{|G|}(X,C)$, and taking the topology on $G\cup\Omega$ this generates.
Here, $\cO_{|G|}(X,C)$ is the union of the 1-complex of $C$, the set of all inner edge points of edges between $X$ and the component $C$, and the set of all ends of $G$ all whose rays have tails in $C$.\footnote{For the experts: Here, we introduced $|G|$ using a basis that slightly differs from the usual one, but which generates the same topology.
With our basis it will be easier to see the similarities to the basis of the tangle compactification.}

Now, we explain the elegant solution:
replacing, in the wording of the theorems, paths with homeomorphic images of the unit interval (\textit{arcs} for short) in the Freudenthal compactification $|G|$, cycles with homeomorphic images of the unit circle (\textit{circles} for short) in $|G|$, 
and spanning trees with uniquely arc-connected subspaces of $|G|$ including the vertex set of the graph (\textit{topological spanning trees} for short, see \cite{Bible} for a precise definition), 
does suffice to extend these theorems. 
The arcs, circles and topological spanning trees considered are allowed to contain ends of $G$ (and in a moment we will see that sometimes they have to).

For a nice illustration of how this solution works, consider the so-called `tree-packing' \cite[Theorem 2.4.1]{Bible}, 
proved independently by Nash-Williams and Tutte in 1961: \textit{A finite multigraph contains $k$ edge-disjoint spanning trees if and only if for every partition $P$ of its vertex set it has at least $k(|P|-1)$ cross-edges.}
Aharoni and Thomassen~\cite{aharoniThom} constructed, for every $k$, a locally finite graph $G(k)$ witnessing that the naive extension of tree-packing to locally finite graphs fails in that the graph $G(k)$ has enough edges across every finite partition of its vertex set but no $k$ edge-disjoint spanning trees (see \cite{Bible} for details).
However, as Diestel~\cite{DiestelBook05} has shown in 2005, as soon as we replace `spanning trees' with `topological spanning trees' we do obtain a correct extension (also see \cite[Theorem 8.5.7]{Bible}):
\textit{A locally finite multigraph contains $k$ edge-disjoint topological spanning trees if and only if for every finite partition $P$ of its vertex set it has at least $k(|P|-1)$ cross-edges.}
In particular, at least one of any $k$ edge-disjoint topological spanning trees of $G(k)$ must use an end. 

Another example is a theorem by Fleischner \cite[Theorem 10.3.1]{Bible} from 1974:
\textit{If $G$ is a finite 2-connected graph, then $G^2$ has a Hamilton cycle.}\footnote{For every graph $G$ and each natural number $d>0$ we write $G^d$ for the graph on $V(G)$ in which two vertices are adjacent if and only if they have distance at most $d$ in $G$.}
This theorem was extended by Georgakopoulos~\cite{fleisch} in 2009:
\textit{If $G$ is a locally finite 2-connected graph, then $G^2$ has a Hamilton circle.}

More generally, Diestel and Kühn~\cite{CyclesI,CyclesII} (2004) and Berger and Bruhn~\cite{BergerBruhnDeg} (2009) were able to generalise the full cycle space theory of finite graphs to locally finite ones. We refer the reader to \cite[Theorem 8.5.10]{Bible} for details since these go beyond the scope of this introduction.

But graphs that are not locally finite---in general---cannot be compactified by adding their ends, 
so the elegant solution no longer applies.
For example, in the introduction of Chapter~\ref{AuxArcTP} we will see that a $K_{2,\aleph_0}$ (a complete bipartite graph with one countable bipartition class and the other of size 2) has enough edges across every finite partition of its vertex set for $k=2$. But this graph has no ends, so the naive extension of topological spanning trees defaults to spanning trees, any two of which must share an edge (otherwise, none would be connected).
Thus, it is considered one of the most important problems in infinite graph theory to come up with an extension that allows us to extend the cycle space theory to non-locally finite graphs.
Recently, in 2015, Diestel \cite{EndsAndTangles} proposed a possible solution to this problem and constructed a new compactification which uses tangles instead of ends, the \textit{tangle compactification}. 
Before we present the special characteristics of this compactification, we give a brief introduction to tangles.

A \textit{separation of finite order} of a graph $G$ is a set $\{A,B\}$ with $A\cap B$ finite and $A\cup B=V(G)$ such that $G$ has no edge between $A\setminus B$ and $B\setminus A$.
If $\omega$ is an end of a graph $G$, then it orients each separation $\{A,B\}$ of finite order towards a \textit{big side} $K\in\{A,B\}$ in that every ray contained in $\omega$ has some tail in $K$ (clearly, it cannot both have a tail in $A$ and a tail in $B$ since $A\cap B$ is a finite separator).
Observe that every end of a graph $G$ orients all the finite order separations consistently in that e.g. for every two finite order separations $\{A,B\}$ and $\{C,D\}$ with $A\subseteq C$ and $B\supseteq D$ the end does not choose $A$ and $D$ as big sides.
From a more abstract point of view (but for a different purpose), Robertson and Seymour~\cite{GMX} defined an $\aleph_0$-\textit{tangle} (or \textit{tangle}) in a graph as an orientation of all its finite order separations towards a big side that is consistent in some sense including the above.
Every end induces an $\aleph_0$-tangle, 
so each graph's end space can be considered as a natural subset of its tangle space, and in fact,
if a graph is locally finite and connected, then its $\aleph_0$-tangles turn out to be precisely its ends (and its tangle compactification coincides with its Freudenthal compactification).
However, for graphs that are not locally finite, there may be $\aleph_0$-tangles that are not induced by an end, and adding these on top of the ends suffices to compactify those graphs---whereas adding only the ends does not.

Understanding the tangles that are not induced by an end is important, but for this,
we need some notation first: If $G$ is a graph, then we write $\TC$ for its tangle compactification, and we denote by $\cX$ the collection of all finite sets of vertices of this graph, partially ordered by inclusion. Furthermore, for every $X\in\cX$ we write $\cC_X$ for the collection of all components of $G-X$. Finally, for every subcollection $\cC\subseteq \cC_X$ we denote the vertex set of $\bigcup\cC$ by $V[\cC]$.
Now we may start: Every finite order separation $\{A,B\}$ corresponds to the bipartition $\{\cC,\cC'\}$ of $\cC_X$ with $X=A\cap B$ and
\begin{align*}
\{A,B\}=\{V[\cC]\cup X,X\cup V[\cC']\},
\end{align*}
and this correspondence is bijective for fixed $X\in\cX$.
Hence if $\tau$ is an $\aleph_0$-tangle of the graph, then
for each $X\in\cX$ it also chooses one \textit{big side} from each bipartition $\{\cC,\cC'\}$ of $\cC_X$, namely the $\cK\in\{\cC,\cC'\}$ with $X\cup V[\cK]=K$ where $K\in\{A,B\}$ is the big side of the corresponding finite order separation $\{A,B\}$. 
Since it chooses these sides consistently, for each $X\in\cX$ they form an ultrafilter on $\cC_X$.
Furthermore, these ultrafilters are compatible in that they are limits of a natural inverse system $\{\cU_X,f_{X',X},\cX\}$. 
Here, each $\cU_X$ is the Stone-Čech \HDcomp{} of $\cC_X$ (where $\cC_X$ is endowed with the discrete topology), i.e each $\cU_X$ is the set of all ultrafilters on $\cC_X$, equipped with a natural topology.
The bonding maps $f_{X',X}$ are the unique continuous maps extending the maps $\phi_{X',X}$ which send, for all $X\subseteq X'\in\cX$, every component of $G-X'$ to the unique component of $G-X$ including it.
Strikingly, it turns out that the $\aleph_0$-tangles are precisely the limits of this inverse system, 
and the ends of a graph are precisely those of its $\aleph_0$-tangles which induce for each $X\in\cX$ a principal ultrafilter on $\cC_X$. 
In particular, if a graph $G$ is locally finite and connected, then all $\cC_X$ are finite, and hence all ultrafilters on them are principal.
So by the above correspondence, we see that every $\aleph_0$-tangle of $G$ is induced by an end, and so the tangle compactification coincides with the Freudenthal compactification for connected locally finite graphs.

This inverse limit description of the $\aleph_0$-tangles is the key to a better understanding of the tangles that are not induced by ends: 
every $\aleph_0$-tangle which is not induced by an end does induce a non-principal ultrafilter on some $\cC_X$, and, as shown by Diestel, each of these non-principal ultrafilters alone determines that tangle.
Therefore, we call these tangles \textit{ultrafilter tangles}. 
It turns out that, for every ultrafilter tangle $\tau$ there exists a unique element $X_\tau$ of $\cX$ whose up-closure in $\cX$ consists precisely of those $X$ for which the tangle $\tau$ induces a non-principal ultrafilter on $\cC_X$.

We conclude our general introduction with a brief description of the tangle compactification. More details will be given in Section~\ref{SummaryEndsAndTangles}.
To obtain the tangle compactification $\TC$ of a graph $G$ we extend the 1-complex of $G$ to a topological space $G\cup\cU$ (where $\cU=\invLim\cU_X$ corresponds to the tangle space) by declaring as open, for every $X\in\cX$ and each subcollection $\cC\subseteq\cC_X$, the set $\cO_{\TC}(X,\cC)$, and taking the topology on $G\cup\cU$ this generates.
Here, $\cO_{\TC}(X,\cC)$ is the union of the 1-complex of $\bigcup\cC$, the set of all inner edge points of edges between $X$ and $\bigcup\cC$, and the set of all limits $(U_Y\,|\,Y\in\cX)\in\cU$ with $\cC\in U_X$.
Note that 
$\cC\in U_X$ means that 
the $\aleph_0$-tangle corresponding to the limit $(U_Y\,|\,Y\in\cX)$ orients the finite order separation
\begin{align*}
\{V[\cC_X\setminus\cC]\cup X,X\cup V[\cC]\}
\end{align*}
towards $X\cup V[\cC]$, so $\cO_{\TC}(X,\cC)\cap\cU$ consists precisely of those $\aleph_0$-tangles which orient this finite order separation towards $X\cup V[\cC]$.
Clearly, all singleton subsets of the tangle compactification are closed in it.
Now we know enough to understand the topics and results of this work.
In the remainder of this introduction let me indicate briefly what awaits the reader later.

\paragraph*{Chapter~\ref{DefAndFacts}.}

In this chapter we introduce basic notation, inverse limits and some lemmas from general topology. Furthermore, we provide a summary of Diestel's original paper on the tangle compactification \cite{EndsAndTangles}, and we give an overview for the various topologies used for infinite graphs in this work.

\paragraph*{Chapter~\ref{closerLook}.}
In this chapter we prove some basic results about the tangle compactification needed later, such as a version of the Jumping Arc Lemma from \cite{Bible}.
Also, we find a combinatorial description of the sets $X_\tau$ (for ultrafilter tangles $\tau$):
\begin{customlem}{\ref{UltrafilterCrit}}
Let $G$ be any graph. The following are equivalent for all $X\in\cX$:
\begin{enumerate}
\item There exists an ultrafilter tangle $\tau$ with $X=X_\tau$.
\item Infinitely many components of $G-X$ have neighbourhood equal to $X$.
\end{enumerate}
\end{customlem}
We call the sets $X\in\cX$ satisfying (i) and (ii) of this lemma the \textit{critical} elements of $\cX$.
The following theorem involves these sets:
\begin{customthm}{\ref{InfIndPaths}}
Let $G$ be any graph.
The following are equivalent for all distinct vertices $u$ and $t$ of the graph $G$:
\begin{enumerate}
\item There exist infinitely many independent $u$--$t$ paths in $G$.
\item There exists an end of $G$ dominated\footnote{A vertex $u$ of a graph $G$ \textit{dominates} an end $\omega$ of $G$ if no finite subset of $V(G-u)$ separates $u$ from a ray in $\omega$.} by both vertices $u$ and $t$ or\\
there exists some critical element of $\cX$ containing both vertices $u$ and $t$.
\end{enumerate}
\end{customthm}
For the next result we first need some additional definitions.
An \textit{edge end} of a graph is an equivalence class of rays, where two rays are equivalent whenever no finite set of edges separates them in the graph.
Two vertices of a graph $G$ are said to be \textit{finitely separable} if there exists some finite set $F$ of edges such that the two vertices are contained in distinct components of $G-F$.
Briefly speaking, for connected graphs $G$, the compact Hausdorff topological space $(\cE G,\textsc{ETop})$ is obtained from $G$ and its edge ends by equipping this set with a very coarse topology\footnote{whose basic open sets can be thought of as components of $G-F$ plus certain edge ends and half-open partial edges of $F$ for each finite set $F$ of edges} first, and then identifying every two points which share the same open neighbourhoods (a precise definition is provided in Section~\ref{subsec:topsOverview}).

By generalising the notion of `not finitely separable' to an equivalence relation $\sim$ on $V\cup\cU$ (where $\cU=\invLim\cU_X$) we derive this space from the tangle compactification as a natural quotient:
\begin{customthm}{\ref{ETopInvLimAndGq}}
If $G$ is a connected graph, then $(\cE G,\textsc{ETop})$ is homeomorphic to the quotient $\TC/{\sim}$ of the tangle compactification $\TC$.
\end{customthm}

\paragraph*{Chapter~\ref{tangleInvLim}.}
Since inverse limits have claimed their place in the infinite topological graph theory as useful tools to construct limit objects such as circles, arcs and topological spanning trees from finite minors (see the 5th edition of \cite{Bible} for an inverse limit description of the Freudenthal compactification), 
it makes sense to investigate whether it is possible to describe the tangle compactification via a similar inverse limit.
As our main result in this chapter, we show that this is indeed possible.
For this, we construct an inverse system $\{G_\gamma,f_{\gamma',\gamma},\Gamma\}$ of topological spaces $G_\gamma$ that are based on multigraphs with finite vertex set but possibly infinite edge set. 
These multigraphs are obtained from the graph $G$ by contraction of possibly disconnected vertex sets, but we will see that this is best possible.
The topological spaces $G_\gamma$ are compact and all of their singleton subsets are closed in them.
As promised, we show that the inverse limit 
\begin{align*}
\iL=\invLim (G_\gamma\,|\,\gamma\in\Gamma)
\end{align*}
of our inverse system describes the tangle compactification:
 \begin{customthm}{\ref{G*}}
For every graph $G$ its tangle compactification is homeomorphic to the inverse limit $\iL$.
\end{customthm}

\paragraph*{Chapter~\ref{ECOcompactification}.}

Whenever we consider a compactification of a topological space, three particular questions come to mind: Is it the coarsest compactification? If not, how does a coarsest one look like, and why can we not just take the one-point compactification and be done?
In this chapter, we only consider compactifications of the 1-complex of $G$ extending the end space in a meaningful way, and we call these $\Omega$-\textit{compactifications} (since the end space is denoted $\Omega$).
First, we characterise the graphs admitting a \textit{one-point} \ecomp{} $\alpha G$, one with $|\alpha G-(G\cup\Omega)|=1$ (see Proposition~\ref{OnePointCompVG}), and we give an example of such a graph showing that---in general---the one-point \ecomp{} does not reflect the structure of the graph at all. 
However, there exist simple examples admitting a one-point \ecomp{} reflecting their structure while their tangle compactification adds at least $2^{\A{c}}$ many points on top of the ends. Hence Diestel \cite{EndsAndTangles} asked:
\begin{enumerate}
\item For which graphs is their tangle compactification also their coarsest\\\ecomp{}?
\item If it is not, is there a unique such \ecomp{}, and is there a canonical way to obtain it from the tangle compactification?
\end{enumerate}
To answer these questions, we first construct an inverse system $\{\aF_X,\af_{X',X},\cX\}$ of \HDcomp{}s $\aF_X$ of the $\cC_X$ (where each $\cC_X$ is equipped with the discrete topology) whose inverse limit $\aF=\invLim\aF_X$ we use to obtain an \ecomp{} $\FG$ of the graph $G$ in the way Diestel used $\cU=\invLim \cU_X$ to compactify it.
Here, for each $X\in\cX$ the \HDcomp{} $\aF_X$ of $\cC_X$ adds as many points to $\cC_X$ as $X$ includes critical elements of $\cX$.
We will see that $\aF$ includes the end space as a natural subspace since the bonding maps $\af_{X',X}$ respect the natural maps $\phi_{X',X}$ (recall that $\phi_{X',X}$ sends every component of $G-X'$ to the unique component of $G-X$ including it), and:
\begin{customprop}{\ref{FandY}}
There exists a natural bijection between $\aF\setminus\Omega$ and the collection of all critical elements of $\cX$.
\end{customprop}
Since one-point \ecomp{}s in general do not reflect the structure of the original graph in a meaningful way, we wish to impose further conditions on the \ecomp{}s considered in (ii).
For this, we introduce $\cC$-systems: these are inverse systems of \HDcomp{}s of the $\cC_X$ whose inverse limits generalise the directions\footnote{A map $f$ with domain $\cX$ is a \textit{direction}\index{direction} of $G$ if $f$ maps every $X\in\cX$ to a component of $G-X$ and $f(X)\supseteq f(X')$ whenever $X\subseteq X'\in\cX$ (this condition says that $f$ chooses the components consistently). Diestel and Kühn~\cite{Ends} have shown that the directions of a graph are precisely its ends.}  of the graph $G$, and hence its ends.
Both $\aF$ and $\cU$ come from $\cC$-systems, and:

\begin{customthm}{\ref{alphaGcomp}}
Every $\cC$-system induces an \ecomp{} of its graph $G$. In particular, $\FG$ and $\TC$ are \ecomp{}s of the graph $G$.
\end{customthm}
Moreover, we obtain the following analogue of \cite[Theorem 1]{EndsAndTangles} for $\FG$:
\begin{customthm}{\ref{FGcompHDetc}}
Let $G$ be any graph.
\begin{enumerate}
\item $\FG$ is a compact space in which $G$ is dense and $\FG\setminus G$ is totally disconnected.
\item If $G$ is locally finite and connected, then $\aF=\Omega$ and $\FG$ coincides with the Freudenthal compactification of $G$.
\end{enumerate}
\end{customthm}
Studying the technical $\cC$-systems leads us to the following result comparing our new compactification $\FG$ and the tangle compactification $\TC$:
\begin{customthm}{\ref{CFleCU}}
For every graph $G$ its \ecomp{} $\FG$ is coarser than its tangle compactification $\TC$.
\end{customthm}
Then we are finally in a position to answer the first question and half of the second question of Diestel from above:

\begin{customthm}{\ref{ECOminimal} and \ref{ECOequivalent}}
Let $G$ be any graph.
$\FG$ is the coarsest \ecomp{} of the graph $G$ induced by a $\cC$-system while $\TC$ is the finest one. Furthermore, the following are equivalent:
\begin{enumerate}
\item There exists a homeomorphism between $\FG$ and $\TC$ fixing $G\cup\Omega$.
\item Every $\cC_X$ is finite.
\item $\FG=G\cup\Omega=\TC$.
\end{enumerate}
\end{customthm}
As our third main result of this chapter, we answer the second half of Diestel's second question from above, and show that there is a canonical way to obtain $\FG$ from the tangle compactification.
For this, we define the natural equivalence relation $\asymp$ on the collection of all ultrafilter tangles by letting $\tau\asymp\tau'$ whenever $X_\tau=X_{\tau'}$ holds. 
\begin{customthm}{\ref{FGandQuotient}}
For every graph $G$ the \ecomp{} $\FG$ is homeomorphic to the quotient $\TC/{\asymp}$ of the tangle compactification $\TC$.
\end{customthm}
If $\Upsilon$ denotes the set of all ultrafilter tangles of a graph $G$, then we find the following cardinality bound and comparison:
\begin{customprop}{\ref{FGcardComparison}}
For every graph $G$ the following hold:
\begin{enumerate}
\item $|\aF-\Omega|=|\crit(\cX)|\le |V(G)|$,
\item $|\aF-\Omega|\cdot 2^{\A{c}}\le |\Upsilon|$.
\end{enumerate}
\end{customprop}
Strikingly, we will see an explicit definition of a set $S'$ of finite order separations yielding our fourth main result in this chapter:
\begin{customthm}{\ref{FasTangles}}
For every graph $G$ 
the elements of the inverse limit $\aF$ are precisely the $\aleph_0$-tangles of the graph $G$ with respect to the set $S'$.
\end{customthm}
In particular, $\FG$ actually is another `tangle compactification' for a smaller separation system.
Finally, we find an inverse subsystem of the inverse system $\{G_\gamma,f_{\gamma',\gamma},\Gamma\}$ whose inverse limit $\ECO$ describes the \ecomp{} $\FG$:

\begin{customthm}{\ref{FGinvLim}}
For every graph $G$ the inverse limit $\ECO$ is homeomorphic to the \ecomp{} $\FG$.
\end{customthm}

\paragraph*{Chapter~\ref{AuxArcTP}.}
Earlier, I expressed my hopes for the tangle compactification to generalise the elegant solution from the locally finite case to the general case.
To explain why I think that modifications to the tangle compactification cannot be avoided, 
I show that every possible notion of a topological spanning tree I could think of does not meet my expectations. More precisely, for the case that our graph $G$ is a $K_{2,\aleph_0}$ (see Fig.~\ref{fig:K_2N}) I will a name set of edges which I expect to induce a topological spanning tree for any sensible notion of a topological spanning tree, but none of whose candidates are 
topologically connected.

\begin{figure}[H]
    \centering
    \def\svgwidth{\columnwidth}
    \scalebox{.85}{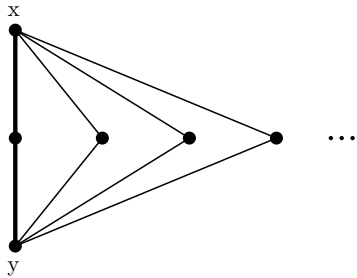}
    \caption{The heavy edge set of this $K_{2,\aleph_0}$ is the sum of all facial cycles.}
    \label{fig:K_2N}
\end{figure}

The heavy edge set from the drawing of our graph $G$ in Fig.~\ref{fig:K_2N} is the thin sum\footnote{A family $(D_i)_{i\in I}$ of subsets of $E(G)$ is \textit{thin} if no edge lies in $D_i$ for infinitely many $i$. 
Then the thin \textit{sum} $\sum_{i\in I}D_i$ is the collection of all edges that lie in $D_i$ for an odd number of indices $i$. See \cite{Bible} for details.} of all facial cycles, and I expect it to be an element of the cycle space for any notion of a cycle space of this graph.
Similarly, I further expect every two edges at a middle vertex together to form an element of any such cycle space.
Consequently, I think that the upper fan in Fig.~\ref{fig:K_2N} should induce a topological spanning tree which contains no edge of the lower fan (since this would create an element of the cycle space).
But if we take as $\Tau$ the upper fan plus the lower vertex $y$ and add any subset of the tangle space, then this turns out to be a (topologically) disconnected subspace of the tangle compactification: Indeed, if we cover $y$ with any basic open neighbourhood $O$ of the 1-complex of $G$, and if we cover $\Tau-\{y\}$ with $\TC-\{y\}$ (which is open since all singleton subsets of the tangle compactification are closed in it), then $O$ and $\TC-\{y\}$ meet only in inner edge points of the lower fan which are no points of $\Tau$.
Thus $\{O,\TC-\{y\}\}$ induces an open bipartition of $\Tau$.
Informally, the problem here is that the tangles are not `sufficiently connected' to certain vertices of the graph, which allows us to separate $y$ from $\Tau-\{y\}$ so easily.
This is why I think that it makes sense to consider modifications to the tangle compactification.

If we modify the tangle compactification to yield a new compactification with the potential of overcoming this and so many other difficulties, then it would be of great advantage if we could see to it that this new compactification also be Hausdorff: then, for connected graphs, the whole field of (non-metric) continuum\footnote{A \textit{continuum} is a compact connected Hausdorff topological space. See Section~\ref{subsec:GeneralTopology} for further details.} theory would open up, providing us with a useful topological tool box. 
This is why I construct two new spaces:
First, in this chapter we study classic hindrances to earlier attempts, and we find inspiration leading to the construction of the \textit{auxiliary space} $\GH$. 
This auxiliary space is Hausdorff, but in general it is not compact.
Thus in the next section we will enhance the idea behind the auxiliary space $\GH$ to obtain a \HDcomp{} $\LG$ from the tangle compactification.

The auxiliary space $\GH$ is obtained from $|G|$ (for the experts: 
here, $|G|$ is endowed with \textsc{MTop}) in two steps:
First, we add \textit{auxiliary edges} (which---formally---are internally disjoint copies of the unit interval) between every end of the graph and each of its dominating vertices, and between any two distinct vertices $u$ and $t$ of the same critical element of $\cX$ (one new auxiliary edge for each critical element of $\cX$ both vertices are contained in, and internally disjoint from the edge $ut$ of $G$ in case it exists).
Second, we generalise the topology of $|G|$ onto this new space in a natural way.

In the first half of this chapter, we study the arc-connected subspaces of $\GH$ induced by the auxiliary edges. 
More precisely, we have a closer look at the auxiliary arc-components of $\GH$, where an \textit{auxiliary arc} is an arc in $\GH$ which is included in the closure of the auxiliary edges. Rather strikingly, our first main result of this chapter holds without imposing any cardinality bounds on the graph considered:

\begin{customthm}{\ref{AuxArcsChar}}
Let $G$ be any graph.
Then between every two distinct points $x$ and $y$ of $V\cup\Omega$ there exists an auxiliary arc if and only if $x$ and $y$ are not finitely separable.
\end{customthm}
This suggests that we might be able to take advantage of the results in Chapter~\ref{closerLook}, namely that $(\cE G,\textsc{ETop})$ is the quotient $\TC/{\sim}$ of the tangle compactification, in order to generalise statements about $\cE G$ to statements about the auxiliary space $\GH$. 
We will see that every normal spanning tree of the graph $G$ induces a topological spanning tree (with respect to the common definition in terms of arcs) of $\GH$, so the auxiliary space $\GH$ overcomes one of the classic hindrances from \cite{TST} (which will be presented in detail in the introduction of Chapter~\ref{AuxArcTP}).

Motivated by these findings, we use the synergy between $\GH$ and $\cE G$ to prove a generalised version of tree-packing\footnote{As mentioned earlier, the so-called `tree-packing' \cite[Theorem 2.4.1]{Bible} was proved independently by Nash-Williams and Tutte in 1961: \textit{A finite multigraph contains $k$ edge-disjoint spanning trees if and only if for every partition $P$ of its vertex set it has at least $k(|P|-1)$ cross-edges.}} for countable graphs:
\begin{customthm}{\ref{auxTreePacking}}
Let $G$ be a countable connected graph. Then the following are equivalent for all $k\in\N$:
\begin{enumerate}
\item $G$ has $k$ topological spanning trees in $\GH$ which are edge-disjoint on $E(G)$.
\item $G$ has at least $k(|P|-1)$ edges across any finite vertex partition $P$.
\end{enumerate}
\end{customthm}
In the outlook of Chapter~\ref{AuxArcTP} we will see an idea on a generalisation of thin sums for circles of $\GH$, but I abandoned further investigation when the idea of $\LG$ came to my mind as a better candidate than $\GH$.

\paragraph*{Chapter~\ref{LGsection}.}

In the previous chapter we have seen that since none of the possible notions of a topological spanning tree of the tangle compactification I could think of met my expectations, I suggested to modify the tangle compactification.
 Furthermore, I claimed that it would be of great advantage if our modifications would yield a \HDcomp{}, since then---for connected graphs---the whole field of (non-metric) continuum theory would open up, providing us with a useful topological tool box.
Then we studied the auxiliary space $\GH$ which---in general---is only Hausdorff but not compact.
In this chapter we enhance the idea behind this auxiliary space to obtain a \HDcomp{} $\LG$ from the tangle compactification.

Starting from the tangle compactification, we add \textit{limit edges} (which---formally---are internally disjoint copies of the unit interval) between every end of the graph and each vertex dominating it, and between every ultrafilter tangle $\tau$ and each vertex in $X_\tau$. 
Treating the inner limit edge points almost like their incident tangles allows us to turn the topology of the tangle compactification into a compact Hausdorff one of the new space, yielding the \HDcomp{} $\LG$.
(As mentioned earlier, we will modify the inverse system $\{G_\gamma,f_{\gamma',\gamma},\Gamma\}$ to construct $\LG$ formally, see Chapter~\ref{LGsection} for details.)

To be precise, the \HDcomp{} $\LG$ only compactifies the 1-complex of $G$ endowed with a slightly coarser topology: at each vertex we only take $\epsilon$-balls instead of stars of arbitrary half-open partial edges (for the experts: $\LG$ is a \HDcomp{} of $(|G|,\textsc{MTop})$).
Since we wish to study graphs of arbitrary big cardinality while the cardinality of
arcs and circles is that of the unit interval (and hence constant), this discrepancy potentially prohibits us from fully understanding these graphs (e.g. sufficiently big graphs would not have a Hamilton circle by definition).
Hence we suggest generalisations of arcs, circles and topological spanning trees for the \HDcomp{} $\LG$ solely in terms of continua (see Chapter~\ref{LGsection} for details).

When $\LG$ came to my mind, this work had already reached critical length, so I only provide sketches. However, the examples I studied so far looked really promising, and I am eager to continue my research on this space.

\paragraph*{Chapter~\ref{sec:GqNatural}.}

Since the tangle compactification in general is not Hausdorff, but the quotient $\TC/{\sim}$ is, and since this quotient is homeomorphic to $(\cE G,\textsc{ETop})$ (cf. Chapter~\ref{closerLook}) for connected graphs $G$, one might ask whether $\cE G$ is the maximal Hausdorff quotient of the tangle compactification. 
Strinkingly, an already known example witnesses that---in general---this is not the case. 
In this chapter we study several graphs, but I did not succeed in finding a combinatorial description of the equivalence relation on the tangle compactification yielding its maximal Hausdorff quotient (it seems like wild auxiliary arcs are the key here, where an arc is \textit{wild} if it induces the ordering of the rationals on some subset of its vertices).

However, as our main contribution we at least present sufficient combinatorial conditions for when $\cE G$ is the maximal Hausdorff quotient of the tangle compactification of a connected graph $G$. 
For these results we need two new pieces of notation: First, if $G$ is a graph and $A$ is a set, then we write $E(A)$ for the collection of all edges of $G$ with both endvertices in $A$. Second, if $T$ is a normal spanning tree of a graph $G$, 
then by \cite[Lemma 8.2.3]{Bible} every end $\omega$ of $G$ contains precisely one \textit{normal ray} of $T$ (a ray in $\omega$ starting at the root of $T$) which we denote $R_\omega^T$.

\newpage
\begin{customprop}{\ref{firstSufCombCond}}
Let $G$ be a connected graph such that for all two distinct vertices $x$ and $y$ of $G$ the following are equivalent:
\begin{enumerate}
\item The vertices $x$ and $y$ are not finitely separable.
\item There exists some $X\in\cX$ containing both $x$ and $y$ such that no $Y\in\cX$ disjoint from $X$ separates $x$ and $y$ in $G-E(X)$.
\end{enumerate}
Then $\cE G$ is the maximal Hausdorff quotient of the tangle compactification $\TC$.
\end{customprop}
If (ii) holds for two distinct vertices $x$ and $y$ of a graph $G$, witnessed by such an $X$, then we can inductively find infinitely many edge-disjoint $x$--$y$ paths in $G-E(X)$, so $x$ and $y$ are not finitely separable. In particular, (ii) implies (i).
Therefore, all connected finitely separable graphs satisfy the premise of this proposition.
The next result involves normal spanning trees and binary trees:
\begin{customprop}{\ref{sufficientAllAuxArcsTame}}
If $G$ is a connected graph such that for every $x\in\TC/{\sim}$ the graph $G$ has a normal spanning tree $T(x)$ whose subtree $\bigcup_{\omega\in x\cap\Omega}R_\omega^{T(x)}$ contains no subdivision of the (infinite) binary tree,
then $\cE G$ is the maximal Hausdorff quotient of the tangle compactification $\TC$.
\end{customprop}
In the outlook of this chapter we point out the difficulties preventing us from giving a characterisation of these graphs, and we state our desired result as a conjecture.\\ \\ \\

Personally, I hope that a (possibly modified) tangle compactification allows us to further generalise the elegant solution from the locally finite case to the general case, and that is why in this work I study the tangle compactification of infinite graphs.

\newpage

\section{Definitions \& general facts}\label{DefAndFacts}
\subsection{Basic Notation}\label{Basics}

Any terms regarding graphs that are not defined in this work can be found in~\cite{Bible}.

A finite partition of a set is said to be \textit{cofinite} if at most one partition class is infinite.\index{cofinite partition/subset}
A set $A$ is \textit{cofinite in} a set $B$ if $B\setminus A$ is finite. If $A\subseteq B$ is cofinite in $B$, then $A$ is called a \textit{cofinite subset} of $B$.
If $A\subseteq B$ are two sets and $R$ is an equivalence relation on $A$ we denote by $B/R$ the set $(B\setminus A)\cup A/R$. 

The set $\N$ contains $0$, and for every $n\in\N$ we denote by $[n]$ the set $\{1,\hdots,n\}$. We denote the unit interval $[0,1]$ by $\I$,\index{$\I$} and for $\lambda\in\R$ and $\epsilon>0$ we write $(\lambda\pm\epsilon)$ for the open interval $(\lambda-\epsilon,\lambda+\epsilon)$ and $[\lambda\pm\epsilon]=[\lambda-\epsilon,\lambda+\epsilon]$.\index{$(\lambda\pm\epsilon)$, $[\lambda\pm\epsilon$]}

A handful of statements are modified or generalised versions of statements from the lecture courses by Diestel (winter 2015--summer 2016); we flagged them with an `$\L$'.\index{$\L$}

If $G$ is a graph,\index{1-complex of $G$}\index{$G$} then we denote by $\CG$ the 1-complex of $G$, i.e. in $\CG$ every edge $e=xy$ is a homeomorphic copy $[x,y]:=\{x\}\cup\mathring{e}\cup\{y\}$ of $[0,1]$ with $\mathring{e}$ corresponding to $(0,1)$ and $\mathring{e}\cap\mathring{e}'=\emptyset$ for every other edge $e'$ of $G$, and $e$ also inherits the euclidean metric from $\I$. 
The points of $\mathring{e}$ are called \textit{inner edge points}, and they inherit their basic open neighbourhoods from $\I$. The space $[x,y]$ is called a \textit{topological edge}, but we refer to it simply as \textit{edge}. Furthermore, for every vertex $u$ of $G$ the set $\bigcup_{e\in E(u)}[u,j_e)$ with each $j_e$ some point of $\mathring{e}$ is basic open.
If every $j_e$ is at distance $\epsilon$ from $u$ with respect to the metric of $e$, then we write $\cO_G(u,\epsilon)=\bigcup_{e\in E(u)}[u,j_e)$.\index{$\cO_G(u,\epsilon)$} For every $F\subseteq E$ we write $\smash{\mathring{F}=\bigcup_{e\in F}\mathring{e}}$.

If $e=[x,y]$ is a topological edge, and $d_e$ is its inherited metric from $\I$, then we denote by $m(e)$ the point of $\mathring{e}$ corresponding to $1/2$, and for all $0\le\epsilon<\delta\le 1$ we write\index{$x(\epsilon,\delta)y$}
\begin{align*}
x(\epsilon,\delta)y=\{i\in e\,|\,d_e(i,x)>\epsilon\text{ and }d_e(i,x)<\delta\}
\end{align*}
for the subset of $e$ corresponding to the open interval $(\epsilon,\delta)$ (with $x$ corresponding to 0 and $y$ corresponding to $1$).

If $X$ is a finite set of vertices of $G$ and $\omega$ is an end of $G$, then $C(X,\omega)$\index{$C(X,\omega)$} is the unique component of $G-X$ such that every ray in $\omega$ has a tail in it. If $C$ is a component of $G-X$, we write $\Omega(X,C)=\{\omega\in\Omega\,|\,C(X,\omega)=C\}$.\index{$\Omega(X,C)$} Furthermore, if $\omega$ is an end of $G$, we write $\Omega(X,\omega)$\index{$\Omega(X,\omega)$} for $\Omega(X,C(X,\omega))$.

For every set $A$ we denote by $E(A)$ the set of all edges of $G$ with both endvertices in $A$.
If $A$ and $B$ are two disjoint sets and $\epsilon\in (0,1]$, then we write $\mathring{E}_\epsilon(A,B)^*$\index{$\mathring{E}_\epsilon(A,B)^*$} for the set of all inner points of $A$--$B$ edges (of $G$) at distance less than $\epsilon$ from their endpoint in $B$ (with respect to the metric of the edge). The `*' on the right side is supposed to help us remember `from where to take our $\epsilon$-balls'.

Two vertices of $G$ are said to be \textit{finitely separable}\index{finitely separable} whenever there exist some finitely many edges separating them.\footnote{i.e. $\exists\;$finite $F\subseteq E(G)$ such that the two vertices are contained in distinct components of $G-F$.} 
If every two distinct vertices of $G$ are finitely separable, then we call $G$ \textit{finitely separable}.

If $T$ is a normal spanning tree (NST) of $G$
and $\omega$ is an end of $G$, we denote by $R_\omega^T$\index{$R_\omega^T$} the normal ray of $T$ in $\omega$ (see ~\cite[Lemma 8.2.3]{Bible}).

If $R$ is a ray and $u$ is a vertex of $R$, then $uR$ denotes the tail of $R$ starting with $u$, and $Ru$ denotes the finite intial segment of $R$ ending with $u$.\index{$uR$, $Ru$}

We denote by $T_2$ the (infinite) binary tree on the set of finite 0--1 sequences (with the empty sequence as the root).\index{$T_2$}

\subsection{Inverse Limits}

Below we give a minimal introduction to inverse limits of inverse systems accumulated from \cite[Chapter 8.7 of the 5th edition]{Bible}, \cite{ProfiniteGroups} and \cite{Field}:\index{inverse system}\index{inverse limit}

A partially ordered set $(I,\le)$ is called \textit{directed} if for every $i,j\in I$ there is some $k\in I$ with $k\ge i,j$. Assume that $(X_i\,|\,i\in I)$ is a family of topological spaces indexed by some directed poset $(I,\le)$. Furthermore suppose that we have a family $(\varphi_{ji}\colon X_j\to X_i)_{i\le j\in I}$ of continuous maps which are \textit{compatible} in that
$\varphi_{ki}=\varphi_{ji}\circ\varphi_{kj}$ for all $i\le j\le k\in I$, and which are the identity on $X_i$ in case of $i=j$.
Then both families together form an \textit{inverse system}, and the maps $\varphi_{ji}$ are called its \textit{bonding maps}. We denote such a system by $\{X_i,\varphi_{ji},I\}$, or $\{X_i,\varphi_{ji}\}$ for short if $I$ is clear from context. The \textit{inverse limit} $\invLim (X_i\,|\,i\in I)$ (or $\invLim X_i$ for short) of this system is the subset
\begin{align*}
\{(x_i)_{i\in I}\,|\,\varphi_{ji}(x_j)=x_i\text{ for all }i\le j\in I\}
\end{align*}
of $\prod_{i\in I}X_i$ whose product topology we pass on to $\invLim X_i$ via the subspace topology. 
Whenever we define an inverse system without specifying a topology for the spaces $X_i$, we tacitly assume them to carry the discrete topology.
We end this introduction by listing some Lemmas which we will put to use later:

\begin{lemma}[{\cite[Lemma 1.1.2]{ProfiniteGroups}}]
If $\{X_i,\varphi_{ji},I\}$ is an inverse system of Hausdorff topological spaces, then $\invLim X_i$ is a closed subspace of $\prod_{i\in I}X_i$.
\end{lemma}

A topological space is \textit{totally disconnected} if every point in the space is its own connected component.\index{totally disconnected}
\begin{lemma}[{\cite[Proposition 1.1.3]{ProfiniteGroups}}]\label{InvLimTotDisc}
Let $\{X_i,\varphi_{ji},I\}$ be an inverse system of compact Hausdorff totally disconnected topological spaces. Then $\invLim X_i$ is also a compact Hausdorff totally disconnected topological space.
\end{lemma}

\begin{lemma}[{\cite[Lemma 1.1.3]{Field}}]\label{InvLimCptHD}
The inverse limit of an inverse system of non-empty compact Hausdorff spaces is a non-empty compact Hausdorff space.
\end{lemma}

\begin{lemma}[Generalized Infinity Lemma, {\cite[Proposition 1.1.4]{ProfiniteGroups}} and {\cite[Corollary 1.1.4]{Field}}]\label{GIL}
The inverse limit of an inverse system of non-empty finite sets is non-empty.
\end{lemma}

\begin{lemma}[{\cite[Lemma 1.1.1]{Field}}]\label{invLimBasis}
Let $\{X_i,\varphi_{ji},I\}$ be an inverse system of topological spaces and denote by $\pi_i$ the restriction of the $i$th projection map $pr_i\colon\prod_{j\in I}X_j\to X_i$ to $X$ where $X=\invLim X_i$.
Then the collection of all subsets of $X$ of the form $\pi_i^{-1}(U_i)$ with $U_i$ open in $X_i$ is a basis for the topology of $X$.

Moreover, if for every $i$ the set $\cB_i$ is a basis of the topology of $X_i$, then the collection of all subsets of $X$ of the form $\pi_i^{-1}(U_i)$ with $U_i$ in $\cB_i$ is a basis for the topology of $X$.\footnote{This last sentence is not part of \cite[Lemma 1.1.1]{Field}, but its claim follows immediately from the original statement.}
\end{lemma}

A topological space $X$ is T$_1$ if for every pair of distinct points, each has an open neighbourhood avoiding the other. Equivalently, $X$ is T$_1$ if and only if every finite subset of $X$ is closed. A topological space $X$ is T${}_2$ if it is Hausdorff. Note that we use normal font here, whereas the binary tree $T_2$ uses italic.\index{T\textsubscript{1}, T\textsubscript{2}}
\begin{lemma}[{\cite[Corollary 1.1.6]{Field}}]\label{InvLimSurjEmbOnto}
Let $T$ be a compact space, $\{X_i,\varphi_{ji},I\}$ an inverse system of T$_1$ topological spaces, and $\sigma_i\colon T\to X_i$ a compatible system of continuous surjective maps. Let $\sigma\colon T\to \invLim X_i$ map each $x$ to $(\sigma_i(x)\,|\,i\in I)$. Then $\sigma$ is a continuous surjection.\footnote{In \cite{Field} the $X_i$ are required to be Hausdorff (T$_2$), but T$_1$ suffices since the proof only uses that singleton subsets of the $X_i$ are closed. Furthermore, \cite{Field} does not state that $\sigma$ is continuous.}
\end{lemma}
\begin{proof}
Since \cite{Field} only states that $\sigma$ is surjective, we quickly show that it is also continuous. For this, by Lemma~\ref{invLimBasis} consider any basic open set $\pi_i^{-1}(U_i)$ of $\invLim X_i$ where $U_i$ is open in $X_i$. Then $\pi_i\circ\sigma=\sigma_i$ implies that $\sigma^{-1}(\pi_i^{-1}(U_i))=\sigma_i^{-1}(U_i)$ which is open in $T$ since $\sigma_i$ is continuous.
\end{proof}

\begin{lemma}[{\cite[Corollary 1.1.5]{Field}}]\label{InvLimCompSurjOnto}
Let $\{X_i,\varphi_{ji},I\}$ and $\{X_i',\varphi_{ji}',I\}$ be inverse systems of compact Hausdorff spaces. Let $\sigma_i\colon X_i\to X_i'$ be a compatible system of continuous surjections. Then the map
\begin{align*}
\invLim X_i&\to\invLim X_i'\\
(x_i\,|\,i\in I)&\mapsto (\sigma_i(x_i)\,|\,i\in I)
\end{align*}
is a continuous surjection.
\end{lemma}

A subset $J$ of $I$ is \textit{cofinal} in $I$ if for every $i\in I$ there is some $j\in J$ with $j\ge i$.\index{cofinal}
If $X$ and $Y$ are two topological spaces, we write $X\simeq Y$ to say that $X$ and $Y$ are homeomorphic.\index{$\simeq$}

\begin{lemma}[{\cite[Lemma 8.7.3 of the 5th edition]{Bible}}]\label{BibleInvLimCofinal}
Let $\{X_i,\varphi_{ji},I\}$ be an inverse system of compact spaces, and let $J\subseteq I$ be cofinal in $I$. Then $\{X_i,\varphi_{ji},J\}$ satisfies $\invLim (X_i\,|\,i\in I)\simeq\invLim (X_i\,|\,i\in J)$ with the homeomorphism that maps every point $(x_i\,|\,i\in I)$ to its restriction $(x_i\,|\,i\in J)$.
\end{lemma}

A function $f\colon X\to Y$ is \textit{monotone} if $f^{-1}(y)$ is connected for every $y\in f[X]$.\index{monotone} Assume that $(X_n\,|\,n\in\N)$ is a family of topological spaces, and furthermore suppose that for every $n\in\N-\{0\}$ we have a continuous map $f_n\colon X_n\to X_{n-1}$. Then the family of the $X_n$ together with the family of all $f_n$ forms an \textit{inverse sequence}\index{inverse sequence}, denoted $\{X_n,f_n,\N\}$. Clearly, every such inverse sequence gives rise to an inverse system $\{X_n,f_{m,n},\N\}$ where $f_{m,n}=f_{n+1}\circ\cdots \circ f_m$ for $m>n$ and $f_{n,n}=\id_{X_n}$. Hence, given an inverse sequence $\{X_n,f_n,\N\}$, we write $\invLim X_n$ for the inverse limit of the inverse system it induces.

\begin{theorem}[Capel, {\cite[Theorem 4.11]{Capel}, \cite[Theorem 200]{invLim}}]\label{Capel}
If $\{A_n,f_n,\N\}$ is an inverse sequence such that $A_n$ is an arc and $f_n$ is monotone and surjective for every $n$ then $\invLim A_n$ is an arc.
\end{theorem}

\subsection{`Ends and tangles' plus further notation}\label{SummaryEndsAndTangles}

This section not only serves as a summary of `Ends and tangles' (\cite{EndsAndTangles}), but also as an introduction of its basic definitions and notation, some of which we modified to meet our needs. At the end of this section, we introduce some additional notation, and we remind of a useful construction from the proof of \cite[Theorem 2.2]{Ends}. But first, we start with the promised summary:

A \textit{separation}\index{separation} of a graph $G$ is a set $\{A,B\}$ with $A$ and $B$ subsets of $V(G)$ such that $V(G)=A\cup B$ and $G$ has no edge from $A\setminus B$ to $B\setminus A$. 
Clearly, every separation $\{A,B\}$ induces a bipartition of the set of components of $G-A\cap B$, and vice versa.
The cardinal $|A\cap B|$ is the \textit{order}\index{order} of the separation $\{A,B\}$.
Every separation $\{A,B\}$ has two \textit{orientations}\index{orientation}, namely
the ordered pairs $(A,B)$ and $(B,A)$, which we also refer to as \textit{oriented separations}\index{oriented separation}.\footnote{Usually we refer to an oriented separation simply as `separation', relying upon context instead.} 
Informally, we think of $A$ and $B$ as the \textit{small side}\index{small side} and the \textit{big side}\index{big side} of $(A,B)$, respectively.
The \textit{order} of an oriented separation $(A,B)$ simply is the order of $\{A,B\}$, namely $|A\cap B|$.
If $S$ is a set of separations of the graph, then we denote by $\vec{S}$ the collection of the orientations of its elements.\index{$\vec{S}$}
We shall call $(B,A)$ the \textit{inverse}\index{inverse} of $(A,B)$ and vice versa. For the sake of readability we use the more intuitive 'arrow notation' known from vector spaces: when referring to an element of $\vec{S}$ as $\vec{s}$\index{$\vec{s}$, $\cev{s}$} (or $\cev{s}$), we denote its inverse by $\cev{s}$ (or $\vec{s}$).
Then the map $\vec{s}\mapsto(\vec{s})^*:=\cev{s}$ is an involution on $\vec{S}$. 
We define a partial ordering $\le$ on $\vec{S}$ by letting
\begin{align*}
(A,B)\le (C,D)\,:\Leftrightarrow\,A\subseteq C\text{ and }B\supseteq D.
\end{align*}
Note that our involution reverses this partial ordering, i.e. for $\vec{r},\vec{s}\in\vec{S}$ we have
\begin{align*}
\vec{r}\le\vec{s}\,\Leftrightarrow\,\cev{r}\ge\cev{s}.
\end{align*}
The triple $(\vec{S},\le,{}^*)$ is known as a \textit{separation system}\index{separation system}.

An \textit{orientation}\index{orientation} $O$ of $S$ is a subset of $\vec{S}$ with $|\{\vec{s},\cev{s}\}\cap O|=1$ for every $\vec{s}\in\vec{S}$. 
If no two distinct $\vec{r},\vec{s}\in O$ satisfy $\cev{r}<\vec{s}$ then we say that $O$ is consistent.
We say that an orientation $O$ of $S$ \textit{avoids} some $\cF\subseteq 2^{\vec{S}}$ if $2^O$ and $\cF$ intersect emptily. 
A non-empty set $\sigma\subseteq\vec{S}$ is a said to be a \textit{star}\index{star} if $\vec{r}\le\cev{s}$ holds for all distinct $\vec{r},\vec{s}\in\sigma$.
The \textit{interior} of a star $\sigma=\{(A_i,B_i)\,|\,i\in I\}$ is the set $\bigcap_{i\in I}B_i$.
In the context of a given graph $G$, the set $[V(G)]^{<\aleph_0}$ will be denoted by $\cX=\cX(G)$, and $S=S(G)$ will denote the set of all separations of $G$ of finite order. \index{$\cX$}\index{$S$}
Furthermore, $\cS=\cS(G)$ will denote the set of all stars in $\vec{S}$.
For the rest of this chapter, we let $G$ be a fixed infinite graph.
By $\Tau_{<\aleph_0}$ we denote the set of all finite stars $\sigma\subseteq\vec{S}$ of finite interior, and by $\Tau$ we denote the set of all stars $\sigma\subseteq\vec{S}$ of finite interior. Outside this section, $\Tau$ will be used to denote topological spanning trees.

For every $\cF\subseteq\cS$ we say that an $\cF$-\textit{tangle}\index{$\cF$-tangle} of $G$ is a consistent orientation of $S$ avoiding $\cF$.
Moreover, an $\aleph_0$-\textit{tangle}\index{$\aleph_0$-tangle} of $G$ is said to be a $\Tau_{<\aleph_0}$-tangle of $G$, and we write $\Theta=\Theta(G)$\index{$\Theta$} for the collection of all $\aleph_0$-tangles of $G$.\footnote{As Diestel showed in \cite{EndsAndTangles}, this definition is equivalent to the definition known from Robertson \& Seymour.}
If $\omega$ is an end of $G$, then by \cite[Corollary 1.7]{EndsAndTangles} letting
\begin{align*}
\tau_\omega:=\{(A,B)\,|\,C(A\cap B,\omega)\subseteq B\}
\end{align*}
defines a bijection $\omega\mapsto\tau_\omega$ from the ends of $G$ to the $\Tau$-tangles of $G$. Therefore, we call these $\aleph_0$-tangles the \textit{end tangles}\index{end tangle} of $G$. By abuse of notation, we will write $\Omega=\Omega(G)$\index{$\Omega$} for the collection of all end tangles of $G$.
The elements of $\Theta\setminus\Omega$ we call the \textit{ultrafilter tangles}\index{ultrafilter tangle} of $G$, and we write $\Upsilon=\Upsilon(G)$\index{$\Upsilon$} for the collection of all these.\footnote{This definition differs from the one given by Diestel, but both turn out to be equivalent due to \cite[Theorem 2]{EndsAndTangles}.}
In particular, we have $\Theta=\Omega\uplus\Upsilon$.

For every $X\in\cX$ we denote by $\cC_X$\index{$\cC_X$} be the set of all components of $G-X$, and $\cU_X$\index{$\cU_X$} is the set of all ultrafilters on $\cC_X$. For every $X\subseteq X'\in\cX$ we define the map $\phi_{X',X}\colon \cC_{X'}\to\cC_X$\index{$\phi_{X',X}$} be letting it send every component of $G-X'$ to the unique component of $G-X$ including it, i.e. such that $\phi_{X',X}(C')\subseteq C$. For subsets $\cC\subseteq\cC_{X'}$ we write $\cC'\rest X:=\phi_{X',X}[\cC']$,\index{$\cC'\rest X$} and for ultrafilters $U'\in\cU_{X'}$ we write
\begin{align*}
U'\rest X&:=\big\langle\,\big\{\cC'\rest X\,\big|\,\cC'\in U'\big\}\,\big\rangle_{\cC_X}\\
&=\{\cC\subseteq\cC_X\,|\,\exists\,\cC'\in U':\cC\supseteq\cC'\rest X\}
\end{align*}
where $\langle \cA\rangle_B$ for two sets $\cA,B$ with $\cA\subseteq 2^B$ denotes  the collection of all supersets $B'\subseteq B$ of elements of $\cA$, the \textit{set-theoretic up-closure of} $\cA$ \textit{in} $2^B$.
Due to \cite[Lemma 2.1]{EndsAndTangles}, letting $f_{X',X}\colon \cU_{X'}\to\cU_X$\index{$f_{X',X}$} send each $U'\in\cU_{X'}$ to $U'\rest X$ for all $X\subseteq X'\in\cX$ yields an inverse system $\{\cU_X,f_{X',X},\cX\}$\index{$\{\cU_X,f_{X',X},\cX\}$} whose inverse limit we denote by $\cU$.\index{$\cU$} 
Hence taking the up-closure in the definition of $U'\rest X$ ensures that $U'\rest X$ is an ultrafilter on $\cC_X$, even if there is some finite component of $G-X$ whose vertex set is included in $X'$.

Next, for every $\tau\in\Theta$ and $X\in\cX$ we let\index{$U(\upsilon,X)=U(\tau,X)$}
\begin{align*}
U(\tau,X):=\{\cC\subseteq\cC_X\,|\,(V\setminus V[\cC],X\cup V[\cC])\in\tau\}
\end{align*}
where $V[\cC]:=\bigcup_{C\in\cC}V(C)$.\index{$V[\cC]$}
Then by \cite[Lemma 2.3]{EndsAndTangles} the map 
\begin{align*}
\tau\mapsto (U(\tau,X)\,|\,X\in\cX)=:\upsilon_\tau
\end{align*}
is a bijection from $\Theta$ to $\cU$. Furthermore, the ends of $G$ are precisely those of its $\aleph_0$-tangles which this map sends to a family of principal ultrafilters. Now let $\cU_X^\ast$\index{$\cU_X^\ast$} be the set of all non-principal elements of $\cU_X$. For every $X\subseteq X'\in\cX$ we define a map $g_{X,X'}\colon \cU_X^\ast\to\cU_{X'}^\ast$\index{$g_{X,X'}$} by letting
\begin{align*}
g_{X,X'}(U):=\{\cC\subseteq\cC_{X'}\,|\,\exists\cD\in U:\cD\subseteq\cC\}
\end{align*}
for every $U\in\cU_X^*$. By \cite[Lemma 3.1]{EndsAndTangles} this map is well-defined, and it sends each $U\in\cU_X^*$ to the unique $U'\in\cU_{X'}^\ast$ with $f_{X',X}(U')=U$. In particular, we have
\begin{align*}
f_{X',X}\circ g_{X,X'}=\id_{\cU_X^\ast}
\end{align*}
by \cite[Lemma 3.2]{EndsAndTangles}. Combined, these Lemmas yield

\begin{lemma}\label{TCfInv}
For all $X\subseteq X'\in\cX$ the map $f_{X',X}$ restricts to a 
bijection between $f_{X',X}^{-1}(\cU_X^\ast)\subseteq\cU_{X'}^\ast$ and $\cU_X^\ast$ with inverse $g_{X,X'}$.\qed
\end{lemma}
In particular, we have
\begin{corollary}\label{Uextension}
For every $X\in\cX$ each non-principal $U\in\cU_X$ uniquely extends to an element of $\cU$.\qed
\end{corollary}
For every $\tau\in\Upsilon$ we set\index{$\cX_\tau$, $X_\tau$}
\begin{align*}
\cX_\tau=\{X\in\cX\,|\,U(\tau,X)\in\cU_X^\ast\},
\end{align*}
and every element of $\cX_\tau$ is said to \textit{witness} that $\tau$ is an ultrafilter tangle. By \cite[Lemma 3.3 \& 3.4]{EndsAndTangles}, for every $\tau\in\Upsilon$ and $X\in\cX_\tau$ the non-principal ultrafilter $U(\tau,X)$ uniquely determines $\tau$ in that
\begin{align*}
\tau=\{(A,B)\in\vec{S}\,|\,\exists\,\cC\in U(\tau,X):V[\cC]\subseteq B\}.
\end{align*}
According to \cite[Theorem 3.6]{EndsAndTangles}, for each $\tau\in\Upsilon$ the set $\cX_\tau$ has a unique least element $X_\tau$ with $\cX_\tau=\lfloor X_\tau\rfloor_{\cX}=\{X\in\cX\,|\,X_\tau\subseteq X\}$. Furthermore, \cite[Lemma 3.7]{EndsAndTangles} states that, if $X\in\cX$ and $U\in\cU_X$ is not generated by $\{C\}$ for any finite $C\in\cC_X$, then there is some $\tau\in\Theta$ such that $U=U(\tau,X)$.

Finally, we use $\cU$ to compactify $G$. For this, we equip the $\cU_X$ with the Stone topology, i.e. we equip $\cU_X$ with the topology generated by declaring as basic open for every $\cC\subseteq\cC_X$ the set\footnote{In \cite{EndsAndTangles}, the notion for these sets is simply $\cO(\cC)$. Since in this work we will compare the tangle compactification with other spaces, we slightly modify a lot of the topological notation from \cite{EndsAndTangles}.}\index{$\cO_{\cU_X}(\cC)$, $\cO_{\cU}(X,\cC)$, $\cO_{\TC}(X,\cC)$}
\begin{align*}
\cO_{\cU_X}(\cC):=\{U\in\cU_X\,|\,\cC\in U\}.
\end{align*}
Then by \cite[Lemma 4.1]{EndsAndTangles}, the bonding maps $f_{X',X}$ are continuous, so \cite[Proposition 4.2]{EndsAndTangles} tells us that $\cU$ is compact, Hausdorff and totally disconnected.\footnote{We added `Hausdorff' here.}
Next, for every $X\in\cX$ let $\pi_X$ be the restriction of the $X$th projection map $\text{pr}_X\colon \prod_{Y\in\cX}\cU_Y\to\cU_X$, and for every $X\in\cX$ and $\cC\subseteq\cC_X$ write
\begin{align*}
\cO_{\cU}(X,\cC)=\pi_X^{-1}(\cO_{\cU_X}(\cC)).
\end{align*}
Then by \cite[Lemma 4.4]{EndsAndTangles}, the collection
\begin{align*}
\{\cO_{\cU}(X,\cC)\,|\,X\in\cX,\cC\subseteq\cC_X\}
\end{align*}
is a basis for the topology of $\cU$.
Now, we extend the 1-complex of $G$ to a topological space $\TC:=G\cup\cU$ by declaring as open, for all $X\in\cX$ and $\cC\subseteq\cC_X$, the sets
\begin{align*}
\cO_{\TC}(X,\cC):=\medcup\cC\cup\mathring{E}(X,\medcup\cC)\cup\cO_{\cU}(X,\cC),
\end{align*}
and endowing $\TC$ with the topology this generates. Then we arrive at the main result of \cite{EndsAndTangles}:
\newpage
\begin{theorem}[{\cite[Theorem 1]{EndsAndTangles}}]
Let $G$ be any graph.
\begin{enumerate}
\item $\TC$ is a compact space in which $G$ is dense and $\TC\setminus G$ is totally disconnected.
\item If $G$ is locally finite and connected, then all its $\aleph_0$-tangles are ends, and $\TC$ coincides with the Freudenthal compactification of $G$.
\end{enumerate}
\end{theorem}
The following is extracted from the proof of \cite[Theorem 1]{EndsAndTangles}, and it will be reproved in a more general context in the proof of Theorem~\ref{alphaGcomp}:
\begin{lemma}[{\cite[Proof of Theorem 1]{EndsAndTangles}}]\label{obviousLemma}
For all $X\subseteq X'\in\cX$ and every $\cC'\subseteq\cC_{X'}$ we have $\cO_{\TC}(X',\cC')\subseteq\cO_{\TC}(X,\cC'\rest X)$.\qed
\end{lemma}

Even though $\TC$ in general is not Hausdorff, Diestel remarks that there exist two workarounds: First, the space $\TC\setminus\mathring{E}$ is a \HDcomp{} of $V(G)$ which still reflects the structure of $G$. Second, we can exchange `compact' for `Hausdorff' by modifying the topology of $\TC$ similarly to the way we would obtain \textsc{MTop} from \textsc{VTop} (see~\ref{subsec:topsOverview} for definitions of \textsc{MTop} and \textsc{VTop}).

From now on, we will write $\Theta=\cU$ and $\tau=\upsilon_\tau$ as well as $\omega=\tau_\omega$ by abuse of notation. 
If $\omega$ is an end of $G$, we write $\CTC(X,\omega)$\index{$\CTC(X,\omega)$} for the set $\cO_{\TC}(X,\{C(X,\omega)\}$, and we write $\Delta(\omega)$\index{$\Delta(\omega)$} for the set of vertices of $G$ dominating $\omega$. 
Furthermore, we write $V(\Omega)=\bigcup_{\omega\in\Omega}\Delta(\omega)$\index{$V(\Omega)$} and $V(\Upsilon)=\bigcup_{\upsilon\in\Upsilon}X_\upsilon$\index{$V(\Upsilon)$}, as well as $V(\cU)$\index{$V(\cU)$} for the union $V(\Omega)\cup V(\Upsilon)$. For every $\upsilon=(U_X\,|\,X\in\cX)\in\Upsilon$ we write $U^\circ(\upsilon)=U_{X_\upsilon}$\index{$U^\circ(\upsilon)$} and $U(\upsilon,X)=U_X$.\index{$U(\upsilon,X)$} 
On $\Upsilon$ we define the equivalence relation $\asymp$\index{$\asymp$} by letting $\upsilon\asymp \upsilon'$ whenever $X_\upsilon=X_{\upsilon'}$ holds. 

A map $f$ with domain $\cX$ is a \textit{direction}\index{direction} of $G$ if $f$ maps every $X\in\cX$ to a component of $G-X$ and $f(X)\supseteq f(X')$ whenever $X\subseteq X'\in\cX$. Clearly, the directions of $G$ are precisely the elements of the inverse limit of $\{\cC_X,\phi_{X',X},\cX\}$, and we have seen above that these are precisely the ends of $G$.
However, the constructive proof of the original Theorem from Diestel and Kühn \cite{Ends} linking ends to directions yields much more:

\begin{lemma}[$\L$]\label{ctblNbhdBase}
Let $G$ be an arbitrary infinite graph and $\omega$ an end of $G$ which is dominated by at most finitely many vertices.
Then there exists a sequence $X_0,X_1,\hdots$ of non-empty finite sets of vertices of $C(\Delta(\omega),\omega)$ such that for all $n\in\N$ the component $C(X_n\cup \Delta(\omega),\omega)$ includes both $X_{n+1}$ and $C(X_{n+1}\cup\Delta(\omega),\omega)$. In particular, the collection of all $\hat{C}_{\TC}(X_n\cup\Delta(\omega),\omega)$ forms a countable neighbourhood basis of $\omega$ in $\TC$.
\end{lemma}
\begin{proof}
Let an $\omega$ be an arbitrary end of $G$ which is dominated by at most finitely many vertices, and let $H:=C(\Delta(\omega),\omega)$.
We now copy the main part of the proof of \cite[Theorem 2.2]{Ends} for the sake of completeness: 
Denote by $\cY$ the collection of all finite sets of vertices of $H$.
For every $X\in\cY$ we write
\begin{align*}
\hat{X}:=C_{H}(X,\omega)\cup N(C_{H}(X,\omega))
\end{align*}
where $C_{H}(X,\omega):=C(X\cup\Delta(\omega),\omega)$ is the component of $H-X$ in which every ray of $\omega$ has a tail.
Starting with an arbitrary non-empty $X_0\in\cY$ we will construct a sequence $X_1,X_2,\hdots$ of non-empty elements of $\cY$ such that for all $n\in\N$ the component $C_{H}(X_n,\omega)$ includes both $X_{n+1}$ and $C_{H}(X_{n+1},\omega)$. 

Therefore, we proceed inductively, as follows: Suppose that $X_n$ has been constructed. Since $\omega$ is not dominated by a vertex of $H$, we find for every $x\in X_n$ some $X_x\in\cY$ with $x\notin \hat{X}_x$. 
Set $X=\bigcup_{x\in X_n}X_x$ and let $X_{n+1}$ be the neighbourhood of $C_{H}(X,\omega)$ in $H$. Then $X_{n+1}$ is finite due to $X_{n+1}\subseteq X\in\cY$. By the choice of $X_{n+1}$ we have $C_{H}(X_{n+1},\omega)=C_{H}(X,\omega)$.
For all $x\in X_n$, together with $\hat{X}\subseteq\hat{X}_x\not\ni x$ this yields $x\notin\hat{X}=\hat{X}_{n+1}\supseteq X_{n+1}$. Hence $H[\hat{X}_{n+1}]$ is connected and avoids $X_n$, so it is included in $C_{H}(X_n,\omega)$. This completes the construction.

Since all the $X_n$ are disjoint, the descending sequence $\hat{X}_0\supseteq\hat{X}_1\supseteq\cdots$ has empty overall intersection: every vertex in $\hat{X}_n$ has distance at least $n$ from $X_0$ (because every $X_0$-$\hat{X}_n$ path meets all the disjoint sets $X_1,\hdots,X_{n-1}$), so no vertex can lie in $\hat{X}_n$ for every $n$. We now leave the proof of \cite[Theorem 2.2]{Ends}.

Now we for every $n\in\N$ we let $Y_n:=X_n\cup\Delta(\omega)$.
Then the collection of all open sets $\CTC(Y_n,\omega)$ forms a countable neighbourhood base of $\omega$: 
Indeed, let $\cO_{\TC}(Y,\cC)$ be a basic open neighbourhood of $\omega$ in $\TC$. 
Without loss of generality we may suppose that $\Delta(\omega)\subseteq Y$ and write $Y':=Y\setminus\Delta(\omega)$.
Since $\bigcap_{n\in\N}\hat{X}_n$ is empty, there is some $N\in\N$ such that $\hat{X}_N$ avoids $Y'$. 
Hence $H[\hat{X}_N]$ is a connected subgraph of $G-Y$, so in particular it is included in $C(Y,\omega)\in\cC$. 
Therefore $\CTC(Y_N,\omega)\subseteq\cO_{\TC}(Y,\cC)$ holds as desired.
\end{proof}

\subsection{General Topology}\label{subsec:GeneralTopology}

If $X$ is a topological space, then a homeomorphic image of the unit interval in $X$ is an \textit{arc} in $X$.\index{arc} The following Lemma is immediate from the continuity of the quotient map:

\begin{lemma}\label{eqClassesClosed}
If $X$ is a topological space and $R$ is an equivalence relation on $X$ such that $X/R$ is {\normalfont T}$_1$, then every element of $X/R$ is a closed subset of $X$.\qed
\end{lemma}

\begin{lemma}[{\cite[Theorem 7.2 a)$\to$d)]{Willard}}]\label{ctsPreservesClosure}
If $X$ and $Y$ are topological spaces and $f\colon  X\to Y$ is continuous, then for each $A\subseteq X$ we have $f[\cl_X(A)]\subseteq\cl_Y(f[A])$.
\end{lemma}

\begin{lemma}[{\cite[Corollary 31.6]{Willard}}]\label{PathHDArc}
A Hausdorff topological space is path-connected if and only if it is arc-connected.
\end{lemma}

\begin{lemma}[{\cite[Corollary 13.14]{Willard}}]\label{HDdenseSubsetUnique}
If $f,g\colon X\to Y$ are continuous, $Y$ is Hausdorff, and $f$ and $g$ agree on a dense set $D$ in $X$, then $f=g$.
\end{lemma}

\begin{lemma}[The pasting Lemma, {\cite[Theorem 18.3]{munkres}}]\label{PastingLemma}
Let $X$ and $Y$ be topological spaces and $X=A\cup B$, where $A$ and $B$ are closed in $X$. Let $f\colon A\to Y$ and $g\colon B\to Y$ be continuous. If $f(x)=g(x)$ for every $x\in A\cap B$, then $f$ and $g$ combine to give a continuous function $h\colon X\to Y$, defined by setting $h(x)=f(x)$ if $x\in A$ and $h(x)=g(x)$ if $x\in B$.
\end{lemma}

The following lines on continuum theory are accumulated from \cite[Chapter 28]{Willard}:
A \textit{continuum}\index{continuum} is a compact, connected Hausdorff topological space.
\begin{theorem}[{\cite[Theorem 28.2]{Willard}}]\label{directContinua}
Let $\{K_i\,|\,i\in I\}$ be a collection of continua in a topological space $X$ directed by inclusion. Then $\bigcap_{i\in I} K_i$ is a continuum.
\end{theorem}
A continuum $K$ in a topological space $X$ is said to be \textit{irreducible}\index{irreducible} about a subset $A$ of $X$ if  $A\subseteq K$ and no proper subcontinuum of $K$ includes $A$. In case of $A=\{a,b\}$ we say that $K$ is \textit{irreducible between} $a$ \textit{and} $b$.
\begin{theorem}[{\cite[Theorem 28.4]{Willard}}]\label{irredSubcontinuumExists}
If $K$ is a continuum, then any subset $A$ of $K$ lies in a subcontinuum irreducible about $A$.
\end{theorem}
If $X$ is a connected T\textsubscript{1} topological space, then a \textit{cut point}\index{cut point} of $X$ is a point $p\in X$ such that $X-\{p\}$ is not connected. If $p$ is not a cut point of $X$, we call $p$ a \textit{noncut point}\index{noncut point} of $X$.
A \textit{cutting}\index{cutting} of $X$ is a set $\{p,O,O'\}$ where $p$ is a cut point of $X$ and $O$ together with $O'$ disconnects $X-\{p\}$ in that $O$ and $O'$ are disjoint non-empty open subsets of $X$ with $O\uplus O'=X-\{p\}$.
A cut point $p$ \textit{separates}\index{separates} $a$ \textit{from} $b$ if there is a cutting $\{p,O,O'\}$ of $X$ with $a\in O$ and $b\in O'$.
Given $a\neq b\in X$ we write $\aS_X(a,b)$\index{$\aS_X(a,b)$} for the set consisting of $a,b$ and all the points $p\in X$ which separate $a$ from $b$. The \textit{separation ordering}\index{separation ordering} on $\aS_X(a,b)$ is defined by letting $p_1\le p_2$ if and only if $p_1=p_2$ or $p_1$ separates $a$ from $p_2$. This is actually a partial ordering on $\aS_X(a,b)$.
\begin{theorem}[{\cite[Theorem 28.11]{Willard}}]\label{cutPointOrdLin}
The separation ordering on $\aS_X(a,b)$ is a linear ordering.
\end{theorem}
\begin{theorem}[{\cite[Theorem 28.12]{Willard}}]\label{sepOrderSubSame}
If $K$ is a continuum with exactly two noncut points $a$ and $b$, then $\aS_K(a,b)=K$ and the topology on $K$ is the order topology.
\end{theorem}

Next, we have a look at compactifications:
A \textit{compactification}\index{compactification} of a topological space $X$ is an ordered pair $(K,h)$ where $K$ is a compact topological space and $h$ is an embedding of $X$ as a dense subset of $K$. Sometimes we also refer to $K$ as a compactification of $X$ if the map $h$ is clearly understood.
In \cite[chapter 19]{Willard} we find the following definitions:
A \textit{\HDcomp{}}\index{\HDcomp{}} of a topological space $X$ is a compactification $(K,h)$ of $X$ with $K$ Hausdorff.
If $(K,h)$ and $(K',h')$ are compactifications of $X$ we write $(K,h)\le (K',h')$ whenever there exists a continuous mapping $f\colon K'\to K$ with $f\circ h'=h$, i.e. such the diagram
\begin{center}
\begin{tikzpicture}[pil/.style={
           ->,
           thin,
           shorten <=2pt,
           shorten >=2pt,}]
\node (1) {$X$};
\node (2) [right=of 1]{$K'$};
\node (3) [below=of 2]{$K$};
\path[-stealth]
(1) edge node [above] {$h'$} (2)
(1) edge node [below left] {$h$} (3)
(2) edge node [right] {$f$} (3);
\end{tikzpicture}
\end{center}
commutes.\footnote{The class of all compactifications of $X$ need not be a set, hence we do not speak of a partial ordering.}
We write $(K,h)<(K',h')$ whenever $(K,h)\le (K',h')$ holds while $(K',h')\le (K,h)$ fails.
If there exists a homeomorphism $f\colon K'\to K$ witnessing $(K,h)\le (K',h')$ we say that $(K,h)$ and $(K',h')$ are \textit{topologically equivalent}\index{topologically equivalent} (clearly, this is symmetric). 
This definition is stated differently in \cite{Willard}, but both turn out to be equivalent for \HDcomp{}s:

\begin{lemma}[{\cite[Lemma 19.7]{Willard}}]\label{TopEqHomeo}
Two \HDcomp{}s $(K,h)$ and $(K',h')$ of $X$ are topologically equivalent if and only if $(K,h)\le (K',h')\le (K,h)$ holds.\footnote{We adapted the statement of the original Lemma to our definition of `topologically equivalent'.}
\end{lemma}

\newpage
\begin{lemma}[{\cite[Lemma 19.8]{Willard}}]\label{CompactificationSur}
Suppose that $(K,h)$ and $(K',h)$ are two \HDcomp{}s
of $X$ with $(K,h)\le (K',h')$ witnessed by a mapping $f\colon K'\to K$. Then the following hold:
\begin{enumerate}
\item $f\rest h'[X]$ is a homeomorphism from $h'[X]$ to $h[X]$.
\item $f[K'-h'[X]]=K-h[X]$.
\end{enumerate}
\end{lemma}

\begin{lemma}\label{CompactificationComparableUnique}
If $(K,h)$ and $(K',h)$ are two \HDcomp{}s of $X$ and $f\colon K'\to K$ witnesses $(K,h)\le (K',h')$ then $f$ is unique.
\end{lemma}
\begin{proof}
Let $g\colon K'\to K$ be any witness of $(K,h)\le (K',h')$. We have to show $f=g$. By choice of $f$ and $g$, we have $f\circ h'=h$ and $g\circ h'=h$. In particular, both $f$ and $g$ agree on $h'[X]$. Since $h'[X]$ is dense in $K'$, Lemma~\ref{HDdenseSubsetUnique} yields $f=g$ as desired.
\end{proof}

Next, we have a look at two particular \HDcomp{}s of discrete topological spaces which are known as the one-point \HDcomp{} and the Stone-Čech \HDcomp{}. The following insights are accumulated from \cite[Chapter 3.5]{EngelkingBook}\footnote{In \cite{EngelkingBook} the definition of `compactification' seems to vary from our definition of `\HDcomp{}' at first sight since it does not mention `Hausdorff'. A closer look at the definition of `compact' in \cite{EngelkingBook} reveals that it is hidden there.}:

If $X$ is a discrete topological space and $\ast$ is a point that is not in $X$, then we can extend $X$ to a topological space $\omega X:=X\uplus\{\ast\}$\index{$\omega X$} by declaring as open, for every finite subset $A$ of $X$, the set $\omega X-A$. The pair of this space and the identity on $X$ is known as the one-point \HDcomp{}\index{one-point \HDcomp{}} of $X$.
\begin{lemma}\label{OnePointCompMin}
If $X$ is a discrete topological space and $(\omega X,\id_X)$ is its one-point \HDcomp{}, then $(\omega X,\id_X)$ is a least \HDcomp{} of $X$ in that every \HDcomp{} $(K,h)$ of $X$ satisfies $(K,h)\ge (\omega X,\id_X)$.
\end{lemma}

More generally, if $X$ is a topological space and $(K,h)$ is a \HDcomp{} of $X$ such that $K-X$ is a singleton, then $(K,h)$ is called the \textit{one-point \HDcomp{}} of $X$ (which is unique up to topological equivalence) and we write $\omega X=K$ as before.
In order to state an existential Theorem concerning one-point \HDcomp{}s we need the following definitions from~\cite[§29]{munkres}: A topological space $X$ is said to be \textit{locally compact at} $x$ if there is some compact subspace $C$ of $X$ that contains a neighbourhood of $x$. If $X$ is locally compact at each of its points, then $X$ is said to be \textit{locally compact}.\index{locally compact}
\begin{theorem}[{\cite[Theorem 29.1 and subsequent remarks]{munkres}}]\label{OnePtCptIFF}
A topological space $X$ has a one-point \HDcomp{} if and only if $X$ is locally compact and Hausdorff, but not compact.
\end{theorem}

If $X$ is a discrete topological space we let $\beta X$\index{$\beta X$} be the set of all ultrafilters on $X$ equipped with the topology whose basic open sets are those of the form $\{U\in\beta X\,|\,A\in U\}$, one for each $A\subseteq X$. Furthermore, we let $\iota\colon X\to\beta X$ map each $x$ to the non-principal ultrafilter on $X$ generated by $\{x\}$. Then $(\beta X,\iota)$ is known as the Stone-Čech \HDcomp{}\index{Stone-Čech \HDcomp{}} of $X$.
\begin{lemma}\label{StoneCech}
If $X$ is a discrete topological space and $(\beta X,\iota)$ is its Stone-Čech \HDcomp{}, then $(\beta X,\iota)$ is a greatest \HDcomp{} of $X$ in that every \HDcomp{} $(K,h)$ of $X$ satisfies $(K,h)\le (\beta X,\iota)$.
\end{lemma}

\begin{theorem}[\cite{Munster}]\label{minimalHDrelation}
If $X$ is a topological space, then the relation $R_X$ on $X$ given by
\begin{align*}
R_X=\bigcap\big\{R\subseteq X^2\,\big|\,R\text{ is an eq.-rel. on }X\text{ and }X/R\text{ is \normalfont{T}}{}_2\big\}
\end{align*}
is an equivalence relation and $X/R_X$ is the maximal Hausdorff quotient of $X$.\index{maximal Hausdorff quotient}
\end{theorem}
If $X$ is a topological space and $R_X\subseteq X^2$ is given by Theorem~\ref{minimalHDrelation}, then we write $\cH(X)$\index{$\cH(X)$} for the quotient space $X/R_X$. For more details on this topic, e.g. regarding uniqueness of $\cH(X)$ and an intuitive construction of $R_X$, we redirect the reader to \cite{Munster}.

If $X$ is a topological space, we write $\between_X$ for the relation on $X^2$ defined by letting $a\between_X b$ whenever there exist no disjoint open neighbourhoods of $a$ and $b$ in $X$.
Now fix a topological space $X$.
For every ordinal $\alpha$ we define an equivalence relation $r_X^\alpha$ on $X$ and a quotient space $h^\alpha(X)$ of $X$, as follows: Set $r_X^0=\text{diag}(X)$ and $h^0(X)=X/r_X^0$. For successors $\alpha+1$ we let $q_X^\alpha\colon X\to h^\alpha(X)$ be the quotient map and put
\begin{align*}
r_X^{\alpha+1}=\big\{(a,b)\in X^2\,\big|\,(q_X^\alpha(a),q_X^\alpha(b))\in\text{trcl}\big(\between_{h^\alpha(X)}\big)\big\}
\end{align*}
as well as $h^{\alpha+1}(X)=X/r_X^{\alpha+1}$. For limits $\lambda$ we take $r_X^\lambda=\bigcup_{\alpha<\lambda}r_X^\alpha$ and $h^\lambda(X)=X/r_X^\lambda$. Then
\begin{theorem}[{\cite[Lemma 4.11 \& Construction 4.12]{Munster}}]\label{minimalHDconstruction}
If $X$ is a topological space, then there exists a minimal ordinal $\alpha$ with $r_X^\alpha=r_X^{\alpha+1}$. Furthermore, $r_X^\alpha=R_X$ holds (where $R_X$ is as in Theorem~\ref{minimalHDrelation}), i.e. $h^\alpha(X)=\cH(X)$.
\end{theorem}

\begin{lemma}\label{UnitIntervalBlowup}
Let $J$ be a countable non-empty subset of $\I$ and let $\ell\colon J\to\R_{>0}$ be such that $\sum_{j\in J}\ell(j)=1$. Then the map
\begin{align*}
\phi\colon \I\to [0,2],\; i\mapsto i+\sum_{\substack{j\in J\\j<i}}\ell(j)
\end{align*}
has the following property:
For every $\lambda\in [0,2]\setminus\phi[\I\setminus J]$ there is some $j_\lambda\in J$ with $j_\lambda<\lambda$ and 
\begin{align*}
\lambda\in [\phi(j_\lambda),\phi(j_\lambda)+\ell(j_\lambda)].
\end{align*}
\end{lemma}
\begin{proof}
For $\lambda\in\phi[J]$ this is clear, so suppose that $\lambda\in [0,2]\setminus\phi[\I]$.
Assume for a contradiction that there is no such $j_\lambda$ and put 
\begin{align*}
J^-&:=\{j\in J\,|\,\phi(j)<\lambda\}, &&\lambda^-:=\sup\,\{\phi(j)\,|\,j\in J^-\}, &&i^-:=\sup J^-\\
J^+&:=\{j\in J\,|\,\phi(j)>\lambda\}, &&\lambda^+:=\inf\,\{\phi(j)\,|\,j\in J^+\}, &&i^+:=\inf J^+
\end{align*}
Clearly we have $\lambda^-\le\lambda\le\lambda^+$. 

First, we show that $\phi(i^-)=\lambda^-$. If $i^-=\max J^-$ then this is clear, so suppose that $i^-\notin J^-$.
Hence
\begin{align*}
\phi(i^-)&=i^-+\sum_{\substack{j\in J\\j<i^-}}\ell(j)\\
&=\sup J^-+\sum_{j\in J^-}\ell(j)\\
&\overset{(\ast)}{=}\sup\Big\{j+\sum_{\substack{j'\in J\\j'< j}}\ell(j')\,\Big|\,j\in J^-\Big\}\\
&=\sup\,\{\phi(j)\,|\,j\in J^-\}\\
&=\lambda^-
\end{align*}
(at $(\ast)$ we used $i^-\notin J^-$).
This completes the proof of $\phi(i^-)=\lambda^-$.

Second, we show that $\phi(i^+)=\lambda^+$.
If $i^+=\min J^+$ then this is clear, so suppose that $i^+\notin J^+$. Hence
\begin{align*}
\phi(i^+)&=i^++\sum_{\substack{j\in J\\j<i^+}}\ell(j)\\
&=\inf J^++\sum_{\substack{j\in J\\j<i^+}}\ell(j)\\
&=\inf\Big\{j+\sum_{\substack{j'\in J\\j'<j}}\ell(j')\,\Big|\,j\in J^+\Big\}\\
&=\inf\,\{\phi(j)\,|\,j\in J^+\}\\
&=\lambda^+
\end{align*}
This completes the proof of $\phi(i^+)=\lambda^+$.

Therefore $\lambda\notin\phi[\I]$ together with $\lambda^-\le\lambda\le\lambda^+$ implies $\lambda^-<\lambda<\lambda^+$, and hence $i^-<i^+$. 
If $i^-\notin J^-$ then $\phi\rest [i^-,i^+]$ is linear with 
\begin{align*}
\phi(i^-)=\lambda^-<\lambda<\lambda^+=\phi(i^+)
\end{align*}
so in particular we find some $i\in (i^-,i^+)$ with $\phi(i)=\lambda$, contradicting our assumption that $\lambda\notin\phi[\I]$.
Otherwise if $i^-=\max J^-$, then $\hat{\phi}:=\phi\rest (i^-,i^+]$ is linear with 
\begin{align*}
\lim_{t\searrow i^-}\hat{\phi}(t)=\lambda^-+\ell(i^-)=\phi(i^-)+\ell(i^-)<\lambda
\end{align*}
(recall that we have $\phi(i^-)+\ell(i^-)<\lambda$ due to $\lambda\notin [\phi(i^-),\phi(i^-)+\ell(i^-)]$ and $\phi(i^-)=\lambda^-<\lambda$).
But then as before we find some $i\in (i^-,i^+)$ with $\phi(i)=\lambda$, contradicting our assumption that $\lambda\notin\phi[\I]$.
\end{proof}

\newpage
\subsection{Topologies on graphs: an overview}\label{subsec:topsOverview}

\paragraph*{The 1-complex of $G$.}
 Recall Section~\ref{Basics}.\index{1-complex of $G$}

\begin{enumerate}[leftmargin=*,labelindent=0pt,label=\textit{common restr.}:]
\item[\textit{common restr.}:] $G$ connected
\item[\textit{compact}:] if and only if $G$ is finite
\item[\textit{Hausdorff}:] always
\item[\textit{reference}:] \cite{EndsAndTangles}
\end{enumerate}

\paragraph*{$|G|$ aka \textsc{MTop}.}
The topological space $|G|$ is obtained by taking $\CG\cup\Omega$ as ground set and taking the topology generated by the following\index{$\vert G\vert$} basis:\index{\textsc{MTop}}
Inner edge points inhereit their basic open neighbourhoods from $(0,1)$. For every $u\in V$ and $\epsilon\in (0,1]$ we declare as open the set $\cO_G(u,\epsilon)$.
For every end $\omega$ of $G$, each $X\in\cX$ and all $\epsilon\in (0,1]$
we declare as open the set\index{$\hat{C}_\epsilon(X,\omega)$}
\begin{align*}
\hat{C}_\epsilon(X,\omega):=C(X,\omega)\cup\Omega(X,\omega)\cup\mathring{E}_\epsilon (X,C(X,\omega))^*,
\end{align*}
completing the definition of our basis.

For locally finite $G$ the space $|G|$ coincides with the Freudenthal compactification of $\CG$ (see \cite{Ends}), and it turned out to be the `right space' in that it allowed many fundamental theorems from finite graph theory to be generalised to locally finite graphs.
\begin{enumerate}[leftmargin=*,labelindent=0pt,label=\textit{common restr.}:]
\item[\textit{common restr.}:] $G$ connected and locally finite
\item[\textit{compact}:] if and only if $G$ is locally finite
\item[\textit{Hausdorff}:] always
\item[\textit{reference}:] \cite{Bible}, \cite{VTopComp}, \cite{Ends}
\end{enumerate}

\paragraph*{\textsc{VTop}.} The topology \textsc{VTop} is defined on $\CG\cup\Omega$ similarly to $|G|$, with one difference: If $\omega$ is an end of $G$ and $X\in\cX$, then we declare as basic open only the set $\hat{C}_1(X,\omega)$.\index{\textsc{VTop}}
\begin{enumerate}[leftmargin=*,labelindent=0pt,label=\textit{common restr.}:]
\item[\textit{common restr.}:] $G$ connected and locally finite
\item[\textit{compact}:] if and only if every $\cC_X$ is finite
\item[\textit{Hausdorff}:] if and only if every end of $G$ is undominated
\item[\textit{reference}:] \cite{Bible}, \cite{VTopComp}
\end{enumerate}

\paragraph*{\textsc{Top}.} The topology \textsc{Top} is defined on $\CG\cup\Omega$ as follows:\index{\textsc{Top}} For every end $\omega$ of $G$, every $X\in\cX$, and 
every choice of precisely one $j_e\in\mathring{e}$ for each edge $e\in E(X,C(X,\omega))$, we declare as open the set
\begin{align*}
C(X,\omega)\cup\Omega(X,\omega)\cup\bigcup\Big\{[y,j_e)\,\big|\,e=xy\in E(X,C(X,\omega)),y\in C(X,\omega)\Big\},
\end{align*}
and we let \textsc{Top} be the topology on $\CG\cup\Omega$ generated by these sets together with the open sets of (the 1-complex of) $G$. In particular, \textsc{Top} induces the 1-complex topology on $\CG$, which \textsc{MTop} and \textsc{VTop} do not as soon as $G$ is not locally finite.
\begin{enumerate}[leftmargin=*,labelindent=0pt,label=\textit{common restr.}:]
\item[\textit{common restr.}:] $G$ connected and locally finite
\item[\textit{compact}:] if and only if $G$ is locally finite
\item[\textit{Hausdorff}:] always
\item[\textit{reference}:] \cite{Bible}, \cite{VTopComp}
\end{enumerate}

\paragraph*{\textsc{ITop}.}
Suppose that no vertex of $G$ dominates two ends. Then \textsc{ITop} is the topology of the quotient space $\tilde{G}$\index{$\tilde{G}$} obtained from $\CG\cup\Omega$ equipped with \textsc{VTop} by identifying every end of $G$ with all of the vertices dominating it.\index{\textsc{ITop}}
\begin{enumerate}[leftmargin=*,labelindent=0pt,label=\textit{common restr.}:]
\item[\textit{common restr.}:] $G$ connected and finitely separable
\item[\textit{compact}:] if $G$ is 2-connected and finitely separable
\item[\textit{Hausdorff}:] always
\item[\textit{reference}:] \cite{duality}, \cite{TST}, \cite{DualTrees}
\end{enumerate}

\paragraph*{\textsc{ETop}.}
The topological space $\cE G$\index{$\cE G$} is constructed in two steps, as follows: First, let $\Omega'(G)$ denote the set of all edge-ends of $G$.\index{\textsc{ETop}} Let $\cE'G:=G\cup\Omega'$ (with $G$ viewed as 1-complex) be endowed with the topology generated by the following basis: Every inner\footnote{Recall that...this is meant wrt top of edge} edge point inherits its open neighbourhoods from $(0,1)$. Furthermore, for every finite set $F$ of edges of $G$, every component $C$ of $G-F$, and every choice of $\lambda_e\in\mathring{e}$ (one for each $e\in F$) we declare as open the set
\begin{align*}
C\cup\Omega'(F,C)\cup\bigcup_{\substack{e=xy\in F\\x\in C}}[x,\lambda_e),
\end{align*}
where $\Omega'(F,C)$ denotes the set of all edge-ends of $G$ living in $C$.
Then let $\cE G$ be obtained from $\cE'G$ by identifying every two points which have the same open neighbourhoods.

\begin{enumerate}[leftmargin=*,labelindent=0pt,label=\textit{common restr.}:]
\item[\textit{common restr.}:] $G$ connected
\item[\textit{compact}:] if $G$ is connected
\item[\textit{Hausdorff}:] always
\item[\textit{reference}:] \cite{ETop}, \cite{Schulz}
\end{enumerate}

\paragraph*{$\ell$-\textsc{Top}.}
The definition of $\ell$-\textsc{Top} takes a dozen lines and we will not need it, hence we redirect the reader to \cite{ETop} for details on this interesting space.\index{$\ell$-\textsc{Top}}
\begin{enumerate}[leftmargin=*,labelindent=0pt,label=\textit{common restr.}:]
\item[\textit{common restr.}:] $G$ connected and countable
\item[\textit{compact}:] if and only if $G$ is locally finite
\item[\textit{Hausdorff}:] always
\end{enumerate}

\paragraph*{The tangle compactification $\TC$ of $\CG$.} See Section~\ref{SummaryEndsAndTangles}.\index{$\TC$}
\begin{enumerate}[leftmargin=*,labelindent=0pt,label=\textit{common restr.}:]
\item[\textit{common restr.}:] ---
\item[\textit{compact}:] always
\item[\textit{Hausdorff}:] if and only if $G$ is locally finite (Corollary~\ref{TCHausdorffIff})
\item[\textit{reference}:] \cite{Bible}
\end{enumerate}

\paragraph*{A least tangle compactification $\FG$ of $\CG$.} See Section~\ref{FGconstruction}.\index{$\FG$}
\begin{enumerate}[leftmargin=*,labelindent=0pt,label=\textit{common restr.}:]
\item[\textit{common restr.}:] ---
\item[\textit{compact}:] always
\item[\textit{Hausdorff}:] if and only if $G$ is locally finite (Corollary~\ref{FGHausdorffIff})
\end{enumerate}

\paragraph*{$\cV G$, i.e. \textsc{VTop} for the 1-complex of $G$.}
We extend (the 1-complex of) $G$ to a topological space $\cV G=G\cup\Omega$ by declaring as open for every end $\omega$ of $G$ and every $X\in\cX$ the set $\hat{C}_1(X,\omega)$, and taking the topology on $\cV G$ that this generates.\index{$\cV G$} Clearly, $\cV G$ coincides with $\TC\setminus\Upsilon$.
\begin{obs}\label{VGcpt}
$\cV G$ is compact if and only if every $\cC_X$ is finite.
\end{obs}
\begin{proof}
If every $\cC_X$ is finite, then $G$ has not ultrafilter tangle since each ultrafilter tangle $\tau$ induces a non-principal on $\cC_{X_\tau}$ which is impossible.

If $\cV G$ is compact, then every $\cC_X$ must be finite: Otherwise there is some $Y\in\cX$ with $\cC_Y$ infinite. Covering the 1-complex of $G[Y]$ with basic open sets, and extending this cover by adding for each $C\in\cC_Y$ the open set
\begin{align*}
\mathring{E}(Y,C)\cup C\cup\Omega(Y,C)
\end{align*}
clearly yields a cover of $\cV G$ which has no finite subcover, which is impossible.
\end{proof}
\begin{enumerate}[leftmargin=*,labelindent=0pt,label=\textit{common restr.}:]
\item[\textit{common restr.}:] $G$ connected
\item[\textit{compact}:] if and only if every $\cC_X$ is finite (Obs.~\ref{VGcpt})
\item[\textit{Hausdorff}:] if and only if no end is dominated
\end{enumerate}

\paragraph*{The auxiliary space $\GH$.} See Section~\ref{AGsubsection}.\index{$\GH$}
\begin{enumerate}[leftmargin=*,labelindent=0pt,label=\textit{common restr.}:]
\item[\textit{common restr.}:] $G$ connected
\item[\textit{compact}:] if and only if $G$ is locally finite
\item[\textit{Hausdorff}:] always
\end{enumerate}

\paragraph*{The Hausdorff compactification $\LG$ of $\CG$.} See Section~\ref{LGconstruction}.\index{$\LG$}
\begin{enumerate}[leftmargin=*,labelindent=0pt,label=\textit{common restr.}:]
\item[\textit{common restr.}:] ---
\item[\textit{compact}:] always
\item[\textit{Hausdorff}:] always
\end{enumerate}

\newpage
\part{Main results}

\section{A closer look at the tangle compactification}\label{closerLook}

\subsection{First steps}

Among the most frequently used lemmas in the field of topological infinite graph theory is the so-called Jumping Arc Lemma: 
\begin{lemma}[{\cite[Lemma 8.5.3]{Bible}}]
Let $G$ be connected and locally finite, and let $F\subseteq E(G)$ be a cut with sides $V_1,V_2$.
\begin{enumerate}
\item If $F$ is finite, then $\overline{V_1}\cap\overline{V_2}=\emptyset$ (with the closures taken in $|G|$), and there is no arc in $|G|\setminus\mathring{F}$ with one endpoint in $V_1$ and the other in $V_2$.
\item If $F$ is infinite, then $\overline{V_1}\cap\overline{V_2}\neq\emptyset$, and there may be such an arc (for example, there exists such an arc if both graphs $G[V_1]$ and $G[V_2]$ are connected).
\end{enumerate}
\end{lemma}

The first statement of this lemma admits a straightforward generalisation to the tangle compactification:
\begin{lemma}\label{easyJAL}
Let $G$ be any graph, and let $F\subseteq E(G)$ be a finite cut with sides $V_1$ and $V_2$. Then 
\begin{align*}
\overline{G[V_1]}\uplus\overline{G[V_2]}=\TC\setminus\mathring{F}
\end{align*}
with the closures taken in the tangle compactification (in particular $\overline{V_1}\cap\overline{V_2}=\emptyset$), and no connected subset of $\TC\setminus\mathring{F}$ meets both $\overline{G[V_1]}$ and $\overline{G[V_2]}$.
\end{lemma}
\begin{proof}
Consider $X=V[F]\in\cX$ and pick a bipartition $\{\cC,\cC'\}$ of $\cC_X$ respecting $F$ in that $V[\cC]\subseteq V_1$ and $V[\cC']\subseteq V_2$ hold. 
If $\upsilon$ is a tangle of $G$, then $\upsilon$ is in one of $\cO_{\TC}(X,\cC)$ and $\cO_{\TC}(X,\cC')$, say in $\cO_{\TC}(X,\cC)$, and this neighbourhood witnesses $\upsilon\notin \overline{G[V_2]}$. Furthermore, we have $\upsilon\in\overline{G[V_1]}$: Let $O$ be any open neighbourhood of $\upsilon$ in $\TC$. Then $O\cap\cO_{\TC}(X,\cC)$ avoids $\overline{G[V_2]}\cup\mathring{F}$ and hence must meet $G[V_1]$, since otherwise $G$ is not dense in $\TC$ which is impossible.
\end{proof}

\begin{figure}[H]
\centering
\includegraphics[scale=1]{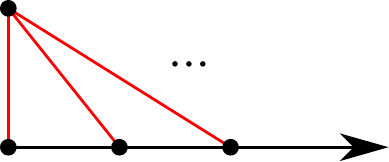}
\caption{A dominated ray whose set of red edges forms an infinite cut.}
    \label{fig:JALiiExample}
\end{figure}

But the second statement extends only partially: 
Indeed, consider the graph $G$ from Fig.~\ref{fig:JALiiExample}, and let $F\subseteq E(G)$ consist precisely of the red edges. Furthermore, let $V_1$ be the singleton of the vertex of infinite degree, and let $V_2$ be the vertex set of the black ray.
Then $F$ is an infinite cut with sides $V_1$ and $V_2$, but $\overline{V_1}$ and $\overline{V_2}$ intersect emptily since we have $\overline{V_1}=V_1$ and $\overline{V_2}=V_2\cup\{\omega\}$ where $\omega$ is the sole end of the graph $G$. 
In particular, $G[V_1]$ and $G[V_2]$ are connected graphs while there is no arc in $\TC\setminus\mathring{F}$ with one endpoint in $V_1$ and the other in $V_2$ since we also have
\begin{align*}
\overline{G[V_1]}\uplus\overline{G[V_2]}=\TC\setminus\mathring{F}.
\end{align*}
However, since the tangle compactification of connected locally finite graphs coincides with their Freudenthal compactification, the following partial generalisation of the second statement of the Jumping Arc Lemma is immediate:
\begin{lemma}
Let $G$ be any graph, and let $F\subseteq E(G)$ be an infinite cut with sides $V_1$ and $V_2$. 
If the graph $G$ is connected and locally finite, then the tangle compactification coincides with the Freudenthal compactification, and we have $\overline{V_1}\cap\overline{V_2}\neq\emptyset$.
Otherwise, $\overline{V_1}\cap\overline{V_2}=\emptyset$ is possible.
However, in case of $\overline{V_1}\cap\overline{V_2}\neq\emptyset$ there may be an arc in $\TC\setminus\mathring{F}$ with one endpoint in $V_1$ and the other in $V_2$.\qed
\end{lemma}

The following lemma shows that arcs in the tangle compactification do not take advantage of the ultrafilter tangles:

\begin{lemma}\label{TCarcsAvoidUltras}
If $G$ is any graph, then every arc in its tangle compactification avoids all ultrafilter tangles.
\end{lemma}
\begin{proof}
Let $A$ be any are in the tangle compactification of the graph $G$.
Pick a homeomorphism $\sigma\colon \I\bij A$ and without loss of generality suppose for a contradiction that $\sigma(1)=:\upsilon$ is in $\Upsilon$. 

First we show that $A$ meets $G$: If not, then in particular $\sigma(0)=:\upsilon'$ is in $\cU$. 
Choose $Y\in\cX$ with $U(\upsilon,Y)\neq U(\upsilon', Y)$ and pick $\cC\subseteq\cC_Y$ witnessing this, i.e. with $\cC\in U(\upsilon,Y)$ and $\cC_Y\setminus\cC\in U(\upsilon',Y)$. Then 
\begin{align*}
\big\{\cO_{\TC}(Y,\cC),\;\cO_{\TC}(Y,\cC_Y\setminus\cC)\big\}
\end{align*}
induces an open bipartition on $A$ which is impossible, so without loss of generality we may assume that $\sigma(0)$ is a point of $G$. Since $A$ is an arc, we may even assume that $\sigma(0)$ is a vertex $u$ of $G$.

Now let $Z:=X_\upsilon\cup\{u\}$. 
Then $A$ meets $\mathring{E}(Z,G-Z)$, since otherwise 
\begin{align*}
\Big\{\cO_{\TC}(Z,\cC_Z),\;\bigcup_{z\in Z}\cO_G(z,1)\Big\}
\end{align*}
induces an open bipartition on $A$ which is impossible.
Hence we may let $e_0$ be an edge in $E(Z,G-Z)$ that $A$ traverses and write $u_0$ for the endvertex of $e_0$ in $G-Z$. Pick $C_0\in\cC_Z$ with $u_0\in C_0$. Since $U(\upsilon,Z)$ is non-principal by choice of $Z$, we know that $\cC_Z\setminus\{C_0\}$ is an infinite element of $U(\upsilon,Z)$. Then $A$ meets $\mathring{E}(Z,G-Z)- \mathring{e}_0$ in some $\mathring{e}_1$, since otherwise 
\begin{align*}
\Big\{\cO_{\TC}\big(Z,\cC_Z\setminus\{C_0\}\big),\;\cO_{\TC}\big(Z,\{C_0\}\big)\cup\bigcup_{z\in Z}\cO_G(z,1)\Big\}
\end{align*}
induces an open bipartition on $A$ which is impossible. Proceeding inductively, we find infinitely many edges $e_0,e_1,\hdots$ in $E(Z,G-Z)$ which $A$ traverses. Since $Z$ is finite, by pigeon-hole principle we find some $t\in Z$ which is incident with infinitely many of the $e_n$. In particular, $t$ is a point of $A$ since $t$ lies in the closure of those $\mathring{e_n}$. But then $A\cap\cO_G(t,1)$ has degree at least 3 at $t$, a contradiction.
\end{proof}

To see that the tangle compactification in general is not sequentially compact, we consider the leaves of a $K_{1,\aleph_0}$ and apply the following lemma:
\begin{lemma}\label{VertexSeqToEnd}
Let $G$ be any graph, and let $\upsilon$ be an ultrafilter tangle. Then no sequence of vertices of $G$ converges to $\upsilon$ in $\TC$.
\end{lemma}
\begin{proof}
Assume for a contradiction that there is some sequence $(t_n)_{n\in\N}$ of vertices $t_n$ of $G$ with $t_n\to\upsilon$ in $\TC$ for $n\to\infty$. Without loss of generality we may assume that no $t_n$ is in $X_\upsilon$. 
Our sequence $(t_n)_{n\in\N}$ meets infinitely many components of $G-X_\upsilon$, since otherwise $t_n\not\to\upsilon$ is a contradiction.
Pick a subsequence (without loss of generality the whole sequence) such that $t_n$ and $t_m$ live in different components of $G-X_\upsilon$ for all $n\neq m$.
Furthermore, let $\cC$ be the set of all components $C$ of $G-X_\upsilon$ for which there exists some even $n\in\N$ with $t_n\in C$, and set $\cC'=\cC_{X_\upsilon}\setminus\cC$.

If $\cC\in U^\circ(\upsilon)$ holds, then there is no $N\in\N$ with $t_n\in\cO_{\TC}(X,\cC)$ for all $n\ge N$, so $t_n\not\to\upsilon$ is a contradiction. Otherwise $\cC'\in U^\circ(\upsilon)$ similarly results in $t_n\not\to\upsilon$ as desired.
\end{proof}

\subsection{Obstructions to Hausdorffness}

We define the relation $\between$ on $\TC$ by letting $x\between y$ whenever there are no disjoint open neighbourhoods of $x$ and $y$ in $\TC$, i.e. whenever $x$ and $y$ are not `Hausdorff topologically distinguishable'. Clearly, $\between$ in general is not transitive, hence we denote by $\transcl$ the transitive closure of $\between$, which is an equivalence relation.\index{$\between$, $\transcl$}
The relations $\between$ and $\transcl$ on $\TC\setminus\mathring{E}$ are actually included in the set $(V\times\cU)\cup (\cU\times V)$.

\begin{lemma}\label{Xtau}
Let $G$ be a any graph and $\upsilon\in\Upsilon$. Then $X_\upsilon$ contains precisely those vertices $u$ of $G$ with $u\between\upsilon$.
\end{lemma}
\begin{proof}
For the backward inclusion consider any vertex $u$ of $G$ with $u\between\upsilon$ and assume for a contradiction that $u$ is not in $X_{\upsilon}$. Put $X=X_\upsilon\uplus\{u\}$ and let $\cC$ be the set of those components of $G-X$ which contain neighbours of $u$.
Since $u\between\upsilon$ holds, we know that $\cC_X\setminus\cC$ is not in $U(\upsilon,X)$, and hence $\cC$ must be $U(\upsilon,X)$. But then $\cC\rest X_\upsilon\in U^\circ(\upsilon)$ is a singleton by the choice of $\cC$ and $u\notin X_\upsilon$, contradicting the fact that $U^\circ(\upsilon)$ is non-principal.

For the forward inclusion let any $u\in X_\upsilon$ be given and put $X=X_\upsilon\setminus\{u\}$. By minimality of $X_\upsilon$, the ultrafilter $U(\upsilon,X)$ is generated by $\{D\}$ for some component $D$ of $G-X$. 
In particular, $u$ is in $D$ since otherwise $U^\circ(\upsilon)$ would be principal which is impossible. Put $\cC=\phi_{X_\upsilon,X}^{-1}(\{D\})$, i.e. $\cC$ is the set of all components of $D-u$, and note that $\cC$ is in $U^\circ(\upsilon)$, since otherwise
\begin{align*}
\cC_X\setminus\{D\}=(\cC_{X_\upsilon}\setminus\cC)\rest X\in U(\upsilon,X)
\end{align*}
contradicts $\{D\}\in U(\upsilon,X)$. 
Since $U^\circ(\upsilon)$ is non-principal, we know that $\cC$ must be an infinite subset of $\cC_{X_\upsilon}$.
Now assume for a contradiction that $u\between\upsilon$ fails, witnessed by some basic open neighbourhoods $O$ of $u$ and $\cO(X',\cD)$ of $\upsilon$ in $\TC$, without loss of generality with $X_\upsilon\subseteq X'$. 
Let $\cC'$ consist of those $C\in\cC$ avoiding $X'$, and note that $\cC'$ is cofinite in $\cC$, so $\cC'\in U^\circ(\upsilon)$ holds as well as $\cC'\subseteq\cC_{X_\upsilon}\cap\cC_{X'}$.
Together with 
\begin{align*}
g_{X_\upsilon,X'}(U^\circ(\upsilon))=U(\upsilon,X')
\end{align*}
this yields $\cC'\in U(\upsilon,X')$. Now set $\cD'=\cC'\cap\cD$ which is in $U(\upsilon,X')$ and hence must be infinite. Since $\cD'$ is a subset of $\cC$, we know that $u$ sends at least one edge to each of the infinitely many $C\in\cD'\subseteq\cD$. In particular, $u$ does send at least one edge to $\bigcup\cD$, so $O$ and $\cO(X',\cD)$ must meet, a contradiction.
\end{proof}

\begin{corollary}\label{NbhdCompsInUX}
Let $\upsilon\in\Upsilon$ and $Y\subseteq X_\upsilon$ be given. Then for every $X\in\lfloor Y\rfloor_{\cX}$ the set of all components of $G-X$ which send an edge to every vertex in $Y$ is contained in $U(\upsilon,X)$.
\end{corollary}
\begin{proof}
If $Y$ is empty, then $\cC_X$ is the set of all components of $G-X$ which send an edge to every vertex in $Y=\emptyset$, and $\cC_X\in U(\upsilon,X)$ holds since $U(\upsilon, X)$ is an ultrafilter on $\cC_X$. Hence we may suppose that $Y$ is non-empty.

For every vertex $u$ of $Y$ we denote by $\cC_u$ the set of all components of $G-X$ which send an edge to $u$.
Then every set $\cC_u$ is contained in $U(\upsilon,X)$: Otherwise, for some $u\in Y$ the set $\cC_X\setminus\cC_u$ of all components of $G-X$ which do not send an edge to $u$ is contained in $U(\upsilon,X)$. 
But then $\cO_{\TC}(X,\cC_X\setminus\cC_u)$ is an open neighbourhood of $\upsilon$ in $\TC$ which avoids every basic open neighbourhood of $u$, contradicting Lemma~\ref{Xtau}. Hence $\cC_u$ is contained in $U(\upsilon,X)$ for every $u$ in $Y$. Since $Y\subseteq X_\upsilon$ is finite, the set $\bigcap_{u\in Y}\cC_u$ is also in $U(\upsilon,X)$.
\end{proof}

\begin{obs}
A graph $G$ is not planar as soon as there is some $\upsilon\in\Upsilon$ or $\omega\in\Omega$ with $|X_\upsilon|\ge 3$ or $|\Delta(\omega)|\ge 3$, respectively.
\end{obs}
\begin{proof}
If $\upsilon\in\Upsilon$ satisfies $|X_\upsilon|\ge 3$, then we let $\cC$ be the set of all components of $G-X_\upsilon$ whose neighbourhood is precisely $X_\upsilon$.
By Corollary~\ref{NbhdCompsInUX}, we have $\cC\in U^\circ(\upsilon)$, so $\cC$ must be infinite.
Consider the subgraph $H:=G[X_\upsilon\cup\bigcup\cC]-E(X_\upsilon)$ and obtain $H'$ from $H$ by contracting every element of $\cC$ to a singleton, deleting loops and reducing parallel edges. Then $H'$ is a $K_{|X_\upsilon|,|\cC|}$ and a minor of $G$,
so the statement follows from Kuratowski's Theorem (\cite[Theorem 4.4.6]{Bible}).

If $\omega\in\Omega$ is dominated by three distinct vertices, pick a ray in $\omega$ which avoids all three vertices, and find for each of the three vertices an infinite fan to that ray such that no two fans meet (this can be achieved by inductively constructing all three fans simultaneously, adding paths in turn).
Then it is easy to find a $TK_{3,\aleph_0}$ in the union of the ray and the three infinite fans. Again, the statement follows from Kuratowski's Theorem (\cite[Theorem 4.4.6]{Bible}).
\end{proof}

\begin{lemma}\label{PathSysXtau}
Let $x$ and $y$ be two distinct vertices of $G$. Then the following are equivalent:
\begin{enumerate}
\item There exist infinitely many independent $x$--$y$ paths in $G$.
\item There is some $\upsilon\in\cU$ satisfying $x\between\upsilon\between y$.
\end{enumerate}
\end{lemma}
\begin{proof}
(i)$\to$(ii). Let $\cP$ be some set of infinitely many independent $x$--$y$ paths in $G$ and discard from it the trivial $x$--$y$ path.
Write $\cP=\{P_i\,|\,i\in I\}$ and choose $\cF$ to be some non-principal ultrafilter on $I$. Furthermore, put $X=\{x,y\}$.
Given $Y\in\cX$ with $X\subseteq Y$, we define an ultrafilter $U_Y$ on $\cC_Y$, as follows: 

Let $I_Y$ denote the set of those $i\in I$ for which $\mathring{P}_i$ avoids $Y$. For every $\cC\subseteq\cC_Y$ we write 
\begin{align*}
I_Y(\cC)=\{i\in I_Y\,|\,\mathring{P}_i\text{ meets }\medcup\cC\}.
\end{align*}
Then for each bipartition $\{\cC,\cC'\}$ of $\cC_Y$, the set $\{I_Y(\cC),I_Y(\cC')\}$ is a bipartition of $I_Y$: indeed, for $i\in I_Y$ we have $i\in I_Y(\cC)$ if and only if there is some $C\in\cC$ containing $\mathring{P}_i$ (since $i\in I_Y$ implies that $\mathring{P}_i$ is connected in $G-Y$), so $\{I_Y(\cC),I_Y(\cC')\}$ is a bipartition of $I_Y$ as claimed. 
In order to define $U_Y$, we have to choose for every bipartition $\{\cC,\cC'\}$ of $\cC_Y$ precisely one of $\cC$ and $\cC'$, which we do now:
Since $Y$ meets only finitely many of the $\mathring{P}_i$ while $\cF$ is non-principal and $I$ can be written as
\begin{align*}
I=I_Y(\cC)\uplus I_Y(\cC')\uplus\{\{i\}\,|\,i\in I\setminus I_Y\},
\end{align*}
we know that $\cF$ picks exactly one of $I_Y(\cC)$ and $I_Y(\cC')$, say $I_Y(\cC)$. Then we let $U_Y$ choose $\cC$, and $\cC'$ otherwise. This completes the definition of $U_Y$, and clearly $U_Y$ inherits the filter properties from $\cF$. 

Our next aim is to show that the $U_Y$ are compatible with respect to the bonding maps of the inverse system of ultrafilters. For this, let $Y\subseteq Y'\in\cX$ both be supersets of $X$ and write $U$ for the ultrafilter $f_{Y',Y}(U_{Y'})$ on $\cC_Y$.
Assume for a contradiction that $U$ and $U_Y$ are distinct, and consider a bipartition $\{\cC,\cC_Y\setminus\cC\}$ of $\cC_Y$ witnessing this with $\cC_Y\setminus\cC\in U_Y$ and $\cC\in U$, say. By definition of $f_{Y',Y}$ we find some $\cC'\in U_{Y'}$ with $\cC'\rest Y\subseteq\cC$.
Therefore, $I_Y(\cC)\in \cF$ holds due to $I_{Y'}(\cC')\in\cF$ and
\begin{align*}
I_{Y'}(\cC')\subseteq I_Y(\cC'\rest Y)\subseteq I_Y(\cC),
\end{align*}
so $\cC$ is in $U_Y$ as well as $\cC_Y\setminus\cC$, a contradiction. 
Hence the $U_Y$ are compatible as desired.

Finally we extend the family of the $U_Y$ to an element $\upsilon$ of $\cU$ by letting $U_Z:=f_{X\cup Z,Z}(U_{X\cup Z})$ for every $Z\in\cX$ with $X\not\subseteq Z$. Then $x\between\upsilon$ holds: Otherwise there is some basic open neighbourhood $O$ of $x$ and some basic open neighbourhood $\cO(Y,\cC)$ of $\upsilon$ with $O\cap\cO(Y,\cC)=\emptyset$. Without loss of generality we may assume that $X$ is included in $Y$. Thus $\cC\in U_Y$ implies $I_Y(\cC)\in \cF$ by choice of $U_Y$, so in particular $I_Y(\cC)$ is infinite since $\cF$ is non-principal. Hence infinitely many of the $P_i$ meet $\bigcup\cC$, therefore witnessing that $x$ sends infinitely many edges to $\bigcup\cC$. In particular, $x$ does send an edge to $\bigcup\cC$, so $O$ and $\cO(Y,\cC)$ must meet, a contradiction. Similarly, $y\between \upsilon$ holds.

(ii)$\to$(i).
If $\upsilon$ is an end tangle, then put $X_0=\{x,y\}$ and pick some $N(x)$--$N(y)$ path $P_0$ in the sole component $C_0$ generating $U(\upsilon,X_0)$. Next, put $X_1=X_0\cup V(P_0)$ and pick some $N(x)$--$N(y)$ path $P_1$ in the sole component $C_1$ generating $U(\upsilon,X_1)$. Proceeding inductively, we find some disjoint $N(x)$--$N(y)$ paths $P_n$ for every $n\in\N$. Then $\{xP_ny\,|\,n\in\N\}$ is a set of infinitely many independent $x$--$y$ paths. Alternatively, it suffices to note that $\upsilon$ must be dominated by $x$ and $y$.

If $\upsilon$ is an ultrafilter tangle, then $\{x,y\}\subseteq X_\upsilon$ holds by Lemma~\ref{Xtau}. Let $\cC_x$ be the set of all components of $G-X_\upsilon$ sending an edge to $x$, and define $\cC_y$ analogously. 
Then $x\between\upsilon\between y$ implies that both $\cC_x$ and $\cC_y$ are in $U^\circ(\upsilon)$, and hence $\cC:=\cC_x\cap\cC_y\in U^\circ(\upsilon)$ is an infinite subset of $\cC_{X_\upsilon}$. Next we choose some infinite subset $\{C_n\,|\,n\in\N\}$ of $\cC$, and for every $n\in\N$ we pick some $N(x)$--$N(y)$ path $P_n$ in $C_n$. Again, $\{xP_ny\,|\,n\in\N\}$ is a set of infinitely many independent $x$--$y$ paths.
\end{proof}

\newpage
\begin{corollary}\label{ApproxOnV}
Let $x$ and $y$ be two distinct vertices of $G$. Then the following are equivalent:
\begin{enumerate}
\item $x\transcl y$.
\item There is some $X\in\cX$ containing $x$ and $y$ such that no $Y\in\cX$ disjoint from $X$ separates $x$ and $y$ in $G-E(X)$.
\end{enumerate}
\end{corollary}
\begin{proof}
(ii)$\to$(i). 
Let $X\in\cX$ be as in (ii).
Hence, we inductively find infinitely many $x$--$y$ paths in $G-E(X)$ which meet only in $X$. By pigeon-hole principle, infinitely many of them meet exactly the same vertices of $X$ and traverse these in the same order (starting in $x$), say $x=u_0,u_1,\hdots,u_n=y$ for some $n\in\N$.
We denote the set of these paths by $\cP=\{P_i\,|\,i\in I\}$ and write $P^k=u_{k-1}Pu_k$ for each $P\in\cP$ and $k\in [n]$. 
Applying, for every $k\in [n]$, Lemma~\ref{PathSysXtau} to $\{P_i^k\,|\,i\in I\}$ and the vertices $u_{k-1}$ and $u_k$ yields some $\upsilon_k\in\cU$ with $u_{k-1}\between\upsilon_k\between u_k$. Hence 
\begin{align}\label{xwy}
x=u_0\between\upsilon_1\between u_1\between\cdots\between\upsilon_n\between u_n=y
\end{align}
holds as desired.

(i)$\to$(ii).
Suppose that $x\transcl y$ holds, witnessed by some vertices $u_0,\hdots, u_n$ of $G$ and some $\upsilon_1,\hdots,\upsilon_n\in\cU$ satisfying~(\ref{xwy}). Applying, for every $k\in [n]$, Lemma~\ref{PathSysXtau} to $u_{k-1}\between\upsilon_k\between u_k$ yields some collection $\{P_m^k\,|\,m\in\N\}$ of infinitely many independent paths from $u_{k-1}$ to $u_k$. Set $X=\{u_k\,|\,k\le n\}$. Then for every $Y\in\cX$ disjoint from $X$ we find for every $k\in [n]$ some $m_k\in\N$ such that $P^k_{m_k}$ avoids $Y$. Then $\bigcup_{k\in [n]}P_{m_k}^k$ admits a path from $x$ to $y$ avoiding $Y$.
\end{proof}

\subsection{Critical vertex sets}\label{subsec:criticalVsets}

In this section we yield a combinatorial description of the sets $X_\tau$ of ultrafilter tangles.\\

For every $X\in\cX$ and $Y\subseteq X$ we write $\cC_X(Y)$ for the set of components of $G-X$ whose neighbourhood is precisely $Y$.\index{$\cC_X(Y)$} Furthermore, we let
\begin{align*}
\crit(X):=\{Y\subseteq X\,|\,\cC_X(Y)\text{ is infinite}\}.
\end{align*}
\begin{obs}\label{OneTimeUse0}
Suppose that $X\in\cX$ and $Y\in \crit(X)$ are given. Then every $X'\in\lfloor X\rfloor_{\cX}$ meets only finitely many elements of $\cC_X(Y)$, so the set $\cC_{X'}(Y)=\cC_X(Y)\cap\cC_{X'}$ is cofinite in $\cC_X(Y)$, and in particular $Y\in \crit(X')$ holds. On the other hand, $Y\in\crit(X)$ implies $Y\in\crit(Y)$: since every component in $C_X(Y)$ has neighbourhood precisely $Y$, we have $C_X(Y)\subseteq C_Y(Y)$.
\end{obs}
We call $X\in\cX$ \textit{critical} if $X\in\crit(X)$ holds, i.e. if $\cC_X(X)$ is infinite.\index{critical}
Moreover, we write $\crit(\cX)$ for the set of all critical $X\in\cX$.\index{$\crit(X)$, $\crit(\cX)$}
Then Observation~\ref{OneTimeUse0} implies
$\crit(\cX)=\bigcup_{X\in\cX}\crit(X)$. 

\begin{lemma}\label{New36}
For every $\upsilon\in\Upsilon$ and every $X\in\lfloor X_\upsilon\rfloor_{\cX}$ we have $X_\upsilon\in \crit(X)$ and $\cC_X(X_\upsilon)\in U(\upsilon,X)$.
\end{lemma}
\begin{proof}
Put $\cC=\cC_{X_\upsilon}(X_\upsilon)$. By Corollary~\ref{NbhdCompsInUX} (applied to $X_\upsilon=X=Y$) we know that $\cC$ is in $U^\circ(\upsilon)$. In particular, $\cC$ is infinite, witnessing $X_\upsilon\in \crit(X_\upsilon)$.
Hence $X_\upsilon\in \crit(X)$ holds by Observation~\ref{OneTimeUse0}.
Let $\cD$ be the set obtained from $\cC$ by discarding the finitely many components meeting $X$ from it, i.e. set $\cD=\cC\cap\cC_X$. Then $\cD=\cC_X(X_\upsilon)$ holds.
Since $\cD$ is cofinite in $\cC$ we have $\cD\in U^\circ(\upsilon)$. Recall that by definition of $g_{X_\upsilon,X}$ we know that
\begin{align*}
U(\upsilon,X)=\{\cC'\subseteq\cC_X\,|\,\exists\cD'\in U^\circ(\upsilon)\,:\,\cD'\subseteq\cC'\}.
\end{align*}
Since $\cD\in U^\circ(\upsilon)$ by definition is a subset of $\cC_X$, the equation above yields $\cD\in U(\upsilon,X)$ as claimed.
\end{proof}

\begin{lemma}\label{ECOextentUltra}
For all $X\in\cX$, every $Y\in \crit(X)$ and each infinite $\cC\subseteq\cC_X(Y)$ there is some $\upsilon\in\Upsilon$ with $\upsilon\in\cO_{\TC}(X,\cC)$ and $X_\upsilon=Y$.
\end{lemma}
\begin{proof}
Choose some non-principal ultrafilter $U$ on $\cC_X$ which contains $\cC$.
By Corollary~\ref{Uextension}, $U$ uniquely extends to an element $\upsilon$ of $\Upsilon$.
In particular we have $Y\in\cX_{\upsilon}$.
 For every $Y^-\subsetneq Y$ the set $\cC\rest Y^-$ is a singleton contained in $U(\upsilon,Y^-)$, therefore witnessing $Y^-\notin\cX_\upsilon$, so $Y=X_\upsilon$ follows. 
\end{proof}

\begin{lemma}\label{UltrafilterCrit}
The map $[\upsilon]_\asymp\mapsto X_\upsilon$ is a bijection between $\Upsilon/{\asymp}$ and $\crit(\cX)$.
\end{lemma}
\begin{proof}
The map is well defined by Lemma~\ref{New36}. By definition of $\asymp$ it is injective. It remains to verify surjectivity. If $X$ is in $\crit(\cX)$, then $X\in\crit(X)$ holds, and Lemma~\ref{ECOextentUltra} yields some $\upsilon\in\Upsilon$ with $\upsilon\in\cO_{\TC}(X,\cC_X(X))$ and $X_\upsilon=X$, so $[\upsilon]_\asymp$ gets mapped to $X$ as desired.
\end{proof}

\begin{theorem}\label{InfIndPaths}
For every two distinct vertices $u$ and $t$ of $G$ the following are equivalent:
\begin{enumerate}
\item There exist infinitely many independent $u$--$t$ paths in $G$.
\item There exists some $\upsilon\in\cU$ with $u\between\upsilon\between t$.
\item There exists some end of $G$ dominated by both $u$ and $t$ or\\
there exists some $\upsilon\in\Upsilon$ with $\{u,t\}\subseteq X_\upsilon$.
\item There exists an end of $G$ dominated by both $u$ and $t$ or\\
there exists some $X\in\crit(\cX)$ containing both $u$ and $t$.
\end{enumerate}
\end{theorem}
\begin{proof}
(i)$\leftrightarrow$(ii) is Lemma~\ref{PathSysXtau}.

(ii)$\leftrightarrow$(iii) is due to Lemma~\ref{Xtau}.

(iii)$\leftrightarrow$(iv) is due to Lemma~\ref{UltrafilterCrit}.
\end{proof}

\begin{corollary}\label{GGwellDef}
For every $X\in\cX$ the following hold:
\begin{enumerate}
\item For every $\omega\in\Omega$ the set $\Delta(\omega)$ meets at most one component of $G-X$.
\item Every $Y\in\crit(\cX)$ meets at most one component of $G-X$.
\item For every $\upsilon\in\Upsilon$ the set $X_\upsilon$ meets at most one component of $G-X$.
\end{enumerate}\qed
\end{corollary}

\begin{lemma}
There exists a connected infinite graph $G$ such that $\crit(\cX)$ is an infinite chain with respect to inclusion.
\end{lemma}
\begin{proof}
Put $A=\{a_n\,|\,n\in\N\}$, and let $\{B_n\,|\,n\in\N\}$ be a collection of pairwise disjoint countably infinite sets $B_n$ avoiding $A$.
For every $n\in\N$ write $X_n=\{a_k\,|\,k\le n\}$.
Let $G$ be the graph on $A\cup\bigcup_{n\in\N}B_n$ in which, for every $n\in\N$, each $b\in B_n$ is joined precisely to all $x\in X_n$. 
Then $A$ consists precisely of the vertices of $G$ of infinite degree (every $b\in B_n$ has degree $n+1$), and for every $n\in\N$ we have $X_n\in\crit(\cX)$ witnessed by $\cC_{X_n}(X_n)=\{\{b\}\,|\,b\in B_n\}\cup \{C\}$ where $C$ is the component of $G-X_n$ including all $B_k$ and $a_k$ with $k>n$. 
Every finite $X\subseteq A$ which is distinct from all $X_n$ misses some $a_k$ for some $k<\max\{n\in\N\,|\,a_n\in X\}$. Thus $\cC_X(X)$ is empty, so $X$ is not in $\crit(\cX)$. In particular, we have $\crit(\cX)=\{X_n\,|\,n\in\N\}$.
\end{proof}

\begin{lemma}\label{neverFinished}
Let $X\in\cX$ and $\cC\subseteq\cC_X$ be given, and let $u\in X$ send infinitely many edges to $\bigcup\cC$. Then there is some $\xi\in\cO_{\cU}(X,\cC)$ with $u\between\xi$.
\end{lemma}
\begin{proof}
Assume for a contradiction that no such $\xi$ exists. We will construct a cover of $\TC$ consisting of open sets which does not admit a finite subcover. For this, we cover $\TC$ as follows: 

For $t\in V(G)$ we pick $\cO_G(t,1/2)$.

For every edge $e$ we chose $\mathring{e}$.

For each $\upsilon\in \cO_{\cU}(X,\cC_X\setminus\cC)$ we pick $\cO_{\TC}(X,\cC_X\setminus\cC)$.

For each $\omega\in\Omega\cap\cO_{\cU}(X,\cC)$ we have $u\notin\Delta(\omega)$ since $u\between\omega$ fails. Thus we find some $X(\omega)\in\cX$ with $u$ and $\omega$ living in distinct components of $G-X(\omega)$, and we choose $\hat{C}_{\TC}(X(\omega),\omega)$ which avoids $\mathring{E}(u)$.

For every $\upsilon\in\Upsilon\cap\cO_{\cU}(X,\cC)$ we have $u\notin X_\upsilon$ by Lemma~\ref{Xtau} since $u\between\upsilon$ fails.
Then we choose $\cO_{\TC}(X\cup X_\upsilon,\cC_{X\cup X_\upsilon}(X_\upsilon))$ (which contains $\upsilon$ by Lemma~\ref{New36} and avoids $\mathring{E}(u)$ due to $u\in X\setminus X_\upsilon$).
This completes the construction of the cover.

Since every $m(e)$ with $e\in E(u,\bigcup\cC)$ is covered only by $\mathring{e}$, our cover of $\TC$ has no finite subcover, contradicting the compactness of $\TC$.
\end{proof}

\begin{lemma}\label{InfDegree} 
The vertices in $V(\cU)=V(\Omega)\cup\bigcup\crit(\cX)$ are precisely the vertices of infinite degree.
\end{lemma}
\begin{proof}
By Lemma~\ref{UltrafilterCrit} we have $V(\cU)=V(\Omega)\cup\bigcup\crit(\cX)$. Clearly, every vertex in $V(\cU)$ has infinite degree. Conversely, for every vertex $u$ of infinite degree there is some $\upsilon\in\cU$ with $u\between\upsilon$ by Lemma~\ref{neverFinished}. If $\upsilon$ is an end, then $u\in\Delta(\upsilon)\subseteq V(\cU)$ follows. Otherwise $\upsilon$ is an ultrafilter tangle, and Lemma~\ref{Xtau} yields $u\in X_\upsilon\subseteq V(\cU)$.
\end{proof}

\begin{corollary}\label{TCHausdorffIff}
The tangle compactification of $G$ is Hausdorff if and only if $G$ is locally finite.
\end{corollary}
\begin{proof}
If $G$ is locally finite, then the tangle compactification coincides with $|G|$ which is Hausdorff. Conversely, if the tangle compactification is Hausdorff, then clearly $V(\Omega)$ is empty, and by Lemma~\ref{Xtau} we know that for each ultrafilter tangle its set $X_\upsilon$ must be empty.
Corollary~\ref{InfDegree} then yields that $G$ is locally finite.
\end{proof}

\newpage
\subsection{ETop as a quotient}\label{ETopAsQuotient}

The coarsest topology for arbitrary infinite graphs known to be of any use is \textsc{ETop}. In general, it is impossible to obtain \textsc{ETop} from \textsc{VTop} as a quotient. Indeed, suppose that $G$ is a $K_{1,\aleph_0}$. Since this graph has no ends and every two vertices are finitely separable, no two points of $G$ may be identified. 
But for this $G$, the topology \textsc{VTop} does not coincide with \textsc{ETop}, since the set of all leaves of $G$ is closed in \textsc{VTop} but not in \textsc{ETop}.
Now if we consider the tangle compactification of $G\simeq K_{1,\aleph_0}$ instead of \textsc{VTop} and identify all of the (ultrafilter) tangles of $G$ with the center vertex, the resulting space is homeomorphic to \textsc{ETop}.\footnote{Indeed, an open neighbourhood of the center vertex and all ultrafilter tangles may miss at most finitely many leaves, since otherwise we could construct an ultrafilter tangle outside of that neighbourhood.}

In this chapter we generalise the common notion of `finitely separable' from vertices to $\aleph_0$-tangles in order to define an equivalence relation on the tangle compactification whose quotient we show to be homeomorphic to $\cE G$.\\

If $x$ and $y$ are two points of $V\cup\cU$ and $F$ is a finite cut of $G$ with sides $V_1$ and $V_2$ such that $x\in\overline{G[V_1]}$ and $y\in\overline{G[V_2]}$, then we call $x$ and $y$ \textit{finitely separable}.
If $x$ and $y$ are both vertices of $G$, then by Lemma~\ref{easyJAL} this definition coincides with the standard definition of `finitely separable' for vertices.
Letting $x\sim y$ whenever $x$ and $y$ are not finitely separable defines an equivalence relation on $V\cup\cU$:\index{finitely separable, ${\sim}$}
\begin{lemma}
The relation $\sim$ is transitive.
\end{lemma}
\begin{proof}
Given $x,y,z\in V\cup\cU$ with $x\sim y$ and $y\sim z$ we have to show $x\sim z$.
Assume for a contradiction that $x\not\sim z$ holds, witnessed by some finite cut $F$ of $G$ with sides $V_1$ and $V_2$. By Lemma~\ref{easyJAL} we have precisely one of $y\in\overline{G[V_1]}$ and $y\in\overline{G[V_2]}$, and each case contradicts one of $x\sim y$ and $y\sim z$.
\end{proof}

Hence $\Gq:=\TC/{\sim}$ is well defined, and it is a compact space since $\TC$ is. Furthermore, we know that\index{$\Gq$}

\begin{obs}
If $G$ is locally finite and connected, then $|G|=\TC=\Gq$.
\end{obs}

The next few technical Lemmas already give a hint that $\Gq$ and $\cE G$ might be related:

\begin{lemma}\label{Tripartition}
Let $F$ be a finite cut of $G$ with sides $V_1$ and $V_2$, and let $\Lambda$ consist of precisely one inner edge point $\lambda_e$ from each $e\in F$.
If this Lemma is applied outside the context of any $\Lambda$, then we tacitly assume $\Lambda=\{m(e)\,|\,e\in F\}$.
Put 
\begin{align*}
O_i=\overline{G[V_i]}\cup\bigcup_{\substack{e=xy\in F\\x\in V_i}}[x,\lambda_e)
\end{align*}
for both $i=1,2$ (with the closures taken in $\TC$). Then both $O_i$ are $\sim$-closed open sets and we have $\TC=O_1\uplus\Lambda\uplus O_2$ as well as $\Gq=(O_1/{\sim})\uplus\Lambda\uplus (O_2/{\sim})$.
\end{lemma}
\begin{proof}
Since $\overline{G[V_1]}$ and the union of the $[x,\lambda_e]$ with $e=xy\in F$ and $x\in V_1$ are closed, by Lemma~\ref{easyJAL} we know that $O_2$ is open. Similarly, $O_1$ is open.
\end{proof}

\begin{corollary}\label{HD}
$\Gq$ is Hausdorff.\qed
\end{corollary}

\begin{corollary}\label{ArcEdgesDense}
Let $A\subseteq\Gq$ be an arc. Then $A\cap\mathring{E}$ is dense in $A$.\qed
\end{corollary}

\begin{lemma}\label{Forms}
Let $G$ be any graph with only finitely many components. Then every point of $\Gq$ is exactly of one of the following forms:
\begin{enumerate}
\item an inner edge point;
\item $[u]_\sim$ for some vertex $u$ of $G$;
\item $\{\omega\}$ for some undominated end $\omega$ of $G$.
\end{enumerate}
If $\omega$ is some undominated end of $G$, then $[\omega]_\sim$ is not necessarily of the form \textnormal{(iii)}, e.g. consider the example from Fig.~\ref{fig:NonHD} where it is of the form \textnormal{(ii)}.
\end{lemma}
\begin{proof}
Consider any point $\xi$ of $\Gq$. 

If $\xi$ contains an inner edge point, then $\xi$ is of the form (i). 

If $\xi$ contains a vertex, then $\xi$ is of the form (ii).

If $\xi$ contains some $\upsilon\in\Upsilon$ ,then $\xi$ is of the form (ii) for each $u\in X_\upsilon$ (which is non-empty since $G$ has only finitely many components). 

Finally suppose that $\xi\subseteq \Omega$, and suppose for a contradiction that $\xi$ is not of the form (iii), i.e. that $\xi$ is not a singleton. Pick distinct elements $\omega,\omega'$ of $\xi$ and let $X\in\cX$ witness $\omega\neq\omega'$. 
Then since $X$ avoids $\xi$ there exists for every $x\in X$ some finite cut $F_x$ with sides $A_x,B_x$ such that $x\in A_x$ and $\omega$ lives in $B_x$.
Let $A:=\bigcup_{x\in X}A_x$ and $B:=\bigcap_{x\in X}B_x$ as well as $F:=E(A,B)$ which is a finite cut due to $F\subseteq\bigcup_{x\in X}F_x$.
Now $\omega\sim\omega'$ implies that there is some component $C$ of $G-F$ in which both $\omega$ and $\omega'$ live. Since $\omega$ lives in $B_x$ for every $x\in X$ and $X$ is finite, we know that $V(C)\subseteq B$. In particular, $C$ avoids $X$, and hence is a connected subgraph of $G-X$. But then $C(X,\omega)=C(X,\omega')$ follows, a contradiction.
\end{proof}

\begin{corollary}\label{FinSepClassesMeetX}
Let $G$ be any graph with only finitely many components. Furthermore let $X\in\cX$ be given together with some bipartition $\{\cC,\cC'\}$ of $\cC_X$ and two distinct points $x$ and $y$ of $\TC$ such that $x\in\cO_{\TC}(X,\cC)$ and $y\in\cO_{\TC}(X,\cC')$. If $x\sim y$ then $[x]_\sim$ meets $X$.
\end{corollary}
\begin{proof}
Assume for a contradiction that $\xi:=[x]_\sim$ avoids $X$.
By Lemma~\ref{Forms} there is some vertex $u$ of $G$ in $\xi$, and since $\xi$ avoids $X$ there is some $C\in\cC_X$ with $u\in C$, without loss of generality with $C\in\cC$.
Now $X\cap\xi=\emptyset$ implies that
there exists for every $t\in X$ some finite cut $F_t$ of $G$ with sides $A_t$ and $B_t$ such that $t\in A_t$ and $u\in B_t$. 
Put $A=\bigcup_{t\in X}A_t$ and $B=\bigcap_{t\in X}B_t$, and let $F$ be the finite cut $E(A,B)$.
Let $D$ be the component of $G-F$ containing $u$.
Then $D\subseteq G[B]$ is a connected subgraph of $G$ avoiding $X\subseteq A$ and $D$ contains $u$, so we have $u\in D\subseteq C$.
 Next consider the finite bond $B:=E(V\setminus D,D)\subseteq F$.
Then $\bigcup\cC'\subseteq V\setminus D$ (due to $D\subseteq C\in\cC$) together with $y\in\cO_G(X,\cC')$ implies $y\notin\overline{G[D]}$, and hence Lemma~\ref{easyJAL} yields $y\in\overline{G[V\setminus D]}$. Thus $B$ witnesses $u\not\sim y$, a contradiction.
\end{proof}

\begin{center}
\textit{For the rest of this chapter, $G$ is assumed to be connected.}
\end{center}

Next we define an inverse system of multigraphs whose inverse limit is homeomorphic to $\cE G$. For this, we begin by defining the multigraphs:
For every $F\subseteq E(G)$ we denote by $G\boldsymbol{.}F$ the (multi-)graph which is obtained from $G$ by contracting the vertex sets of the components of $G-F$ (therefore turning edges of these components into loops). 
For the rest of this chapter let $\cE:=[E(G)]^{<\aleph_0}$ be ordered by inclusion.
We equip $G\boldsymbol{.}F$ with the topology generated by the following basis: For inner edge points we declare as open the usual neighbourhoods. 
For every vertex $d$ of $G\boldsymbol{.}F$ we declare as open all neighbourhoods of the form
\begin{align}\label{GpBOS}
\{d\}\cup\bigcup_{e\in E_{G\boldsymbol{.}F}(d,G\boldsymbol{.}F-d)}(j_e,d]\cup\bigcup_{e\in L}(e\setminus I_e)\cup \bigcup (\ell(d)\setminus L)
\end{align}
where $\ell(d)$ denotes the set of all loops at $d$, $j_e\in\mathring{e}$ for every $e\in E_{G\boldsymbol{.}F}(d,G\boldsymbol{.}F-d)$, $L\in [\ell(d)]^{<\aleph_0}$ and $I_e$ is some closed interval included in $\mathring{e}$ for every $e\in L$.

For every $F\subseteq F'\in \cE$ we define bonding maps $f_{F',F}\colon G\boldsymbol{.}F'\to G\boldsymbol{.}F$ which are the identity on $\mathring{E}(G\boldsymbol{.}F')=\mathring{E}(G)$ and which send every vertex $d$ of $G.F'$ to the vertex of $G\boldsymbol{.}F$ whose corresponding component of $G-F$ includes the component of $G\boldsymbol{.}F'$ corresponding to $d$.
The following Lemmas are case checking:
\begin{lemma}
The maps $f_{F',F}$ are continuous.\qed
\end{lemma}
\begin{lemma}
The maps $f_{F',F}$ are compatible.\qed
\end{lemma}
\begin{obs}\label{GpHD}
The $G\boldsymbol{.}F$ are compact and Hausdorff.
\end{obs}
Since $\{G\boldsymbol{.}F,f_{F',F},\cE\}$ is an inverse system of non-empty compact Hausdorff spaces\index{$\{G\boldsymbol{.}F,f_{F',F},\cE\}$}, its inverse limit\index{$\CL G\CL$}
\begin{align*}
\CL G\CL:=\invLim(G\boldsymbol{.}F\,|\,F\in \cE)
\end{align*}
is again non-empty, compact and Hausdorff by Lemma~\ref{InvLimCptHD}.
Furthermore, Miraftab~\cite{Babak} showed that

\begin{lemma}\label{ETopInvLim}
If $G$ is a connected graph, then $\CL G\CL$ is homeomorphic to $\cE G$.
\end{lemma}

Given $F\in\cE$ we define $\sigma_F\colon \Gq\to G\boldsymbol{.}F$ as follows: Let $x\in\Gq$ be given. By Lemma~\ref{Forms} there are three cases to distinguish:
If $x$ is an inner edge point of some edge $e$, then we let $\sigma_F$ map $x$ to $x$.
For $x=[u]_\sim$ with $u$ some vertex of $G$ we let $\sigma_F$ map $x$ to the vertex of $G.F$ whose corresponding component of $G-F$ contains $u$. This is well-defined since the finite cuts of $G\boldsymbol{.}F$ are finite cuts of $G$.
Finally, if $x=\{\omega\}$ we let $\sigma_F$ map $x$ to the vertex of $G\boldsymbol{.}F$ whose corresponding component includes a tail of every ray in $\omega$.\footnote{If $G$ has infinitely many components, then there is some $\upsilon\in\Upsilon$ with $X_\upsilon=\emptyset$ and there is no natural choice for $\sigma_F([\upsilon]_\sim)=\sigma_F(\{\upsilon\})$ in $G\boldsymbol{.}F$. Furthermore, no $G\boldsymbol{.}F$ would be compact. Hence we require $G$ to be connected.}

\begin{lemma}\label{sigmaCOMP}
The maps $\sigma_F$ are compatible.
\end{lemma}
\begin{proof}
Let $F\subseteq F'\in\cE$ and $x\in\Gq$ be given. We have to show that 
\begin{align*}
(f_{F',F}\circ\sigma_{F'})(x)=\sigma_F(x)
\end{align*}
holds, but this is clear from the definition of $\sigma_F,\sigma_{F'}$ and $f_{F',F}$.
\end{proof}

\begin{lemma}\label{sigmaCTS}
The maps $\sigma_F$ are continuous surjections.
\end{lemma}
\begin{proof}
Clearly, the $\sigma_F$ are surjective by construction.
Let $x\in\Gq$ be given together with $W$ some basic open neighbourhood of $y:=\sigma_F(x)$ in $G\boldsymbol{.}F$. By Lemma~\ref{Forms} we distinguish three cases:

If $x$ is an inner edge point we are done, so assume that $x$ is of the form $[u]_\sim$ for some vertex $u$ of $G$. Hence $y$ is a vertex of $G\boldsymbol{.}F$ and $W$ is of the form~(\ref{GpBOS}). Let $C$ be the component of $G-F$ corresponding to $y$.
Then the non-loop edges at $y$ in $G\boldsymbol{.}F$ form a finite cut $\bar{F}=E_{G\boldsymbol{.}F}(y,G\boldsymbol{.}F-y)$ of $G\boldsymbol{.}F$ and thus also of $G$. 
Put $\Lambda=\{j_e\,|\,e\in \bar{F}\}$ and use Lemma \ref{Tripartition} to yield a tripartition $\Gq=(O_1/{\sim})\uplus\Lambda\uplus (O_2/{\sim})$. 
Since $C$ is one side of $\bar{F}$ in $G$ and $u\in C$, we without loss of generality have $[u]_\sim\in O_1/{\sim}$, so $\sigma_F$ maps $(O_1/{\sim})-\bigcup_{e\in L}I_e$ into $W$.

Finally assume $x=\{\omega\}$ where $\omega$ is an undominated end of $G$. Then $y$ is a vertex of $G\boldsymbol{.}F$ and we are done by the previous case.
\end{proof}

\begin{theorem}\label{ETopInvLimAndGq}
If $G$ is a connected graph, then $\CL G\CL$ is homeomorphic to the quotient space $\Gq$ of $\TC$. In particular, $\Gq$ is homeomorphic to $\cE G$.
\end{theorem}
\begin{proof}[First proof]
We start by letting $\Psi\colon \Gq\to\CL G\CL$ map $x$ to $(\sigma_F(x)\,|\,F\in \cE)$.
Then $\Psi(x)$ is an element of $\CL G\CL$ by Lemma \ref{sigmaCOMP} and $\Psi$ is injective: Consider distinct elements $x$ and $y$ of $\Gq$.
Without loss of generality none of $x$ and $y$ is an inner edge point. 
Let $F$ be some finite cut of $G$ witnessing $x\neq y$, then
$\sigma_F(x)\neq\sigma_F(y)$ follows as desired. 
But $\Psi$ is also a continuous surjection by Lemmas~\ref{sigmaCOMP}, \ref{sigmaCTS} and \ref{InvLimSurjEmbOnto}. 
Therefore, $\Psi$ is a continuous bijection from a compact space onto a Hausdorff space. By general topology, $\Psi$ is a homeomorphism, so $\CL G\CL$ is homeomorphic to $\Gq$ as claimed. Due to Lemma~\ref{ETopInvLim}, we also know that $\Gq$ is homeomorphic to $\cE G$.
\end{proof}
\begin{proof}[Second proof (sketch)]
Define $\Psi\colon \Gq\to\cE G$ using Lemma~\ref{Forms} as follows: Let $x\in\Gq$ be given. If $x$ is an inner edge point put $\Psi(x)=x$. Else if $x$ is of the form $[u]_\sim$ for some vertex $u$ of $G$ let $\Psi$ send $u$ to the point of $\cE G$ containing $u$. Otherwise $x$ is of the form $\{\omega\}$ for some end $\omega$ of $G$ and we put $\Psi(x)=\{\omega\}$. It is straightforward to check that $\Psi$ is a well defined bijection. Since $\Gq$ is compact and $\cE G$ is Hausdorff, it suffices to check that $\Psi$ is continuous, which easily follows from Lemma~\ref{Tripartition}.
\end{proof}

\begin{corollary}
If $G$ is a countable connected graph and $\ell\colon E(G)\to\R_{>0}$ satisfies $\sum_{e\in E(G)}\ell(e)<\infty$, then $|G|_\ell$ is homeomorphic to $\Gq$.
\end{corollary}
\begin{proof}
By \cite[Theorem 3.1]{ETop} we know that $|G|_\ell$ and $\cE G$ are homeomorphic. Hence Theorem~\ref{ETopInvLimAndGq} implies that $|G|_\ell$ and $\Gq$ are also homeomorphic.
\end{proof}

\begin{obs}
In Section~\ref{FGconstruction} we learn about a new compactification $\FG$ which is quite similar to $\TC$. For this space, Lemmas~\ref{easyJAL},~\ref{Tripartition} and~\ref{Forms} admit straightforward analogues, so $\cE G$ is also a quotient of $\FG$.
\end{obs}

\newpage
\section{The tangle compactification as inverse limit}\label{tangleInvLim}

\subsection{Introduction}

In recent years, inverse limits emerged as useful tools to obtain `limit objects' of graphs, such as circles and TSTs, from compatible choices of `finite objects' of some carefully chosen `finite minors'. 
Such inverse systems are known for the Freudenthal compactification of locally finite graphs (\cite[Theorem 8.7.2, 5th edition]{Bible}) and for $\cE G$ with \textsc{ETop} (\cite{Babak}, or see $\CL G\CL$ from section~\ref{ETopAsQuotient}).
In this chapter, we will set up such an inverse system whose inverse limit describes the tangle compactification. 
The `finite minors' we are going to use for this actually are multigraphs with finite vertex set and possibly infinite edge set, and they are obtained from $G$ by contraction of possibly non-connected sets of vertices.
Hence the question comes to mind whether 
there is a better inverse system based on finite minors of $G$? 
\begin{figure}[H]
\includegraphics[clip,page=9,trim=40 490 40 35,width=\textwidth]{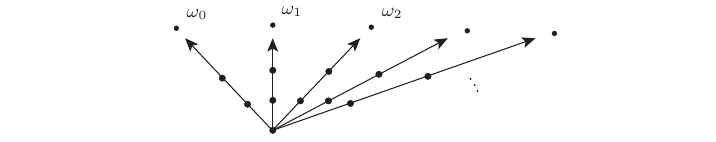}
\caption{An example graph from \cite[Fig. 3]{EndsAndTangles}.}
\label{fig:tangleExample}
\end{figure}
First, we consider an example graph whose tangle compactification cannot be described in a meaningful way by an inverse system of minors.
Let $G$ be the graph pictured in Fig.~\ref{fig:tangleExample}, and for every $n$ denote by $R_n$ the ray in $\omega_n$ starting at the centre vertex of the infinite star. As Diestel showed in \cite[Example 1.9]{EndsAndTangles}, every ultrafilter on the naturals induces an $\aleph_0$-tangle of this graph. Indeed, since the sides of a finite order separation of this graph induce a bipartition of its end space, and hence of the naturals via the enumeration $\omega_1,\omega_2,\hdots$ of $\Omega$, we can use the `big' elements of the ultrafilter to declare the corresponding sides of the finite order separations as `big' (i.e. orient the finite order separations such that their big sides induce elements of the ultrafilter).
Then the principal ultrafilters on the naturals induce precisely the end-tangles of this graph.
Interestingly, it is not possible to describe the tangle compactification of this graph as an inverse limit of minors of $G$ (not even infinite ones), as long as we deem the following premise meaningful: 

\textit{If $\{H_i,\varphi_{ji},I\}$ is an inverse system of minors $H_i$ of $G$ and for every $i$ there is a mapping $\sigma_i\colon \TC\to H_i$ such that $\invLim\sigma_i\colon \TC\to\invLim H_i$ is a homeomorphism, then for every $i$ the map $\sigma_i$ sends every ultrafilter tangle of $G$ to an infinite branch set of $H_i$ which is not included in some $R_n$.}

Now if $U$ and $U'$ are distinct non-principal ultrafilters on the naturals, they induce distinct ultrafilter tangles $\tau$ and $\tau'$ of $G$. Since $\invLim\sigma_i$ is injective, there is some $i$ in $I$ such that $\sigma_i$ sends the two ultrafilter tangles to different branch sets of $H_i$. Since $H_i$ is a minor of $G$, its branch sets are connected, so by the premise they both meet in the centre vertex of the infinite star, which contradicts the fact that distinct branch sets are disjoint. Hence contracting non-connected sets of vertices in general cannot be avoided.

But perhaps we can use finite multigraphs instead? Indeed, this is possible, and as it will turn out there is an easy way to obtain such multigraphs from ours. 
But making their edge sets finite comes at the cost of quality of life, which is why we forgo this option.

\subsection{Setup}

Given a graph $G$, for every $X\in\cX$ we denote by $\cP_X$\index{$\cP_X$} the set of all finite partitions of $\cC_X$. Then we let\index{$\Gamma$}
\begin{align*}
\Gamma:=\{(X,P)\,|\,X\in\cX\text{ and }P\in\cP_X\}.
\end{align*}
If $(X,P)$ is in $\Gamma$, we denote by $p(X,P)$\index{$p(X,P)$} the finite partition of $V(G)$ induced by $P$ and the singleton subsets of $X$. 
Letting $(X,P)\le (Y,Q)$ whenever $X\subseteq Y$ and $p(Y,Q)$ refines $p(X,P)$ defines a directed partial ordering on $\Gamma$:

\begin{lemma}
\label{directed}
$(\Gamma,\le)$ is a directed poset.
\end{lemma}
\begin{proof}
The poset properties are inherited from the subset and the refinement relation; it remains to verify that $(\Gamma,\le)$ is directed: For this, let any two elements $(X,P)$ and $(Y,Q)$ of $\Gamma$ be given.
First, we will define finite partitions $P'$ and $Q'$ of $\cC_{X\cup Y}$ such that $(X,P)\le (X\cup Y,P')$ and $(Y,Q)\le (X\cup Y,Q')$ hold; but this is easy: Recall that the map $\phi_{X\cup Y,X}$ sends every component of $G-(X\cup Y)$ to the unique component of $G-X$ containing it, so the preimage of every partition class $\cC$ of $P$ under $\phi_{X\cup Y,X}$ consists precisely of those components of $G-(X\cup Y)$ which are included in some component of $G-X$ that is also an element of $\cC$.
Hence letting 
\begin{align*}
P'&:=\{\phi_{X\cup Y,X}^{-1}(\cC)\,|\,\cC\in P\}\setminus\{\emptyset\}\\
Q'&:=\{\phi_{X\cup Y,Y}^{-1}(\cC)\,|\,\cC\in Q\}\setminus\{\emptyset\}
\end{align*}
will do.
Finally, choosing $R$ to be the coarsest refinement of $P'$ and $Q'$ results in a finite partition of $\cC_{X\cup Y}$ which satisfies the two inequalities
\begin{align*}
(X,P)&\le (X\cup Y,P')\le (X\cup Y,R)\\
(Y,Q)&\le (X\cup Y,Q')\le (X\cup Y,R),
\end{align*}
so we are done by transitivity.
\end{proof}

Next we define the topological spaces of our inverse limit. For ever $(X,P)\in\Gamma$, we let $G/p(X,P)$ be the multigraph on $p(X,P)$ whose edges are precisely the cross-edges of $p(X,P)$.\index{$G/p(X,P)$}
The vertices of $G/p(X,P)$ that are singleton subsets $\{x\}$ of $X$ we consider to be vertices of $G$ and refer to them as $x$; the other vertices of $G/p(X,P)$ are its \textit{dummy vertices}\index{dummy vertex}\footnote{This definition does not coincide with that of dummy vertices of the $G/p$ from the construction of $\|G\|$ in \cite[5th edition]{Bible}: e.g. if some $p\in p(X,P)$ is a singleton, but not a subset of $X$, then $p$ is a dummy vertex here but a non-dummy vertex in \cite{Bible}.}.
Now we let $G_{(X,P)}$\index{$G_\gamma$} be the topological space obtained from the ground set of the 1-complex of $G/p(X,P)$ by endowing it with the topology generated by the following basis:

For inner edge points we choose the usual open neighbourhoods, and for vertices $x\in X$ we choose as open neighbourhoods the sets of the form $\bigcup_{e\in E(x)} [x,j_e)$ where the inner edge points $j_e\in\mathring{e}$ may be chosen individually for each $e$.
Finally, for dummy vertices $p$ of $G/p(X,P)$ we declare as open the sets of the form
\begin{align*}
\Big(\mathring{E}(X,p)\cup\{p\}\Big)\,\setminus\,\bigcup_{e\in F}(x_e,j_e]
\end{align*}
where $F$ is some finite subset of $E(X,p)$, 
and for every $e\in F$ the point $j_e$ is an arbitrary inner edge point of $e$ while $x_e$ is the endvertex of $e$ in $X$.\footnote{The purpose of $F$ is to make our topological spaces T$_1$. Only allowing $F=\emptyset$ would also work, but then our topological spaces are not T$_1$ while the tangle compactification is.} This completes the definition of the basis.

The spaces $G_{(X,P)}$ are easily seen to be compact. To complete the setup of our inverse system, for all $(X',P')\ge (X,P)$ we choose the bonding map
\begin{align*}
f_{(X',P'),(X,P)}\colon G_{(X',P')}\to G_{(X,P)}
\end{align*}
which sends the vertices of $G/p(X',P')$ to the vertices of $G/p(X,P)$ including them; which is the identity on the interior of the edges of $G/p(X',P')$ that are also edges of $G/p(X,P)$; and which sends any other edge of $G/p(X',P')$ to the dummy vertex of $G/p(X,P)$ that includes both its endvertices in $G/p(X',P')$. 
An easy proof by cases shows that these bonding maps are continuous.\footnote{This is basically the definition of Diestel.}
Therefore, we arrived at an inverse system $\{G_\gamma,f_{\gamma',\gamma},\Gamma\}$\index{$\{G_\gamma,f_{\gamma',\gamma},\Gamma\}$}, and we set\index{$\iL$}
\begin{align*}
\iL=\invLim (G_\gamma\,|\,\gamma\in\Gamma).
\end{align*}
Before we verify that $\iL$ describes the tangle compactification, we introduce some notation concerning $\iL$:
If $W$ is an open set in $G_\eta$ for some $\eta\in\Gamma$, then we denote by $\cO_{\iL}(W,\eta)$\index{$\cO_{\iL}(W,\eta)$} the open subset $\iL\cap\prod_\gamma W_\gamma$ of $\iL$ where $W_\gamma$ is the whole space $G_\gamma$ for all $\gamma$ except for $W_\eta$, which we choose to be $W$. Given $\eta\in\Gamma$ we define $f_\eta$ to be the continuous projection map\index{$f_\eta$}
\begin{align*}
f_\eta\colon \iL\to G_\eta,\;(x_\gamma\,|\,\gamma\in\Gamma)\mapsto x_\eta.
\end{align*}
Finally, whenever $x$ is an element of $\iL$ we denote $f_\eta(x)$ by $x_\eta$.

Before we verify that $\iL$ suits its purpose, we have a last look at their possibly infinite edge sets.
Obviously, the multigraphs $G/p(X,P)$ in general are not finite, but their vertex set is finite. It is possible to turn these multigraphs into finite graphs, simply by deleting the inner edge points of edges incident with dummy vertices. If we then think of these inner edge points as `hidden' in the dummy vertices incident with the edges containing them, we can adjust the bonding maps so as to yield another inverse system whose inverse limit would still describe the tangle compactification. But obviously, the deletion of these points renders this inverse system useless, since it contradicts our aim to improve everyone's quality of life. Therefore, we decided to keep these points, and instead rely on the compactness of the spaces $G_{(X,P)}$.

\subsection{The homeomorphism}

\begin{theorem}\label{G*}
If $G$ is an arbitrary infinite graph, then the tangle compactification $\TC$ of $G$ is homeomorphic to $\iL$.
\end{theorem}
\begin{proof}
The first task is to construct a bijection $\Psi\colon \iL\bij \TC$. 
For this, let an arbitrary $x=(x_\gamma\,|\,\gamma\in\Gamma)\in\iL$ be given. We distinguish two cases:

First assume that there is some $\gamma=(X,P)\in\Gamma$ such that $x_\gamma$ is not a dummy vertex of $G/p(X,P)$, so $x_\gamma$ is an element of $X$ or an inner edge point of $G_\gamma$. In this case we let $\Psi(x):=x_\gamma$, which is a well defined choice: If $x_\gamma$ and $x_{\gamma'}$ are two such points we pick some $\eta\ge\gamma,\gamma'$. Then the point $x_\eta$ is not a dummy vertex either, so the definition of $f_{\eta,\gamma}$ and $f_{\eta,\gamma'}$ implies $x_\eta=x_\gamma=x_{\gamma'}$.

In the second case, every $x_\gamma$ is a dummy vertex. Then for every $X\in\cX$ we let
$U_X$ be the collection of all the sets $\cC\in 2^{\cC_X}$ for which there is some finite partition $P$ of $\cC_X$ such that $x_{(X,P)}=V[\cC]$,
and put $\Psi(x):=(U_X\,|\,X\in\cX)$. To complete the definition of $\Psi$, it remains to show that $(U_X\,|\,X\in\cX)$ is an element of $\cU$. 
We start by verifying that

\begin{claim_tcInvLim}\label{UXareUltra}
The $U_X$ are ultrafilters on $\cC_X$.
\end{claim_tcInvLim}
\begin{proof}[Proof of the Claim]\renewcommand{\qedsymbol}{$\diamond$}
Since dummy vertices formally are non-empty sets of vertices, the empty set is not in $U_X$.

Next we show that, for every two elements $\cC$ and $\cC'$ of $U_X$, the intersection $\cC\cap \cC'$ is again an element of $U_X$. Given two elements $\cC$ and $\cC'$ of $U_X$ we choose witnesses $(X,P)$ and $(X,P')$ of $\cC\in U_X$ and $\cC'\in U_X$, respectively.
Let $Q$ be the coarsest refinement of $P$ and $P'$, and let $d\in Q$ be the partition class satisfying $d=\cC\cap \cC'$. Since $(X,Q)\ge (X,P),(X,P')$ and by compatibility, the only choice for $x_{(X,Q)}$ is the dummy vertex $V[d]$. Hence $x_{(X,Q)}$ witnesses that $d=\cC\cap\cC'$ is contained in $U_X$.

Now we show that $U_X$ is upwards closed. For this, let any $\cC\in U_X$ be given together with a superset $\cC'\in 2^{\cC_X}$; we have to show $\cC'\in U_X$.
Pick some witness $(X,P)$ of $\cC\in U_X$ and let $Q$ be the coarsest refinement of $P$ and the bipartition $P'=\{\cC',\cC_X\setminus\cC'\}$ of $\cC_X$. Assume for a contradiction that $x_{(X,P')}$ is not the dummy vertex $V[\cC']$ of $G/p(X,P')$. Then $x_{(X,P')}$ must be the dummy vertex $V[\cC_X\setminus\cC']$ since it is the only other dummy vertex of $G/p(X,P')$. Now let $d$ be the partition class of $Q$ satisfying $x_{(X,Q)}=V[d]$. By compatibility
we know that $x_{(X,Q)}$ is a subset of both $x_{(X,P)}$ and $x_{(X,P')}$. But then, since $\cC$ is a subset of $\cC'$, we also know that
\begin{align*}
V[d]\subseteq V[\cC_X\setminus\cC']\cap V[\cC]=\emptyset
\end{align*}
holds, which
is a contradiction since $d$ is non-empty.

So far we have seen that $U_X$ is a filter, so finally we verify that $U_X$ is a maximal one. For this, consider any bipartition $\{\cC,\cC'\}$ of $\cC_X$; we show that exactly one of $\cC$ and $\cC'$ is contained in $U_X$. Since $x_{(X,\{\cC,\cC'\})}$ is a dummy vertex, we know that at least one of $\cC$ and $\cC'$ is in $U_X$, say $\cC$. Now assume for a contradiction that $\cC'$ is contained in $U_X$, too, witnessed by some $(X,P)$. Then $P$ refines the bipartition $\{\cC,\cC'\}$ since $\cC'$ is a partition class of $P$. In particular $(X,P)\ge (X,\{\cC,\cC'\})$ holds, and $x_{(X,P)}$ is the dummy vertex $V[\cC']$ of $G/p(X,P)$. But then, by definition of $f_{(X,P),(X,\{\cC,\cC'\})}$, the only possibility for $x_{(X,\{\cC,\cC'\})}$ is the dummy vertex $V[\cC']$ of $G/p(X,\{\cC,\cC'\})$, which is a contradiction.
\end{proof}

\begin{claim_tcInvLim}
For all $X\subseteq X'\in\cX$ we have $f_{X',X}(U_{X'})=U_X$.
\end{claim_tcInvLim}
\begin{proof}[Proof of the Claim]\renewcommand{\qedsymbol}{$\diamond$}
Let $X\subseteq X'\in\cX$ be given. 
By Claim~\ref{UXareUltra} both $U_X$ and $U_{X'}$ are ultrafilters on $\cC_X$ and $\cC_{X'}$, respectively, so $U:=f_{X',X}(U_{X'})$ is again an ultrafilter on $\cC_X$.
Assume for a contradiction that $U$ and $U_X$ are distinct,
and pick some bipartition $\{\cC,\cC_X\setminus\cC\}$ of $\cC_X$ with $U$ containing $\cC$ and $U_X$ containing $\cC_X\setminus\cC$.
By definition of $f_{X',X}$ there is some $\cC'$ in $U_{X'}$ witnessing $\cC\in U$, i.e. $\cC'$ satisfies $\cC'\rest X\subseteq\cC$.
Now choose witnesses $(X,P)=\gamma$ and $(X',P')=\gamma'$ of $\cC_X\setminus\cC\in U_X$ and $\cC'\in U_{X'}$, respectively, and pick some $\eta=(X',Q)$ in $\Gamma$ with $\eta\ge (X',P'),(X,P)$. 
Since $\cC'\rest X=\phi_{X',X}[\cC']$ holds by definition and $\cC$ includes $\cC'\rest X$, the set $V[\cC']$ is included in $V[\cC]$.
But $x_\gamma$ is the dummy vertex $V[\cC_X\setminus\cC]$ of $G/p(X,P)$, and $x_{\gamma'}$ is the dummy vertex $V[\cC']$ of $G/p(X',P')$, 
so $V[\cC']\subseteq V[\cC]$ implies $x_\gamma\cap x_{\gamma'}=\emptyset$, and
hence $x_\eta\subseteq x_\gamma\cap x_{\gamma'}=\emptyset$ (by definition of $f_{\eta,\gamma}$ and $f_{\eta,\gamma'}$) yields the desired contradiction.
\end{proof}

Therefore, $\Psi(x)$ is indeed an element of $\cU$ and our construction of $\Psi$ is complete. It remains to show that $\Psi$ is bijective and bicontinuous; we begin with injectivity:

\begin{claim_tcInvLim}
$\Psi$ is injective.
\end{claim_tcInvLim}
\begin{proof}[Proof of the Claim]\renewcommand{\qedsymbol}{$\diamond$}
Let $x=(x_\gamma\,|\,\gamma\in\Gamma)$ and $y=(y_\gamma\,|\,\gamma\in\Gamma)$ be two distinct elements of $\iL$ with $x_\eta\neq y_\eta$ for some $\eta\in\Gamma$. 
If $\Psi(x)\in G$ and $\Psi(y)\in\cU$ or vice versa, we are done, so suppose that $\Psi(x)$ and $\Psi(y)$ either are both in $G$ or both in $\cU$.

For the first case suppose that both $\Psi(x)$ and $\Psi(y)$ are in $G$, witnessed by $(X,P)$ and $(Y,Q)$, respectively, so neither of $x_{(X,P)}$ and $y_{(Y,Q)}$ is a dummy vertex. 
Pick some $\xi\ge\eta,(X,P),(Y,Q)$ and note that, by definition of the bonding maps and well-definition of $\Psi$, we have $x_\xi=x_{(X,P)}=\Psi(x)$ as well as $y_\xi=y_{(Y,Q)}=\Psi(y)$.
In particular, $x_\eta\neq y_\eta$ implies $x_\xi\neq y_\xi$ (again by definition of the bonding maps) and hence $\Psi(x)\neq\Psi(y)$.

For the second case suppose that both $\Psi(x)=(U_X\,|\,X\in\cX)$ and $\Psi(y)=(U_X'\,|\,X\in\cX)$ are elements of $\cU$, and write $\eta=(Y,P)$. 
Let $\cC$ and $\cC'$ be the partition classes of $P$ with $V[\cC]=x_\eta$ and $V[\cC']=y_\eta$. Then in particular $\cC$ is contained in $U_Y$ and $\cC'$ is contained in $U_Y'$. Since $\cC$ and $\cC'$ are disjoint, the ultrafilters $U_Y$ and $U_Y'$ are distinct, so $\Psi(x)\neq \Psi(y)$ holds.
\end{proof}

\begin{claim_tcInvLim}
$\Psi$ is surjective.
\end{claim_tcInvLim}
\begin{proof}[Proof of the Claim]\renewcommand{\qedsymbol}{$\diamond$}
Let any $y\in\TC$ be given.
If $y$ is not an element of $\cU$ we choose some $X$ in $\cX$ such that $y$ is contained in the 1-complex of $G[X]$, and let $\xi:=(X,\{\cC_X\})$. Then for every $\eta\ge\xi$ in $\Gamma$ we set $x_\eta=y$, and for every $\gamma'<\eta$ for some such $\eta$ we put $x_{\gamma'}=f_{\eta,\gamma'}(y)$. Then $(x_\gamma\,|\,\gamma\in\Gamma)$ is a well defined element of $\iL$ which $\Psi$ sends to $y$.

Otherwise $y$ is in $\cU$ and of the form $(U_X\,|\,X\in\cX)$. 
For every $\gamma=(X,P)\in\Gamma$ the ultrafilter $U_X$ chooses excatly one partition class $\cC(\gamma)$ of $P$. Then 
\begin{align*}
x:=\big(V[\cC(\gamma)]\,\big\vert\,\gamma\in\Gamma\big)
\end{align*}
is a well-defined point in $\iL$: Assume for a contradiction that there are some $\eta\le\eta'$ in $\Gamma$ with $x_\eta,x_{\eta'}$ incompatible in $\iL$ in that $f_{\eta',\eta}(x_{\eta'})\neq x_\eta$. Write $\eta=(Y,Q)$ and $\eta=(Y',Q')$. Then the dummy vertex $x_{\eta'}$ satisfies
\begin{align*}
x_{\eta'}\subseteq f_{\eta',\eta}(x_{\eta'})\in p(Y,Q)-\{x_\eta\}
\end{align*}
which implies $x_{\eta'}\cap x_\eta=\emptyset$. Furthermore, using the definition of $x$ yields
\begin{align}\label{tcInvLimEq}
V[\cC(\eta)]\cap V[\cC(\eta')]=\emptyset
\end{align}
Since $\cC(\eta')\rest Y$ and $\cC(\eta)$ are both in $U_Y$ it suffices to show that these sets have empty intersection in order to yield a contradiction. For this, let any $C\in\cC(\eta')\rest Y$ be given. By definition of $\cC(\eta')\rest Y$
there is some $C'\in\cC(\eta')$ with $C'\subseteq C$. Hence the component $C$ is not in $\cC(\eta)$, since otherwise the vertices of $C'$ would be contained in the left hand side of~(\ref{tcInvLimEq}) which is empty, a contradiction. Thus $\cC(\eta')\rest Y$ and $\cC(\eta)$ intersect emptily as desired.
\end{proof}

By now we know that $\Psi$ is a bijection. Before we verify that $\Psi$ is also bicontinuous, let us quickly recall Lemma~\ref{invLimBasis}: According to this Lemma, the collection of all open sets of the form $\cO_{\iL}(W,\eta)$ with $\eta$ some element of $\Gamma$ and $W$ some basic open set in $G_\eta$ is a basis of the topology of $\iL$. Now that we know this basis, we are almost done:
\begin{claim_tcInvLim}
$\Psi^{-1}$ maps basic open sets of $\TC$ to basic open sets of $\iL$.
\end{claim_tcInvLim}
\begin{proof}[Proof of the Claim]\renewcommand{\qedsymbol}{$\diamond$}
Let any point $x$ of $\TC$ be given together with some basic open neighbourhood $W$.

If $x$ is not in $\cU$ we choose some $X$ in $\cX$ such that $x$ is contained in the 1-complex of $G[X]$, and set $\eta:=(X,\{\cC_X\})$.
Then $W$ is also a basic open neighbourhood of $x$ in $G_\eta$, so $\Psi^{-1}[W]=\cO_{\iL}(W,\eta)$ holds.

Otherwise $x$ is an element of $\cU$, so $W$ is of the form $\cO_{\TC}(X,\cC)$ for some $X\in\cX$ and some non-empty $\cC\subseteq\cC_X$. Let $\eta:=(X,\{\cC,\cC_X\setminus\cC\})$ and consider the basic open neighbourhood $W':=\mathring{E}(X,\bigcup\cC)\cup \{V[\cC]\}$ of the dummy vertex $V[\cC]$ in $G_\eta$. Then one easily checks that $\Psi^{-1}[W]=\cO_{\iL}(W',\eta)$ holds.
\end{proof}

\begin{claim_tcInvLim}
$\Psi$ maps basic open sets of $\iL$ to open sets of $\TC$.
\end{claim_tcInvLim}
\begin{proof}[Proof of the Claim]\renewcommand{\qedsymbol}{$\diamond$}
Let any basic open set $\cO_{\iL}(W,\eta)$ of $\iL$ be given and write $(X,P)=\eta$.

If $W$ is a basic open neighbourhood of some non-dummy point of $G_\eta$,
then $\Psi[\cO_{\iL}(W,\eta)]=W$ is basic open in $\TC$.

Otherwise $W$ is of the form
\begin{align*}
\left(\mathring{E}(X,d)\cup\{d\}\right)\,\setminus\,\bigcup_{e\in F}(x_e,j_e]
\end{align*}
where $d$ is some dummy vertex of $G_\eta$, the set $F$ is some finite subset of $E(X,d)$, 
and for every $e\in F$ the point $j_e$ is an arbitrary inner edge point of $e$ while $x_e$ is the endvertex of $e$ in $X$.
Let $\cC$ be the partition class of $P$ with $V[\cC]=d$. Then
\begin{align*}
\Psi[\cO_{\iL}(W,\eta)]=\cO_{\TC}(X,\cC)\setminus \bigcup_{e\in F}[x_e,j_e]
\end{align*}
which is open since the finitely many intervals $[x_e,j_e]$ are closed in $\TC$.
\end{proof}
This completes the proof that $\Psi$ is a homeomorphism.
\end{proof}

From now on we treat $\iL\simeq\TC$ as $\iL=\TC$.




\newpage
\section{A least tangle compactification}\label{ECOcompactification}

We call a compactification $\alpha G$ of (the 1-complex of) $G$ an \textit{\ecomp{}}\index{\ecomp{}} of $G$ if $\alpha G$ is also a compactification of $\cV G$ and $\alpha G\setminus\mathring{E}$ is Hausdorff.
Even though we speak of an \ecomp{} `of' $G$, we formally treat it as a compacitification of $\cV G$, e.g. if $\alpha G\le\delta G$ holds for another \ecomp{} $\delta G$ of $G$, then any witness $f\colon \delta G\to\alpha G$ of this is required to fix $\Omega$ as well as $G$.
Whenever $\alpha G$ is an \ecomp{} of $G$ and $\alpha G\setminus\cV G$ is a singleton, we call $\alpha G$ a \textit{one-point} \ecomp{}\index{one-point \ecomp{}} of $G$.
For every $X\in\cX$ we denote by $c_X$\index{$c_X$} the map sending each $C\in\cC_X$ to the principal ultrafilter on $\cC_X$ generated by $\{C\}$. Furthermore, in this chapter every $\cC_X$ carries the discrete topology.

\subsection{Introduction}

As the graph depicted in Fig.~\ref{fig:tangleExample} (p.~\pageref{fig:tangleExample}) shows, there exist graphs $G$ for which $\TC$ is not the coarsest \ecomp{}: 
Indeed, if $G$ is the graph from Fig.~\ref{fig:tangleExample}, then $G$ has a one-point \ecomp{} reflecting the structure of $G$ while $\TC$ adds $2^{2^{\aleph_0}}=2^{\A{c}}\ge\aleph_2$ many ultrafilter tangles to $\CG\cup\Omega$.
But as a $T_{\aleph_0}$ shows, it is not always possible to just take the one-point \ecomp{} of $G$ in order to obtain a strictly coarser\footnote{
By `strictly coarser' we mean `coarser but not topologically equivalent'.
} \ecomp{} than $\TC$, since $\cV (T_{\aleph_0})$ does not even have a one-point compactification (see Proposition~\ref{OnePointCompVG} for a characterisation of the graphs $G$ admitting a one-point \ecomp{}). 
Hence Diestel \cite{EndsAndTangles} asked:
\begin{enumerate}
\item For which [graphs] $G$ is $\TC$ the coarsest [$\Omega$-]compactification [...]?\footnote{In his paper, Diestel asked `For which $G$ is $\TC$ the coarsest compactification in which its ends appear as distinct points?'.
In order to obviate coarser versions of the topology of $\cV G$ for some special $G$ from leading this question ad absurdum, we interpreted it as stated in (i).}
\item If it is not, is there a unique such topology [i.e. \ecomp{}],\\and is there a canonical way to obtain it from $\TC$?
\end{enumerate}

In order to tackle these questions, we proceed as follows:
In Section~\ref{FGconstruction} we construct for arbitrary $G$ an inverse system $\{\aF_X,\af_{X',X},\cX\}$ of T\textsubscript{2}-compactifica-\\tions $(\aF_X,c_X)$ of the $\cC_X$ whose inverse limit $\aF=\invLim\aF_X$ we use to obtain an \ecomp{} $\FG$ of $G$ in a similar way Diestel used $\cU$ to compactify $G$. The purpose of $\FG$ on the one hand is to serve as a witness when we characterise those graphs $G$ for which $\TC$ is not the coarsest \ecomp{}, and on the other hand it is our candidate for a positive answer regarding the first part of the second question. But we should have a closer look at the second question itself before we consider possible solutions.

Obviously, the second question tacitly assumes that any answer imposes some further meaningful conditions on the \ecomp{}s considered, as the example graph $G$ from Fig.~\ref{fig:FG3} shows.
\begin{figure}[t]
    \centering
    \def\svgwidth{\columnwidth}
    \scalebox{.75}{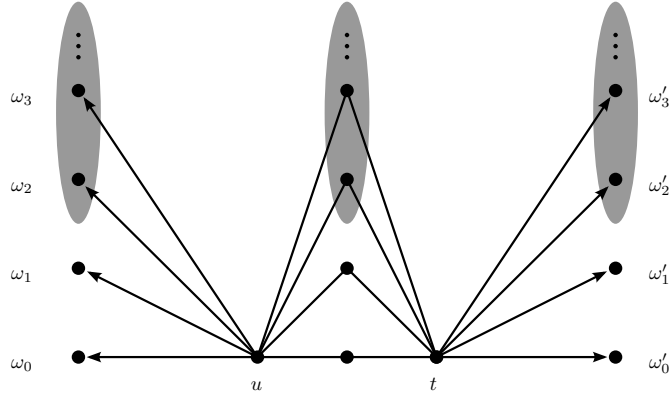}
    \caption{A graph $G$ whose one-point \ecomp{} does not reflect its structure.}
    \label{fig:FG3}
\end{figure}
By Proposition~\ref{OnePointCompVG} we know that this graph has a one-point \ecomp{}, so $\TC$ clearly is not the coarsest \ecomp{} of $G$.
But from an intuitive point of view, the one-point \ecomp{} of $G$ does not reflect the structure of the graph at all. Instead of one point, three points as indicated by the grey ovals in the drawing feel far more intuitive.
Indeed, if $x$ is one of $u$ and $t$, then $x$ splits up the drawing of the graph into a left hand side and a right hand side, which together induce a separation $s_x$ of the graph whose separator consists precisely of $x$.

The tangle compactification of $G$ respects $s_u$ and $s_t$ in that every $\aleph_0$-tangle lives on precisely one side of each. Since the two separations are nested, say $\vec{s_u}\le\vec{s_t}$ as in the drawing, they partition $\Theta$ into three classes $\Theta^-$, $\Theta^\circ$ and $\Theta^+$.
Here, the tangles in $\Theta^-$ are those living on the small side of $\vec{s_u}$, the tangles in $\Theta^+$ are those living on the big side of $\vec{s_t}$, and the tangles in $\Theta^\circ$ are those living in on the big side of $\vec{s_u}$ and on the small side of $\vec{s_t}$.
But $s_u$ and $s_t$ are not the only finite order separations of $G$ with two infinite sides, so $G$ has more than three $\aleph_0$-tangles. 
Hence our intuition suggests that some finite order separations reflect the structure of $G$ better than others. So what makes our two separations so special? The answer is rather simple: both separations respect the $\asymp$-classes\footnote{Recall that the equivalence relation $\asymp$ was defined on $\Upsilon$ by letting $\upsilon\asymp\upsilon'$ whenever $X_\upsilon=X_{\upsilon'}$ holds.}
which turn out to be $\Theta^-$, $\Theta^\circ$ and $\Theta^+$. Furthermore, we have $\crit(\cX)=\{\{u\},\{t\},\{u,t\}\}$, and $|\crit(\cX)|=|\Theta/{\asymp}|=3$ gives an indication that these might be related; but this was to be expected somewhat (in fact, we will see that the map $[\upsilon]_\asymp\mapsto X_\upsilon$ defines a bijection between $\Upsilon/{\asymp}$ and $\crit(\cX)$, see Corollary~\ref{UltrafilterCrit}).

Here, my intuition tells me that it might be worth to have a closer look at the quotient $\TC/{\asymp}$. 
Surprisingly, I did not find graph theoretical indications substantiating my intuition, let alone graph theoretical reasons to consider certain \ecomp{}s of $\CG$ revealing $\TC/{\asymp}$ as the coarsest one.
That is why I try a more technical approach for a change.
Remember that, on the one hand $\cU$ is a useful technical description of $\Theta$, but on the other hand $\cU$ can be understood as a generalisation of $\Omega$ (with the ends viewed as directions).
Therefore, if we manage to abstract the essential ideas behind the inverse system $\{\cU_X,f_{X',X},\cX\}$ generalising $\{\cC_X,\phi_{X',X},\cX\}$, then this might lead to meaningful conditions regarding question (ii), and maybe $\TC/{\asymp}$ turns out to be the coarsest \ecomp{} satisfying these.

For this, let us take a closer look at how $\cU$ manages to induce an \ecomp{} of $G$. 
As is well known, $\Omega$ is the inverse limit of the inverse system $\{\cC_X,\phi_{X',X}\}$. Hence, if every $\cC_X$ is finite, then $\Omega$ is compact Hausdorff by Lemma~\ref{InvLimCptHD}, and so is $\cV G$ (Observation~\ref{VGcpt}). But as soon as some $\cC_X$ is infinite, the space $\cV G$ is no longer compact. 
Let us have a closer look at the tangle compactification overcoming this hindrance. 
For every $X\in\cX$, the pair $(\cU_X,c_X)$ is the Stone-Čech \HDcomp{} of $\cC_X$, so $\cU=\invLim\cU_X$ is compact Hausdorff by Lemma~\ref{InvLimCptHD}.
For $\cV G$, the space $\Omega$ being compact in general does not imply that $\cV G$ is compact, e.g. consider a $K_{1,\aleph_0}$. This is because every infinite $\cC_X$ gives rise to a `bad' cover of $\cV G$ consisting of open sets (one that has no finite subcover).
Since every $\cU_X$ is compact, this cannot possibly happen in $\TC$. 
Next, we investigate $\Omega\subseteq\cU$. The bonding maps $f_{X',X}$ of the inverse system $\{\cU_X,f_{X',X}\}$ by definition respect the bonding maps $\phi_{X',X}$ of the inverse system $\{\cC_X,\phi_{X',X}\}$ in that for all $X\subseteq X'\in\cX$ the diagram
\begin{center}
\begin{tikzpicture}[pil/.style={
           ->,
           thin,
           shorten <=2pt,
           shorten >=2pt,}]
\node (1) {$\cC_X$};
\node (2) [right=of 1]{$\cC_{X'}$};
\node (3) [below=of 2]{$\cU_{X'}$};
\node (4) [below=of 1]{$\cU_X$};
\path[-stealth]
(2) edge node [above] {$\phi_{X',X}$} (1)
(3) edge node [below] {$f_{X',X}$} (4)
(2) edge node [right] {$c_{X'}$} (3)
(1) edge node [left] {$c_X$} (4);
\end{tikzpicture}
\end{center}
commutes, so as a result $\cU$ includes a homeomorphic copy of $\Omega$. Capturing these properties formally leads us to the notion of a $\cC$-system: A $\cC$-system is an inverse system of \HDcomp{}s of the $\cC_X$ whose bonding maps respect the maps $\phi_{X',X}$, and on the class of all $\cC$-systems (for fixed $G$) we define a natural order-relation $\le_{\cC}$. 
We prove, making use of the fact that $(\cU_X,c_X)$ is the Stone-Čech \HDcomp{} of $\cC_X$ for every $X\in\cX$, that $\cC_{\cU}=\{(\cU_X,c_X),f_{X',X}\}$ is the greatest $\cC$-system with respect to $\le_{\cC}$ and up to $\cC$-equivalence (this will be introduced formally in Section~\ref{Minimality}). Furthermore, every $\cC$-system gives rise to an \ecomp{} of $G$ in the same way $\cC_{\cU}$ yields $\TC$.
Hence the following two questions arise:
\begin{enumerate}
\item[(ii$'$)] Is there a coarsest \ecomp{} induced by a least $\cC$-system?
\item[(iii)] If there is, can it be used as a witness regarding question (i)?
\end{enumerate}

As our first main result of this chapter, we will show in Theorem~\ref{ECOminimal} that $\{(\aF_X,c_X),\af_{X',X}\}$ is indeed the least $\cC$-system (up to $\cC$-equivalence), and $\FG$ is the coarsest \ecomp{} of $G$ induced by a $\cC$-system while $\TC$ is the finest \ecomp{} of $G$ induced by a $\cC$-system. This settles question (ii$'$).

Furthermore, as our second main result of this chapter, we will show in Theorem~\ref{ECOequivalent} that $\FG$ and $\TC$ are topologically equivalent \ecomp{}s if and only if every $\cC_X$ is finite. In particular, $\TC$ is the coarsest \ecomp{} of the graph $G$ if and only if every $\cC_X$ is finite (i.e. $\TC$ coincides with $\cV G$). This settles question (i) and (iii).

In Section~\ref{AnaturalQuotient} we show that $\FG$ is homeomorphic to $\TC/{\asymp}$ (the quotient we deemed worth a second look earlier) by a homeomorphism fixing $\Omega$, and we give a description of $\asymp$ solely in terms of tangles. Together with the answer of (ii$'$) this settles question (ii).

As it turns out, it is even possible to describe $\aF$ using $\aleph_0$-tangles, simply by adjusting the underlying separation system (see Section~\ref{FGandTangles}).

Finally, in Section~\ref{FGinvLimsubsec} we show that there is an inverse subsystem of the inverse system $\{G_\gamma,f_{\gamma',\gamma},\Gamma\}$ (from Chapter~\ref{tangleInvLim}) whose inverse limit describes the space $\FG$.

\newpage
\subsection{The one-point $\Omega$-compactification of $G$}

If $G$ is a graph, then $\cV G$ is compact if and only if every $\cC_X$ is finite (Observation~\ref{VGcpt}). Hence $G$ has a trivial \ecomp{} if and only if every $\cC_X$ is finite.
Now if $G$ is a graph with not all $\cC_X$ finite, then $G$ in general does not have a one-point \ecomp{}. In this section we characterise the graphs $G$ admitting a one-point \ecomp{}.\\

Recall that by Theorem~\ref{OnePtCptIFF} a topological space $T$ has a one-point \HDcomp{} if and only if $T$ is locally compact and Hausdorff, but not compact.
We start with a characterisation of the graphs admitting a locally compact end space:

\begin{lemma}
For every graph $G$ the following are equivalent:
\begin{enumerate}
\item $\Omega$ is a locally compact subspace of $\cV G$.
\item For every end $\omega$ of $G$ there is some $X\in\cX$ such that for all $X'\in\lfloor X\rfloor_{\cX}$ only finitely many components $C$ of $G-X'$ with $C\subseteq C(X,\omega)$ contain ends of $G$.
\end{enumerate}
\end{lemma}
\begin{proof}
(i)$\to$(ii). 
Assume for a contradiction that there is some end $\omega$ of $G$ for which (ii) fails. Since $\Omega$ is locally compact, we find some compact $A\subseteq\Omega$ together with an open neighbourhood $O$ of $\omega$ in $\Omega$ such that $O\subseteq A$. Without loss of generality $O$ is induced by a basic open neighbourhood $\hat{C}(X,\omega)$ for some $X\in\cX$. Since (ii) fails for $\omega$ we find some $X'\in\lfloor X\rfloor_{\cX}$ such that infinitely many components $C$ of $G-X'$ with $C\subseteq C(X,\omega)$ contain at least one end of $G$. Then the collection $\{\Omega(X',C)\,|\,C\in\cC_{X'}\}$ is an open cover of $\Omega$ which has no finite subcover of $\Omega\cap\hat{C}(X,\omega)$. In particular, it is an open cover of $A$ that has no finite subcover, contradicting the compactness of $A$.

(ii)$\to$(i).
In order to show that $\Omega(G)$ is locally compact, let $\omega$ be an end of $G$.
Our task is to find a compact neighbourhood of $\omega$ in $\Omega$.
Let $A:=\Omega\cap\hat{C}(X,\omega)$ for some $X\in\cX$ as in (ii), and note that $A$ is a closed neighbourhood of $\omega$ in $\Omega$.
For every $X'\in\lfloor X\rfloor_{\cX}$ let
\begin{align*}
\cD_{X'}:=\big\{C\in\phi_{X',X}^{-1}(C(X,\omega))\,\big|\,\Omega(X',C)\neq\emptyset\big\}
\end{align*}
carry the discrete topology and note that $\cD_{X'}$ is finite and non-empty by choice of $X$. Then by Lemma~\ref{InvLimCptHD} the inverse system 
$\{\cD_Y,\phi_{Y',Y}\rest\cD_{Y'},\lfloor X'\rfloor_{\cX}\}$ has a compact inverse limit $\cI$ whose elements extend precisely to the ends of $G$ in $\hat{C}(X,\omega)$. 
In particular, it is easy to see that $\cI$ and $A$ are even homeomorphic, so $A$ is compact as desired.
\end{proof}

As the following example shows, a locally compact end space in general does not suffice to ensure the existence of a one-point \ecomp{} of $\cV G$: indeed, if $G$ is the graph depicted in Fig.~\ref{fig:OmLocCompVGnot}, then $\cV G$ obviously is not compact, and every open neighbourhood of the sole end of the graph includes a copy of $\cV G$. Hence $\cV G$ and $\cV G\setminus\mathring{E}$ are not locally compact at the sole end of the graph, so $\cV G\setminus\mathring{E}$ has no one-point \HDcomp{} by Theorem~\ref{OnePtCptIFF}.

\begin{figure}[H]
    \centering
    \def\svgwidth{\columnwidth}
    \scalebox{.75}{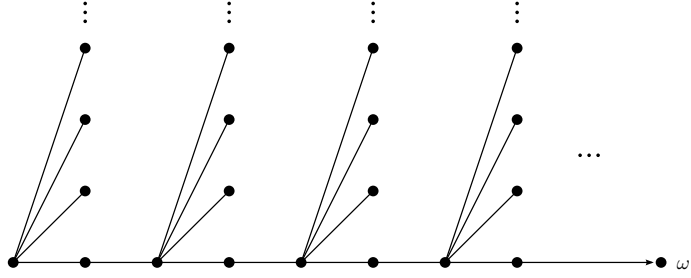}
    \caption[]{A graph $G$ with $\Omega(G)$ locally compact, but $\cV G$ not.\footnotemark}
    \label{fig:OmLocCompVGnot}
\end{figure}
\footnotetext{Interestingly, this graph looks precisely like the topological space from \cite[Ex. 27.15 and Fig. 27.3]{Willard}.}

\begin{proposition}\label{OnePointCompVG}
For every graph $G$ the following are equivalent:
\begin{enumerate}
\item $G$ has a one-point \ecomp{}.
\item There is some $Y\in\cX$ with $\cC_Y$ infinite and for every end $\omega$ of $G$ there is some $X\in\cX$ such that for all $X'\in\lfloor X\rfloor_{\cX}$ the set $\phi_{X',X}^{-1}(C(X,\omega))$ is finite. 
\end{enumerate}
\end{proposition}
\begin{proof}
(i)$\to$(ii). 
If $\alpha G$ is a one-point \ecomp{} of $G$, then $\alpha G\setminus\mathring{E}$ is a one-point \HDcomp{} of $\cV G\setminus\mathring{E}$.
It follows from Theorem~\ref{OnePtCptIFF} that $\cV G\setminus\mathring{E}$ is locally compact but not compact, so Observation~\ref{VGcpt} implies that there is some $Y\in\cX$ with $\cC_Y$ infinite. 
Now assume for a contradiction that there is some $\omega\in\Omega$ not satisfying the condition from (ii). 
Since $\cV G\setminus\mathring{E}$ is locally compact, we find a compact neighbourhood $A$ of $\omega$ in $\cV G\setminus\mathring{E}$. 
Pick a basic open neighbourhood $O=\hat{C}(X,\omega)$ of $\omega$ in $\cV G$ with $O\setminus\mathring{E}\subseteq A$. Since (ii) fails, there is some $X'\in\lfloor X\rfloor_{\cX}$ with $\phi_{X',X}^{-1}(C(X,\omega))$ infinite. Then the collection 
\begin{align*}
\big\{\overline{C}\setminus\mathring{E}\,\big|\,C\in\cC_{X'}\big\}\cup\big\{\{x\}\,\big|\,x\in X'\big\}
\end{align*}
is a cover of $\cV G\setminus\mathring{E}$ of open sets which has no finite subcover of $A$, contradicting the compactness of $A$.

(ii)$\to$(i). Since there is some $Y\in\cX$ with $\cC_Y$ infinite, we know that $\cV G$ is not compact.
Let $\ast$ be a point that is not in $G\cup\Omega$, and extend $\cV G$ to a topological space $\alpha G=\cV G\uplus\{\ast\}$ by declaring as open for every $X\in\cX$ and $\cC\subseteq\cC_X$ the set
\begin{align*}
O^\ast(X,\cC):=\medcup\cC\cup\mathring{E}(X,\medcup\cC)\cup\Omega(X,\cC)\cup\{\ast\}
\end{align*}
whenever $\cC$ is cofinite in $\cC_X$ and the set
\begin{align*}
\cK_X:=\big\{C\in\cC_X\,\big|\,\exists X'\in\lfloor X\rfloor_{\cX}:\phi_{X',X}^{-1}(C)\text{ is infinite}\big\}
\end{align*}
is included in $\cC$, and taking on $\alpha G$ the topology this generates. To see that this yields a topology, it suffices to show that for every two open neighbourhoods $O^\ast (X,\cC)$ and $O^\ast(Y,\cD)$ of $\ast$ there is some open neighbourhood of $\ast$ included in their intersection.
Due to $\cK_X\subseteq\cC$ we know that $\cC':=\phi_{X\cup Y,X}^{-1}(\cC)$ is cofinite in $\cC_{X\cup Y}$. Furthermore, $\cK_{X\cup Y}\subseteq\cC'$ follows from $\cK_{X\cup Y}\subseteq\phi_{X\cup Y,X}^{-1}(\cK_X)$ and the choice of $\cC'$. Similarly, $\cD':=\phi_{X\cup Y,Y}^{-1}(\cD)$ is cofinite in $\cC_{X\cup Y}$ and $\cK_{X\cup Y}\subseteq\cD$, so
\begin{align*}
\ast\in O^\ast(X\cup Y,\cC'\cap\cD')\subseteq O^\ast(X,\cC)\cap O^\ast(Y,\cD)
\end{align*}
holds as desired. Clearly, $\cV G$ is a dense subspace of $\alpha G$. Furthermore, $\alpha G\setminus\mathring{E}$ is Hausdorff: For this, let $\omega$ be an end of $G$, and pick $X\in\cX$ as in (ii). Then $\hat{C}(X,\omega)$ and $O^\ast(X,\cC_X\setminus\{C(X,\omega)\})$ are disjoint open neighbourhoods of $\omega$ and $\ast$, respectively. For other points this is clear, so $\alpha G\setminus\mathring{E}$ is Hausdorff as claimed. It remains to show that $\alpha G$ is compact.
For this, let $O=O^\ast(X,\cC)$ be an arbitrary open neighbourhood of $\ast$. It suffices to show that $\alpha G\setminus O$ is compact. Let $H:=G-\bigcup\cC$. Clearly, $\cV H$ is homeomorphic to $\alpha G\setminus O$, so it suffices to show that $\cV H$ is compact. Since $\cK_X$ is included in $\cC$ and $\cC_X\setminus\cC$ is finite, for every $X'\in\lfloor X\rfloor_{\cX}$ the set $\phi_{X',X}^{-1}(\cC_X\setminus\cC)$ is also finite. Hence $\cV H$ is compact by Observation~\ref{VGcpt} as desired.
\end{proof}

\subsection{An inverse limit of Hausdorff compactifications}\label{FGconstruction}

In order to create a coarser \ecomp{} than $\TC$, we create an inverse system $\{\aF_X,\af_{X',X},\cX\}$\index{$\{\aF_X,\af_{X',X},\cX\}$} of \HDcomp{}s $(\aF_X,c_X)$ of the $\cC_X$, which we then use to obtain an \ecomp{} $\FG=G\cup\aF$ of $G$.\\

If $M$ is a set, then we denote by $\text{cof}(M)$ the cofinite filter on $M$. 
Recall that for two sets $\cA,B$ with $\cA\subseteq 2^B$ we denote by $\langle \cA\rangle_B$  the collection of all supersets $B'\subseteq B$ of elements of $\cA$, and we call $\langle \cA\rangle_B$ the \textit{set-theoretic up-closure of} $\cA$ \textit{in} $2^B$.
Now we start the construction: For every $X\in\cX$ and every $Y\in \crit(X)$ let $\cF_X(Y)$ be the set-theoretic up-closure of $\text{cof}(\cC_X(Y))$ in $2^{\cC_X}$, i.e.\index{$\cF_X(Y)$}
\begin{align*}
\cF_X(Y)=\langle\,\text{cof}(\cC_X(Y))\,\rangle_{\cC_X}.
\end{align*}
In particular, $\cF_X(Y)$ is again a filter on $\cC_X$.
Recall that, for every $X\in\cX$, we chose $c_X$ to be the function with domain $\cC_X$ which sends every $C\in\cC_X$ to the principal ultrafilter on $\cC_X$ generated by $\{C\}$.
Write\index{$\aF_X$, $\aF_X^{\text{p}}$, $\aF_X^\ast$}
\begin{align*}
\A{F}_X^\text{p}&=c_X[\cC_X],\\
\A{F}_X^\ast&=\{\cF_X(Y)\,|\,Y\in \crit(X)\},\\
\aF_X&=\A{F}_X^\text{p}\uplus\A{F}_X^\ast.
\end{align*}
We endow the spaces $\aF_X$ with the standard topology whose basic open sets are of the form\index{$\cO_{\aF_X}(\cC)$}
\begin{align*}
\cO_{\aF_X}(\cC):=\{\cF\in\aF_X\,|\,\cC\in\cF\},
\end{align*}
one for each $\cC\subseteq\cC_X$. 
This topology has a basis which suits our needs far better:
For every $X\in\cX$ let $\aC_X$\index{$\aC_X$} be the collection of all singleton subsets of $\cC_X$ and all sets $\cC$ which are cofinite in $\cC_X(Y)$ for some $Y\in \crit(X)$. Let $\aB_X$ be the collection of all sets $\cO_{\aF_X}(\cC)$ with $\cC\in\aC_X$. Then\index{$\aB_X$}

\begin{lemma}\label{FXbasis}
$\aB_X$ is a basis for the topology of $\aF_X$.
\end{lemma}
\begin{proof}
Given $\cF\in\aF_X$ and $\cC\subseteq\cC_X$ with $\cF\in\cO_{\aF_X}(\cC)$ we have to find some $\cD\in\aC_X$ with $\cF\in\cO_{\aF_X}(\cD)\subseteq\cO_{\aF_X}(\cC)$. If $\cF=c_X(C)$ holds for some $C\in\cC_X$, then we may choose $\{C\}$ as $\cD$. Otherwise $\cF$ is of the form $\cF_X(Y)$ for some $Y\in \crit(X)$. Then $\cC\in\cF$ implies that $\cC\cap\cC_X(Y)$ is cofinite in $\cC_X(Y)$, so we may pick $\cD=\cC\cap\cC_X(Y)$.
\end{proof}

Whenever we speak of a basic open neighbourhood of some element of $\aF_X$, we mean `basic' with respect to $\aB_X$.

The $\aF_X$ compactify the $\cC_X$ as promised:

\begin{lemma}\label{aCXcomp}
$(\A{F}_X,c_X)$ is a \HDcomp{} of $\cC_X$, for every $X\in\cX$.
\end{lemma}
\begin{proof}
Since we equipped $\cC_X$ with the discrete topology and $\A{F}_X$ also induces the discrete topology on $\aF_X^{\text{p}}=c_X[\cC_X]$, the map $c_X$ clearly is an embedding. Clearly, $c_X[\cC_X]$ is dense in $\aF_X$, and $\aF_X$ is Hausdorff since for every two distinct $Y$ and $Y'$ in $\crit(X)$ we have $\cC_X(Y)\cap\cC_X(Y')=\emptyset$.
Finally we show that $\A{F}_X$ is compact: Let $\{O_i\,|\,i\in I\}$ be any cover of $\A{F}_X$ by open sets $O_i=\cO_{\A{F}_X}(\cC_i)$. For every $Y\in \crit(X)$ pick some $i_Y\in I$ such that $\cF_X(Y)\in\cO_{i_Y}$ and note that $\cC_X(Y)\setminus\cC_{i_Y}$ is finite. Hence $\cC_X\setminus\bigcup_{Y\in \crit(X)}\cC_{i_Y}$ is finite, and we are done.
\end{proof}

\begin{lemma}\label{FXtotDisc}
The spaces $\aF_X$ are totally disconnected.
\end{lemma}
\begin{proof}
Let $\cA$ be a subset of $\aF_X$ with $|\cA|>1$.
Pick two distinct elements $\cF$ and $\cF'$ of $\cA$.
If one of $\cF$ and $\cF'$ is principal, say $\cF=c_X(C)$ for some $C\in\cC_X$, the open sets $\cO_{\aF_X}(\{C\})$ and $\cO_{\aF_X}(\cC_X\setminus\{C\})$ induce an open bipartition of $\cA$.
Otherwise we find distinct elements $Y$ and $Y'$ of $\crit(X)$ with $\cF=\cF_X(Y)$ and $\cF'=\cF_X(Y')$. 
Then the open sets $\cO_{\aF_X}(\cC_X(Y))$ and $\cO_{\aF_X}(\cC_X\setminus\cC_X(Y))$ induce an open bipartition of $\cA$.
\end{proof}

In order to obtain an inverse system $\{\aF_X,\af_{X',X},\cX\}$ from the spaces $\aF_X$ we define compatible bonding maps $\af_{X',X}\colon \aF_{X'}\to\aF_X$ for all $X\subseteq X'\in\cX$ as follows: Let $\cF'\in\aF_{X'}$ be given.\index{$\af_{X',X}$}

Suppose first that $\cF'$ is of the form $c_{X'}(C')$ for some $C'\in\cC_{X'}$. Then we put $\af_{X',X}(\cF')=c_X(\phi_{X',X}(C'))$, i.e. $\af_{X',X}(\cF')=c_X(C)$ where $C$ is the unique component of $G-X$ including $C'$.

For the second case suppose that $\cF'$ is of the form $\cF_{X'}(Y)$ for some $Y\in \crit(X')$. If $Y\subseteq X$ holds, then we also have $Y\in \crit(X)$ witnessed by $\cC_{X'}(Y)\subseteq\cC_X(Y)$,
 so we may put $\af_{X',X}(\cF')=\cF_X(Y)$. Otherwise $Y\not\subseteq X$ implies that there is some unique component $C$ of $G-X$ including $\bigcup\cC_{X'}(Y)$ since every vertex in $Y\setminus X$ has a neighbour in each element of $\cC_{X'}(Y)$, so we put $\af_{X',X}(\cF')=c_X(C)$. This completes the definition of the bonding maps. It remains to show that they are continuous:

\begin{lemma}\label{Afcts}
The bonding maps $\A{f}_{X',X}\colon \A{F}_{X'}\to\A{F}_X$ are continuous.
\end{lemma}
\begin{proof}
Let $\cF'\in\A{F}_{X'}$ be given together with some basic open neighbourhood $\cO_{\A{F}_X}(\cC)$ of $\cF:=\A{f}_{X',X}(\cF')$.

Suppose first that $\cF'$ is of the form $c_{X'}(C')$ for some $C'\in\cC_{X'}$, and hence $\cF$ is of the form $c_X(C)$ for $C=\phi_{X',X}(C')$. Then $\cO_{\A{F}_{X'}}(\{C'\})$ is an open neighbourhood of $\cF'$ which $\A{f}_{X',X}$ maps to $\cO_{\A{F}_X}(\{C\})\subseteq \cO_{\A{F}_X}(\cC)$.

For the second case suppose that $\cF'=\cF_{X'}(Y)$ for some $Y\in \crit(X')$. 
If $Y\not\subseteq X$ then there exists a unique component $C$ of $G-X$ including $\bigcup\cC_{X'}(Y)$, so $\cF$ is of the form $c_X(C)$. In particular, $\cO_{\A{F}_{X'}}(\cC_{X'}(Y))$ is an open neighbourhood of $\cF'$ which $\A{f}_{X',X}$ maps to $\cO_{\A{F}_X}(\{C\})\subseteq \cO_{\A{F}_X}(\cC)$. 
Otherwise $Y\subseteq X$ holds, then we have $\cF=\cF_X(Y)$ and moreover $\cC$ is a cofinite subset of $\cC_X(Y)$ since $\cO_{\aF_X}(\cC)$ is basic open.
In particular, $\cC':=\cC\cap\cC_{X'}(Y)$ is a cofinite subset of $\cC_{X'}(Y)$ with $\phi_{X',X}[\cC']\subseteq\cC$. Hence $\cO_{\A{F}_{X'}}(\cC')$ is an open neighbourhood of $\cF'$ which $\A{f}_{X',X}$ maps into $\cO_{\A{F}_X}(\cC)$.
\end{proof}

Therefore, $\{\aF_X,\af_{X',X},\cX\}$ is an inverse system, and we write\index{$\aF$}
\begin{align*}
\aF=\invLim(\aF_X\,|\,X\in\cX).
\end{align*}
\begin{proposition}\label{FGtotDisc}
The topological space $\aF$ is compact, Hausdorff and totally disconnected.
\end{proposition}

\begin{proof}
According to Lemmas~\ref{aCXcomp}~and~\ref{FXtotDisc} the topological spaces $\aF_X$ are compact, Hausdorff and totally disconnected. By Lemma~\ref{InvLimTotDisc}, so is $\aF$.
\end{proof}

Analogously to \cite[Lemma 3.2]{EndsAndTangles}, we note a very useful fact:
\begin{lemma}\label{ECObondingInverse}
For all $X\subseteq X'\in\cX$ the map $\af_{X',X}$ restricts to a 
bijection between $\af_{X',X}^{-1}(\aF_X^\ast)\subseteq\aF_{X'}^\ast$ and $\aF_X^\ast$ with inverse $\ag_{X,X'}$.\qed
\end{lemma}

\begin{corollary}\label{Fextension}
For every $X\in\cX$, each $\cF\in\aF_X^\ast$ uniquely extends to an element of $\aF$.\qed
\end{corollary}

For every $\A{v}=(\cF_X\,|\,X\in\cX)\in\aF$ let\index{$\cX_{\A{v}}$, $X_{\A{v}}$}
\begin{align*}
\cX_{\A{v}}:=\{X\in\cX\,|\,\cF_X\notin\aF_X^{\text{p}}\},
\end{align*}
i.e. $X\in\cX$ is in $\cX_{\A{v}}$ if and only if there is some $Y\in\crit(X)$ with $\cF_X=\cF_X(Y)$. The following is immediate from Lemma~\ref{ECObondingInverse} and Corollary~\ref{Fextension}:

\begin{proposition}\label{FandY}
If $\A{v}$ is not in $\Omega$, then the set $\cX_{\A{v}}$ has a least element $X_{\A{v}}\in\crit(\cX)$ such that $\cX_{\A{v}}=\lfloor X_{\A{v}}\rfloor_{\cX}$. The set $X_{\A{v}}$ determines $\A{v}$, and $\cF_X=\cF_X(X_{\A{v}})$ holds for every $X\in\cX_{\A{v}}$. Furthermore, for every $Y\in\crit(\cX)$ there is a unique $\A{v}\in\aF$ with $X_{\A{v}}=Y$. This defines a bijection\index{$\chi\colon \crit(\cX)\bij\aF\setminus\Omega$} $\chi\colon \crit(\cX)\bij\aF\setminus\Omega$.\qed
\end{proposition}

Finally, we use $\aF$ to compactifiy $G$ in the same way Diestel used $\cU$ to compactify $G$. 
For this, let $\A{f}_X\colon \A{F}\to\A{F}_X$\index{$\af_X$} send each $(\cF_Y\,|\,Y\in\cX)$ to $\cF_X$, i.e. $\af_X$ is the restriction of the $X$th projection map $\text{pr}_X\colon \prod_{Y\in\cX}\aF_Y\to\aF_X$ to $\aF$.
By Lemma~\ref{invLimBasis} and Lemma~\ref{FXbasis}, the collection of all sets of the form\index{$\cO_{\aF}(X,\cC)$, $\cO_{\FG}(X,\cC)$}
\begin{align*}
\cO_{\aF}(X,\cC):=\af_X^{-1}(\cO_{\aF_X}(\cC))=\{(\cF_Y\,|\,Y\in\cX)\in\aF\,|\,\cC\in\cF_X\}
\end{align*}
with $X\in\cX$ and $\cC\in\aC_X$ forms a basis for the topology of $\aF$. Now consider the 1-complex of $G$ and extend its topology to the topological space\index{$\FG$}
\begin{align*}
\FG=\CG\cup\aF
\end{align*}
by declaring as open, for every $X\in\cX$ and $\cC\in\aC_X$, the set
\begin{align*}
\cO_{\FG}(X,\cC):=\medcup\cC\cup\mathring{E}(X,\medcup\cC)\cup\cO_{\aF}(X,\cC)
\end{align*}
and taking the topology on $\FG$ this generates. Here, the graph $\bigcup\cC$ carries the 1-complex topology. Note that the subspace topology on $\aF\subseteq\FG$ is our original topology on $\aF$.
Let $\aB$ denote the basis we used for $\FG$.\index{$\aB$}
That $\FG$ is indeed an \ecomp{} of $G$ will follow from the more general Theorems~\ref{alphaGcomp}~and~\ref{FGcompHDetc}.

\subsection{$\Omega$-compactifications induced by $\cC$-systems}\label{Minimality}

In this section we study $\cC$-systems which generalise the known concepts of directions and the inverse systems $\{\cU_X,f_{X',X},\cX\}$ and $\{\aF_X,\af_{X',X},\cX\}$ from the previous section. Furthermore, we study \ecomp{}s induced by $\cC$-systems in order to answer questions (i) and (ii$'$) from the introduction of this chapter.\\

Fix any infinite graph $G$.
For every $X\in\cX$ denote by $\cC_X^-$\index{$\cC_X^-$} the set of all components of $G-X$ whose neighbourhood is not in $\crit(X)$, i.e. 
\begin{align*}
\cC_X^-=\cC_X-\bigcup_{Y\in \crit(X)}\cC_X(Y),
\end{align*}
and note that this set must be finite due to pigeonhole principle.
If $(\alpha(\cC_X),\alpha_X)$ is a \HDcomp{} of $\cC_X$ for every $X\in\cX$ and $\{\alpha(\cC_X),\A{a}_{X',X}\}$ is an inverse system whose bonding maps respect the maps $\phi_{X',X}$ in that
\begin{align}\label{CompactificationBondingMapsRest}
\A{a}_{X',X}\circ\alpha_{X'}=\alpha_X\circ\phi_{X',X}
\end{align}
holds for all $X\subseteq X'\in\cX$, i.e. the diagram
\begin{center}
\begin{tikzpicture}[pil/.style={
           ->,
           thin,
           shorten <=2pt,
           shorten >=2pt,}]
\node (1) {$\;\;\cC_X\;\;$};
\node (2) [right=of 1]{$\;\;\cC_{X'}\;\;$};
\node (3) [below=of 1]{$\alpha(\cC_X)$};
\node (4) [below=of 2]{$\alpha(\cC_{X'})$};
\path[-stealth]
(2) edge node [above] {$\phi_{X',X}$} (1)
(4) edge node [below] {$\A{a}_{X',X}$} (3)
(1) edge node [left] {$\alpha_X$} (3)
(2) edge node [right] {$\alpha_{X'}$} (4);
\end{tikzpicture}
\end{center}
commutes, then we call $\{(\alpha(\cC_X),\alpha_X),\A{a}_{X',X}\}$ a $\cC$-\textit{system} (of $G$).\index{$\cC$-system}
Note that\index{$\cC_{\aF}$}
\begin{align*}
\cC_{\aF}:=\{(\A{F}_X,c_X),\A{f}_{X',X}\}
\end{align*}
is a $\cC$-system
by Lemmas~\ref{aCXcomp}~and~\ref{Afcts}, and\index{$\cC_{\cU}$}
\begin{align*}
\cC_{\cU}:=\{(\cU_X,c_X),f_{X',X}\}
\end{align*}
is a $\cC$-system, too.
If every $\cC_X$ is finite, then $\{(\cC_X,\id_{\cC_X}),\phi_{X',X}\}$ is a $\cC$-system, too.
As our first main result of this section, we generalise Diestel's construction of the tangle compactification and show that every $\cC$-system gives rise to an \ecomp{} of the graph $G$:

If $\{(\alpha(\cC_X),\alpha_X),\A{a}_{X',X}\}$ is a $\cC$-system, put $\cI_\alpha:=\invLim\alpha(\cC_X)$\index{$\cI_\alpha$} and for every $X\in\cX$ let $\pi_X^\alpha\colon \cI_\alpha\to\alpha(\cC_X)$ be the restriction of the $X$th projection map $\text{pr}_X\colon \prod_{Y\in\cX}\alpha(\cC_Y)\to\alpha(\cC_X)$. Clearly, $\pi_X^\alpha$ is continuous. If clear from context, we write $\pi_X$ instead of $\pi_X^\alpha$.\index{$\pi_X$, $\pi_X^\alpha$}
Now we extend $G$ (viewed as 1-complex) to a topological space $\alpha G=\CG\cup\cI_\alpha$\index{$\alpha G$} by declaring as open, for all $X\in\cX$ and every open set $O$ of $\alpha(\cC_X)$, the set\index{$\cO_{\alpha G}(X,O)$}
\begin{align*}
\cO_{\alpha G}(X,O):=\medcup\cC\cup\mathring{E}(X,\medcup\cC)\cup(\pi_X^\alpha)^{-1}(O)
\end{align*}
where $\cC=\alpha_X^{-1}(O)$, and taking the topology on $\alpha G$ this generates.

\begin{theorem}\label{alphaGcomp}
If $G$ is any infinite graph, then $\alpha G$ is an \ecomp{} of $G$. In particular, so are $\FG$ and $\TC$.
\end{theorem}
\begin{proof}
First, we show that $G$ is dense in $\alpha G$.
Every open neighbourhood of an element of $\cI_\alpha$ is of the form $\cO_{\alpha G}(X,O)$. Since $\alpha_X[\cC_X]$ is a dense subset of $\alpha(\cC_X)$, we know that $O$ meets $\alpha_X[\cC_X]$, and hence $\bigcup\alpha_X^{-1}(O)$ meets $G$. In particular, $\cO_{\alpha G}(X,O)$ meets $G$, so $G$ is dense in $\alpha G$.

Next, we show that $\alpha G$ is compact.
For this, we mimic the proof of \cite[Theorem 1 (i)]{EndsAndTangles}, replacing \cite[Lemmas 2.3 and 3.7]{EndsAndTangles} by topological arguments. Hence consider any cover $\cO$ of $\alpha G-G$ by open sets $\cO_{\alpha G}(X,O)$. Since the subspace topology of $\cI_\alpha\subseteq\alpha G$ is the original topology of $\cI_\alpha$ and $\cI_\alpha$ is compact by Lemma~\ref{InvLimCptHD}, the cover $\cO$ has a finite subset of the form
\begin{align*}
\cO'=\{\cO_{\alpha G}(X,O_X)\,|\,X\in\cX'\}
\end{align*}
(with $\cX'\subseteq\cX$ finite) that covers $\cI_\alpha$. 
Our aim is to show that $G\setminus\bigcup\cO'$ is the 1-complex of a finite graph: then $G\setminus\bigcup\cO'$ will be compact as desired. For this, put $X'=\bigcup\cX'$, and
for each $X\in\cX'$ let $O_X'$ be the open set $\A{a}_{X',X}^{-1}(O_X)$ of $\alpha(\cC_{X'})$. 
We claim that for each $X\in\cX'$ the inclusion 
\begin{align}\label{alphaCompSecTimeUse}
\cO_{\alpha G}(X,O_X)\supseteq\cO_{\alpha G}(X',O_X')
\end{align}
holds:
Indeed, 
\begin{align}\label{alphaCompOneTimeUse}
\pi_X^{-1}(O_X)=\pi_{X'}^{-1}(O_X')
\end{align}
holds by choice of $O_X'$ and since the diagram
\begin{center}
\begin{tikzpicture}[pil/.style={
           ->,
           thin,
           shorten <=2pt,
           shorten >=2pt,}]
\node (1) {$\cI_\alpha$};
\node (2) [right=of 1]{$\alpha(\cC_{X'})$};
\node (3) [below=of 2]{$\alpha(\cC_X)$};
\path[-stealth]
(1) edge node [above] {$\pi_{X'}$} (2)
(1) edge node [below left] {$\pi_X$} (3)
(2) edge node [right] {$\A{a}_{X',X}$} (3);
\end{tikzpicture}
\end{center}
commutes.
According to the definitions of $\cO_{\alpha G}(X,O_X)$ and $\cO_{\alpha G}(X',O_X')$, and due to $X\subseteq X'$, it remains to show that $\bigcup\alpha_X^{-1}(O_X)\supseteq\bigcup\alpha_{X'}^{-1}(O_X')$ holds. Using~(\ref{CompactificationBondingMapsRest}), this is easily calculated:
\begin{align*}
&\medcup\alpha_{X'}^{-1}(O_X')\\
=&\medcup\alpha_{X'}^{-1}(\A{a}_{X',X}^{-1}(O_X))\\
\overset{(\ref{CompactificationBondingMapsRest})}{=}&\medcup\phi_{X',X}^{-1}(\alpha_X^{-1}(O_X))\\
\subseteq&\medcup\alpha_X^{-1}(O_X)
\end{align*}
Hence (\ref{alphaCompSecTimeUse}) holds as claimed, and the sets $\cO_{\alpha G}(X',O_X')$ still cover $\cI_\alpha$ by~(\ref{alphaCompOneTimeUse}).
Now consider the set
\begin{align*}
\cC':=\cC_{X'}\setminus\bigcup_{X\in\cX'}\alpha_{X'}^{-1}(O_X').
\end{align*}
If $\bigcup\cC'$ is finite, then so is $G[X']\cup\bigcup\cC'=G\setminus\bigcup\cO'$, and we are done. 
Hence we assume for a contradiction that $\bigcup\cC'$ is infinite. Consider the closed set
\begin{align*}
A:=\alpha(\cC_{X'})\setminus\bigcup_{X\in\cX'}O_X'
\end{align*}
and for every $Y\in\lfloor X'\rfloor_{\cX}$ set $K_Y=\A{a}_{Y,X'}^{-1}(A)$. In particular, every $K_Y$ is closed in $\alpha(\cC_Y)$, and hence compact. Furthermore, since $\bigcup\cC'$ is infinite, each $\phi_{Y,X'}^{-1}(\cC')$ is non-empty. Together with $\alpha_{X'}[\cC']\subseteq A$ and~(\ref{CompactificationBondingMapsRest}), this implies that every $K_Y$ is non-empty. By Lemma~\ref{InvLimCptHD} we find an element in the inverse limit of the inverse system $\{K_Y,\A{a}_{Y',Y}\rest K_{Y'},\lfloor X'\rfloor_{\cX}\}$, which is easily extended to an element $x$ of $\cI_\alpha$ since $\lfloor X'\rfloor_{\cX}$ is cofinal in $\cX$. In particular, $x\in \cI_\alpha\cap A=\cI_\alpha\setminus\bigcup\cO'$ is a contradiction. Thus $\alpha G$ is compact, so in particular $\alpha G$ is a compactification of (the 1-complex of) $G$.

Since $\cI_\alpha$ is also Hausdorff by Lemma~\ref{InvLimCptHD}, it is easy to see that so is $\alpha G\setminus\mathring{E}$.
It remains to show that $\alpha G$ is a compactification of $\cV G$. We have already seen that $\alpha G$ is compact and $G$ is dense in $\alpha G$.
By (\ref{CompactificationBondingMapsRest}) and abuse of notation we may assume $\Omega\subseteq\cI_\alpha\subseteq\alpha G$. 
Hence it suffices to show that $\alpha G$ induces the right subspace topology on $G\cup\Omega$.
For every $X\in\cX$ and $\omega\in\Omega$ the open neighbourhood $\cO_{\alpha G}(X,\alpha_X(C(X,\omega)))$ of $\omega$ in $\alpha G$ induces on $G\cup\Omega$ the basic open neighbourhood $\hat{C}(X,\omega)$ of $\omega$. Conversely, it is easy to see that every open set $\cO_{\alpha G}(X,O)$ of $\alpha G$ induces on $G\cup\Omega$ a set that is also open in $\cV G$. Therefore, $\cV G$ is a subspace of $\alpha G$.
\end{proof}

In particular, we obtain the following analogue of \cite[Theorem 1]{EndsAndTangles} for $\FG$:

\begin{theorem}\label{FGcompHDetc}
Let $G$ be any graph.
\begin{enumerate}
\item $\FG$ is an \ecomp{} of $G$ and $\FG\setminus G$ is totally disconnected.
\item If $G$ is locally finite and connected, then $\aF=\Omega$ holds and $\FG$ coincides with the Freudenthal compactification of $G$.
\end{enumerate}
\end{theorem}
\begin{proof}
(i) By Theorem~\ref{alphaGcomp} we know that $\FG$ is an \ecomp{} of $G$, and $\FG\setminus G=\aF$ is totally disconnected by Proposition~\ref{FGtotDisc}.

(ii) If $G$ is locally finite, then $\crit(\cX)$ is empty, so $\aF=\Omega$ holds.
\end{proof}

In order to compare $\cC$-systems, we introduce the following notion:\\
If $\cC_\alpha=\{(\alpha(\cC_X),\alpha_X),\A{a}_{X',X}\}$ and $\cC_\delta=\{(\delta(\cC_X),\delta_X),\A{d}_{X',X}\}$ are two $\cC$-systems we write $\cC_\alpha\le_\cC\cC_{\delta}$\index{$\le_{\cC}$}
whenever $(\alpha(\cC_X),\alpha_X)\le (\delta(\cC_X),\delta_X)$ holds for every $X\in\cX$ (witnessed by unique $f_X\colon \delta(\cC_X)\to\alpha(\cC_X)$, see Lemma~\ref{CompactificationComparableUnique}) and
\begin{align}\label{Commute}
f_X\circ\A{d}_{X',X}=\A{a}_{X',X}\circ f_{X'}
\end{align}
holds for all $X\subseteq X'\in\cX$. 
Condition (\ref{Commute}) together with condition (\ref{CompactificationBondingMapsRest}) ensures that the diagram
\begin{center}
\begin{tikzpicture}[pil/.style={
           ->,
           thin,
           shorten <=2pt,
           shorten >=2pt,}]
\node (1) {$\delta(\cC_X)$};
\node (2) [above left=of 1]{$\cC_X$};
\node (3) [right=of 1]{$\delta(\cC_{X'})$};
\node (4) [above right=of 3]{$\cC_{X'}$};
\node (5) [below=of 1]{$\alpha(\cC_X)$};
\node (6) [below=of 3]{$\alpha(\cC_{X'})$};
\path[-stealth]
(2) edge node [below left] {$\delta_X$} (1)
(2) edge[bend right] node [below left] {$\alpha_X$} (5)
(3) edge[double] node [above] {$\A{d}_{X',X}$} (1)
(1) edge[dashed] node [left] {$f_X$} (5)
(3) edge[dashed] node [right] {$f_{X'}$} (6)
(6) edge[double] node [below] {$\A{a}_{X',X}$} (5)
(4) edge[double] node [above] {$\phi_{X',X}$} (2)
(4) edge node [below right] {$\delta_{X'}$} (3)
(4) edge[bend left] node [below right] {$\alpha_{X'}$} (6);
\end{tikzpicture}
\end{center}
commutes for all $X\subseteq X'\in\cX$, so that the mapping\index{$\psi_{\delta\alpha}$}
\begin{equation}\label{CsystemSur}
\begin{alignedat}{2}
\psi_{\delta\alpha}\colon &&\invLim \delta(\cC_X)&\to\invLim \alpha(\cC_X)\\
&&(p_X\,|\,X\in\cX)&\mapsto (f_X(p_X)\,|\,X\in\cX)
\end{alignedat}
\end{equation}
is a well-defined continuous surjection by Lemma~\ref{InvLimCompSurjOnto}.
Furthermore, this yields
\begin{lemma}\label{ECOcompsComparable}
If $\cC_\alpha$ and $\cC_\delta$ are two $\cC$-systems with $\cC_\alpha\le_\cC\cC_\delta$ and $\cC_\alpha$ induces the \ecomp{} $\alpha G$ of $\CG$ whereas $\cC_\delta$ induces the \ecomp{} $\delta G$ of $\CG$, then $\alpha G\le\delta G$ is witnessed by $\psi_{\delta\alpha}\cup\id_{\CG}$.
\end{lemma}
\begin{proof}
Put $\psi=\psi_{\delta\alpha}\cup\id_{\CG}$.
By (\ref{CompactificationBondingMapsRest}) we know that $\psi$ fixes $\Omega$ (by abuse of notation).
Clearly, it suffices to show that $\psi$ is continuous at every $p=(p_X\,|\,X\in\cX)\in\cI_\delta$ (recall that $\cI_\delta=\invLim\delta(\cC_X)$).
For this, consider any basic open neighbourhood $\cO_{\alpha G}(Y,O)$ of $\psi(p)$ in $\alpha G$, and recall that $O$ is open in $\alpha(\cC_Y)$. 
By definition of $\psi$ we have $f_Y(p_Y)\in O$.
Since $f_Y$ is continuous, we find some open neighbourhood $O'$ of $p_Y$ in $\delta(\cC_Y)$ with $f_Y[O']\subseteq O$. 
Then $\cO_{\delta G}(Y,O')$ is an open neighbourhood of $p$ in $\delta G$.
Hence it suffices to show that $\psi$ maps it to $\cO_{\alpha G}(Y,O)$. 
By choice of $O'$ and definition of $\psi_{\delta\alpha}$ we know that $\psi_{\delta\alpha}$ sends the open subset $(\pi_Y^\delta)^{-1}(O')$ of $\cI_\delta$ to the open subset $(\pi_Y^\alpha)^{-1}(O)$ of $\cI_\alpha$, so by definition of $\cO_{\alpha G}(Y,O)$ and $\cO_{\delta G}(Y,O')$ is suffices to show $\delta_Y^{-1}(O')\subseteq\alpha_Y^{-1}(O)$. 
But this is clear from $f_Y[O']\subseteq O$ combined with the commuting diagram
\begin{center}
\begin{tikzpicture}[pil/.style={
           ->,
           thin,
           shorten <=2pt,
           shorten >=2pt,}]
\node (1) {$\cC_Y$};
\node (2) [right=of 1]{$\delta(\cC_Y)$};
\node (3) [below=of 2]{$\alpha(\cC_Y)$};
\path[-stealth]
(1) edge node [above] {$\delta_Y$} (2)
(1) edge node [below left] {$\alpha_Y$} (3)
(2) edge node [right] {$f_Y$} (3);
\end{tikzpicture}
\end{center}
(recall that it commutes since $f_Y$ witnesses $(\alpha(\cC_Y),\alpha_Y)\le(\delta(\cC_Y),\delta_Y)$).
\end{proof}

If both $\cC_\alpha\le_\cC \cC_{\delta}$ and $\cC_{\delta}\le_\cC \cC_\alpha$ hold, we say that $\cC_\alpha$ and $\cC_{\delta}$ are $\cC$-\textit{equivalent}.\index{$\cC$-equivalent}
The following lemma and corollary indicate that this definition is meaningful:

\begin{lemma}\label{CsystemPartialWell}
Let $\cC_\alpha$ and $\cC_\delta$ be two $\cC$-systems with $\cC_\alpha\le_\cC\cC_\delta$ and $\cC_\delta\le_{\cC}\cC_\alpha$. For every $X\in\cX$ suppose that
\begin{enumerate}
\item $f_X\colon \delta(\cC_X)\to\alpha(\cC_X)$ witnesses $(\alpha(\cC_X),\alpha_X)\le (\delta(\cC_X),\delta_X)$, and
\item $g_X\colon \alpha(\cC_X)\to\delta(\cC_X)$ witnesses $(\delta(\cC_X),\delta_X)\le (\alpha(\cC_X),\alpha_X)$.
\end{enumerate}
Let $\psi_{\delta\alpha}$ be as in (\ref{CsystemSur}), and let $\psi_{\alpha\delta}$ be defined analogously using the $g_X$ instead of the $f_X$.
Then $\psi_{\delta\alpha}$ is a homeomorphism with inverse $\psi_{\alpha\delta}$.
\end{lemma}
\begin{proof}
For every $X\in\cX$ Lemma~\ref{TopEqHomeo} yields that $f_X$ is a homeomorphism, and by Lemma~\ref{CompactificationComparableUnique} we know that $g_X$ must be its inverse. In particular, every $g_X$ is injective, and so must be $\psi_{\alpha\delta}$. Furthermore, we already know that $\psi_{\alpha\delta}$ is also a continuous surjection. Similarly, $\psi_{\delta\alpha}$ is a continuous bijection, and $\psi_{\alpha\delta}$ is its continuous inverse.
\end{proof}

\begin{corollary}\label{CFleCUextension}
Let $\cC_\alpha$ and $\cC_\delta$ be two $\cC$-equivalent $\cC$-systems. Let $\alpha G$ and $\delta G$ be the two \ecomp{}s of $G$ induced by $\cC_\alpha$ and $\cC_\delta$, respectively. Then $\alpha G$ and $\delta G$ are topologically equivalent, witnessed by the homeomorphism $\psi_{\delta\alpha}\cup\id_G$ with inverse $\psi_{\alpha\delta}\cup\id_G$.
\end{corollary}
\begin{proof}
Combine Lemmas~\ref{CsystemPartialWell}~and~\ref{ECOcompsComparable}.
\end{proof}

At first glance condition (\ref{Commute}) complicates the definition of $\le_\cC$ for $\cC$-systems a lot, giving us a hard time working with it. Surprisingly, the following Lemma takes this extra work load off our shoulders:

\begin{lemma}\label{CsystemComparableNotHard}
Suppose that $\{(\alpha(\cC_X),\alpha_X),\A{a}_{X',X}\}$ and $\{(\delta(\cC_X),\delta_X),\A{d}_{X',X}\}$ are two $\cC$-systems with
$(\alpha(\cC_X),\alpha_X)\le(\delta(\cC_X),\delta_X)$ (witnessed by $f_X$) for all $X\in\cX$. Then (\ref{Commute}) holds for all $X\subseteq X'\in\cX$.
\end{lemma}
\begin{proof}
Recall that for every $X\in\cX$ the witness $f_X$ satisfies $f_X\circ\delta_X=\alpha_X$. Hence for every $X\subseteq X'\in\cX$ we compute:
\begin{align*}
f_X\circ\A{d}_{X',X}\circ\delta_{X'}\overset{(\ref{CompactificationBondingMapsRest})}{=}&\,f_X\circ\delta_X\circ\phi_{X',X}\\
=&\;\alpha_X\circ\phi_{X',X}\\
\overset{(\ref{CompactificationBondingMapsRest})}{=}&\;\A{a}_{X',X}\circ \alpha_{X'}\\
=&\;\A{a}_{X',X}\circ f_{X'}\circ\delta_{X'}
\end{align*}
Thus both sides of (\ref{Commute}) agree on $\delta_{X'}[\cC_{X'}]$, so by Lemma~\ref{HDdenseSubsetUnique} they agree on all of $\delta(\cC_{X'})$ as desired.
\end{proof}

\begin{theorem}\label{CFleCU}
If $G$ is an infinite graph, then the following hold:
\begin{enumerate}
\item For every $X\in\cX$ we have $(\aF_X,c_X)\le (\cU_X,c_X)$.
\item $\cC_{\A{F}}\le_{\cC}\cC_{\cU}$.
\item $\FG\le\TC$.
\end{enumerate}
\end{theorem}
\begin{proof}
(i). This is immediate from Lemma~\ref{StoneCech}. For further insight, a constructive proof follows:

Let any $X\in\cX$ be given and define $f\colon \cU_X\to\A{F}_X$ as follows: 
If $U\in\cU_X$ is principal we set $f(U)=U$. Otherwise $U$ is non-principal, so $\cC_X(Y)\in U$ holds for exactly one $Y\in \crit(X)$ due to
\begin{align*}
\cC_X=\cC_X^-\uplus\biguplus_{Y\in \crit(X)}\cC_X(Y)
\end{align*}
and since $\cC_X^-$ is finite. Hence, we set $f(U)=\cF_X(Y)$.

It remains to verify that $f$ is continuous: For this, consider any $U\in\cU_X$ together with some basic open neighbourhood $\cO_{\A{F}_X}(\cC)$ of $f(U)$ in $\aF_X$. 
If $U$ is generated by some $\{C\}$ with $C\in\cC_X$, then $\cO_{\cU_X}(\{C\})$ is a basic open neighbourhood of $U$ in $\cU_X$ which $f$ maps to $\cO_{\A{F}_X}(\cC)$. 
Otherwise $U$ is non-principal there is some $Y\in \crit(X)$ with $f(U)=\cF_X(Y)$ and $\cC$ is a cofinite subset of $\cC_X(Y)$, so $f$ maps $\cO_{\cU_X}(\cC)$ to $\cO_{\A{F}_X}(\cC)$ as desired. Thus $f$ is continuous as desired.

(ii) is immediate from (i) and Lemma~\ref{CsystemComparableNotHard}.

(iii) is immediate from (ii) and Lemma~\ref{ECOcompsComparable}.
\end{proof}

Next, we prove two technical Lemmas, which we then combine in Theorem~\ref{ECOminimal} in order to show that $\cC_{\aF}$ is the least $\cC$-system (with respect to $\le_{\cC}$).

\begin{lemma}\label{CsystemInjFail}
Let $X\in\cX$ be given together with a compactification $(\alpha(\cC_X),\alpha_X)$ of $\cC_X$ and suppose that
\begin{align}\label{assumption0}
\overline{\alpha_X[\cC_X(Y)]}\cap\overline{\alpha_X[\cC_X(Y')]}=\emptyset
\end{align}
holds for all $Y\neq Y'\in \crit(X)$. Then $(\A{F}_X,c_X)\le (\alpha(\cC_X),\alpha_X)$.
\end{lemma}
\begin{proof}
To construct a witness $f\colon \alpha(\cC_X)\to\aF_X$ of $(\aF_X,c_X)\le (\alpha(\cC_X),\alpha_X)$, we proceed as follows:

First, we write $K^-=c_X[\cC_X^-]$ (recall that $\cC_X^-$ is the set of all components of $G-X$ whose neighbourhood is not in $\crit(X)$), and for every $Y\in \crit(X)$ we put
\begin{align*}
K_Y=c_X[\cC_X(Y)]\cup\{\cF_X(Y)\}.
\end{align*}
Then 
\begin{align*}
\{K^-\}\cup\{K_Y\,|\,Y\in\crit(X)\}
\end{align*}
is a finite partition of $\aF_X$ into closed---and hence clopen---subsets.
Similarly, we write $L^-:=\alpha_X[\cC_X^-]$, and for every $Y\in\crit(X)$ we set
\begin{align*}
L_Y=\overline{\alpha_X[\cC_X(Y)]}
\end{align*}
with the closure taken in $\alpha(\cC_X)$.
Since $\cC_X^-$ is finite and $\alpha(\cC_X)$ is Hausdorff, we know that $L^-$ is closed in $\alpha(\cC_X)$. Hence (\ref{assumption0}) yields that 
\begin{align*}
\{L^-\}\cup\{L_Y\,|\,Y\in\crit(X)\}
\end{align*}
is a finite partition of $\alpha(\cC_X)$ into closed---and hence clopen---subsets.

Next, we define mappings whose union we will take as $f$ in the end. Let $f^-\colon L^-\to \aF_X$ map $\alpha_X(C)$ to $c_X(C)$ for each $C\in\cC_X^-$, and note that $f^-$ is continuous since $L^-$ carries the discrete subspace topology.
Since for every $Y\in\crit(X)$ the pair $(K_Y,c_X\rest\cC_X(Y))$ is the one-point \HDcomp{} of $\cC_X(Y)$ (equipped with the discrete subspace topology inherited from $\cC_X$) and $(L_Y,\alpha_X\rest\cC_X(Y))$ is a \HDcomp{} of $\cC_X(Y)$, by Lemma~\ref{OnePointCompMin} we have
\begin{align}\label{1ptComp}
(K_Y,c_X\rest\cC_X(Y))\le (L_Y,\alpha_X\rest\cC_X(Y)),
\end{align}
witnessed by some $f_Y\colon L_Y\to K_Y$. Since $K_Y$ carries the subspace topology of $\aF_X$, we may widen the codomain of $f_Y$ to $\aF_Y$ without losing continuity.
Finally, we take $f\colon \alpha(\cC_X)\to\aF_X$ to be the union of $f^-$ and the all of the $f_Y$. By Lemma~\ref{PastingLemma} we know that $f$ is continuous.
\end{proof}

\begin{lemma}\label{weirdLemma}
Let $\{(\alpha(\cC_X),\alpha_X),\A{a}_{X',X}\}$ be a $\cC$-system and let $\Xi\in\cX$ be such that $(\A{F}_X,c_X)\le (\alpha(\cC_X),\alpha_X)$ holds for all $X\subsetneq\Xi$. Then
\begin{align*}
\overline{\alpha_\Xi[\cC_\Xi(Y)]}\cap\overline{\alpha_\Xi[\cC_\Xi(Y')]}=\emptyset
\end{align*}
holds for all $Y\neq Y'\in\crit(\Xi)$.
\end{lemma}
\begin{proof}
Assume for a contradiction that there is a pair $Y\neq Y'\in\crit(\Xi)$ with
\begin{align*}
\overline{\alpha_\Xi[\cC_\Xi(Y)]}\cap\overline{\alpha_\Xi[\cC_\Xi(Y')]}\neq\emptyset.
\end{align*}
Then one of $Y\setminus Y'$ and $Y'\setminus Y$ is non-empty, so without loss of generality we find some $\xi\in Y'\setminus Y$. 
Set $\Xi^-=\Xi-\{\xi\}$. Then $(\A{F}_{\Xi^-},c_{\Xi^-})\le(\alpha(\cC_{\Xi^-}),\alpha_{\Xi^-})$ holds by assumption, witnessed by some $f\colon \alpha(\cC_{\Xi^-})\to\A{F}_{\Xi^-}$.
By choice of $Y$ and $Y'$ we may pick some $z\in \overline{\alpha_{\Xi}[\cC_\Xi(Y)]}\cap \overline{\alpha_{\Xi}[\cC_\Xi(Y')]}$ (with the closure taken in $\alpha(\cC_\Xi)$). 
Now put $z^-=\A{a}_{\Xi,\Xi^-}(z)$.
Since $\A{a}_{\Xi,\Xi^-}$ is continuous, Lemma~\ref{ctsPreservesClosure} implies
\begin{align*}
z^-\in\overline{\A{a}_{\Xi,\Xi^-}[\alpha_{\Xi}[\cC_\Xi(Y)]]}\cap \overline{\A{a}_{\Xi,\Xi^-}[\alpha_{\Xi}[\cC_\Xi(Y')]]}
\end{align*}
(with the closures taken in $\alpha(\cC_{\Xi^-})$). Using (\ref{CompactificationBondingMapsRest}) yields
\begin{align}\label{zMinusOneTimeUse}
z^-\in\overline{\alpha_{\Xi^-}[\phi_{\Xi,\Xi^-}[\cC_\Xi(Y)]]}\cap\overline{\alpha_{\Xi^-}[\phi_{\Xi,\Xi^-}[\cC_\Xi(Y')]]}.
\end{align}
By choice of $\xi$ and $\Xi^-$ there exists a unique component $C$ of $G-\Xi^-$ including $\bigcup\cC_\Xi(Y')$, i.e. with $\{C\}=\phi_{\Xi,\Xi^-}[\cC_\Xi(Y')]$. 
Furthermore, $\xi\notin Y$ implies $\phi_{\Xi,\Xi^-}[\cC_\Xi(Y)]=\cC_\Xi(Y)$.
Hence, combining (\ref{zMinusOneTimeUse}) with $\xi\notin Y$ yields 
\begin{align*}
z^-\in\overline{\alpha_{\Xi^-}[\cC_\Xi(Y)]}\cap\overline{\{\alpha_{\Xi^-}(C)\}}.
\end{align*}
 But then $z^-\in \overline{\{\alpha_{\Xi^-}(C)\}}$ implies $z^-=\alpha_{\Xi^-}(C)$ since $\alpha(\cC_{\Xi^-})$ is Hausdorff.
Recall that $c_{\Xi^-}(C)=\langle C\rangle_{\Xi^-}$ is an isolated point of $\A{F}_{\Xi^-}$, witnessed by the open neighbourhood $\cO_{\A{F}_{\Xi^-}}(\{C\})=\{c_{\Xi^-}(C)\}$. 
Put $O=f^{-1}(\{c_{\Xi^-}(C)\})$ and note that this is open in $\alpha(\cC_{\Xi^-})$ since $f$ is continuous.
Moreover, $f\circ\alpha_{\Xi^-}=c_{\Xi^-}$ together with Lemma~\ref{CompactificationSur}~(i)~and~(ii) yields $O=\{z^-\}$.
In particular, $z^-$ is also isolated in $\alpha(\cC_{\Xi^-})$, witnessed by $O$.
Due to $Y\neq Y'$ we have $\cC_\Xi(Y)\cap\cC_\Xi(Y')=\emptyset$, and furthermore $\xi\in Y'\setminus Y$ yields $C\notin\cC_\Xi(Y)$ (recall $\{C\}=\phi_{\Xi,\Xi^-}[\cC_\Xi(Y')]$). 
Hence $z^-\notin\alpha_{\Xi^-}[\cC_\Xi(Y)]$ follows. Moreover, since $z^-$ is isolated in $\alpha(\cC_{\Xi^-})$, we even have $z^-\notin\overline{\alpha_{\Xi^-}[\cC_\Xi(Y)]}$, witnessed by $O$, contradicting our choice of $z^-$ as desired.
\end{proof}

\begin{theorem}\label{ECOminimal}
For every infinite $G$ the following hold up to $\cC$-/topological equivalence:
\begin{enumerate}
\item $\cC_{\A{F}}$ is the least $\cC$-system (with respect to $\le_{\cC}$).
\item $\cC_{\cU}$ is the greatest $\cC$-system (with respect to $\le_{\cC}$).
\item $\FG$ is the coarsest \ecomp{} of $\CG$ induced by a $\cC$-system.
\item $\TC$ is the finest \ecomp{} of $\CG$ induced by a $\cC$-system.
\end{enumerate}
\end{theorem}
\begin{proof}
(i). Assume for a contradiction that there is another $\cC$-system $\cC_\alpha$ with $\cC_{\aF}\not\le_{\cC}\cC_\alpha$. Write $\{(\alpha(\cC_X),\alpha_X),\A{a}_{X',X}\}=\cC_\alpha$.
Due to Lemma~\ref{CsystemComparableNotHard} we find some minimal $\Xi\in\cX$ with $(\A{F}_\Xi,c_\Xi)\not\le (\alpha(\cC_\Xi),\alpha_\Xi)$. By Lemma~\ref{CsystemInjFail} there exist $Y\neq Y'$ in $\crit(\Xi)$ with
\begin{align*}
\overline{\alpha_\Xi[\cC_\Xi(Y)]}\cap\overline{\alpha_\Xi[\cC_\Xi(Y')]}\neq\emptyset,
\end{align*}
contradicting Lemma~\ref{weirdLemma} by choice of $\Xi$.

(ii). 
Let $\cC_\alpha=\{(\alpha(\cC_X),\alpha_X),\A{a}_{X',X}\}$ be any $\cC$-system. Since $(\cU_X,c_X)$ is the Stone-Čech \HDcomp{} of $\cC_X$ for every $X\in\cX$, Lemma~\ref{StoneCech} yields $(\alpha(\cC_X),\alpha_X)\le (\cU_X,c_X)$ for every $X\in\cX$. Thus Lemma~\ref{CsystemComparableNotHard} implies $\cC_\alpha\le_{\cC}\cC_{\cU}$.

(iii) and (iv) are immediate from Lemma~\ref{ECOcompsComparable} combined with (i) and (ii), respectively.
\end{proof}

Since we now know that $\FG$ is the coarsest \ecomp{} of $G$ induced by a $\cC$-system while $\TC$ is the finest one, the question for a characterisation of those graphs with $\FG$ and $\TC$ being topologically equivalent immediately comes to mind. 

\begin{theorem}\label{ECOequivalent}
For every infinite $G$ the following are equivalent:
\begin{enumerate}
\item The \ecomp{}s $\FG$ and $\TC$ of $G$ are topologically equivalent.
\item $\cC_{\aF}$ and $\cC_{\cU}$ are $\cC$-equivalent.
\item Every $\cC_X$ is finite.
\item $\Omega=\aF=\cU=\Theta$.
\end{enumerate}
\end{theorem}
\begin{proof}
Both (iv)$\leftrightarrow$(iii) and (iii)$\to$(ii) are clear, whereas (ii)$\to$(i) holds by Corollary~\ref{CFleCUextension}. Hence it suffices to show (i)$\to$(iii).

(i)$\to$(iii). 
Assume for a contradicting that (iii) fails, witnessed by some $Y\in\cX$ with $\cC_Y$ infinite. Then $\aF_Y^\ast$ is finite while $\cU_Y^\ast$ is infinite, so $(\cU_Y,c_Y)\le(\aF_Y,c_Y)$ is impossible. 
Let $g\colon \TC\to\FG$ be a homeomorphism witnessing (i).
To yield a contradiction, we will use $g$ to find a witness of $(\cU_Y,c_Y)\le(\aF_Y,c_Y)$.

For every $X\in\cX$ let $f_X\colon \cU_X\to\aF_X$ be a witness of $(\aF_X,c_X)\le (\cU_X,c_X)$ (which exists by Theorem~\ref{CFleCU}). 
Furthermore, let $\bar{f}\colon \cU\to\aF$ be obtained from the $f_X$ as in (\ref{CsystemSur}), i.e. by setting
\begin{align*}
\bar{f}((U_X\,|\,X\in\cX))=(f_X(U_X)\,|\,X\in\cX)
\end{align*}
for every $(U_X\,|\,X\in\cX)\in\cU$, and recall that this is a well-defined continuous surjection 
by Lemma~\ref{CsystemComparableNotHard}. 
By Lemma~\ref{ECOcompsComparable}, the map $f:=\bar{f}\cup\id_{\CG}$ witnesses $\FG\le\TC$. In particular, $\hat{f}:=\bar{f}\cup\id_V$ witnesses $(\FG\setminus\mathring{E},\id_V)\le (\TC\setminus\mathring{E},\id_V)$. 
Since so does $\hat{g}:=g\rest (\TC\setminus\mathring{E})$, Lemma~\ref{CompactificationComparableUnique} yields $\hat{f}=\hat{g}$.
In particular, $\bar{f}\colon \cU\to\aF$ is bijective. 

Put $h_Y=f_Y\rest\cU_Y^\ast$ and note $\im(h_Y)=\aF_Y^\ast$ by Lemma~\ref{CompactificationSur}. Next, we use the injectivity of $\bar{f}$ to yield that $h_Y$ is injective. Assume for a contradiction that $h_Y$ is not injective, so there exist $U_Y\neq U'_Y$ in $\cU_Y^\ast$ with $f_Y(U_Y)=f_Y(U'_Y)=\cF_Y$ for some $\cF_Y\in\aF_Y^\ast$.
By Corollary~\ref{Uextension}, both $U_Y$ and $U'_Y$ uniquely extend to elements $(U_X\,|\,X\in\cX)=\upsilon$ and $(U_X'\,|\,X\in\cX)=\upsilon'$ of $\cU$, respectively. By Corollary~\ref{Fextension}, $\cF_Y$ uniquely extends to an element $(\cF_X\,|\,X\in\cX)=\A{v}$ of $\aF$. Since $\A{v}$ is unique, every other $(\cF_X'\,|\,X\in\cX)\in\aF$ satisfies $\cF_Y'\neq\cF_Y$. Thus $f_Y(U_Y)=f_Y(U'_Y)=\cF_Y$ combined with the definition of $\bar{f}$ forces $\bar{f}(\upsilon)=\A{v}$ and $\bar{f}(\upsilon')=\A{v}$, contradicting the injectivity of $\bar{f}$. Hence $h_Y=f_Y\rest\cU_Y^\ast$ is injective as claimed.

Together with $f_Y\circ c_Y=c_Y$ this yields that $f_Y$ is injective.
Since $f_Y$ is also a continuous surjection by Lemma~\ref{CompactificationSur}, and both $\aF_Y$ and $\cU_Y$ are compact Hausdorff, it follows from general topology that $f_Y$ is a homeomorphism. In particular, the inverse of $f_Y$ witnesses $(\cU_Y,c_Y)\le (\aF_Y,c_Y)$, the desired contradiction.
\end{proof}

\newpage
\subsection{A natural quotient of the tangle compactification}\label{AnaturalQuotient}

In this section, we show that $\FG$ is homeomorphic to the quotient space $\TC/{\asymp}$, but first we describe $\asymp$ in terms of tangles.\\

Recall that we defined the equivalence relation $\asymp$ on $\Upsilon$ by letting $\upsilon\asymp \upsilon'$ whenever $X_\upsilon=X_{\upsilon'}$ holds. 
Via the natural correspondence between $\Upsilon$ and $\Theta-\Omega$, this induces an equivalence relation $\tasymp$\index{$\tasymp$} 
on $\Theta-\Omega$ (here we distinguish between elements of $\cU$ and elements of $\Theta$).
Since $\cU$ originally only served as a technical tool to help us better understand the ultrafilter tangles, we wish to be able to describe the $\tasymp$-classes in terms of tangles. For this,
we investigate the following natural class of stars: For every $X\in\cX$ and $C\in\cC_X$ define\index{$s_{C\to X}$, $s_{X\to C}$}
\begin{align*}
s_{C\to X}:=(C\cup X,V\setminus C),\quad s_{X\to C}:=(s_{C\to X})^\ast=(V\setminus C,X\cup C)
\end{align*}
and put\index{$\sigma_X$}
\begin{align*}
\sigma_X=\{s_{C\to X}\,|\,C\in\cC_X\}.
\end{align*}
Next, we will use these stars to describe $X_\upsilon$ (for $\upsilon\in\Upsilon$) in terms of tangles:

\begin{lemma}
For every $\tau\in\Theta-\Omega$ we have
\begin{align*}
\{X\in\cX\,|\,\sigma_X\subseteq\tau\}=\cX_\tau=\lfloor X_\tau\rfloor_{\cX}.
\end{align*}
In particular, $X_\tau$ is the unique minimal $X\in\cX$ with $\sigma_X\subseteq\tau$.
\end{lemma}
\begin{proof}
It suffices to show $\{X\in\cX\,|\,\sigma_X\subseteq\tau\}=\cX_\tau$.

For the forward inclusion let any $X\in\cX$ with $\sigma_X\subseteq\tau$ be given.
Then $U(\tau,X)$ is non-principal: Otherwise 
$U(\tau,X)$ is generated by some $\{C\}$ with $C\in\cC_X$,
so $s_{X\to C}\in\tau$ implies $s_{C\to X}\notin\tau$, and hence
$\sigma_X\not\subseteq\tau$ is a contradiction. Therefore, $U(\tau,X)\in\cU_X^\ast$ implies $X\in\cX_{\tau}$.

For the backward inclusion let any $X\in\cX_{\tau}$ be given. Since $U(\tau,X)$ is non-principal, we know that $\cC_X\setminus\{C\}\in U(\tau,X)$ holds for every $C\in\cC_X$. In particular we have $s_{C\to X}\in\tau$ for every $C\in\cC_X$, and therefore $\sigma_X\subseteq\tau$ holds as desired.
\end{proof}

Thus $\tasymp$ can be described solely in terms of tangles, as desired. Next, we compute loads of technical Lemmas leading to an explicit description of a basis of the quotient topology of $\TC/{\asymp}$.

Let $G$ be any graph. Define $\Psi\colon \cU/{\asymp}\to\aF$\index{$\Psi\colon \cU/{\asymp}\bij\aF$} by setting $\Psi\rest\Omega=\id_\Omega$ and letting $\Psi\rest (\Upsilon/{\asymp})$ map each $[\upsilon]_\asymp\in\Upsilon/{\asymp}$ to $\chi(X_\upsilon)$ (where $\chi\colon \crit(\cX)\bij\aF\setminus\Omega$ is as in Proposition~\ref{FandY})

\begin{proposition}\label{UltrafiltersEconomic}
The map $\Psi$ is a well defined bijection between $\cU/{\asymp}$ and $\aF$.
\end{proposition}
\begin{proof}
The map $\Psi$ is well defined by Lemma~\ref{New36}. Clearly, it suffices to show that $\Psi\rest (\cU/{\asymp})\colon \cU/{\asymp}\to\aF\setminus\Omega$ is bijective, which is the case by Proposition~\ref{FandY} and Corollary~\ref{UltrafilterCrit}.
\end{proof}

\begin{lemma}\label{UpsilonLivesInComp}
For every $\upsilon\in\Upsilon$ and $X\in\cX$ with $X_\upsilon\not\subseteq X$ there is a unique component $C$ of $G-X$ with $X_\upsilon\subseteq X\cup C$ and $\upsilon\in\cO_{\TC}(X,\{C\})$.
\end{lemma}
\begin{proof}
By Lemma~\ref{GGwellDef} the set $X_\upsilon$ meets at most one component of $G-X$. Since $X_\upsilon\not\subseteq X$ holds, $X_\upsilon$ meets exactly one component $C$ of $G-X$. Set $X'=X\cup X_\upsilon$. 
Then by Lemma~\ref{New36}, $X_\upsilon$ is in $\crit(X')$ and $\cC:=\cC_{X'}(X_\upsilon)$ is in $U(\upsilon,X')$. Hence $\cC\rest X=\{C\}$ is in $U(\upsilon,X)$, so $\upsilon\in\cO_{\TC}(X,\{C\})$ holds.
\end{proof}

\begin{lemma}\label{ECOasympclosed}
For every $X\in\cX$ and $C\in\cC_X$ the set $\cO_{\TC}(X,\{C\})$ is closed under the equivalence relation $\asymp$.
\end{lemma}
\begin{proof}
Let any $\upsilon\in\Upsilon\cap\cO_{\TC}(X,\{C\})$ be given together with some $\upsilon'\in [\upsilon]_\asymp$. 
In particular $X_\upsilon=X_{\upsilon'}\not\subseteq X$ holds, 
so by Lemma~\ref{UpsilonLivesInComp} there exist for both $\upsilon$ and $\upsilon'$ unique components $D$ and $D'$ of $G-X$, respectively, with $X_\upsilon\subseteq X\cup D$ and $\upsilon\in\cO_{\TC}(X,\{D\})$ as well as $X_{\upsilon'}\subseteq X\cup D'$ and $\upsilon'\in\cO_{\TC}(X,\{D'\})$. 
Since $X_\upsilon=X_{\upsilon'}$ meets both $D$ and $D'$, the uniqueness of $D$ implies $D=D'$. 
Furthermore $D=D'=C$ follows since otherwise $\{D\}$ and $\{C\}$ would be disjoint elements of $U(\upsilon,X)$, which is impossible.
Hence $\upsilon'\in\cO_{\TC}(X,\{C\})$ holds.
\end{proof}

\begin{lemma}\label{notherLemma}
Let $\upsilon\in\Upsilon$, $X\in\cX_\upsilon$ and $Y\in\crit(X)$ be given.
Furthermore, suppose that $\cC\subseteq\cC_X$ is cofinite in $\cC_X(Y)$ with $\cC\in U(\upsilon,X)$. Then we have $X_\upsilon=Y$.
\end{lemma}
\begin{proof}
Assume for a contradiction that $X_\upsilon$ is distinct from $Y$.
By Lemma~\ref{New36} we have $\cC_X(X_\upsilon)\in U(\upsilon,X)$.
Now $X_\upsilon\neq Y$ implies that $\cC_X(X_\upsilon)$ is disjoint from $\cC_X(Y)$, and hence from $\cC$, a contradiction.
\end{proof}

\begin{lemma}\label{ECOfinDiffClosedAsymp}
For all $X\in\cX$, every $Y\in \crit(X)$ and each cofinite subset $\cC$ of $\cC_X(Y)$ the set $\cO_{\TC}(X,\cC)$ is closed under $\asymp$.
\end{lemma}
\begin{proof}
Put $\cD=\cC_X(Y)$.
First, we show that $\cO_{\TC}(X,\cD)$ is closed under $\asymp$. For this, let any $\varkappa\in \Upsilon\cap\cO_{\TC}(X,\cD)$ be given together with any $\varkappa'\in [\varkappa]_\asymp$. 

If $X_{\varkappa}\not\subseteq X$ holds, then
$U(\varkappa,X)$ is generated by some $\{C\}$ with $C$ a component of $G-X$, and $\cD\in U(\varkappa,X)$ implies $C\in\cD$.
In particular we have $\varkappa\in\cO_{\TC}(X,\{C\})$, so Lemma~\ref{ECOasympclosed} implies $\varkappa'\in \cO_{\TC}(X,\{C\})$ and we are done.
 
Otherwise $X_{\varkappa}\subseteq X$ holds, so $\cD\in U(\varkappa,X)$ together with Lemma~\ref{notherLemma} implies $X_\varkappa=Y$. 
Hence $\cD\in U(\varkappa',X)$ follows from Lemma~\ref{New36} and $X_{\varkappa'}=Y$.
Therefore, $\cO_{\TC}(X,\cD)$ is closed under $\asymp$. 

Since $\cC$ is cofinite in $\cD$ we have
\begin{align*}
\cO_{\TC}(X,\cC)=\cO_{\TC}(X,\cD)-\bigcup_{D\in\cD\setminus\cC}\cO_{\TC}(X,\{D\})
\end{align*}
which is closed under $\asymp$ since $\cO_{\TC}(X,\cD)$ is closed under $\asymp$ and the $\cO_{\TC}(X,\{D\}$ are also closed under $\asymp$ by Lemma~\ref{ECOasympclosed}.
\end{proof}

\begin{lemma}\label{ECOcompact}
For every $\upsilon\in\Upsilon$ the set $[\upsilon]_\asymp$ is closed in $\TC$. In particular, $[\upsilon]_\asymp$ is a compact subset of $\TC$.
\end{lemma}
\begin{proof}
Clearly, the closure of $[\upsilon]_\asymp$ avoids $G$. 
It remains to show that no $\varkappa\in\cU\setminus [\upsilon]_\asymp$ lies in the closure of $[\upsilon]_\asymp$.
If we have $\varkappa\in\Omega$, then $\CTC(X_\upsilon,\varkappa)$ avoids $[\upsilon]_\asymp$ since for every $\upsilon'\in [\upsilon]_\asymp$ the ultrafilter $U^\circ(\upsilon')$ is non-principal. 
Otherwise we have $\varkappa\in\Upsilon\setminus [\upsilon]_\asymp$, so $\cO_{\TC}(X,\cC_X(X_\varkappa))$ with $X=X_\upsilon\cup X_\varkappa$ avoids $[\upsilon]_\asymp$ since $\cC_X(X_\upsilon)$ and $\cC_X(X_\varkappa)$ are disjoint due to $X_\upsilon\neq X_{\varkappa}$. 
Thus $[\upsilon]_\asymp$ is a closed subset of the compact space $\TC$ and hence also compact.
\end{proof}

Let $\cB$ be the collection of all basic open sets of the 1-complex topology of $G$, all sets $\cO_{\TC}(X,\{C\})/{\asymp}$ with $X\in\cX$ and $C\in\cC_X$, and all sets $\cO_{\TC}(X,\cC)/{\asymp}$ with $X\in\cX$ and $\cC$ cofinite in $\cC_X(Y)$ for some $Y\in \crit(X)$.\index{$\cB$}
\begin{proposition}\label{TCquotientBasis}
The collection $\cB$ is a basis for the topology of $\TC/{\asymp}$.
\end{proposition}
\begin{proof}
By Lemmas~\ref{ECOasympclosed}~and~\ref{ECOfinDiffClosedAsymp}, the elements of $\cB$ are open in $\TC/{\asymp}$.

Let $x$ be any point of $\TC/{\asymp}$ together with some arbitrary open neighbourhood $O$. Our task is to find an element $B$ of $\cB$ with $x\in B\subseteq O$.

If $x$ is a vertex or an inner edge point, we are done. Else if $x$ is an end, we write $\omega=x$ and pick a basic open neighbourhood $\CTC(X,\omega)$ of $\omega$ in $\TC$ which is included in $\bigcup O$. Then we are done by taking $B=\CTC(X,\omega)/{\asymp}$.

Finally suppose that there is some $\upsilon\in\Upsilon$ with $x=[\upsilon]_\asymp$.
For every $\upsilon'\in [\upsilon]_\asymp$ we pick some basic open neighbourhood $O(\upsilon')=\cO_{\TC}(X(\upsilon'),\cC(\upsilon'))$ in $\TC$ which is included in $\bigcup O$. 

If $X_\upsilon\not\subseteq X(\varkappa)$ holds for some $\varkappa\in [\upsilon]_\asymp$ we put $X=X(\varkappa)$.
Then by Lemma~\ref{UpsilonLivesInComp} there is some unique component $C$ of $G-X$ with $X_\upsilon=X_\varkappa\subseteq X\cup C$ and $U(\varkappa,X)$ is generated by $\{C\}$. 
Then 
\begin{align*}
\cO_{\TC}(X,\{C\})\subseteq\cO_{\TC}(X,\cC(\varkappa))\subseteq \medcup O
\end{align*}
is a basic open neighbourhood of $\varkappa$ in $\TC$ which is closed under $\asymp$ by Lemma~\ref{ECOasympclosed} and included in $\bigcup O$. 
In particular, $\cO_{\TC}(X,\{C\})/{\asymp}$ is an open neighbourhood of $[\upsilon]_\asymp$ in $\TC/{\asymp}$ which is an element of $\cB$ and a subset of $O$, so we are done.

Hence we may assume that $X_\upsilon\subseteq X(\upsilon')$ holds for all $\upsilon'\in[\upsilon]_\asymp$.
Since the $O(\upsilon')$ form an open cover in $\TC$ of the set $[\upsilon]_\asymp$ which is compact by Lemma~\ref{ECOcompact}, we find some finite subset $\{\upsilon_k\,|\,k<n\}$ of $[\upsilon]_\asymp$ such that the $O(\upsilon_k)$ cover $[\upsilon]_\asymp$. 
Next we put $X=\bigcup_{k<n}X(\upsilon_k)$ and note that $X_\upsilon\subseteq X\in\cX$ holds. Furthermore, we set $\cD=\cC_X(X_\upsilon)$, and for all $k<n$ we put $\cD_k=\cC(\upsilon_k)\cap\cD$. 
Letting $\cD^+:=\bigcup_{k<n}\cD_k\subseteq\cD$, our hope is that $\cO_{\TC}(X,\cD^+)/{\asymp}$ is a suitable candidate for $B$.

We claim that $\cD^+$ is a cofinite subset of $\cD$.
Assume not for a contradiction, so $\cD^-:=\cD\setminus\cD^+$ is infinite
and by Lemma~\ref{ECOextentUltra} we find some $\varkappa\in\Upsilon\cap\cO_{\TC}(X,\cD^-)$ with $X_\varkappa=X_\upsilon$.
But then there is some $\ell<n$ with $\varkappa\in O(\upsilon_\ell)$, so in particular $\cC(\upsilon_\ell)\in U(\varkappa,X(\upsilon_\ell))$ holds. 
Since $X$ meets only finitely many elements of $\cC(\upsilon_\ell)$, 
the set $\cD_\ell$ is a cofinite subset of $\cC(\upsilon_\ell)$,
so in particular $\cD_\ell$ is in $U(\varkappa,X(\upsilon_\ell))$.
Furthermore, $\cD_\ell\in U(\varkappa,X(\upsilon_\ell))\cap 2^{\cC_X}$ implies $\cD_\ell\in U(\varkappa,X)$ by definition of $g_{X(\upsilon_\ell),X}$. This contradicts $\cD^-\in U(\varkappa,X)$ since $\cD^-$ and $\cD_\ell\subseteq\cD^+$ are disjoint. Therefore $\cD^-$ is finite as claimed, so by Lemma~\ref{ECOfinDiffClosedAsymp} the set $\cO_{\TC}(X,\cD^+)$ is $\asymp$-closed.

Furthermore, every $\upsilon'\in [\upsilon]_\asymp$ is contained in $\cO_{\TC}(X,\cD^+)$ since $\cD^+$ is cofinite in $\cD=\cC_X(X_\upsilon)=\cC_X(X_{\upsilon'})$ and $\cC_X(X_{\upsilon'})$ is in $U(\upsilon',X)$ by Lemma~\ref{New36}.

It remains to show that $\cO_{\TC}(X,\cD^+)$ is included in $\bigcup O$.
Since $\cD^+$ is the finite union of the sets $\cD_k$, it follows that $\cO_{\TC}(X,\cD^+)$ is the finite union of the sets $\cO_{\TC}(X,\cD_k)$. Hence it suffices to verify that $\cO_{\TC}(X,\cD_k)\subseteq\bigcup O$ holds for all $k<n$.
For this, we observe that
\begin{align*}
\cD_k\rest X(\upsilon_k)=\cD_k=\cC(\upsilon_k)\cap\cD\subseteq\cC(\upsilon_k)
\end{align*}
holds for all $k<n$. Then Lemma~\ref{obviousLemma} together with $\cD_k\rest X(\upsilon_k)\subseteq\cC(\upsilon_k)$ implies
\begin{align*}
\cO_{\TC}(X,\cD_k)&\subseteq\cO_{\TC}(X(\upsilon_k),\cD_k\rest X(\upsilon_k))\\
&\subseteq\cO_{\TC}(X(\upsilon_k),\cC(\upsilon_k))\\
&\subseteq\medcup O
\end{align*}
as desired.
\end{proof}


Now that we know a basis of the topology of $\TC/{\asymp}$, 
only two technical Lemmas stand in between of us and our desired result,
namely that $\Psi\colon \cU/{\asymp}\bij\aF$ induces a homeomorphism $\Psi\cup\id_G$ between $\FG$ and $\TC/{\asymp}$.

\begin{lemma}\label{PsiSimpleNbhd}
For every $X\in\cX$ and $C\in\cC_X$ we have
\begin{align*}
\Psi[\cO_{\cU}(X,\{C\})/{\asymp}]=\cO_{\aF}(X,\{C\}).
\end{align*}
\end{lemma}
\begin{proof}
By Lemma~\ref{ECOasympclosed} the set $\cO_{\TC}(X,\{C\})$ is $\asymp$-closed, and so is $\cO_{\cU}(X,\{C\})$.

We start with the forward inclusion.
If $\omega$ is any element of $\Omega\cap\cO_{\cU}(X,\{C\})$, then clearly $\Psi(\omega)=\omega$ is contained in $\cO_{\aF}(X,\{C\})$.
Else if $\upsilon$ is any element of $\Upsilon\cap\cO_{\cU}(X,\{C\})$, then the ultrafilter $U(\upsilon,X)$ is generated by $\{C\}$ and $X_\upsilon\not\subseteq X$ must hold. 
Furthermore, Lemma~\ref{UpsilonLivesInComp} implies $X_\upsilon\subseteq X\cup C$. Hence for $\cC:=\cC_{X\cup X_\upsilon}(X_\upsilon)$ we have $\bigcup\cC\subseteq C$.
Let $\A{v}=(\cF_Y\,|\,Y\in\cX)$ be the element of $\aF$ to which $\Psi$ sends $[\upsilon]_\asymp$. In particular, $X_\upsilon=X_{\A{v}}$ holds and furthermore $\cF_X$ is generated by $\{D\}$ where $D$ is the unique component of $G-X$ including $\bigcup\cC$. In particular, $C=D$ must hold,
so $\A{v}$ is an element of $\cO_{\aF}(X,\{C\})$ as desired.

Finally, we show the backward inclusion. For this, let any element $\A{v}=(\cF_Y\,|\,Y\in\cX)$ of $\cO_{\aF}(X,\{C\})$ be given, and let $x$ be the element of $\cU/{\asymp}$ which $\Psi$ sends to $\A{v}$.
If $x$ is an end we are done, so suppose that $x$ is of the form $[\upsilon]_\asymp$ for some $\upsilon\in\Upsilon$.
 Since $\cF_X$ is generated by $\{C\}$ we have $X_\upsilon\not\subseteq X$ by definition of $\Psi$. By Lemma~\ref{UpsilonLivesInComp} there exists a unique component $D$ of $G-X$ with $X_\upsilon\subseteq X\cup D$ and $\upsilon\in\cO_{\cU}(X,\{D\})$. 
By the forward inclusion we have
\begin{align*}
\Psi[\cO_{\cU}(X,\{D\})/{\asymp}]\subseteq\cO_{\aF}(X,\{D\}),
\end{align*}
so $\A{v}\in\cO_{\aF}(X,\{D\})$ implies $C=D$. Hence $\upsilon\in\cO_{\cU}(X,\{C\})$ holds as desired.
\end{proof}

\begin{lemma}\label{PsiCofiniteNbhd}
For all $X\in\cX$, every $Y\in \crit(X)$ and each cofinite subset $\cC$ of $\cC_X(Y)$ we have
\begin{align*}
\Psi[\cO_{\cU}(X,\cC)/{\asymp}]=\cO_{\aF}(X,\cC).
\end{align*}
\end{lemma}
\begin{proof}
By Lemma~\ref{ECOfinDiffClosedAsymp} the set $\cO_{\TC}(X,\cC)$ is $\asymp$-closed.

We start with the forward inclusion.
If $\omega$ is any element of $\Omega\cap\cO_{\cU}(X,\cC)$, then clearly $\Psi(\omega)=\omega$ is contained in $\cO_{\aF}(X,\cC)$.
If $\upsilon$ is any element of $\Upsilon\cap\cO_{\TC}(X,\cC)$ we consider two cases. First suppose that $X_\upsilon\not\subseteq X$ holds. Then $U(\upsilon,X)$ is generated by $\{C\}$ for some $C\in \cC$. In particular, $\upsilon$ is contained in $\cO_{\cU}(X,\{C\})$, so $\Psi$ maps $[\upsilon]_\asymp$ to an element of $\cO_{\aF}(X,\{C\})\subseteq\cO_{\aF}(X,\cC)$ by Lemma~\ref{PsiSimpleNbhd} as desired. Second suppose that $X_\upsilon\subseteq X$ holds. Then $\cC\in U(\upsilon,X)$ and $\cC$ being a cofinite subset of $\cC_X(Y)$ implies $X_\upsilon=Y$ by Lemma~\ref{notherLemma}. Let $\A{v}=(\cF_Z\,|\,Z\in\cX)$ be the element of $\aF$ to which $\Psi$ sends $[\upsilon]_\asymp$. Then $\cF_X=\cF_X(Y)$ holds by definition of $\Psi$. In particular, $\cC$ being a cofinite subset of $\cC_X(Y)$ implies $\cC\in\cF_X(Y)$, resulting in $\A{v}\in\cO_{\aF}(X,\cC)$ as desired.

Finally, we show the backward inclusion. For this, let any element $\A{v}=(\cF_Z\,|\,Z\in\cX)$ of $\cO_{\aF}(X,\cC)$ be given, and let $x$ be the element of $\cU/{\asymp}$ which $\Psi$ sends to $\A{v}$.
If $x$ is an end we are done, so suppose that $x$ is of the form $[\upsilon]_\asymp$ for some $\upsilon\in\Upsilon$.
If $\cF_X$ is generated by $\{C\}$ for some $C\in\cC$ then $\A{v}$ is in $\cO_{\aF}(X,\{C\})$ and by Lemma~\ref{PsiSimpleNbhd} we have
\begin{align*}
[\upsilon]_\asymp\in\cO_{\cU}(X,\{C\})/{\asymp}\subseteq\cO_{\cU}(X,\cC)/{\asymp}
\end{align*}
as desired.
Otherwise $\cF_X=\cF_X(Y)$ is the only possibility for $\cF_X$, so by definition of $\Psi$ we have $X_\upsilon=Y$. 
Furthermore, Lemma~\ref{New36} implies $\cC_X(Y)=\cC_X(X_\upsilon)\in U(\upsilon,X)$. 
Since $U(\upsilon,X)$ is non-principal due to $X_\upsilon=Y\subseteq X$, the fact that $\cC$ is a cofinite subset of $\cC_X(Y)$ implies $\cC\in U(\upsilon,X)$. Hence $\upsilon$ is contained in $\cO_{\cU}(X,\cC)$, so $[\upsilon]_\asymp$ is in $\cO_{\cU}(X,\cC)/{\asymp}$ as desired.
\end{proof}

\begin{theorem}\label{FGandQuotient}
For every graph $G$ combining the bijection $\Psi\colon \cU/{\asymp}\bij\aF$ with the identity on $G\cup\Omega$ yields a homeomorphism between $\TC/{\asymp}$ and $\FG$.
\end{theorem}
\begin{proof}
Set $\hPsi:=\Psi\cup\id_G\colon \TC/{\asymp}\bij\FG$ which is a bijection by Proposition~\ref{UltrafiltersEconomic}. 
By Proposition~\ref{TCquotientBasis} we may consider the basis $\cB$ for the quotient topology of $\TC/{\asymp}$, and we consider the basis $\aB$ for the topology of $\FG$ (recall that $\aB$ simply is the basis we defined to generate the topology of $\FG$).
Since $\hPsi$ extends the identity on $G$, Lemmas~\ref{PsiSimpleNbhd}~and~\ref{PsiCofiniteNbhd} yield that $\hPsi$ is bicontinuous.
\end{proof}
\begin{corollary}
$\Psi$ is a homeomorphism between $\cU/{\asymp}$ and $\aF$.\qed
\end{corollary}
\begin{corollary}\label{FGHausdorffIff}
$\FG$ is Hausdorff if and only if $G$ is locally finite.
\end{corollary}
\begin{proof}
Combine Theorem~\ref{FGandQuotient} and Corollary~\ref{TCHausdorffIff}.
\end{proof}

We close this chapter with a comparison of the cardinalities of $\aF-\Omega$ and $\Upsilon$:

\begin{lemma}\label{ExtendCrit}
For all $X\in\cX$, every $Y\in\crit(X)$ and each non-principal ultrafilter $U$ on $\cC_X(Y)$ there is some unique $\upsilon\in\Upsilon$ with $X_\upsilon=Y$ and $U\subseteq U^\circ(\upsilon)$.
\end{lemma}
\begin{proof}
Put $U'=\langle U\rangle_{\cC_X}$ and note that this is the only ultrafilter on $\cC_X$ extending $U$. Using Corollary~\ref{Uextension} we uniquely extend it to an element $\upsilon$ of $\cU$. In particular, $\upsilon$ is uniquely determined by $U$, and clearly we have $Y\in\cX_{\upsilon}$.
 For every $Y^-\subsetneq Y$ the set $\cC_X(Y)\rest Y^-$ is a singleton contained in $U(\upsilon,Y^-)$, therefore witnessing $Y^-\notin\cX_\upsilon$. Thus $Y=X_\upsilon$ follows due to $\cX_\upsilon=\lfloor X_\upsilon\rfloor_{\cX}$.
\end{proof}

\begin{proposition}\label{FGcardComparison}
If $G$ is an infinite graph, then
\begin{enumerate}
\item $|\aF-\Omega|=|\crit(\cX)|\le |V|$,
\item $|\Upsilon|\ge |\Upsilon/{\asymp}|\cdot 2^{\A{c}}=|\crit(\cX)|\cdot 2^{\A{c}}$,
\item $|\aF-\Omega|\cdot 2^{\A{c}}\le |\Upsilon|$.
\end{enumerate}
\end{proposition}
\begin{proof}
(i). By Proposition~\ref{FandY} the map $\A{v}\mapsto X_{\A{v}}$ is a bijection from $\aF-\Omega$ onto $\crit(\cX)$. Furthermore, $\crit(\cX)\subseteq\cX=[V]^{<\aleph_0}$ implies $|\crit(\cX)|\le |V|$.

(ii). By Lemma~\ref{UltrafilterCrit} we have $|\Upsilon/{\asymp}|=|\crit(\cX)|$.
By Lemma~\ref{New36} we know that $\cC_{X_\upsilon}(X_\upsilon)\in U^\circ(\upsilon)$ holds for all $\upsilon\in\Upsilon$. Since the set $\cC_{X_\upsilon}(X_\upsilon)$ is infinite, there exist at least $2^{\A{c}}$ many non-principal ultrafilters on it, and by Lemma~\ref{ExtendCrit} each of these extends to a unique element of $[\upsilon]_\asymp$. In particular, $[\upsilon]_\asymp$ has cardinality at least $2^{\A{c}}$. Since $\asymp$ is an equivalence relation, this implies $|\Upsilon|\ge |\Upsilon/{\asymp}|\cdot 2^{\A{c}}$.

(iii) follows from (i) and (ii) combined.
\end{proof}

\subsection{$\aF$ described by tangles}\label{FGandTangles}

In \cite{EndsAndTangles}, $\cU$ was use as a technical description of $\Theta$. Recently, we created $\aF$ inspired by $\cU$, and we showed that $\aF$ is homeomorphic to a natural quotient of $\Theta$, namely $\Theta/{\tasymp}$. So perhaps it is possible to describe $\aF$ using tangles of $G$?\\

Indeed, it turns out that considering a smaller separation system suffices:
Let $S'$\index{$S'$} be the set of all $\{\bigcup\cC\cup X,X\cup\bigcup\cC'\}\in S$ with $\cC\uplus\cC'=\cC_X$ for which there is no $Y\in\crit(X)$ such that both $\cC_X(Y)\cap\cC$ and $\cC_X(Y)\cap\cC'$ are infinite.
Write $\Theta'$ for the set of all $\aleph_0$-tangles with respect to $S'$.
For every $\tau\in\Theta'$ and $X\in\cX$ we let
\begin{align*}
\cF(\tau,X):=\big\{\cC\subseteq\cC_X\,\big|\,\big(V[\cC_X\setminus\cC]\cup X,X\cup V[\cC]\big)\in\tau\big\}.
\end{align*}
Like for elements of $\Theta$, it is easy to check that this is indeed an element of $\aF_X$.

\begin{lemma}\label{FinSymDiffClosedS}
Let $\{A,B\}\in S'$ and $\{C,D\}\in S$ be such that both $A\triangle C$ and $B\triangle D$ are finite. Then $\{C,D\}\in S'$.
\end{lemma}
\begin{proof}
Assume for a contradiction that $\{C,D\}\notin S'$ holds, witnessed by some $Y\in\crit(C\cap D)$. Let $\{\cC,\cD\}$ be the bipartition of $\cC_{C\cap D}(Y)$ with $V[\cC]\subseteq C$ and $V[\cD]\subseteq D$. By choice of $Y$, both $\cC$ and $\cD$ are infinite. 
Next, put
\begin{align*}
\cC'&=\{C\in\cC\,|\,C\cap A\cap B=\emptyset\}\\
\cD'&=\{C\in\cD\,|\,D\cap A\cap B=\emptyset\}
\end{align*}
and note that both $\cC\setminus\cC'$ and $\cD\setminus\cD'$ are finite, since $A\cap B$ is finite. By choice of $\{C,D\}$ we know that all but finitely many element of $\cC'$ are included in $A$, and all but finitely elements of $\cD'$ are included in $B$. We write $\cC''$ and $\cD''$ for the sets of those elements, respectively.

If $Y\not\subseteq A\cap B$ holds, then there is some $K\in\cC_{A\cap B}$ with $\bigcup\cC''\cup \bigcup\cD''\subseteq V(K)$. Without loss of generality we may assume that $V(K)\subseteq A\setminus B$. Since both $\cC''$ and $\cD''$ are infinite, so is $A\triangle C$, the desired contradiction.

Otherwise $Y\subseteq A\cap B$ holds, so $\cC''$ and $\cD''$ together with $Y\in\crit(A\cap B)$ witness $\{A,B\}\notin S'$, a contradiction.
\end{proof}
\begin{lemma}\label{FinSymDiffParallel}
Let $\tau\in\Theta'$ and $(A,B)\in\tau$ be given.
If $(A',B')\in \vec{S}$ is such that both $A\triangle A'$ and $B\triangle B'$ are finite, then $(A',B')\in\tau$.
\end{lemma}
\begin{proof}
We mimic the proof of \cite[Lemma 1.10]{EndsAndTangles}. 
By Lemma~\ref{FinSymDiffClosedS} we know that the three separations $\{A',B'\},\{A\cup A',B'\}$ and $\{A,B\cup B'\}$ are in $S'$.
First note that it suffices to show that $(A,B)\in\tau$ implies $(A\cup A',B')\in\tau$: then $(A',B')\in\tau$ follows from $(A',B')\le (A\cup A',B')\in\tau$ and the consistency of $\tau$.

As $(A,B\cup B')\le (A,B)\in\tau$ holds, we have $(A,B\cup B')\in\tau$ by consistency. Due to $\{(A,B\cup B'),(B',A\cup A')\}\in\Tau_2$ the only possibility for $\{A\cup A',B'\}$ is $(A\cup A',B')\in\tau$ as desired.
\end{proof}

The next lemma is an analogue of \cite[Lemma 2.2]{EndsAndTangles}, and with a bit of extra work we can mimic Diestel's proof.

\begin{lemma}\label{tanglesInduceF}
Let $\tau\in\Theta'$ and $X\subseteq X'\in\cX$ be given. Then
\begin{align*}
\af_{X',X}(\cF(\tau,X'))=\cF(\tau,X).
\end{align*}
\end{lemma}
\begin{proof}
Set $\cF=\af_{X',X}(\cF(\tau,X'))$ and note that this is in $\aF_X$. We check two cases:

First suppose that $\cF=c_X(C)$ holds for some $C\in\cC_X$ and let $(A,B)\in\vec{S'}$ be such that $A\cap B=X$ and $B\setminus A=V(C)$, i.e. $(A,B)=s_{X\to C}$.
We wish to show $(A,B)\in\tau$, since then $\{C\}\in\cF(\tau,X)$ implies $\cF(\tau,X)=c_X(C)=\cF$ as desired. 
For this, set $\cC'=\phi_{X',X}^{-1}(\{C\})$ and let $(A',B')\in\vec{S'}$ be such that $A'\cap B'=X'$ and $B'\setminus A'=V[\cC']$. 
By choice of $\cC'$ we have $\{A',B'\}\in S'$. 
Now $\{C\}\in\cF$ implies $\cC'\in\cF(\tau,X')$ by definition of $\af_{X',X}$, so in particular we have $(A',B')\in\tau$. 
By Lemma~\ref{FinSymDiffClosedS}, $\{A,B\}\in S'$ implies $\{A,B\cup X'\}\in S'$.
Furthermore, $(A,B\cup X')\le (A',B')$ holds since we have $A'=A\cup X'$ and $B'=B\cup X'$ by choice of $\cC'$.
Hence $(A',B')\in\tau$ together with $(A,B\cup X')\le (A',B')$ implies $(A,B\cup X')\in\tau$ by the consistency of $\tau$. In particular, Lemma~\ref{FinSymDiffParallel} implies $(A,B)\in\tau$ as desired.

Second suppose that $\cF$ is of the form $\cF_X(Y)$ for some $Y\in\crit(X)$, and assume for a contradiction that $\cF\neq\cF(\tau,X)$. 
In particular, we find some cofinite subset $\cC$ of $\cC_X(Y)$ with $\cC\notin\cF(\tau,X)$. 
Let $(A,B)\in\vec{S'}$ be such that $A\cap B=X$ and $B\setminus A=V[\cC]$. 
We wish to show $(A,B)\in\tau$, since then $\cC\in\cF(\tau,X)$ yields a contradiction as desired.
For this, set $\cC'=\phi_{X',X}^{-1}(\cC)$ and let $(A',B')\in\vec{S}$ be such that $A'\cap B'=X'$ and $B'\setminus A'=V[\cC']$. 
Since $\cC'$ is a cofinite subset of $\cC_{X'}(Y)$, we have $(A',B')\in\vec{S'}$. In particular, since $\cF(\tau,X')=\cF_{X'}(Y)$ holds by definition of $\af_{X',X}$, we also have $(A',B')\in\tau$.
By Lemma~\ref{FinSymDiffClosedS}, $\{A,B\}\in S'$ implies $\{A,B\cup X'\}\in S'$.
As in the first case this yields $(A,B)\in\tau$ as desired.
\end{proof}

\begin{theorem}\label{FasTangles}
The $\aleph_0$-tangles of $G$ with respect to $S'$ are precisely the limits of the inverse system $\{\aF_X,\af_{X',X}\}$.
\end{theorem}
\begin{proof}
Define $\Phi\colon \Theta'\to\aF$ by setting $\Psi(\tau)=(\cF(\tau,X)\,|\,X\in\cX)$ for every $\tau\in\Theta'$. By Lemma~\ref{tanglesInduceF} this is well-defined, and it is easy to see that $\Psi$ is injective. It remains to show that $\Psi$ is surjective. For this, let any $\A{v}=(\cF_X\,|\,X\in\cX)\in\aF$ be given.

If $\cF_X$ is principal for all $X$, then $\A{v}$ comes from an end $\omega$ of $G$. Furthermore, $\omega$ induces an $\aleph_0$-tangle $\tau_\omega$ of $G$ with respect to $S$, and $\tau_\omega$ induces the element $\tau':=\tau_\omega\cap \vec{S'}$ of $\Theta'$. Clearly, $\Psi$ sends $\tau'$ to $\A{v}$.

Otherwise $\cX_{\A{v}}$ is non-empty. By Proposition~\ref{UltrafiltersEconomic}, the map $\Psi\colon \cU/{\asymp}\to\aF$ is bijective, so there is some $\upsilon\in \Upsilon$ with $\Psi([\upsilon]_\asymp)=\A{v}$. Then $\upsilon$ corresponds to the $\aleph_0$-tangle $\tau_\upsilon$ of $G$ with respect to $S$, and $\Psi$ sends $\tau_\upsilon\cap \vec{S'}$ to $\A{v}$.
\end{proof}

\subsection{$\FG$ as inverse limit}\label{FGinvLimsubsec}

In section~\ref{tangleInvLim} we created an inverse system $\{G_\gamma,f_{\gamma',\gamma},\Gamma\}$ whose inverse limit describes the tangle compactification in that there exists a natural homeomorphism between $\iL=\invLim (G_\gamma\,|\,\gamma\in\Gamma)$ and $\TC$. In this chapter we define a subset $\Delta$ of $\Gamma$, and we show that $\invLim (G_\gamma\,|\,\gamma\in\Delta)$ describes $\FG$.\\

For every $X\in\cX$ let\index{$P_X^-$}
\begin{align*}
P_X^-:=\{\{C\}\,|\,C\in\cC_X^-\}=\big\{\{C\}\,|\,C\in\cC_X\text{ and }N(C)\notin \crit(X)\big\}
\end{align*}
and let $\Delta\subseteq\Gamma$ be the set of all $(X,P)\in\Gamma$ with $P$ of the form\index{$\Delta$}
\begin{align}\label{LForm}
P_X^-\uplus\biguplus_{Y\in \crit(X)}P_Y
\end{align}
where each $P_Y$ is a cofinite partition of $\cC_X(Y)$.
Letting $\Delta$ inherit the partial ordering from $(\Gamma,\le)$ yields a directed poset:

\begin{lemma}
$(\Delta,\le\cap\Delta^2)$ is a directed poset.
\end{lemma}
\begin{proof}
Let any two $(X,P),(Y,Q)\in\Delta$ be given. 
Our task is to find some $(Z,R)\in\Gamma$ with $(X,P)\le (Z,R)$ and $(Y,Q)\le (Z,R)$.
Put $Z=X\cup Y$ and suppose that $P$ is of the form
\begin{align*}
P_X^-\uplus\biguplus_{K\in \crit(X)}P_K
\end{align*}
with $P_K$ some cofinite partition of $\cC_X(K)$ for every $K\in \crit(X)$.
 For every $K\in \crit(X)\subseteq \crit(Z)$ set
\begin{align*}
P_K'=\{\cC\cap\cC_Z\,|\,\cC\in P_K\}\setminus\{\emptyset\}
\end{align*}
(this is a cofinite partition of $\cC_Z(K)$)
and for every $K\in \crit(Z)\setminus \crit(X)$ put $P_K'=\{\cC_Z(K)\}$. Then let
\begin{align*}
P':=P_Z^-\uplus\biguplus_{K\in \crit(Z)}P_K'
\end{align*}
and note that $(Z,P')$ is an element of $\Delta$.
Define $Q'$ similarly. Then both $P'$ and $Q'$ include $P_Z^-$.
For every $K\in \crit(Z)$ choose $R_K$ to be the coarsest partition of $\cC_Z(K)$ refining both $P_K'$ and $Q_K'$, and note that $R_K$ is again cofinite. Let
\begin{align*}
R:=P_Z^-\uplus\biguplus_{K\in \crit(Z)}R_K
\end{align*}
Then $(Z,R)\ge (X,P),(Y,Q)$ holds as desired.
\end{proof}

If we recall the inverse system $\{G_\gamma,f_{\gamma',\gamma},\Gamma\}$ from section~\ref{tangleInvLim}, then obviously the inverse limit $\invLim (G_\gamma\,|\,\gamma\in\Delta)$ should describe $\FG$.
But the elements of $\Delta$ are rather complicated, while a much easier attempt seems possible: For every $X\in\cX$ write\index{$\A{P}_X$}
\begin{align*}
\A{P}_X=P_X^-\uplus\biguplus_{Y\in \crit(X)}\{\cC_X(Y)\}
\end{align*}
and consider the subset\index{$\Delta'$}
\begin{align*}
\Delta':=\{(X,\A{P}_X)\,|\,X\in\cX\}
\end{align*}
of $\Delta$ which is also directed since $(X,\A{P}_X)\le (X',\A{P}_{X'})$ holds for all $X\subseteq X'\in\cX$. Furthermore, we can show that

\begin{lemma}\label{cofinalDelta}
$\Delta'$ is cofinal in $\Delta$.
\end{lemma}
\begin{proof}
Let any $(X,P)\in\Delta$ be given where $P$ is of the form~(\ref{LForm}). For every element $C$ of
\begin{align*}
\bigcup\{\cC\subseteq\cC_X\,|\,\exists Y\in \crit(X):\cC\in P_Y\text{ is finite}\}=:\cD\subseteq\cC_X
\end{align*}
pick a vertex $u(C)\in V(C)$. Set $X'=X\uplus\{u(C)\,|\,C\in\cD\}$ which is finite. Then $(X',\A{P}_{X'})\ge (X,P)$ holds as desired.
\end{proof}
Hence if $\invLim (G_\gamma\,|\,\gamma\in\Delta)$ describes $\FG$, then so should $\invLim (G_\gamma\,|\,\gamma\in\Delta')$ by Lemma~\ref{BibleInvLimCofinal}. But given $X\in\cX$ and $C\in\cC_X\setminus\cC_X^-$ it is not straightforward to construct a $\gamma\in\Delta'$ together with an open set in $G_\gamma$ corresponding to $\cO_{\FG}(X,\{C\})$ (an approach similar to the idea of the proof of Lemma~\ref{cofinalDelta} works to derive $\gamma$ from $(X,\A{P}_X)$). Hence if we want to put our inverse limit description to use, the superior simplicity of $\Delta'$ comes at the cost of `quality of life'.
Since the whole purpose of this inverse limit description is to increase our `quality of life', we take $\Delta$ over $\Delta'$, keeping in mind that $\Delta'$ always is an option by Lemma~\ref{BibleInvLimCofinal}.

Thus we set\index{$\ECO$}
\begin{align*}
\ECO=\invLim(G_\gamma\,|\,\gamma\in\Delta).
\end{align*}
It remains to show $\ECO\simeq\FG$. Since the definition of the $\af_{X',X}$ deviates from that of the $f_{X',X}$ too much, we cannot simply claim that mimicking the proof of Theorem~\ref{G*} will do, but neither do we wish to redo all of the work. 
Since it is clear that a homeomorphism defined similarly to the one constructed in the proof of Theorem~\ref{G*} should do, we do not lose any insights when we use Lemma~\ref{InvLimSurjEmbOnto} over an explicit construction.

For this, we define a continuous surjection $\sigma_\gamma\colon \FG\to G_\gamma$ for every $\gamma=(X,P)\in\Delta$ as follows: We let $\sigma_\gamma$ map every vertex to the partition class of $p(X,P)$ containing it. If $j$ is an inner edge point of some edge $e$ of $G$, then if $e$ is an edge of $G_\gamma$ we let $\sigma_\gamma(j):=j$, and otherwise we let $\sigma_\gamma$ map $j$ to the non-singleton partition class of $p(X,P)$ containing the endvertices of $e$. 
Finally if $\A{v}=(\cF_X\,|\,X\in\cX)$ is an element of $\aF$ we let $\sigma_\gamma$ assign $V[\cC]$ to $\A{v}$ where $\cC$ is the unique element of $P\cap\cF_X$.

If $O$ is open in $H_\eta$ for some $\eta\in\Delta$, then we denote by $\cO_{\ECO}(O,\eta)$\index{$\cO_{\ECO}(O,\eta)$} the set $\ECO\cap\prod_{\gamma\in\Delta} O_\gamma$ with $O_\gamma=G_\gamma$ for all $\gamma\in\Gamma-\{\eta\}$ and $O_\eta=O$.

\begin{lemma}\label{ECOsigmaCOMP}
The maps $\sigma_\gamma$ are compatible.
\end{lemma}
\begin{proof}
Let $\gamma'>\gamma$ be elements of $\Delta$, and let $\xi\in\FG$ be given; we have to show
\begin{align}\label{EQcompatible}
(f_{\gamma',\gamma}\circ\sigma_{\gamma'})(\xi)=\sigma_\gamma(\xi).
\end{align}
This is clear for $\xi\in G\cup\Omega$, so we may assume that that $\xi=(\cF_Z\,|\,Z\in\cX)$ is an element of $\aF\setminus\Omega$.
Write $(X',Q)=\gamma'$, $(X,P)=\gamma$, and let $\cC'$ and $\cC$ be the unique elements of $Q\cap\cF_{X'}$ and $P\cap\cF_X$, respectively. 
By definition of $f_{\gamma',\gamma}$ we know that in order to verify (\ref{EQcompatible}) it suffices to show 
\begin{align}\label{EQcompsinc}
V[\cC']\subseteq V[\cC].
\end{align}
For this we check several cases:

First suppose that $\cF_{X'}=\cF_{X'}(Y)$ holds for some $Y\in\crit(X')$. Then $Q$ is of the form
\begin{align*}
Q_{X'}^-\uplus\biguplus_{K\in \crit(X')}Q_K
\end{align*}
where the $Q_K$ are cofinite partitions of $\cC_{X'}(K)$ and $\cC'$ is the unique infinite element of $Q_Y$. 
If $Y\not\subseteq X$ holds, then $\cF_X=\A{f}_{X',X}(\cF_{X'})$ is of the form $c_X(C)$ where $C$ is the unique component of $G-X$ including $\bigcup\cC_{X'}(Y)$.
Then $\cC'\subseteq\cC_{X'}(Y)$ and $\cC\in c_X(C)$ yield
\begin{align*}
V[\cC']\subseteq V[\cC_{X'}(Y)]\subseteq V(C)\subseteq V[\cC]
\end{align*}
as desired.
Otherwise $Y\subseteq X$ implies $\cF_X=\cF_X(Y)$. We may assume that $P$ is of the form~(\ref{LForm}). In particular, $\cC$ is the unique infinite element of $P_Y$.
Since $\cC'$ is cofinite in $\cC_{X'}(Y)$ and $\cC_{X'}(Y)$ is cofinite in $\cC_X(Y)$ (Observation~\ref{OneTimeUse0}) we know that $\cC'$ is cofinite in $\cC_X(Y)$.
Now some $\cD\in P_Y$ satisfies $V[\cD]\supseteq V[\cC']$ due to $\gamma<\gamma'$, and since $\cC'$ is cofinite in $\cC_X(Y)$, the only possibility for $\cD$ is $\cD=\cC$ as desired.

For the second case suppose that $\cF_{X'}=c_{X'}(C')$ holds for some component $C'$ of $G-X'$.
In particular $C'$ is in $\cC'$ due to $\cC'\in\cF_{X'}$. 
Let $C$ be the unique component of $G-X$ including $C'$. 
Then the definition of $\A{f}_{X',X}$ yields $\cF_X=c_X(C)$, and
 $\cC\in\cF_X$ implies $C\in\cC$. Let $\cD\in P$ be the partition class with $V[\cD]\supseteq V[\cC']$ (which exists due to $\gamma<\gamma'$). In particular, $\cD$ must contain $C$ since $\cC'$ does, so $\cD\in\cF_X$ follows. Assume for a contradiction that $\cD$ and $\cC$ are distinct: then $\cC$ and $\cD$ are disjoint partition classes of $P$, so $C\in\cD$ implies $C\notin\cC$ contradicting $\cC\in\cF_X=c_X(C)$. Hence $\cD=\cC$ yields (\ref{EQcompsinc}) as desired.
\end{proof}

\begin{lemma}\label{ECOsigmaCTS}
The maps $\sigma_\gamma$ are continuous.\qed
\end{lemma}

\begin{theorem}\label{FGinvLim}
For every graph $G$ we have $\FG\simeq\ECO$.
\end{theorem}
\begin{proof}
Consider the map
\begin{align*}
\Phi\colon \FG&\to\invLim(G_\gamma\,|\,\gamma\in\Delta)\\
\xi&\mapsto(\sigma_\gamma(\xi)\,|\,\gamma\in\Delta)
\end{align*}
which clearly is injective.
By Theorem~\ref{FGcompHDetc} we know that $\FG$ is compact.
Since the $\sigma_\gamma\rest G$ are surjective, so are the $\sigma_\gamma$.
Hence by Lemmas~\ref{ECOsigmaCOMP}~and~\ref{ECOsigmaCTS} the $\sigma_\gamma$ form a compatible system of continuous surjections, and applying Lemma~\ref{InvLimSurjEmbOnto} yields that $\Phi$ is also a continuous surjection.
It remains to verify that $\Phi^{-1}$ is continuous. For this, let any point $x$ of $\FG$ be given together with some basic open neighbourhood $O$ of $x$.

If $x$ is in $G$ we choose some $X\in\cX$ such that $x$ is contained in the 1-complex of $G[X]$, and set $\eta=(X,\A{P}_X)$. Then $O$ is also a basic open neighbourhood of $x$ in $G_\eta$, so $\cO_{\ECO}(O,\eta)=\Phi^{-1}[O]$ holds.

Otherwise $x$ is an element of $\aF$ and we check two subcases. First, if $O$ is of the form $\cO_{\FG}(X,\{C\})$ for some $X\in\cX$ and $C\in\cC_X$, then we let $P$ be a partition of $\cC_X$ such that $\{C\}$ is a singleton partition class of $P$ and $(X,P)\in\Delta$.
Next, we consider the basic open neighbourhood $O':=\mathring{E}(X,C)\cup\{V(C)\}$ of the dummy vertex $V(C)$ in $G_\eta$. 
Clearly, $\cO_{\ECO}(O',\eta)=\Phi^{-1}[O]$ holds as desired. 
For the second subcase suppose that $O$ is of the form $\cO_{\FG}(X,\cC)$ for some $X\in\cX$ with $\cC$ a cofinite subset of $\cC_X(K)$ for some $K\in \crit(X)$. 
Then we let $P$ be a partition of $\cC_X$ such that $\cC$ is a partition class of $P$ and $(X,P)\in\Delta$.
Considering the basic open neighbourhood $O':=\mathring{E}(X,\bigcup\cC)\cup\{V[\cC]\}$ of the dummy vertex $V[\cC]$ in $G_\eta$ yields $\cO_{\ECO}(O',\eta)=\Phi^{-1}[O]$ as expected. Thus $\Phi^{-1}$ is continuous.
\end{proof}

\begin{corollary}
For every infinite graph $G$ we have 
\begin{align*}
\FG\simeq\ECO\simeq\invLim (G_\gamma\,|\,\gamma\in\Delta').
\end{align*}
\end{corollary}
\begin{proof}
Combine Theorem~\ref{FGinvLim} with Lemmas~\ref{BibleInvLimCofinal}~and~\ref{cofinalDelta}.
\end{proof}

\subsection{Outlook}

Of course, finding meaningful graph theoretical reasons to consider $\cC$-systems or $\aF\simeq\cU/{\asymp}$ is of highest interest.
Since we have
\begin{align*}
\invLim (G_\gamma\,|\,\gamma\in\Delta')\simeq\invLim (G_\gamma\,|\,\gamma\in\Delta)\simeq\FG\simeq\TC/{\asymp}
\end{align*}
it might be possible to show that the $\aleph_0$-tangles of $G$ w.r.t $S'$ induce precisely the $\aleph_0$-tangles of $G$ w.r.t $S''$, where $S''$ is the set of all $\{A,B\}\in S$ respecting the finite partition $\A{P}_{A\cap B}$ of $\cC_{A\cap B}$.

Finally, it remains to compare $\aF$ and $\cU$ in the wild. Since $\cU$ is so much bigger than $\aF$, it might be easier to get some generalisation of thins sums to work in a modified version of $\TC$ than in a modified version of $\FG$. On the other hand, we might run into dependencies on certain models of ZFC.

\newpage
\section{Auxiliary edges and tree-packing}\label{AuxArcTP}

\begin{center}
\textit{In this chapter, $G$ is always assumed to be connected.}
\end{center}

\subsection{Introduction}

By now, the topological cycle space of locally finite graphs has been the subject of extensive studies, but several problems are known to occur in straightforward generalisations to arbitrary infinite graphs. For example consider the heavy edge sets in the following three graphs:

\begin{figure}[H]
\centering
\includegraphics[scale=0.85]{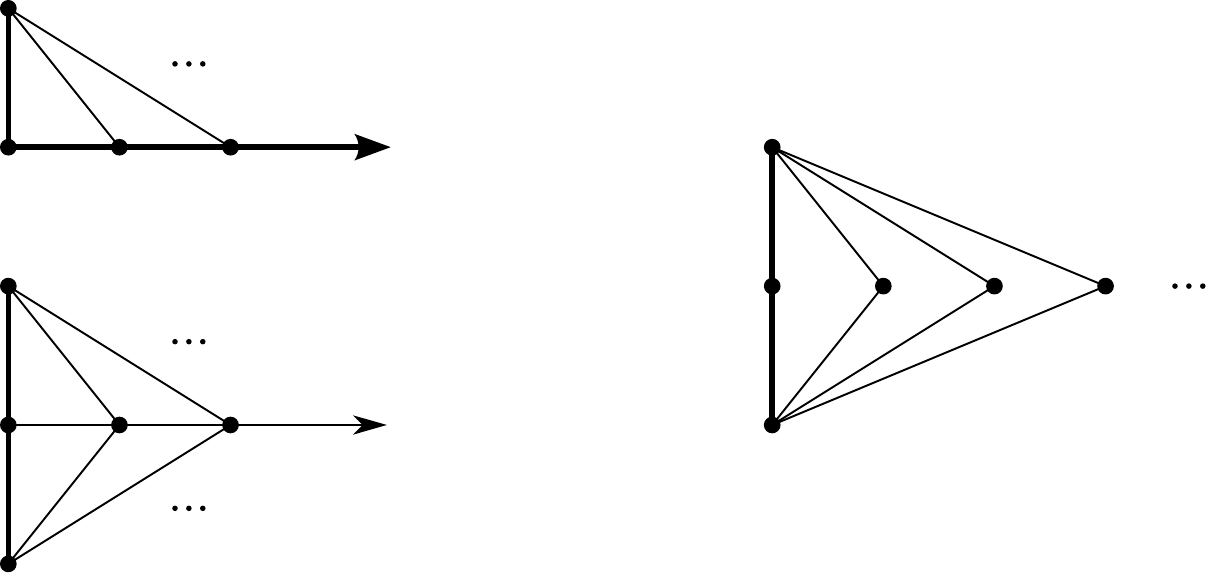}
\caption{Classical cycle space obstructions}
    \label{fig:ProbsAuxMot}
\end{figure}

All heavy edge sets can be obtained as thin sum of all the facial cycles of their respective graphs, and hence should be elements of their cycle spaces. 
For the top left graph, a solution is known: Identifying the end with its dominating vertex yields a compact Hausdorff quotient $\tilde{G}$ of $(G\cup\Omega,\textsc{VTop})$ whose topology is known as \textsc{ITop}. In this quotient, 
the ray actually converges to its dominating vertex, as pictured below:
\begin{figure}[H]
\centering
    \includegraphics[scale=0.85]{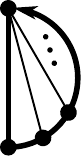}
\end{figure}
Studying this quotient proved rewarding, e.g. see \cite{duality} and \cite{DualTrees}.
For the graph $G$ on the bottom left of Fig.~\ref{fig:ProbsAuxMot}, however, $(\tilde{G},\textsc{ITop})$ is not defined since the relation we would use to yield $\tilde{G}$ no longer is transitive.
Obviously, we could fix this by taking the transitive closure of that relation, but this would result in vertex identification, which we want to avoid at all cost. 
Finally, consider the graph $G$ on the right in Fig.~\ref{fig:ProbsAuxMot}. This graph has no end, and hence \textsc{VTop} is Hausdorff, but not compact, and identifying the two vertices of infinite degree does not change that. 
A space solving all our problems by using vertex identification is already known: $\cE G$ is compact Hausdorff, and many of the theorems known for locally finite graphs admit easy generalisations to $\cE G$ for arbitrary infinite graphs (among these we find a working cycle space):

\begin{figure}[H]
\centering
\includegraphics[scale=0.85]{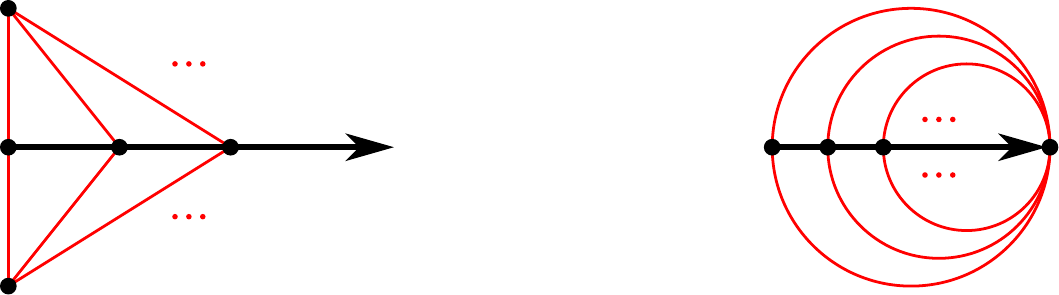}
\caption{A ray dominated by two vertices (left) and visualised in $\cE G$ (right). I expect the black edges and vertices to induce a TST for any sensible notion of a TST.}
\label{fig:OneEndTwoDom}
\end{figure}
\begin{figure}[H]
\centering
\includegraphics[scale=0.85]{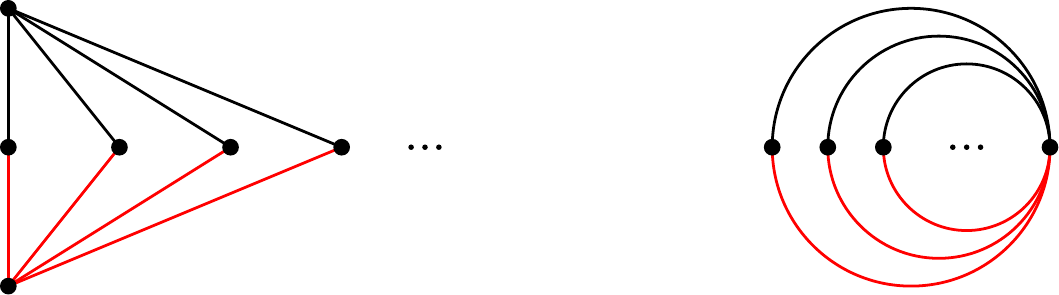}
\caption{A $K_{2,\aleph_0}$ (left) and visualised in $\cE G$ (right). I expect the black edges and vertices to induce a TST for any sensible notion of a TST.}
    \label{fig:K_2N_NoTST_v2}
\end{figure}
This is possible thanks to the inverse limit $\CL G\CL$ describing $\cE G$. But $\cE G$ has a huge downside: since it uses vertex identification, we lose the structure of the graph.
For an extreme example, consider the graph $G$ from Fig.~\ref{fig:badGAux}.
Clearly, this graph admits a very rich structure, e.g. its underlying binary tree plus the edge $\emptyset x$ yields an NST and in $(\tilde{G},\textsc{ITop})$ we find an edge-less Hamilton circle (see \cite[Proposition 3.4 and Corollary 3.5]{TST}).\footnote{This is one reason why the usage of \textsc{ITop} is commonly restricted to finitely separable graphs.}
But since no two of its vertices are finitely separable, $\cE G$ is just a hawaiian earring, and we may say that
$\cE G$ takes the sledgehammer approach by sacrificing the structure of our graphs in exchange for easy generalisations.
\begin{figure}[h]
\includegraphics[clip,page=11,trim=40 170 40 285,width=\textwidth]{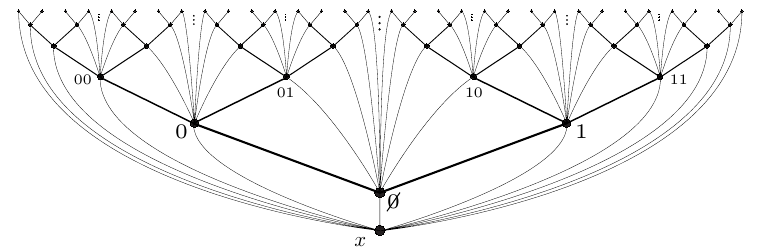}
\caption{A graph from \cite[Fig. 6]{TST} admitting an edge-less Hamilton circle in $(\tilde{G},\textsc{ITop})$.}
\label{fig:badGAux}
\end{figure}

Let us have a look on the topologies considered so far from another perspective.
Generally speaking, \textsc{ITop} approaches the generalisation problem `from bottom up' by using the straightforward generalisation \textsc{VTop} of the Freudenthal compactification---which is natural on locally finite graphs---on arbitrary graphs, then fixing problems by moving on to a Hausdorff quotient and restriction to a suitable class of graphs (which is still big enough to be of interest).

The tangle compactification on the other hand naturally extends the Freu-\\denthal compactification to arbitrary infinite graphs in that it uses $\aleph_0$-tangles to compactify the 1-complex of $G$. 
Yet it seems like the tangle compactification is not the `right setting' for arbitrary infinite graphs:
First of all, we recall from the introduction of this work that a $K_{2,\aleph_0}$ admits some set of edges which we expect to induce a TST of the tangle compactification, but which has no meaningful acirclic and topologically connected superset.
Second, Lemma~\ref{TCarcsAvoidUltras} tells us that arcs in the tangle compactification default to arcs in $\cV G$, i.e. the arcs avoid the ultrafilter tangles. But it would be great if at least leanly structured countable graphs would admit arc-connected circles and TSTs using ultrafilter tangles in order to overcome the difficulties discussed earlier.

Therefore, we have two choices. First, we may stick to our definitions of circles and TSTs in terms of arcs while modifying the tangle compactification in some way deemed natural.  
Second, we may adjust both the tangle compactification and our definitions of circles and TSTs. 
Since our definitions of circles and TSTs heavily rely on the unit interval while we do not impose any cardinal bounds on our graphs, the second choice is more appealing.

If we modify the tangle compactification, we should see to it that the resulting space $\LG$ is both compact and Hausdorff:
then $\LG$ is a continuum and the field of continuum theory provides us with a useful tool box.
But what could $\LG$ look like? Of course, it should solve our earlier problems from Fig.~\ref{fig:ProbsAuxMot}. On the other hand, it should admit a TST even for the graph $G$ from Fig.~\ref{fig:badGAux}. 
Here, the obstruction in $(\tilde{G},\textsc{ITop})$ is the existence of an edge-less Hamilton circle. This circle does not look very `circle-like' in that drawing, hence let us draw the graph again, but this time we start with a circle in the plane ($\S^1$, say) and embed the vertices of $G$ into it discretely, as sketched in Fig.~\ref{fig:badGnicelyDrawn}.
\begin{figure}[t]
    \centering
    \def\svgwidth{\columnwidth}
    \scalebox{.6}{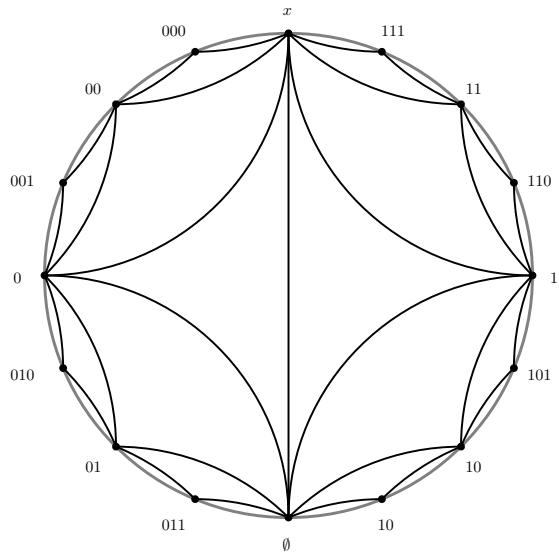}
    \caption{A sketch of the graph from Fig.~\ref{fig:badGAux} embedded into the plane with all vertices lying on $\S^1$.}
    \label{fig:badGnicelyDrawn}
\end{figure}
Now the idea of a circle consisting exactly of the vertices and ends of $G$ appears much less inconvenient, doesn't it? Also, the (finite) cycles induced by the boundaries of the inner faces seem to converge to $\S^1$ in some sense. Let us state this more clear: 

For every $n\in\N$ denote by $\cC_n$ the set of all (finite) cycles which are induced by face boundaries of the drawing above, and which only meet $V(G)$ in $\bigcup_{k\le n}2^k$.
Then for every $\epsilon>0$ there is some $N\in\N$ such that the (finite) cycle $C_n$ with edge set $\sum E[\cC_N]$ is included in $\{x\in\R^2\,|\,d_2(x,0)\in (1-\epsilon,1]\}$. Also note that $\sum E[\bigcup_{n\in\N}\cC_n]=\emptyset$ holds.

This gives us a first hint: if we introduce some sort of limit edges, one between every end of $G$ and each of its dominating vertices, and if we embed these into $\S^1$ appropriately (which is possible since we embedded $V(G)$ discretely), then the edge-less Hamilton circle is no longer edge-less. 
Furthermore, by the looks of the drawing, it should not be a problem to define basic open neighbourhoods of the inner edge points of these limit edges such that the sequence $(\sum E[\cC_n])_{n\in\N}$ `converges' to our Hamilton circle induced by the limit edges.
On top of that, the obvious NST of $G$ (the underlying $T_2$ plus the edge $x\emptyset$) now should induce a TST, and the Hamilton circle minus one of its limit edges should be a TST.

But what about a $K_{2,\aleph_0}$ or graphs admitting ultrafilter tangles in general? 
In case of $G\simeq K_{2,\aleph_0}$, if we remember why the expected TST did not work out, the problem informally amounted to the ultrafilter tangles not being `sufficiently connected' to the two vertices of infinite degree. 
Furthermore, we cannot meaningfully modify the topology of the tangle compactification of this graph into a Hausdorff one without losing compactness.
But if we join every ultrafilter tangle $\upsilon$ of $G$ to each vertex in $X_\upsilon$ via a limit edge, and if we generalise the topology of the tangle compactification in an \textsc{MTop}-like way onto the new space (treating inner limit edge points almost like tangles), then we obtain a continuum.

Since the space $\LG$ is not that easy to work with, 
we dedicate this chapter to a Hausdorff auxiliary space $\GH$ which in general is not compact, but whose structure combinatorially captures the basic structure of $\LG$. 
Informally, $\GH$ is obtained from $|G|$ with \textsc{MTop} by disjointly adding auxiliary (multi-)edges between every end of $G$ and each of its dominating vertices, and between any two distinct vertices of the same critical vertex set (one for each critical vertex set both are contained in). 
As basic open neighbourhoods of the inner edge points of auxiliary edges we take the same ones as for inner edge points of edges of $G$, and we carefully adjust the open neighbourhoods of other points to include also half-open partial auxiliary edges.

As we have seen two times before, if we construct a TST in $\cE G$, then its inner edge points do not induce a TST of $\TC$, e.g. in a $K_{2,\aleph_0}$ the sole non-singleton $\sim$-class of $\TC$ (which consists of all ultrafilter tangles and the two vertices of infinite degree; also recall $\TC/{\sim}=\Gq\simeq\cE G$) is totally disconnected in $\TC$.
But in $\aA(K_{2,\aleph_0})$ that $\sim$-class minus the ultrafilter tangles plus the auxiliary edge between its two vertices is arc-connected. 
As it turns out, between every two distinct points $x$ and $y$ of $V\cup\Omega$ there exists an auxiliary arc from $x$ to $y$ if and only if $x\sim y$, where an auxiliary arc simply is an arc in the closure of the set of all auxiliary edges (see Sections~\ref{subsecAuxArcConstr} and~\ref{subsecAuxArcProp}).
Interestingly, this holds without any cardinality bounds imposed on the graph considered. Hence, if we wish to prove a statement about the existence of certain structures in $\GH$, then the following two step procedure might be worth a try: First, prove it for $\cE G$ using $\CL G\CL$. Second, `lift' the obtained structures to $\GH$ by expanding the non-trivial $\sim$-classes to something auxiliary arc-connected.
For the rest of the section, we carry out this procedure to prove a generalisation of tree-packing (with the usual circle and TST definitions) for countable graphs, demonstrating the synergy of the two spaces $\cE G$ and $\GH$.

For finite graphs, the so-called `tree packing' is a fundamental theorem that was proved by Nash-Williams and Tutte independently in 1961:

\begin{theorem}[{\cite[Theorem 2.4.1]{Bible}}]\label{finiteTreePacking}
The following are equivalent for all finite multigraphs $G$ and $k\in\N$:
\begin{enumerate}
\item $G$ contains $k$ edge-disjoint spanning trees.
\item $G$ has at least $k(|P|-1)$ edges across any finite vertex partition $P$.
\end{enumerate}
\end{theorem}

For infinite graphs, however, the naive generalisation fails.
Indeed, let $G$ be a $K_{2,\aleph_0}$ whose vertices of infinite degree we denote by $x$ and $y$. Furthermore, we enumerate the other vertices of $G-x-y$ as $u_0,u_1,\hdots$. Now we show that $G$ satisfies (ii) for $k=2$. 
For this, let $P=\{p_0,\hdots,p_\ell\}$ be any finite vertex partition of $G$. 
If $x$ and $y$ are contained in different partition classes, with $x\in p_0$ say, then infinitely many edges leave $p_0$ since there are infinitely many disjoint paths from $x$ to $y$ in $G$, so (ii) holds. 
Otherwise $x$ and $y$ are contained in the same partition class $p_0$, say. If there are infinitely many of the $u_n$ not contained in $p_0$, then again infinitely many edges leave $p_0$, witnessing (ii). Hence we may assume that $p_0$ contains all but finitely many of the $u_n$. Then for every $i>0$ the partition class $p_i$ is a finite subset of $\{u_n\,|\,n\in\N\}$, and every edge leaving $p_i$ is incident with precisely one of $x$ and $y$. Thus exactly $2|p_i|$ many edges leave $p_i$. By choice of $p_0$ the partition $P$ has precisely 
\begin{align*}
\sum_{i=1}^\ell 2|p_i|\ge \sum_{i=1}^\ell 2=2\ell=2(|P|-1)
\end{align*}
many cross-edges, so again (ii) holds. This completes the verification of (ii) for $k=2$. But obviously, every spanning tree of $G$ has degree 2 at some $u_n$. In particular, every two spanning trees of $G$ share an edge, so (i) fails as claimed.

The TSTs of known topologies on $G$ (not relying on vertex-identification) face the same problem, but $\GH$ does not:
\begin{figure}[H]
\centering
\includegraphics[scale=1]{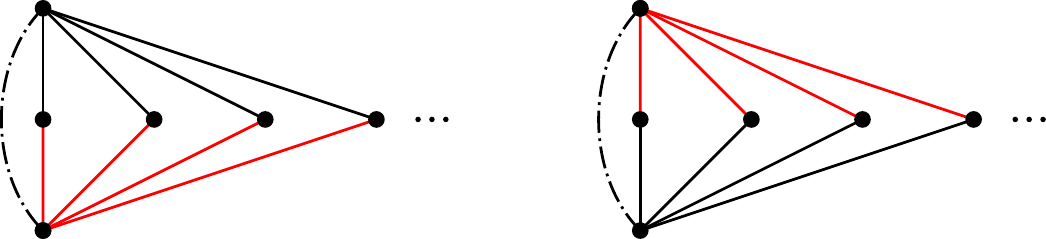}
\caption{The black edges (including the dashed auxiliary edges) form TSTs of $G\simeq K_{2,\aleph_0}$ which are edge-disjoint on $E(G)$.}
    \label{fig:K_2N_2ATSTs}
\end{figure}
Of course the two TSTs of $\GH$ depicted in Fig.~\ref{fig:K_2N_2ATSTs} share an edge, but this is not an edge of $G$, and clearly this is best possible for $\GH$.
Furthermore, the shared auxiliary edge indicates that its endvertices are in the same (infinitely edge-connected) ${\sim}$-class, so both TSTs share an artificial edge which represents infinite edge-connectivity.

\subsection{The space $\GH$}\label{AGsubsection}

In this section, we formally define the topological space $\GH$ which makes use of auxiliary edges.\\

For this, we let\index{$\AuxE$, $\AuxE(\Omega)$, $\AuxE(\cX)$}
\begin{align*}
\AuxE(\Omega)&:=\AuxE(\Omega(G)):=\{u\omega\,|\,\omega\in\Omega(G),u\in\Delta(\omega)\}\\
\AuxE(\cX)&:=\AuxE(\cX(G)):=\bigcup_{X\in\crit(\cX)} [X]^2\times\{X\} \\
\AuxE&:=\AuxE(G):=\AuxE(\Omega)\cup\AuxE(\cX)
\end{align*}
and we let $\GG$\index{$\GG$} be the multigraph on $V(G)\cup\Omega(G)$ with edge set $E(G)\cup\AuxE$ where every $(\{x,y\},X)\in\AuxE(\cX)$ has endvertices $x$ and $y$.
The elements of $\AuxE$ are referred to as \textit{auxiliary edges}.\index{auxiliary edge}
Unless stated otherwise, our standard notation such as $\cX$, $\cC_X$ and $\Omega$ still depends on $G$, not on $\GG$.
Next, we obtain a topological space $\GH$ with ground set the 1-complex of $\GG$ by declaring as (basic) open the following sets:\index{$\GH$}

For inner edge points of $\GG$ we take the usual basic open neighbourhoods.
For every vertex $u$ of $G$ and every $\epsilon\in (0,1]$ we declare as open the star $\cO_{\GG}(u,\epsilon)$ of half-open intervals.
Finally, for every $X\in\cX$, for each $\cC\subseteq\cC_X$ and for all $\epsilon\in (0,1]$ take\index{$\cO_{\GH}(X,\cC,\epsilon)$}
\begin{align*}
\cO_{\GH}(X,\cC,\epsilon):=\cA(X,\cC)\cup\mathring{E}_{\GG}\big(\medcup\cC,\cA(X,\cC)\big)\cup\bigcup_{\xi\in\cA(X,\cC)}\cO_{\GG}(\xi,\epsilon)
\end{align*}
where 
\begin{align*}
\cA(X,\cC)=\medcup\cC\cup\bigcup_{C\in\cC}\Omega(X,C).
\end{align*}
(Informally, if $\cC$ is of the form $\{C(X,\omega)\}$ for some end $\omega$ of $G$, then we may think of $\cO_{\GH}(X,\cC,\epsilon)$ as $\hat{C}_\epsilon(X,\omega)$ plus (possibly half-open partial) auxiliary edges.) 
Using Corollary~\ref{GGwellDef} it is easy to check that this really yields a topology.
Similar to $\hat{C}_\epsilon(X,\omega)$ and $\hat{C}(X,\omega)$, for every $X\in\cX$ and $\omega\in\Omega$ we write\index{$\aC(X,\omega)$, $\hat{\aC}_\epsilon(X,\omega)$}
\begin{align*}
\hat{\aC}_\epsilon(X,\omega)&:=\cO_{\GH}(X,\{C(X,\omega)\},\epsilon)\\
\hat{\aC}(X,\omega)&:=\hat{\aC}_1(X,\omega)
\end{align*}

\begin{obs}
$\GH$ is Hausdorff.
\end{obs}

\begin{lemma}\label{finCutsG''}
The finite cuts of $G$ are finite cuts of $\GG$.
\end{lemma}
\begin{proof}
This follows from Theorem~\ref{InfIndPaths} and the fact that no finite cut separates an end from any of its dominating vertices.
\end{proof}

\begin{lemma}\label{vxJAL}
Let $X\in\cX$ be given together with a bipartition $\{\cC,\cC'\}$ of $\cC_X$. Furthermore suppose that $K$ is a connected subset of $\GH$ meeting both $\cO_{\GG}(X,\cC,1)$ and $\cO_{\GG}(X,\cC',1)$. Then $K$ meets $X$.
\end{lemma}
\begin{proof}
Otherwise $\{\cO_{\GH}(X,\cC,1)\,,\,\cO_{\GH}(X,\cC',1)\}$
induces an open bipartition of $K$ which is impossible.
\end{proof}

The following Lemma basically restates the first part of the Jumping Arc Lemma \cite[Lemma 8.5.3 (i)]{Bible} for $\GH$. Hence, with Lemma~\ref{finCutsG''} the proof is analogue:

\begin{lemma}[$\L$]\label{cutJAL}
Let $F$ be a finite cut of $G$ with sides $V_1$ and $V_2$. Then
\begin{align*}
\GH\setminus\mathring{E}(\GG)=\overline{V_1}^{\;\GH}\uplus\overline{V_2}^{\;\GH}
\end{align*}
and no connected subset of $\GH\setminus\mathring{F}$ meets both $\overline{V_1}^{\;\GH}$ and $\overline{V_2}^{\;\GH}$.\qed
\end{lemma}

An arc $A\subseteq\GH$ is called an \textit{auxiliary arc} if $A\subseteq\overline{\AuxE}^{\;\GH}$.\index{auxiliary arc}

\begin{lemma}[$\L$]\label{SigmaArcEdgesDense}
$\bigcup \mathring{E}_{\GG}(A)$ is dense in $A$ for every arc $A$ in $\GH$.
\end{lemma}
\begin{proof}
First choose a homeomorphism $\sigma\colon \I\to A$.
We claim that $\overline{\bigcup \mathring{E}_{\GG}(A)}=A$. Assume not for a contradiction. Then $\overline{\bigcup \mathring{E}_{\GG}(A)}\subsetneq A$ since $A$ is closed. 
Pick some $\xi\in A\setminus \overline{\bigcup \mathring{E}_{\GG}(A)}$ together with some basic open neighbourhood $O$ of $\xi$ in $\GH$ which avoids $\overline{\bigcup \mathring{E}_{\GG}(A)}$. 
It is impossible to find such a neighbourhood for vertices of $G$ or inner edge points of $\GG$, so $\xi$ must be an end of $G$. Let $\cI$ be some basic open subset of $\I$ which $\sigma$ maps to $O$. Then $\sigma[\cI]\subseteq\Omega(G)$ holds by the previous argument. Hence it suffices to show that there is no arc living entirely in $\Omega(G)$ to yield a contradiction:

Assume for a contradiction that there is an arc $A'\subseteq\Omega(G)$ starting in $\omega$ and ending in $\omega'$, say. Pick some $X\in\cX$ witnessing $\omega\neq\omega'$. Then $A'$ avoids $X$, contradicting Observation~\ref{vxJAL}.
\end{proof}

If $H=(V',E')$ is a subgraph of $\GG$, then we write $\overline{H}=\overline{V'}\cup\mathring{E}'$ for its closure in $\GH$. If a subspace $\Xi$ of $\GH$ is of the form $\overline{H}$ and every $\omega\in V'\setminus V(G)$ is incident with an auxiliary edge from $E'$, then we call $\Xi$ a \textit{standard subspace}\index{standard subspace} of $\GH$ and write $E(\Xi)=E'$.
By Lemma~\ref{SigmaArcEdgesDense} we have $\overline{H}=\overline{\mathring{E}'}$ if $\overline{H}$ is arc-connected.
An $\aA$\textit{TST}\index{\ATST{}} of $G$ is a uniquely arc-connected standard subspace of $\GH$ including $V(G)$.
If $\cA$ is an auxiliary arc-component of $\GH$ and $\Xi=\overline{\mathring{E}'}\subseteq\cA$ is a uniquely arc-connected standard subspace with $E'\subseteq\AuxE$, then $\Xi$ is called an $\aA$\textit{TST of} $\cA$.
If $\Tau$ is an \ATST{} of $G$ or of an auxiliary arc-component of $\GH$ and $f$ is any edge in $E(\Tau)$, then $\Tau-\mathring{f}$ has precisely two arc-components $\Tau_1$ and $\Tau_2$ and we write $D_f^{\Tau}$ for the \textit{fundamental cut}\index{fundamental cut} of $\Tau$ (with respect to $f$) which consists precisely of those edges of $\GG$ with one endvertex in $\Tau_1$ and the other in $\Tau_2$.
A \textit{circle}\index{circle} of $\GH$ is the image of a homeomorphic embedding of $\S^1$ into $\GH$. By Lemma~\ref{SigmaArcEdgesDense} every circle of $\GH$ is a standard subspace of $\GH$.
If $C$ is a circle of $\GH$, then $E(C)$ is a \textit{circuit}.

\begin{obs}
If $T$ is an NST of $G$ and $\Tau$ is the closure of $T$ in $\GH$, then $\le_T$ naturally extends to an ordering $\le_{\Tau}$ of $\Tau$. Furthermore, by Theorem~\ref{InfIndPaths}, all the sets $X_\upsilon$ and $\Delta(\omega)\cup\{\omega\}$ (where $\upsilon\in\Upsilon$ and $\omega\in\Omega$) form chains in $\le_{\Tau}$.
Hence if $f=ut$ is an edge of $T$ with $u<_Tt$, then all edges of $D_f^{\Tau}$ (in particular those of $D_f^{\Tau}\cap\AuxE(G)$) are incident with $\lceil u\rceil_T$.
\end{obs}

\begin{lemma}[$\L$]\label{GH:NSTisTST}
Let $G$ be a any graph and let $T$ be an NST of $G$. Then the closure of $T$ in $\GH$ is a TST of $G$ with respect to $\GH$.
\end{lemma}
\begin{proof}
Let $\Tau$ denote the closure of $T$ in $\GH$. 
Using normal rays it is straightforward to show that $\Tau$ is arc-connected, so assume for a contradiction that there is some circle $C\subseteq\Tau$.
By Lemma~\ref{SigmaArcEdgesDense} we may pick some edge $f=ut$ of $\GG$ which $C$ traverses. Without loss of generality we have $u<_Tt$ since $f$ is an edge of the graph $G$ by choice of $C$. Now consider the arc $A:=C\setminus\mathring{f}$. Since $T$ is an NST we know that $D_f^{\Tau}=E_{\GG}(X,D\cup\Omega(X,\{D\}))$ where $X:=\lceil u\rceil_T\in\cX$ and $D:=\lfloor t\rfloor_T$. But then
\begin{align*}
\Big\{\cO_{\GH}\big(X,\{D\},1/2\big)\enspace,\enspace\cO_{\GH}\big(X,\cC_X\setminus\{D\},1\big)\cup\bigcup_{x\in X}\cO_{\GG}\big(x,1/2\big)\Big\}
\end{align*}
induces an open bipartition of $A$ since $A\subseteq\Tau\setminus\mathring{f}\subseteq\GH\setminus(\mathring{D}_f^{\Tau}\cup\mathring{\AuxE}(G))$.
\end{proof}

\subsection{Construction of auxiliary arcs}\label{subsecAuxArcConstr}

The aim of this section is to construct an auxiliary arc between any two distinct vertices of $G$ which are not finitely separable. In order to do this, we will approximate a topological path in $\overline{\AuxE}$ between the two vertices via a countable linear ordering on some special set of vertices. Then we `fill in the gaps' of that linear ordering with inner edge points of auxiliary edges and ends of $G$, yielding a topological path between the two vertices in $\overline{\AuxE}$. Finally, since we cannot guarantee injectivity at the ends of $G$, we involve general topology to obtain the desired auxiliary arc.\\

Suppose that $x$ and $y$ are two distinct vertices of $G$ with $x\sim y$ and denote by $\cP$ the set of all $x$-$y$ paths in $G$. Every path $P\in\cP$ naturally induces a linear ordering $\le_P$ on its vertex set with $x<_Py$, and for each $X\in\cX$ we denote by $\le_P^X$ the linear ordering on $X\cap V(P)$ induced by $\le_P$. Furthermore, for every $X\in\cX$ we define a map $\psi_X$ with domain $\cP$ by letting
\begin{align*}
\psi_X(P):=(X\cap V(P),\le_P^X)
\end{align*}
for each $P\in\cP$. In addition, for every $X\subseteq X'\in\cX$ we set up a function $\varphi_{X',X}\colon \im(\psi_{X'})\to\im(\psi_X)$ by letting
\begin{align*}
\varphi_{X',X}((Y',\le')):=(X\cap Y',{\le'}\cap (X\cap Y')^2)
\end{align*}
for every $(Y',\le')\in\im(\psi_{X'})$. Hence $\{\im(\psi_X),\varphi_{X',X}\}$ is an inverse system and the diagram
\begin{center}
\begin{tikzpicture}[pil/.style={
           ->,
           thin,
           shorten <=2pt,
           shorten >=2pt,}]
\node (1) {$\cP$};
\node (2) [right=of 1]{$\im(\psi_{X'})$};
\node (3) [below=of 2]{$\im(\psi_X)$};
\path[-stealth]
(1) edge node [above] {$\psi_{X'}$} (2)
(1) edge node [below left] {$\psi_X$} (3)
(2) edge node [right] {$\varphi_{X',X}$} (3);
\end{tikzpicture}
\end{center}
is easily seen to commute.

Since $x\sim y$ holds, we inductively find some countably infinite subset $\cQ$ of $\cP$ consisting of edge-disjoint $x$-$y$ paths.
For every $X\in\cX$ we let $\aL_X$ be the set of all $\cL\in\im(\psi_X)$ with $\psi_X^{-1}(\{\cL\})\cap\cQ$ infinite, i.e. $\cL$ is in $\aL_X$ if and only if $\psi_X$ sends some infinitely many paths of $\cQ$ to $\cL$. 
Then it is easy to check that $\{\aL_X,\varphi_{X',X}\rest\aL_{X'}\}$ is an inverse system. Clearly, every $\im(\psi_X)$ is finite, and so is every $\im(\psi_X\rest\cQ)$. Hence all $\aL_X$ are non-empty by pigeon-hole principle, so the Generalized Infinity Lemma~(\ref{GIL}) yields some 
\begin{align}\label{bundle}
(\cL_X\,|\,X\in\cX)\in\invLim\A{L}_X.
\end{align}

\begin{theorem}\label{SigmaArcConstruction}
Let $G$ be any graph, and let $x$ and $y$ be two distinct vertices of $G$. Then $x\sim y$ if and only if there is some auxiliary arc from $x$ to $y$.
\end{theorem}
\begin{proof}
The backward direction is immediate from Lemma~\ref{cutJAL}.

For the forward direction suppose that $x\sim y$ holds and consider the family 
from~(\ref{bundle}) constructed above. Write $(L_X,\le_X)=\cL_X$ for every $X\in\cX$ and put $V^*=\bigcup_{X\in\cX}L_X$. By standard inverse limit arguments it is straightforward to check that ${\le}:=\bigcup_{X\in\cX}{\le}_X$ is a linear ordering on $V^*$ with least element $x$ and greatest element $y$. Since $V^*$ is a subset of $V[\cQ]$ and $\cQ$ is countable, we know that $V^*$ is countable.

If $V^*$ is finite, then $x\transcl y$ holds: Otherwise Corollary~\ref{ApproxOnV} yields some $Y\in\cX$ disjoint from $V^*$ such that $Y$ separates $x$ and $y$ in $G-E(V^*)$. 
Set $Z=V^*\cup Y$.
Then we have $L_Z\cap Y=\emptyset$ since $Y$ avoids $V^*$. 
Pick some $P\in\cQ$ edge-disjoint from the finite set $E(V^*)$ with $\psi_Z(P)=\cL_Z$.
Then $P$ is an $x$--$y$ path which avoids both $Y$ and $E(V^*)$, which is impossible since $Y$ separates $x$ and $y$ in $G-E(V^*)$. Hence $x\transcl y$ holds, so by Theorem~\ref{InfIndPaths} there exists a path from $x$ to $y$ in $\GG$ using only auxiliary edges. In particular, we find an auxiliary arc from $x$ to $y$. 

Therefore, we may assume that $V^*$ is infinite.
Now consider the set
\begin{align*}
M:=(V^*\times\{-1,0,1\})\setminus \{(x,-1),(y,1)\}
\end{align*}
and let $\preceq$ be the linear ordering on $M$ induced by the lexicographic ordering on $V^*\times\{-1,0,1\}$, where $V^*$ is linearly ordered by $\le$ and $\{-1,0,1\}$ inherits its ordering from $\Z$.
Since $M$ is countable we find some order preserving injection $\iota_M\colon M\hookrightarrow\I\cap\Q$ with $\iota_M((x,0))=0$ and $\iota_M((y,0))=1$. 
Note that $\iota_M[V^*\times\{0\}]$ is discrete\footnote{i.e. every point in $\iota_M[V^*\times\{0\}]$ is isolated} in $\I$. Next define $\iota_{V^*}\colon V^*\hookrightarrow\I\cap\Q$ by letting $\iota_{V^*}(u):=\iota_M((u,0))$ for every $u\in V^*$. Then the linear ordering on $V^*$ inherited from $\I$ via $\iota_{V^*}$ coincides with $\le$. Furthermore, the image of $\iota_{V^*}$ is discrete in $\I$, so for every $u\in V^*$ we may pick some $\delta_u>0$ such that $(\iota_{V^*}(u)\pm \delta_u)$ meets the image of $\iota_{V^*}$ precisely in $\iota_{V^*}(u)$.

An $h$-\textit{blank} (where $h$ is some function into $\R$) is a non-empty open interval $(a,b)\subseteq\I\setminus \im(h)$ with $\{a,b\}\subseteq\overline{\im (h)}\setminus\im(h)$. Clearly, any two $h$-blanks are disjoint. Since every $h$-blank contains a rational number, there are only countably many $h$-blanks, so we may consider some enumeration $\{B_n\,|\,n<\theta\}$ of all $\iota_{V^*}$-blanks where $\theta\le\aleph_0$. Also note that every $\iota_{V^*}$-blank avoids $(\iota_{V^*}(u)\pm\delta_u)$ for every $u\in V^*$.
In order to get rid of the $\iota_{V^*}$-blanks, define $\iota\colon V^*\hookrightarrow\Q$ as follows: let $\iota(x):=0$ and for every other $u\in V^*$ pick
\begin{align*}
\iota(u)\in\left(\iota_{V^*}(u)-\sum\Big\{b-a\,\Big|\,n<\theta,B_n=(a,b),b<\iota_{V^*}(u)\Big\}\pm\delta_{u}/2\right)\cap\Q.
\end{align*}
Then the image of $\iota$ is again discrete by choice of $\delta_u/2$. Furthermore $\iota$ is injective and the ordering on $V'$ inherited from $\I$ via $\iota$ coincides with $\le$. 
By construction, there are no $\iota$-blanks. 
For ease of notation we that $\iota(y)=1$ holds.
As before, for every $u\in V^*$ we pick some $\epsilon_u>0$ such that $(\iota(u)\pm \epsilon_u)$ meets the image of $\iota$ precisely in $\iota(u)$, e.g. set $\epsilon_u=\delta_u/2$.
Let $\cA:=\overline{\im(\iota)}\setminus \im(\iota)$ which is the set of all accumulation points of the discrete image of $\iota$ in $\I$. In order to obtain a topological $x$--$y$ path in $\overline{\AuxE}$ we need some more information about $\iota$ first:

\begin{claim_auxArc}\label{easyCase}
Let $u\in V^*$ be given and suppose that there is some $t\in V^*$ such that $\iota(u)<\iota(t)$ and $(\iota(u),\iota(t))\cap \im(\iota)=\emptyset$. Then there is some $\upsilon\in\cU$ with $u\between\upsilon\between t$.
\end{claim_auxArc} 
\begin{proof}[Proof of the Claim]\renewcommand{\qedsymbol}{$\diamond$}
Assume not for a contradiction. Then by Lemma~\ref{PathSysXtau} we find some $Y\in\cX$ with $u,t\notin Y$ such that $Y$ separates $u$ and $t$ in $G-ut$. 
Put $Z=\{u,t\}\cup Y$ and pick some $P\in\cQ$ with $ut\notin E(P)$ and $\psi_Z(P)=\cL_Z$. 
Then $uPt$ meets $Y$ in some vertex $r$. Hence $r\in L_Z$ implies $r\in V^*$, but then $P$ witnesses $u<_Zr<_Zt$ and thus $\iota(u)<\iota(r)<\iota(t)$, a contradiction.
\end{proof}

\begin{claim_auxArc}\label{SeriesOmega}
Let $\mu\in\cA$ be given together with a sequence $(t_n)_{n\in\N}$ in $V^*$ such that $\iota(t_n)\to \mu$ for $n\to\infty$. Then $(t_n)_{n\in\N}$ has a subsequence which converges to some $\omega\in\Omega(G)$ in $\GH$.
\end{claim_auxArc}
\begin{proof}[Proof of the Claim]\renewcommand{\qedsymbol}{$\diamond$}
Without loss of generality we may assume that our sequence satisfies $\mu<q(t_{n+1})<q(t_n)$ for all $n$.
Inductively we define $X_n\in\cX$ and $P_n\in\cQ$ by letting 
\begin{align*}
X_n:=\{t_{n+1}\}\cup\bigcup_{k<n}V(t_{k+1}P_kt_k)
\end{align*}
and picking some $P_n\in\cQ$ with $\psi_{X_n}(P_n)=\cL_{X_n}$, for every $n\in\N$. Then $Q_n:=t_{n+1}P_nt_n$ avoids all $t_k$ with $k<n$ since otherwise $P_n$ would witness $q(t_{n+1})<q(t_k)<q(t_n)$ for some $k<n$, which is impossible. 
Consider the connected infinite subgraph $H:=\bigcup_{n\in\N}Q_n$ of $G$. We claim that $H$ is even locally finite:

Assume not for a contradiction, witnessed by some vertex $z$ of $H$ of infinite degree. This vertex is none of the $t_n$ since each $Q_n$ avoids all $t_k$ with $k<n$. 
Let $\cI\subseteq\N$ be the infinite set of all $n\in\N$ with $z\in \mathring{Q_n}$, and let $i:=\min \cI$. Then $z$ is in $X_n$ for all $n>i$. Pick some $m\in \cI$ with $m>i$. 
Then $P_m$ contains $z$ while $z\in X_m$ and $\psi_{X_m}(P_m)=\cL_{X_m}$ hold, so $P_m$ witnesses $z\in L_{X_m}\subseteq V^*$. 
But then for every $n\in\cI$ with $n>m$ we know that $z\in \mathring{Q_n}$ implies $\iota(t_{n+1})<\iota(z)<\iota(t_n)$, which is impossible since $\cI$ is infinite.
Hence $H$ is locally finite.

Therefore, since $H$ is connected, infinite and locally finite, applying the Star-Comb Lemma (\cite[Lemma 8.2.2]{Bible}) yields a comb with all teeth in $\{t_n\,|\,n\in\N\}$. 
Let $R$ be its spine, and let $\omega$ be the end of $G$ containing $R$. Since every basic open neighbourhood of $\omega$ in $\GH$ is of the form $\hat{\aC}_\epsilon(X,\omega)$ for some $X\in\cX$, only finitely many teeth of the comb can lie outside of this neighbourhood.
Therefore, the teeth  of the comb form a subsequence of $(t_n)_{n\in\N}$ which converges to $\omega$ in $\GH$ as desired.
\end{proof}

\begin{claim_auxArc}\label{SeriesOmegaExt}
In the context of Claim~\ref{SeriesOmega}, suppose that there is some $u\in V^*$ with $\iota(u)=\lambda<\mu$ and $(\lambda,\mu]\cap \im(\iota)=\emptyset$.
Then $u$ dominates $\omega$.
\end{claim_auxArc}
\begin{proof}[Proof of the Claim]\renewcommand{\qedsymbol}{$\diamond$}
Assume not for a contradiction, witnessed by some $Z\in\cX$. 
Without loss of generality the whole sequence $(t_n)_{n\in\N}$ converges to $\omega$ in $\GH$.
Pick $N\in\N$ such that $t_n\in \hat{\aC}(Z,\omega)$ holds for all $n\ge N$. 
For every $n\ge N$ let $Y_n:=\{u,t_n\}\cup Z$ and pick some $W_n\in\cQ$ with $\psi_{Y_n}(W_n)=\cL_{Y_n}$. 
Then $uW_nt_n$ meets $Z$ in some $z_n$, so $z_n\in V^*$ and $\iota(u)<\iota(z_n)<\iota(t_n)$ follow. 
By pigeon-hole principle we find some infinite $\cJ\subseteq\N_{\ge N}$ with $z_n=z_m=:z$ for all $n,m\in \cJ$. Then $\iota(u)<\iota(z)<\iota(t_n)$ holds for all $n\in\cJ$. Since $\cJ$ is infinite, this implies $\iota(z)\in(\lambda,\mu]$, a contradiction. Therefore, $u$ dominates $\omega$.
\end{proof}

Define $\sigma\colon \I\to\GH$ as follows: 
On the image of $\iota$ we let $\sigma(\iota(u)):=u$ for every $u\in V^*$, so the diagram
\begin{center}
\begin{tikzpicture}[pil/.style={
           ->,
           thin,
           shorten <=2pt,
           shorten >=2pt,}]
\node (1) {$V^*$};
\node (2) [right=of 1]{$\I$};
\node (3) [below=of 2]{$\GH$};
\path[-stealth]
(1) edge node [above] {$\iota$} (2)
(1) edge node [below left] {$\id_{V^*}$} (3)
(2) edge node [right] {$\sigma$} (3);
\end{tikzpicture}
\end{center}
commutes.
Next we write the set $\I\setminus \im(\iota)$ as disjoint union $\biguplus_{j\in J}I_j$ of maximal non-empty intervals $I_j$ over some index set $J$. Let $a_j:=\inf I_j$ and $b_j:=\sup I_j$ for all $j\in J$. 
We distinguish several cases (see Fig.~\ref{fig:cases} for an illustration):
\begin{figure}[H]
    \centering
    \def\svgwidth{\columnwidth}
    \scalebox{.9}{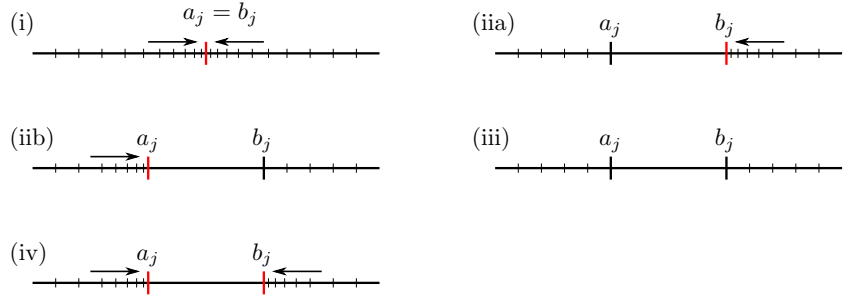}
    \caption{The cases; red indicates elements of $\cA$.}
    \label{fig:cases}
\end{figure}
\begin{enumerate}
\item[(i)] $a_j=b_j$, i.e. $I_j=\{a_j\}=\{b_j\}$
\item[(iia)] $a_j<b_j$ with $a_j\in \im(\iota)$ and $b_j\in \cA$, i.e. $I_j=(a_j,b_j]$
\item[(iib)] $a_j<b_j$ with $a_j\in\cA$ and $b_\alpha\in\im(\iota)$, i.e. $I_j=[a_j,b_j)$
\item[(iii)] $a_j<b_j$ with $a_j\in \im(\iota)$ and $b_\alpha\in\im(\iota)$, i.e. $I_j=(a_j,b_j)$
\end{enumerate}
There is no case (iv) covering `$a_j<b_j$ with $a_j\in\cA$ and $b_j\in\cA$, i.e. $I_j=[a_j,b_j]$' since there are no $\iota$-blanks. 
Now given $j\in J$, we define $\sigma\rest I_j$ as follows:

First we consider \textbf{case (i)}, i.e. $I_j$ is a singleton. Let $a_j$ take on the role of $\mu$ in Claim~\ref{SeriesOmega} and pick an arbitrary sequence $(t_n)_{n\in\N}$ in $V^*$ witnessing $a_j\in\cA$, so Claim~\ref{SeriesOmega} yields some $\omega\in\Omega(G)$. Then we set $\sigma(a_j)=\omega$.

Next we consider \textbf{case (iia)}, i.e. $a_j<b_j$ with $a_j\in \im(\iota)$ and $b_j\in \cA$. We pick an arbitrary sequence $(t_n)_{n\in\N}$ in $V^*$ witnessing $b_j\in\cA$, so Claim~\ref{SeriesOmega} yields some $\omega\in\Omega(G)$.
Letting $\iota^{-1}(a_j)$ take on the role of $u$ in Claim~\ref{SeriesOmegaExt} yields $u\in\Delta(\omega)$. 
Hence we let $\sigma\rest I_j$ be a homeomorphism onto the edge $u\omega$ with $\sigma(b_j)=\omega$.
We treat \textbf{Case (iib)} analogously using the symmetric analogue of Claim~\ref{SeriesOmegaExt}.

Finally we consider \textbf{case (iii)}, i.e. $a_j<b_j$ with $a_j\in \im(\iota)$ and $b_j\in \im(\iota)$.
Letting $\iota^{-1}(a_j)$ and $\iota^{-1}(b_j)$ take on the roles of $u$ and $t$ from  Claim~\ref{easyCase}, respectively, yields some $\upsilon\in\cU$ with $u\between\upsilon\between t$.
If $\upsilon$ is an end of $G$, then we write $\omega=\upsilon$ and note $\{u,t\}\subseteq\Delta(\omega)$. 
Hence we pick some $\mu\in (a_j+\epsilon_{u},b_j-\epsilon_{t})$ and let $\sigma\rest [a_j,\mu]$ and $\sigma\rest [\mu,b_j]$ be homeomorphisms onto the edges $u\omega$ and $\omega t$, respectively, such that $\sigma(\mu)=\omega$ holds.
Otherwise $\upsilon$ is an ultrafilter tangle and we let $\sigma\rest [a_j,b_j]$ be a homeomorphism onto some auxiliary edge between $u$ and $t$ such that $\sigma(a_j)=u$ and $\sigma(b_j)=t$ (such an auxiliary edge exists by Theorem~\ref{InfIndPaths}).
This completes the definition of $\sigma$.

Note that $V^*\subseteq\im(\sigma)$ holds by construction. Due to the Claims used for the definition of $\sigma$, proving that $\sigma$ is well defined amounts to verifying that $\sigma$ is continuous, and we will need the continuity of $\sigma$ in order to obtain an auxiliary arc from $\im(\sigma)$. Hence, we verify that $\sigma$ is continuous: 
\begin{claim_auxArc}
$\sigma\colon \I\to\GH$ is continuous.
\end{claim_auxArc}
\begin{proof}[Proof of the Claim]\renewcommand{\qedsymbol}{$\diamond$}
Given $\mu\in\I$ (without loss of generality $\mu\in\mathring{\I}$) and $O$ some basic open neighbourhood of $\sigma(\mu)$ in $\GH$, we have to find some basic open neighbourhood $(a,b)\subseteq \I$ of $\mu$ which $\sigma$ maps to $O$. 

If $\sigma(\mu)$ is a vertex $u$ of $G$, then we find a suitable neighbourhood of $\mu$ included in $(\mu\pm\epsilon_u)$. 
Else if $\sigma(\mu)$ is an inner edge point we are done, so finally suppose that $\sigma(\mu)=:\omega\in\Omega(G)$. 
Thus $O$ is of the form $\hat{\aC}_\epsilon(X,\omega)$ for some $X\in\cX$ and $\epsilon>0$. If $\mu$ is not in $\cA$, then $\sigma(\mu)$ was defined in case (iii) for some $j\in J$, and we find some suitable open neighbourhood of $\mu$ included in $(a_j,b_j)$. Hence we may assume that $\mu$ is in $\cA$. The rest of the proof is dedicated to this case.

Assume for a contradiction that for every $n\in\N$ there is some $\xi_n\in (\mu\pm 1/n)$ with $\sigma(\xi_n)\notin O$. 
First, we obtain a sequence $(\xi_n')_{n\in\N}$ from $(\xi_n)_{n\in\N}$ with $\sigma(\xi_n')\in V^*\setminus O$ for all $n$ and $\xi_n'\to\mu$, as follows: Let any $n\in\N$ be given.

If $\sigma(\xi_n)$ is a vertex of $G$, then we let $\xi'_n:=\xi_n$.

Else if $\sigma(\xi_n)$ is an inner edge point of some auxiliary edge $e_n$, then we pick $\xi'_n\in\im(\iota)$ such that $\sigma(\xi_n')$ is an endvertex of $e_n$ outside of $O$. For later use, we let $\ell_n$ denote the length of the closed interval $I$ containing $\xi_n$ for which $\sigma\rest I$ is a homeomorphism onto the auxiliary edge $e_n$.

Finally suppose that $\sigma(\xi_n)=\omega_n\in\Omega\setminus O$. We check two subcases:
For the first subcase suppose that $\xi_n$ is not in $\cA$, so $\sigma(\xi_n)$ was defined in case (iii) for some $j_n\in J$. 
Furthermore, $\omega_n\notin O$ implies that $\sigma[I_{j_n}]$ avoids $O$.
Let $u:=\sigma(a_{j_n})$ and $t:=\sigma(b_{j_n})$.
Then neither of $u$ and $t$ is in $O$, since otherwise $X$ would separate one of $u$ and $t$ from $\omega_n$, contradicting $\{u,t\}\subseteq\Delta(\omega_n)$. 
In particular, $\mu$ is not in $[a_{j_n},b_{j_n}]$. If $\mu<a_{j_n}$ we let $\xi_n':=a_{j_n}$, and $\xi_n':=b_{j_n}$ otherwise. 
For the second subcase suppose that $\xi_n$ is in $\cA$.
Then by definition of $\sigma$ there is some sequence $(t_m)_{m\in\N}$ in $V^*$ such that $\iota(t_m)\to \xi_n$ in $\I$ and $t_m\to \omega_n$ in $\GH$ for $m\to\infty$. Since $\GH$ is Hausdorff and $\omega_n\neq\omega$ holds due to $\omega_n\notin O$, we find some $M\in\N$ such that $t_M\notin O$ and $\iota(t_M)\in (\mu\pm 1/n)$ (which is possible due to $\xi_n\in (\mu\pm 1/n)$). Thus we let $\xi'_n:=\iota(t_M)$. This completes the subcase and the definition of the $\xi_n'$.

We still have to verify that $\xi_n'\to\mu$ holds. For this, let $\cI\subseteq\N$ be the set of all $n\in\N$ for which $\sigma(\xi_n)$ is an inner edge point of $\GG$.
Clearly, $\sigma(\mu)$ was defined in one of the cases (i), (iia) and (iib) for some $j\in J$.
If $\sigma(\mu)$ was defined in case (i), i.e. with $I_j=\{\mu\}$, then for every $\epsilon>0$ both intervals $(\mu-\epsilon,\mu)$ and $(\mu,\mu+\epsilon)$ meet $\im(\iota)$ and one easily checks that $\ell_n\to 0$ holds for $n\in\cI$ and $n\to\infty$. Therefore, $|\xi_n'-\xi_n|\to 0$ for $n\to\infty$ follows from the choice of the $\xi_n'$. Thus $\xi_n'\to\mu$ holds for $n\to\infty$.
Else if $\sigma(\mu)$ was defined in case (iia), i.e. with $a_j\in\im(\iota)$ and $b_j=\mu\in\cA$, then for every $\epsilon>0$ the interval $(\mu,\mu+\epsilon)$ meets $\im(\iota)$, but the interval $(a_j,\mu]$ does not.
By construction, $\sigma$ maps $(a_j,\mu)$ to the interior of the auxiliary-edge $\sigma(a_j)\omega$. 
Hence $\xi_n\to\mu$ together with $\sigma(\xi_n)\notin O$ implies that $\xi_n>\mu$ must hold for all but finitely many $n$, without loss of generality for all $n$. But then one easily checks as before that $|\xi'_n-\xi_n|\to 0$ holds, and thus $\xi_n'\to\mu$. Otherwise $\sigma(\mu)$ was defined in case (iib), which follows from (iia) via symmetry, completing the proof of $\xi_n'\to\mu$.

Therefore, $(\xi_n')_{n\in\N}$ is a sequence in $\I$ with $\sigma(\xi_n')\in V^*\setminus O$ for all $n$ and $\xi_n'\to \mu$ for $n\to\infty$.
Applying Claim~\ref{SeriesOmega} to $(\sigma(\xi'_n))_{n\in\N}$ and $\mu$ yields a convergent subsequence $(\sigma(\xi'_{n_k}))_{k\in\N}$ with limit $\omega'\in\Omega(G)$. Since $\GH$ is Hausdorff, the fact that no $\sigma(\xi'_{n_k})$ is in $O$ implies $\omega'\neq\omega$. 
Write $r_k=\sigma(\xi'_{n_k})$ for all $k\in\N$ and note that $\iota(r_k)\to\mu$ holds for $k\to\infty$ due to $\iota(r_k)=\iota(\sigma(\xi'_{n_k}))=\xi'_{n_k}$ and $\xi'_n\to\mu$.
Recall that Claim~\ref{SeriesOmega} was used in order to define $\sigma(\mu)=\omega$, so there is a sequence $(t_n)_{n\in\N}$ in $V^*$ with $\iota(t_n)\to\mu$ in $\I$ and $t_n\to \omega$ in $\GH$.\footnote{This is the convergent subsequence mentioned in the conclusion of Claim~\ref{SeriesOmega}, not the sequence of the same name mentioned in the premise of that claim.}
Pick some $Z\in\cX$ witnessing $\omega\neq\omega'$, and choose some $N\in\N$ such that $t_n\in \hat{\aC}(Z,\omega)$ and $r_n\in \hat{\aC}(Z,\omega')$ hold for all $n\ge N$. For every $n\ge N$ let $X_n:=\{t_n,r_n\}\cup Z$ and pick some $P_n\in\cQ$ with $\psi_{X_n}(P_n)=\cL_{X_n}$. Then $P_n$ meets $Z$ in some $z_n$ between $t_n$ and $r_n$ (i.e. $z_n\in t_nP_nr_n$ if $t_n<r_n$ and $z_n\in r_nP_nt_n$ otherwise) which yields $z_n\in L_{X_n}\subseteq V^*$, and furthermore
\begin{align*}
\min\{\iota(r_n),\iota(t_n)\}<\iota(z_n)<\max\{\iota(r_n),\iota(t_n)\}
\end{align*}
Since $Z$ is finite, by pigeon-hole principle we find some infinite $\cJ\subseteq\N_{\ge N}$ with $z_n=z_m=:z$ for all $n,m\in \cJ$. But then
\begin{align*}
\min\{\iota(r_n),\iota(t_n)\}<\iota(z)<\max\{\iota(r_n),\iota(t_n)\}
\end{align*}
holds for all $n\in \cJ$, yielding $\iota(z)=\mu$ since $\cJ$ is infinite and $\iota(r_n)\to\mu$ holds as well as $\iota(t_n)\to\mu$. Thus $\mu\in\im(\iota)$ contradicts $\mu\in\cA$ as desired. Therefore, $\sigma$ is continuous at $\mu$.
\end{proof}

Since $\GH$ is Hausdorff, so is $\im(\sigma)$, and in particular $\im(\sigma)$ is also path-connected. Hence $\im(\sigma)$ is arc-connected by Lemma~\ref{PathHDArc}, so we may let $A$ be an arc in $\im(\sigma)$ from $x$ to $y$. Since $A$ only traverses auxiliary edges, Lemma~\ref{SigmaArcEdgesDense} implies $A\subseteq\overline{\AuxE}$, so $A$ is an auxiliary arc from $x$ to $y$.
\end{proof}

We involve Lemma~\ref{PathHDArc} at the end of the proof since $\sigma$ may fail to be injective at ends. Indeed, consider the one-ended grid on $4\times\N$, and for each $k<4$ let $R_k$ denote the ray with vertex set $\{k\}\times\N$. Then let $G$ be obtained from this grid by replacing every edge of every $R_k$ and the egde $\{(1,0),(2,0)\}$ with a copy of a $K_{2,\aleph_0}$ (with the two vertices of infinite degree being identified with the endvertices of the original egde). 
\begin{figure}[H]
\centering
\includegraphics[scale=0.85,angle=90]{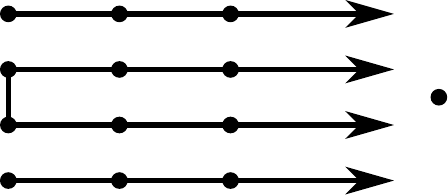}
\caption{The auxiliary edges of $G$.}
    \label{fig:whyUseHD}
\end{figure}
Then $\GH$ admits two topological $(0,0)$--$(3,0)$ paths in $\overline{\AuxE}$ (see Fig.~\ref{fig:whyUseHD}), and only one of them is an auxiliary arc, while the other can be obtained as $\im(\sigma)$ for certain choices of $\cQ$.

\subsection{Basic properties of auxiliary arcs}\label{subsecAuxArcProp}

\begin{lemma}\label{CtblAuxArcStarComb}
Let $G$ be a any graph, let $y$ be a vertex of $\GG$ and $(\xi_n)_{n\in\N}$ a sequence of vertices of $\GG$ with $\xi_n\to\omega$ in $\GH$ for some end $\omega$ of $G$. Suppose that for every $n$ we have an auxiliary arc $A_n$ from $y$ to $\xi_n$. Then one of the following holds:
\begin{enumerate}
\item There is some $k\in\N$ such that $A_k$ contains a vertex of $G$ which is the centre of an infinite star of auxiliary arcs included in the union of the $A_n$ with leaves in $\{\xi_n\,|\,n\in\N\}$. In particular, that vertex dominates $\omega$.
\item There is an auxiliary arc from $y$ to $\omega$ included in $\{\omega\}\cup\bigcup_{n\in\N}A_n$.
\end{enumerate}
In either case there exists an auxiliary arc $A$ from $y$ to $\omega$. Moreover, if all of the $A_n$ are tame, then so is $A$.
\end{lemma}
\begin{proof}
The proof basically mimics two proofs from the lecture course while taking care of new special cases. First, we prove the statement of the Lemma minus the `Moreover'-part.

Inductively we define a function $k:\N-\{0\}\to\N$ together with an ascending sequence $W_0\subseteq W_1\subseteq\cdots$ of closed subsets of $\GH$ 
as follows: Let $W_0:=A_0$ and suppose that we are at step $n>0$ of the construction. Pick some $\sigma_n\colon \I\to A_n$ witnessing that $A_n$ is an $y$--$\xi_n$ auxiliary arc, with $\sigma_n(0)=y$ and $\sigma_n(1)=\xi_n$, and let $\lambda_n\in\I$ be maximal with $\sigma_n(\lambda_n)\in W_{n-1}$. Such a $\lambda_n$ exists since $W_{n-1}$ is closed in $\GH$ and $A_n$ meets $y\in W_0\subseteq W_{n-1}$. Write $\varrho_n:=\sigma_n(\lambda_n)$ and $A_n':=\sigma_n[[\lambda_n,1]]$. Furthermore, let $k(n)<n$ be minimal with $\varrho_n\in A_{k(n)}'$ and let $W_n:=W_{n-1}\cup A_n'$. This completes the definition of the $W_n$.

If there is some $n$ with $\omega\in W_n$ we are done, so suppose $\omega\notin\bigcup_{n\in\N} W_n$ and consider the tree $T:=(\N,\{\{n,k(n)\}\,|\,n\ge 1\})$. 
We check two main cases:

For the first main case, suppose that $T$ is not locally finite, witnessed by some vertex $\ell$ of infinite degree.
Let $\ell_0<\ell_1<\cdots$ denote the $\omega$-sequence of all $\ell_n\in\N$ which $k$ sends to $\ell$. Since $A_\ell'$ is an arc, and hence sequentially compact, the sequence $(\varrho_{\ell_n})_{n\in\N}$ has a convergent subsequence (without loss of generality the whole sequence) with limit point $\varrho\in A_\ell'$. Clearly, $\varrho$ is not an inner edge point. 

First suppose that $\varrho$ is a vertex of $G$ and consider any basic open neighbourhood $O$ of $\varrho$ in $\GH$. 
Without loss of generality every $\varrho_{\ell_n}$ is contained in this neighbourhood.
Since $A'_\ell\cap O$ is the union of two half open partial edges at $\varrho$ and every $A'_{\ell_n}$ meets $A'_\ell$ precisely in $\varrho_{\ell_n}$, the only possibility for each $\varrho_{\ell_n}$ is $\varrho_{\ell_n}=\varrho$. 
Hence for every $n\neq m$ we have $A'_{\ell_n}\cap A'_{\ell_m}=\{\varrho\}$. 
Then $\varrho$ must dominate $\omega$: Assume not for a contradiction, so we find a witness $X\in\cX$ of $\varrho\notin\Delta(\omega)$. 
Pick $N\in\N$ such that $\varrho_{\ell_n}\in\hat{\aC}(X,\omega)$ holds for all $n\ge N$. 
Then all $A'_{\ell_n}$ with $n\ge N$ meet $X$ by Lemma~\ref{vxJAL}, so by pigeon-hole principle some two of them meet the same vertex of $X$, a contradiction. 
Therefore, $\varrho$ must dominate $\omega$, so (i) holds.
In particular, $\varrho\omega\cup A_\ell$ admits an auxiliary arc from $y$ to $\omega$.

Second suppose that $\varrho$ is an end $\omega'$ of $G$ and let $X\in\cX$ witness $\omega'\neq\omega$. Furthermore let $N\in\N$ be big enough that $\varrho_{\ell_n}\in\hat{\aC}(X,\omega)$ holds for all $n\ge N$.
Then every $A'_{\ell_n}$ with $n\ge N$ meets $X$ by Lemma~\ref{vxJAL}, 
so by pigeon-hole principle some two of them meet the same vertex of $X$, which is impossible. Hence $\varrho$ cannot be an end.
This completes the first main case.
 
For the second main case suppose that $T$ is locally finite.
Then by \cite[Proposition 8.2.1]{Bible} there is some ray $R=m_0m_1\hdots$ in $T$.
Clearly, we can see to it that $m_0=0$.
For every $n\ge 1$ let $I_n:=[1-\frac{1}{n}\,,\,1-\frac{1}{n+1})$.
Define $\sigma\colon \I\to\GH$ by letting $\sigma\rest I_n$ traverse $A_{m_n}'$ from $\varrho_{m_n}$ to $\varrho_{m_{n+1}}$, and set $\sigma(1)=\omega$. 
Then $\sigma$ is continuous: 
It suffices to show continuity at 1, so consider some basic open neighbourhood $\hat{\aC}_\epsilon(Z,\omega)$ of $\omega$ in $\GH$ and pick $N'\in\N$ such that $\xi_n\in \hat{\aC}_\epsilon(Z,\omega)$ holds for all $n\ge N'$. 
Since $Z$ is finite, Lemma~\ref{vxJAL} implies that only finitely many $A_{m_n}'$ with $m_n\ge N'$ have points outside of $\hat{\aC}_\epsilon(Z,\omega)$. 
Picking $K\ge N'$ bigger than these finitely many $m_n$ ensures that $\sigma$ maps $(1-1/K,1]$ to $\hat{\aC}_\epsilon(Z,\omega)$. 
Hence $\sigma$ defines a topological path from $y$ to $\omega$ in $\GH$. Since $\I$ is compact, $\GH$ is Hausdorff and $\sigma\colon \I\to\GH$ is a continuous injection, it follows from general topology that $A:=\im(\sigma)$ is an arc (in particular: an auxiliary arc). Thus (ii) holds and the second main case is complete.

For the `Moreover'-part, it remains to show that if all of the $A_n$ are tame, then so is $A$. Since this is clear for (i) we may assume that (ii) holds. Assume for a contradiction that $A$ is wild, witnessed by some infinite $V_{\Q}\subseteq V(G)\cap A$ on which $A$ induces the ordering of the rationals. If there is some $n\in\N$ such that $\im(\sigma\rest I_n)$ meets $V_{\Q}$ in two distinct vertices, then $A_{m_n}'\subseteq A_{m_n}$ is wild, which is impossible. Hence every $\im(\sigma\rest I_n)$ meets $V_{\Q}$ in at most one point. But then 1 is the only accumulation point of $\sigma^{-1}[V_{\Q}]$ in $\I$, a contradiction. Therefore, $A$ is tame as claimed.
\end{proof}

\begin{theorem}\label{AuxArcsChar}
Let $G$ be any graph, and let $x$ and $y$ be two distinct points of $V\cup\Omega$. Then $x\sim y$ if and only if there exists an auxiliary arc in from $x$ to $y$.
\end{theorem}
\begin{proof}
The backward direction holds by Lemma~\ref{cutJAL}.
For the forward direction we check several cases:

If both $x$ and $y$ are vertices of $G$ we are done by Theorem~\ref{SigmaArcConstruction}, so suppose first that $x\in V$ and $y=\omega\in\Omega$. 
If there is some $t\in\Delta(\omega)$, then $x\sim\omega\sim t$ implies $x\sim t$.
Hence we find an auxiliary arc from $x$ to $t$ which we may extend to $\omega$ by adding the auxiliary edge $t\omega$.

Otherwise $\Delta(\omega)$ is empty, so by Lemma~\ref{ctblNbhdBase} we find a sequence $X_0,X_1,\hdots$ of non-empty elements of $\cX$ such that for all $n\in\N$ the component $C(X_n,\omega)$ includes both $X_{n+1}$ and $C(X_{n+1},\omega)$. In particular, the collection of all $\CTC(X_n,\omega)$ forms a countable neighbourhood basis of $\omega$ in $\TC$.
Since $x$ does not dominate $\omega$ we find some $X\in\cX$ with $x\notin X\cup C(X,\omega)$, and we may choose $N\in\N$ big enough that $\CTC(X_N,\omega)\subseteq\CTC(X,\omega)$.

Now for every $n\ge N$ we wish to find some vertex $t_n$ in $C(X_n,\omega)$ with $x\sim t_n$. Hence consider an arbitrary $n\ge N$ and suppose for a contradiction that there is some finite set of edges $F$ which separates $x$ from $X_n$ in $G-C(X_n,\omega)$. Then $X_n\cup C(X_n,\omega)$ is included in some component of $G-F$ which does not contain $x$, so $x\not\sim\omega$ is a contradiction. Therefore, no finite set of edges separates $x$ from $X_n$ in $G-C(X_n,\omega)$. Proceeding inductively we find some infinitely many edge-disjoint $x$--$X_n$ paths in $G-C(X_n,\omega)$. By pigeon-hole principle, some infinitely many of these paths agree on their endvertex in $X_n$ in the same vertex, which we choose to be $t_n$. In particular, $x\sim t_n$ holds.

Then for all $n\ge N$ by Theorem~\ref{SigmaArcConstruction} we find some auxiliary arc $A_n$ from $x$ to $t_n$. Furthermore, the choice of the basic open neighbourhoods $\CTC(X_n,\omega)$ implies $t_n\to\omega$ in $\TC$.
Applying Lemma~\ref{CtblAuxArcStarComb} yields some auxiliary arc $A$ from $x$ to $\omega$. 

Finally suppose that both $x$ and $y$ are ends of $G$, and write $x=\omega$ as well as $y=\omega'$. By Lemma~\ref{Forms} there is some vertex $u$ of $G$ in $[\omega]_\sim$. Then $\omega\sim u\sim\omega'$ together with the previous case yields two auxiliary arcs, one from $u$ to $\omega$ and one from $u$ to $\omega'$, whose union yields an auxiliary arc from $\omega$ to $\omega'$ as desired.
\end{proof}

\subsection{External $\mathfrak{A}$TSTs: Technical preliminaries}

An \ATST{} $\Tau$ of $G$ is \textit{external} if for every auxiliary arc-component $\cA$ of $\GH$ the space $\Tau\cap \cA$ is arc-connected (i.e. $\Tau\cap\cA$ is an \ATST{} of $\cA$).\index{external}
The idea behind this definition is that, if we want to obtain an \ATST{} from a TST of $\cE G$ by blowing up its non-trivial points in $\cE G\setminus\mathring{E}$ to \ATST{}s of their respective auxiliary arc-components of $\GH$, then we hope for the result to be an \ATST{}, and if it is then of course it is an external one.

\begin{corollary}\label{AuxArcCompsClosed}
Let $G$ be a graph such that every end of $G$ has a countable neighbourhood basis in $\GH$. Furthermore, let $X\in\cX$ and let $\cA$ be an auxiliary arc-component of $\GH-X-\mathring{E_{\GG}}(X,\GH)$. Then $\cA$ is closed.
\end{corollary}
\begin{proof}
Assume for a contradiction that there is some $x\in\overline{\cA}- \cA$ (with the closure taken in $\GH$). In order to yield a contradiction, we will find an auxiliary arc from $x$ to a point of $\cA$.
Clearly, $x$ must be an end of $G$, so we write $\omega=x$. Let $y$ be an arbitrary point of $\cA\cap V(\GG)$.
Using that $\omega$ has a countable neighbourhood basis in $\GH$ we find a sequence $(\xi_n)_{n\in\N}$ of points of $(\cA\cap V(\GG))-\{y\}$
such that $\xi_n\to\omega$ holds in $\GH$.
For every $n$ we pick an arc $A_n$ from $y$ to $\xi_n$ in $\cA$. In particular, the $A_n$ are auxiliary arcs. Hence Lemma~\ref{CtblAuxArcStarComb} yields an auxiliary arc from $y$ to $\omega$, a contradiction.
\end{proof}

For every $X\in\cX$ and $\cC\subseteq\cC_X$ we write\index{$\cO_{\TC}^\sim(X,\cC)$, $\cO_{\Gq}(X,\cC)$}
\begin{align*}
\cO_{\TC}^\sim(X,\cC)&=\cO_{\TC}(X,\cC)- \medcup(X/{\sim})\\
\cO_{\Gq}(X,\cC)&=\cO_{\TC}^\sim(X,\cC)/{\sim}
\end{align*}
and furthermore if $\omega$ is an end of $G$ we write\index{$\hat{C}_{\Gq}(X,\omega)$}
\begin{align*}
\hat{C}_{\Gq}(X,\omega)=\cO_{\Gq}(X,\{C(X,\omega)\}).
\end{align*}

\begin{lemma}\label{GqConvSetsClosed}
For every $X\in\cX$ and $\cC\subseteq\cC_X$ the set $\cO_{\Gq}(X,\cC)$ is an open subset of $\Gq$.
\end{lemma}
\begin{proof}
Since $X$ is finite and $\Gq$ is Hausdorff, by Lemma~\ref{eqClassesClosed} the set $\bigcup (X/{\sim})$ is a closed subset of $\TC$.
Thus $\cO_{\TC}^\sim(X,\cC)$ is an open subset of $\TC$. By Corollary~\ref{FinSepClassesMeetX} it is closed under $\sim$.
\end{proof}

\begin{corollary}\label{GqArcDoesNotAlt}
Let $A$ be a connected subset of $\Gq$. Furthermore, let $X\in\cX$ and $\cC\subseteq\cC_X$ be given such that $A$ meets both $\cO_{\Gq}(X,\cC)$ and $\cO_{\Gq}(X,\cC_X\setminus\cC)$. Then $A$ meets $X/{\sim}$.\qed
\end{corollary}

\begin{lemma}\label{GqArcConvEnd}
Let $A$ be an arc in $\Gq$, witnessed by $\sigma\colon \I\bij A$. Suppose that $K:=\sigma(1)\notin\mathring{E}$ is an accumulation point of $A\setminus\mathring{E}$. Then there is some unique $\omega\in K\cap\Omega$ with $A$ converging to $\omega$ in that for every $X\in\cX$ there is some $\epsilon>0$ such that
\begin{align}\label{BigConv}
\sigma[(1-\epsilon,1)]\subseteq\hat{C}_{\Gq}(X,\omega).
\end{align}
\end{lemma}
\begin{proof}
We start by choosing a candidate for $\omega$. For this, we claim that
\begin{claim_bigConv}
For every $X\in\cX$ there is some $\epsilon>0$ and some $C\in\cC_X$ such that
\begin{align*}
\sigma[(1-\epsilon,1)]\subseteq\cO_{\Gq}(X,\{C\})
\end{align*}
\end{claim_bigConv}
\begin{proof}\renewcommand{\qedsymbol}{$\diamond$}
Assume not for a contradiction, witnessed by some $X\in\cX$.
Choose $\epsilon>0$ small enough that $A':=\sigma[(1-\epsilon,1)]$ avoids the finite set $X/{\sim}$.
Then 
\begin{align*}
A'\subseteq\Gq -(X/{\sim})=\biguplus_{C\in\cC_X}\cO_{\Gq}(X,\{C\})
\end{align*}
holds by Lemmas~\ref{Forms} and~\ref{GqConvSetsClosed}.
By assumption, $A'$ meets at least two sets of the union on the right hand side, say for $D\neq D'\in\cC_X$. But then by Corollary~\ref{GqArcDoesNotAlt} we know that $A'$ meets $X/{\sim}$, a contradiction.
\end{proof}

Mapping every $X\in\cX$ to the unique $C\in\cC_X$ from the claim above yields a direction of $G$, and hence an end $\omega$ of $G$. 
In particular, for every $X\in\cX$ there is some $\epsilon>0$ such that $\omega$ satisfies (\ref{BigConv}). Moreover, $\omega$ is unique in $\Omega$ with this property.
It remains to show $\omega\in K$, so assume not for a contradiction and pick some finite cut $F$ of $G$ witnessing this. 
Since this $F$ induces an open neighbourhood of $K$ in $\Gq$ there is some $\delta>0$ such that $\sigma[(1-\delta,1)]$ is included in this neighbourhood.
For $X=V[F]$ this contradicts (\ref{BigConv}).
\end{proof}

\begin{lemma}[Arc Lifting]\label{GqArcExtension}
Let $G$ be a countable connected graph.
Suppose that for every non-singleton auxiliary arc-component $\cA$ of $\GH$ there is some \ATST{} $\Tau_{\cA}$ of $\cA$, and let $A$ be an arc in $\Gq$ witnessed by $\sigma\colon \I\bij A$. Furthermore let $x\in\sigma(0)\subseteq\TC$ and $y\in\sigma(1)\subseteq\TC$ be given. Then there exists an arc $A'$ in $\GH$ from $x$ to $y$ such that $(A'-\mathring{\AuxE})/{\sim}\,=A$ and for every non-singleton auxiliary arc-component $\cA$ of $\GH$ exactly one of the following holds:
\begin{enumerate}
\item $|A'\cap\cA|\le 1$\textnormal{;}
\item $A'\cap\cA$ is an arc in $\Tau_{\cA}$.
\end{enumerate}
\end{lemma}
\begin{proof}
We do not use Theorem~\ref{Capel} here since this would require us to `mix' the spaces $\GH$ and $\Gq$.

Let $\hat{\cA}$ be the set of all non-singleton auxiliary arc-components of $\GH$ meeting $\bigcup (A\setminus\mathring{E})$. By Lemma~\ref{Forms} this set is countable since $V(G)$ is countable.
Let $\bar{\sigma}$ be the function with domain $\I$ which sends every $i\in\I$ to $\sigma(i)\setminus\Upsilon$.
Denote by $I$ the set of all $i\in\I$ which $\bar{\sigma}$ sends to a subset of some element of $\hat{\cA}$.
For every $\cA\in\hat{\cA}$ we let $i_{\cA}$ be the unique point of $\I$ with $\bar{\sigma}(i_{\cA})\subseteq \cA$. Then $I:=\{i_{\cA}\,|\,\cA\in\hat{\cA}\}$ is a countable set. 
Conversely, for each $i\in I$ we denote by $\cA_i$ the unique $\cA\in\hat{\cA}$ with $i_{\cA}=i$.
Write $I^-$ for the set of all $i\in I$ which are accumulation points of $I\cap [0,i)$, and $I^+$ for the set of all $i\in I$ which are accumulation points of $I\cap (i,1]$. For every $i\in I$ we define points $x_i$ and $y_i$ in $\cA_i$ as follows: If $i\in I^-$ then we let $x_i$ be the end of $G$ given by Lemma~\ref{GqArcConvEnd}. Else if $i=0$ we set $x_i=x$. Otherwise there is some edge $e$ of $G$ and some $\epsilon>0$ such that $\sigma[(i-\epsilon,i)]=\mathring{e}$, and we let $x_i$ be the endvertex of $e$ in $\cA_i$. Similarly, we define $y_i$ (with the case `$i=0$' replaced by `$i=1$').
Let $J:=\{i\in I\,|\,x_i\neq y_i\}$ and choose some $\ell\colon  J\to\R_{>0}$ with $\sum_{i\in J}\ell(i)=1$.
For every $i\in J$ we let $A_i$ be the unique arc in $\Tau_{\cA_i}$ from $x_i$ to $y_i$ and we let $\sigma_i\colon [0,\ell(i)]\to A_i$ be a parametrisation of $A_i$. Let
\begin{align*}
\phi\colon \I\to [0,2],\; i\mapsto i+\sum_{\substack{j\in J\\j<i}}\ell(j)
\end{align*}
and recall that for every
$\lambda\in [0,2]\setminus\phi[\I\setminus J]$ there is some $j_\lambda\in J$ with $j_\lambda<\lambda$ and 
\begin{align*}
\lambda\in [\phi(j_\lambda),\phi(j_\lambda)+\ell(j_\lambda)]
\end{align*}
by Lemma~\ref{UnitIntervalBlowup}.
Now we are ready to define a mapping $\psi\colon [0,2]\to\GH$ whose image we will take as $A'$. 
For this, let $\lambda\in [0,2]$ be given, and suppose for the first main case that there is some $i\in\I\setminus J$ with $\phi(i)=\lambda$.
If $\bar{\sigma}(i)$ is an inner edge point, we put $\psi(\lambda)=\sigma(i)$.
Else if $\bar{\sigma}$ sends $i$ to a singleton subset $\{\xi\}$ of $\TC\setminus\mathring{E}$, we set $\psi(\lambda)=\xi$ (note that we have $\xi\in V\cup\Omega$ by definition of $\bar{\sigma}$).
Otherwise $i$ is contained in $I\setminus J$ and we set $\psi(\lambda)=x_i=y_i$. This completes the first main case.

For the second main case suppose that $\lambda$ is not in $\phi[\I\setminus J]$. Thus
\begin{align*}
\lambda\in [\phi(j_\lambda),\phi(j_\lambda)+\ell(j_\lambda)]
\end{align*}
and we put 
\begin{align*}
\psi(\lambda)=\sigma_{j_\lambda}(\lambda-\phi(j_\lambda))\in A_{j_\lambda},
\end{align*}
completing the second main case and thus the definition of $\psi$. Clearly, $\psi$ is injective, and for every $j\in J$ the restriction
\begin{align*}
\psi\rest [\phi(j),\phi(j)+\ell(j)]
\end{align*}
parametrises $A_j$.

Put $A'=\im(\psi)$. In order to verify that $A'$ is the desired arc, it suffices to show that $\psi$ is a continuous injection since $[0,2]$ is compact and $\GH$ is Hausdorff.
For this, let an arbitrary $\lambda\in [0,2]$ be given. Clearly, we may assume that $\psi(\lambda)$ is an end of $G$, so we write $\omega=\psi(\lambda)$.
Consider any basic open neighbourhood $O=\hat{\aC}(X,\omega)$ of $\omega$ in $\GH$. We have to find an open neighbourhood of $\lambda$ in $[0,2]$ which $\psi$ maps to $O$. 

If we have $\omega\in\mathring{A_j}$ for some $j\in J$ we are done by the continuity of $\sigma_j$. 

Else if $\omega\in\{x_j,y_j\}$ for some $j\in J$, say $\omega=x_j$, then by choice of $x_j$ there is some $\epsilon>0$ such that (\ref{BigConv}) holds, and hence $\psi$ maps $(\phi(j-\epsilon),\phi(j)]=(\phi(j-\epsilon),\lambda]$ to $O$. 
Using the continuity of $\sigma_j$ yields some $\delta>0$ such that 
$\psi$ sends the open interval $(\lambda-\epsilon,\lambda+\delta)$ to $O$ as desired. 

Else if $\omega=x_i=y_i$ for some $i\in I\setminus J$ there is some $\epsilon>0$ such that (\ref{BigConv}) symmetrically holds for both sides, and hence $\psi$ sends $(\phi(j-\epsilon),\phi(j+\epsilon))$ to $O$.

Finally if $\{\omega\}$ is a singleton auxiliary arc-component of $\GH$ meeting $\bigcup (A\setminus\mathring{E})$, then let $i$ be the point of $\I$ which $\sigma$ maps to $\{\omega\}$ (Theorem~\ref{AuxArcsChar} and Lemma~\ref{Forms} yield $[\omega]_\sim=\{\omega\}$). Now we use the continuity of $\sigma$ to find some $\epsilon>0$ such that $\sigma$ sends $(i\pm\epsilon)$ to $\hat{C}_{\Gq}(X,\omega)$. Then $\psi$ sends $(\phi(i-\epsilon),\phi(i+\epsilon))$ to $O$ as desired.
\end{proof}

\begin{lemma}\label{InvLimOfArcCompsEnd}
Let $\cA$ be an auxiliary arc-component of $\GH$ and let $u_0,u_1,\hdots$ be an enumeration of $\cA\cap V(G)$. For every $n\in\N$ write $X_n=\{u_k\,|\,k<n\}$ and
\begin{align*}
\cA_n=\cA-X_n-\mathring{E_{\GG}}(X_n,\cA).
\end{align*}
Let some $N\in\N$ be given, and let $(x_n)_{n\ge N}$ be an $\omega$-sequence of arc-components $x_n$ of $\cA_n$ with $x_n\supseteq x_{n+1}$ for all $n\ge N$. Then there exists an end $\omega$ of $G$ such that $\bigcap_{n\ge N}x_n=\{\omega\}$.
\end{lemma}
\begin{proof}
The set $\{x_n\,|\,n\ge N\}$ trivially has the finite intersection property. For every $n\ge N$ set $y_n=x_n\setminus\mathring{E}(\GG)$ and let $z_n$ be the closure of $y_n$ in $\TC$.
Since the set $\{y_n\,|\,n\ge N\}$ also has the finite intersection property, so does $\{z_n\,|\,n\ge N\}$. Using that $\TC$ is compact hence yields some $\xi\in\zeta:=\bigcap_{n\ge N}z_n$.

First we show $\zeta\subseteq\Omega(G)$. Clearly, $\zeta$ avoids $G$, so suppose for a contradiction that there is some $\upsilon$ in $\Upsilon\cap\zeta$. 
If $X_\upsilon$ meets $\cA$, then $X_\upsilon\subseteq\cA$ and we may let $M\in\N$ be big enough that $X_\upsilon\subseteq X_M$. Then due to Lemma~\ref{vxJAL} there is some component $C$ of $G-X_M$ with $x_M\subseteq\cO_{\GH}(X_M,\{C\},1)$, and hence
$y_M\subseteq \cO_{\TC}(X_M,\{C\})$. 
Therefore, $\cO_{\TC}(X_M,\cC_{X_M}-\{C\})$ is an open neighbourhood of $\upsilon$ avoiding $y_M$, yielding $\upsilon\notin z_M$, a contradiction.
Else if $X_\upsilon$ avoids $\cA$, then by Theorem~\ref{AuxArcsChar} there is a finite cut $F$ of $G$ with sides $A$ and $B$ separating $X_\upsilon$ from $\cA$. 
Put $X=V[F]$ and let $\{\cC,\cC'\}$ be a bipartition of $\cC_X$ respecting the sides of $F$, i.e. with $V[\cC]\subseteq A$ and $V[\cC']\subseteq B$. Without loss of generality we have $X_\upsilon\subseteq A$. Then $\upsilon\in\cO_{\TC}(X,\cC)$, since otherwise $\cO_{\TC}(X,\cC')$ is basic open neighbourhood of $\upsilon$ which sends no edges to $X_\upsilon$, contradicting Lemma~\ref{Xtau}.
Hence $\cO_{\TC}(X,\cC)$ is an open neighbourhood of $\upsilon$ avoiding $\cA\setminus\mathring{E}(\GG)=y_0$, resulting in $\upsilon\notin z_0\supseteq\zeta$, a contradiction. Hence $\zeta\subseteq\Omega(G)$ holds as claimed.

Finally, we show that $\zeta$ is a singleton: Assume not for a contradiction, so we find two distinct ends $\omega$ and $\omega'$ of $G$ in $\zeta$. Let $X\in\cX$ be a witness of $\omega\neq\omega'$ and let $K$ be big enough that $X\cap\cA\subseteq X_K$. Then due to Lemma~\ref{vxJAL} the ends $\omega$ and $\omega'$ are contained in distinct auxiliary arc components of $\cA_K$, contradicting $\zeta\subseteq x_K$. Hence $\zeta$ is a singleton subset of $\Omega(G)$.
\end{proof}

\newpage
\subsection{$\mathfrak{A}$TSTs of auxiliary arc-components}

\begin{theorem}\label{ExternalSubTSTs}
If $G$ is a countable graph, then every auxiliary arc-component of $\GH$ has an \ATST{}.
\end{theorem}
\begin{proof}
The proof starts with the essential idea of a proof from the lecture course.

Let $\cA$ be any non-singleton auxiliary arc-component of $\GH$. If $\cA\cap V(G)$ is finite, then any spanning tree of $\cA\subseteq\GG$ will do, so we may assume that $\cA\cap V(G)$ is infinite. Pick an enumeration $u_0,u_1,\hdots$ of $\cA\cap V(G)$ and write $X_n=\{u_k\,|\,k<n\}$ for every $n\in\N$.
Furthermore, for every $n$ we write 
\begin{align*}
\cA_n=\cA-X_n-\mathring{E}_{\GG}(X_n,\cA),
\end{align*}
and we let the set $\cV_n$ consist of all singleton subsets of $X_n$ and all arc-components of $\cA_n$. Then we let $H_n$ be the multigraph on $\cV_n$ whose edges are the cross-edges of $\cV_n$ with respect to $\GG-E(G)$. For each $n$ we denote by $\cN_n$ the arc-component of $\cA_n$ containing $u_n$ (which is in $X_{n+1}\setminus X_n$ by definition).
Hence $\cV_n$ can be obtained from $\cV_{n-1}$ by discarding $\cN_{n-1}$ and adding $\{u_{n-1}\}$ as well as the arc-components of 
\begin{align*}
\cK_{n-1}:=\cN_{n-1}-u_{n-1}-\mathring{E}_{\GG}(u_{n-1},\cN_{n-1}).
\end{align*}

Next, we inductively construct spanning trees $T_n$ of the $H_n$, starting with the spanning tree $T_0=(\{\cA\},\emptyset)$ of $H_0=(\{\cA\},\emptyset)$.
For the induction step, informally we obtain $T_n$ from $T_{n-1}$ by expanding the vertex $\cN_{n-1}$ of $T_{n-1}$ to a star in $H_n$ with centre $\{u_{n-1}\}$ and leaves the arc-components of $\cK_{n-1}$. 
Formally, we proceed as follows: Let $T_n'$ be the subgraph of $H_n$ whose 1-complex coincides with $\overline{\mathring{E}(T_{n-1})}$ where the closure is taken in the 1-complex of $H_n$ (i.e. $T_n'$ is the subgraph of $H_n$ induced by the inner edge points of $T_n$). 
Then we let $T_n$ be the union of $T_n'$ and an arbitrary spanning star of $H_n[\cV_n\setminus \cV_{n-1}]$ (recall that $\cV_n\setminus \cV_{n-1}$ is the set consisting of $\{u_{n-1}\}$ and the arc-components of $\cK_{n-1}$). It is easy to see that $T_n$ is a spanning tree of $H_n$.
Note that, by construction,
for every $n$ we have that every $\{u\}\subseteq X_n$ sends to each arc-component of $\cA_n$ at most one edge of $T_n$. As a consequence, every arc-component of $\cA_n$ has finite degree in $T_n$.

Finally, we let
\begin{align*}
\Tau:=\overline{\bigcup_{n\in\N}\mathring{E}(T_n)}^{\;\GH}
\end{align*}
which will turn out to be our desired \ATST{} of $G$. The rest of the proof is dedicated to a formal verification.
By definition, $\Tau$ is a standard subspace of $\GH$.
First, we show that $\Tau$ is `spanning':

\begin{claim_subTSTs}\label{auxArcCompSingleton}
If $\cK$ is an auxiliary arc component of $\cA_n$ for some $n$ then either $\cK$ meets $V(G)$ or $\cK=\{\omega\}$ for some $\omega\in\Omega$.
\end{claim_subTSTs}
\begin{proof}[Proof of the Claim]\renewcommand{\qedsymbol}{$\diamond$}
If $\cK$ avoids $V(G)$ then $\cK$ also avoids $\mathring{\AuxE}$, so $\cK\subseteq\Omega$. Then $\cK$ must be a singleton: Otherwise there is some $\omega'\in \cK-\{\omega\}$. Pick a witness $X\in\cX$ of $\omega\neq\omega'$. Then $\cK$ meets $X$ by Lemma~\ref{vxJAL}, a contradiction.
\end{proof}

\begin{claim_subTSTs}\label{subTSTspanning}
$\cA\cap V(\GG)\subseteq\Tau$.
\end{claim_subTSTs}
\begin{proof}[Proof of the Claim]\renewcommand{\qedsymbol}{$\diamond$}
By construction we have $\cA\cap V(G)\subseteq\Tau$. Let any $\omega\in\cA\cap\Omega$ be given; we have to show $\omega\in\Tau$.

If for every $n\in\N$ the auxiliary arc-component $\cW_n$ of $\cA_n$ containing $\omega$ meets $\cA\cap V(G)$ we have $\omega\in\overline{\cA\cap V(G)}$: Indeed, consider any basic open neighbourhood $\hat{\aC}_\epsilon(X,\omega)$ of $\omega$ in $\GH$ and pick $N\in\N$ big enough such that $X\cap\cA\subseteq X_N$. 
Then $\cW_N\subseteq \hat{\aC}_\epsilon(X,\omega)$ by Lemma~\ref{vxJAL}. Due to our assumption we know that $\cW_N\subseteq \hat{\aC}_\epsilon(X,\omega)$ meets $\cA\cap V(G)$. Hence $\omega\in\overline{\cA\cap V(G)}$ holds. Since $\Tau$ is closed and $\cA\cap V(G)\subseteq\Tau$ we also have $\omega\in\overline{\cA\cap V(G)}\subseteq\Tau$.

Otherwise there is some $\cW_n$ which avoids $\cA\cap V(G)$ and hence $V(G)$. By Claim~\ref{auxArcCompSingleton} we have $\cW_n=\{\omega\}$, so $T_n$ includes the interior of an auxiliary-edge $e$ from $X_n$ to $\omega$ and $\mathring{e}\subseteq\Tau$ witnesses $\omega\in\Tau$.
\end{proof}

Thus $\Tau$ is `spanning'. Next, we show that $\Tau$ is arc-connected:

\begin{claim_subTSTs}\label{arcFromUiToUj}
For every $i<j$ there exists an arc from $u_i$ to $u_j$ in $\Tau$.
\end{claim_subTSTs}
\begin{proof}[Proof of the Claim]\renewcommand{\qedsymbol}{$\diamond$}
For every $n\ge j$ there exists a unique path $P_n\subseteq T_n$ from $u_i$ to $u_j$. Let $A_n$ be an arc traversing $P_n$ (i.e. formally $A_n$ is the 1-complex of the graph $P_n$). For every $n>j$ we define $f_n\colon A_n\to A_{n-1}$ as follows:
If $\xi\in A_n$ is a point of $P_n$ which is also a point of $P_{n-1}$ we put $f_n(\xi):=\xi$. Else if $\xi$ is not a point of $P_{n-1}$, then $\xi$ formally is one of the following:
\begin{enumerate}
\item the singleton $\{u_{n-1}\}\subseteq X_n\setminus X_{n-1}$;
\item an inner edge point of an auxiliary edge from $u_{n-1}$ to an arc-component of $\cK_{n-1}$;
\item an arc-component of $\cK_{n-1}$.
\end{enumerate}
In either case we let $f_n$ map $\xi$ to $\cN_{n-1}$, which formally is a point of $P_{n-1}$. Note that $f_n^{-1}(\cN_{n-1})$ is a connected subset of (the 1-complex of) $P_n$ and $f_n$ is a continuous surjection. The arcs $A_n$ together with the maps $f_n$ form an inverse sequence $\{A_n,f_n,\N_{\ge j}\}$ whose inverse limit $A:=\invLim A_n$ is an arc by Theorem~\ref{Capel}. We have to translate this arc into an auxiliary arc. 

For this, we define $\varphi\colon A\to\cA$ as follows:
Let any $x=(x_n\,|\,n\ge j)\in A$ be given. If there is some $N\ge j$ such that $x_N$ is an inner edge point, then we put $\varphi(x)=x_N$.
Else if there are some naturals $N\ge j$ and $\ell$ such that $x_N=\{u_\ell\}$, then we put $\varphi(x)=u_\ell$.
Both cases are well defined by definition of the $f_n$. 
For the final case suppose that $x_n$ is an auxiliary arc-component of $\cA_n$ for every $n\ge j$. 
Then $x_n\supseteq x_{n+1}$ holds for all $n\ge j$ by definition of the bonding maps and we set $\varphi(x)=\omega$ for the end $\omega$ of $G$ with $\bigcap_{n\ge j}x_n=\{\omega\}$ which exists by Lemma~\ref{InvLimOfArcCompsEnd}. This completes the definition of $\varphi$.

Let $\sigma\colon \I\to A$ witness that $A$ is an arc. In order to show that $A':=\im(\varphi)$ is an arc it suffices to show that $\psi:=\varphi\circ\sigma\colon \I\to A'$ is a continuous injection since $\I$ is compact and $A'\subseteq\GH$ is Hausdorff.
Clearly, $\psi$ is injective. Since $\sigma$ is continuous, it suffices to show that $\varphi$ is continuous. 
For this, let $x=(x_n\,|\,n\ge j)$ be any point of $A$ and let $O$ be any basic open neighbourhood of $\varphi(x)$ in $\GH$. 

If there exists some $N\ge j$ such that $x_N$ is not an arc-component of $\cA_N$, then $O$ easily translates into an open neighbourhood $W$ of $x_N$ in $A_N$. Letting $W_n:=A_n$ for all $n\ge j$ with $n\neq N$ and $W_N:=W$ results in the desired open neighbourhood $A\cap\prod_{n\ge j}W_n$ of $x$. 

Otherwise for every $n\ge j$ the point $x_n$ is an arc-component of $\cA_n$, and $\varphi(x)$ is an end $\omega$ of $G$ by construction. In particular, 
$O$ is of the form $\hat{\aC}_\epsilon(X,\omega)$ for some $X\in\cX$. 
Let $N\ge j$ be big enough that $X\cap \cA\subseteq X_N$. 
Then Lemma~\ref{vxJAL} implies $x_N\subseteq\hat{\aC}_\epsilon(X,\omega)$. 
Recall that $x_N$ is a vertex of the path $P_N$. 
Due to $N\ge j>i$ there exist two distinct auxiliary edges $e$ and $e'$ such that $\mathring{e}\cup\{x_N\}\cup\mathring{e}'$ is a homeomorphic copy of $(0,1)$ in $A_N$.
Clearly, this set includes an open neighbourhood $W$ of $x_N$ in $A_N$ such that $W-\{x_N\}\subseteq\hat{\aC}_\epsilon(X,\omega)$ holds. In particular, letting $W_n:=A_n$ for all $n\ge j$ with $n\neq N$ and $W_N:=W$ results in an open neighbourhood $A\cap\prod_{n\ge j}W_n$ of $x$ which $\varphi$ sends to $O$.
\end{proof}

\begin{claim_subTSTs}\label{simpleArg}
If $A$ is an arc in $\cA$ which meets two distinct auxiliary arc-components of $\cA_n$, then $A$ meets $X_n$.
\end{claim_subTSTs}
\begin{proof}[Proof of the Claim]\renewcommand{\qedsymbol}{$\diamond$}
If not, then $A$ also avoids $X_n\cup\mathring{E}_{\GG}(X_n,\cA)$. Hence there exists a unique arc-component of $\cA_n$ including $A$, a contradiction.
\end{proof}

\begin{claim_subTSTs}\label{auxArcMeetsEdgeAtXn}
For every $n$ and every arc-component $\cK$ of $\cA_n$: if $A$ is an arc in $\cA$ meeting $\cK$ and $\cA\setminus\cK$, then $A$ meets $\mathring{\AuxE}(X_n,\cK)$.
\end{claim_subTSTs}
\begin{proof}[Proof of the Claim]\renewcommand{\qedsymbol}{$\diamond$}
Let $\sigma\colon \I\to A$ be a parametrisation of $A$, without loss of generality with $\sigma(0)\in\cK$ and $\sigma(1)\in\cA-\cK-\mathring{\AuxE}(X_n,\cK)$. 
Since $\cK$ is a closed subset of $\GH$ by Lemma~\ref{AuxArcCompsClosed} there is a maximal $\lambda\in\I$ with $\sigma(\lambda)\in\cK$. 
By Claim~\ref{simpleArg} we know that $\sigma[[\lambda,1]]$ meets $X_n$, so since $X_n$ is finite there is some minimal $\mu>\lambda$ with $\sigma(\mu)\in X_n$. 
Again by Claim~\ref{simpleArg} and by choice of both $\lambda$ and $\mu$ we know that $\sigma[(\lambda,\mu)]$ avoids $\cA_n$. 
Hence $\sigma[(\lambda,\mu)]$ is a connected subset of $\mathring{\AuxE}(X_n,\cA)$. 
In particular, there is an auxiliary edge $e$ in $\AuxE(X_n,\cA)$ with $\mathring{e}\supseteq\sigma[(\lambda,\mu)]$, so $e=\sigma[[\lambda,\mu]]$ follows.
Write $e=xy$ with $x\in \cA$ and $y\in X_n$. Due to $\sigma(\lambda)\in\cK$, the only possibility for $x$ and $y$ is $x=\sigma(\lambda)$ and $y=\sigma(\mu)$. 
Therefore, we have $e\in \AuxE(X_n,\cK)$ as desired.
\end{proof}

\begin{claim_subTSTs}
For every end $\omega$ of $G$ in $\cA$ there exists an arc from $\omega$ to $u_0$ in $\Tau$.
\end{claim_subTSTs}
\begin{proof}[Proof of the Claim]\renewcommand{\qedsymbol}{$\diamond$}
We check two main cases.

For the first main case suppose that for every $n$ the auxiliary arc-component $\cW_n$ of $\cA_n$ containing $\omega$ meets $V(G)$ in some $t_n$. 
Then we use Claim~\ref{arcFromUiToUj} to find an auxiliary arc $A_n$ from $u_0$ to $t_n$ in $\Tau$ for every $n$. 
We have $t_n\to\omega$ in $\GH$: 
Consider any basic open neighbourhood $\hat{\aC}_\epsilon(X,\omega)$ of $\omega$ in $\GH$ and let $N\in\N$ be big enough that $X\cap \cA\subseteq X_N$. Then Lemma~\ref{vxJAL} yields $\cW_N\subseteq\hat{\aC}_\epsilon(X,\omega)$ which implies $t_n\in \hat{\aC}_\epsilon(X,\omega)$ for all $n\ge N$.
If Lemma~\ref{CtblAuxArcStarComb} yields an auxiliary arc from $u_0$ to $\omega$ included in $\{\omega\}\cup\bigcup_{n\in\N}A_n\subseteq\Tau$ then we are done.
Otherwise Lemma~\ref{CtblAuxArcStarComb} yields some $k\in\N$ and a vertex $u_\ell$ on $A_k$ which is the centre of an infinite star of auxiliary arcs included in $\bigcup_{n\in\N}A_n$ and with leaves in $\{t_n\,|\,n\in\N\}$. But then Claim~\ref{auxArcMeetsEdgeAtXn} applied to the infinitely many arcs from $u$ to the $t_n$ in that fan implies that $\cW_\ell$ is a vertex of $H_\ell$ with infinite degree in $T_\ell$, which is impossible (as argued at the end of the construction of the trees $T_n$).

For the second main case suppose that there is some $N$ such that $\cW_N$ avoids $V(G)$. 
Then $\cW_N=\{\omega\}$ holds by Claim~\ref{auxArcCompSingleton}, so $T_N$ includes the interior of some auxiliary edge from some $u\in X_N$ to $\omega$. 
If $u=u_0$ we are done. Otherwise, by Claim~\ref{arcFromUiToUj} there is an auxiliary arc from $u_0$ to $u$ in $\Tau$, which $u\omega$ extends to an auxiliary arc from $u_0$ to $\omega$ in $\Tau$.
\end{proof}

Thus, $\Tau$ is arc-connected.

\begin{claim_subTSTs}
$\Tau$ is acirclic.
\end{claim_subTSTs}
\begin{proof}[Proof of the Claim]\renewcommand{\qedsymbol}{$\diamond$}
Suppose for a contradiction that $\Tau$ contains a circle $C$. 
By Lemma~\ref{SigmaArcEdgesDense} we know that $C$ traverses an auxiliary-edge, so in particular $C$ contains a vertex of $G$.
Let $N\in\N$ be minimal with $u_N\in C$. 
Then $C-u_N$ avoids $X_{N+1}$. 
Let $e$ be an auxiliary edge at $u_N$ which $C$ traverses, and let $\cK$ be the auxiliary arc-component of $\cA_{N+1}$ containing the other endvertex of $e$ (which exists by choice of $N$).
Now consider the arc $A:=C-\mathring{e}\subseteq\Tau$. 
Since $A$ meets both $X_{N+1}$ and $\cK$, Claim~\ref{auxArcMeetsEdgeAtXn} together with the choice of $N$ yields that $A$ traverses some $e'\in\AuxE(u_N,\cK)-\{e\}$. 
But then we have $\mathring{e}\cup\mathring{e}'\subseteq T_{N+1}$, i.e. $T_{N+1}$ has two parallel edges, contradicting the fact that $T_{N+1}$ is a spanning tree of $H_{N+1}$.
\end{proof}
This completes the proof that $\Tau$ is an \ATST{} of $\cA$.
\end{proof}

\subsection{Tree-packing}

A \textit{circle} in $\cE G$ is a homeomorphic copy of the unit circle.
If $H=(V',E')$ is a subgraph of $G$,
then the closure of $(V'/{\sim}_{\cE})\cup\mathring{E}'$ in $\cE G$ where ${\sim}_{\cE}$ is the equivalence class used to obtain $\cE G$ from $\cE' G=G\cup\Omega'$ (cf. Section~\ref{subsec:topsOverview}) is said to be a \textit{standard subspace} of $\cE G$.
Clearly, every circle in $\cE G$ is a standard subspace.
A standard subspace of $\cE G$ is said to be \textit{spanning} if it includes $\cE G\setminus\mathring{E}$.
A \textit{topological spanning tree} (TST) of $\cE G$ is a uniquely arc-connected spanning standard subspace of $\cE G$. These definitions are equivalent to those of Miraftab \cite[Chapter 5]{Babak}.\footnote{Note that Miraftab writes $(\tilde{G},\textsc{Itop})$ for $(\cE G,\textsc{ETop})$.}

\begin{lemma}[{\cite[Lemma 22]{Babak}}]\label{ETopAconFinCut}
Suppose that $G$ is a countable connected graph.
A standard subspace $\Xi$ of $\cE G$ is arc-connected if and only if $\Xi$ contains an edge from every finite cut of $G$ of which it meets both sides (taken in $\cE G$).
\end{lemma}

\begin{lemma}\label{Lem858}
Suppose that $G$ is a countable connected graph.
If for every finite partition of $V(G)$, into $\ell$ sets say, $G$ has at least $k(\ell-1)$ cross-edges, then $\cE G$ has $k$ edge-disjoint arc-connected spanning standard subspaces.
\end{lemma}
\begin{proof}
This easily follows from mimicking the proof of \cite[Lemma 8.5.8]{Bible} where we replace 
\begin{enumerate}
\item the $G_n$ by the $G\boldsymbol{.}F$ from the inverse system $\{G\boldsymbol{.}F,f_{F',F},\cE\}$,
\item \cite[Lemma 8.5.5]{Bible} by Lemma~\ref{ETopAconFinCut},
\item \cite[Lemma 8.1.2]{Bible} by Lemma~\ref{GIL},
\end{enumerate}
and finally use Lemma~\ref{ETopInvLim} for $\invLim (G\boldsymbol{.}F\,|\,F\in\cE)=\CL G\CL\simeq\cE G$.
\end{proof}

The following Lemma can be proved analogously to \cite[Lemma 8.5.9]{Bible} (with \cite[Lemma 8.5.5]{Bible} replaced by Lemma~\ref{ETopAconFinCut} in the proof):
\begin{lemma}[$\L$]\label{ETopStandardTST}
Suppose that $G$ is a countable connected graph. Then every connected spanning standard subspace of $\cE G$ includes a TST of $\cE G$.\qed
\end{lemma}

\newpage
\begin{theorem}\label{auxTreePacking}
Let $G$ be a countable connected graph.\footnote{Multigraph should not be a problem.} Then the following are equivalent for all $k\in\N$:
\begin{enumerate}
\item $G$ has $k$ external \ATST{}s which are edge-disjoint on $E(G)$.
\item $G$ has at least $k(|P|-1)$ edges across any finite vertex partition $P$.
\end{enumerate}
\end{theorem}
\begin{proof}
(i)$\to$(ii). Using Lemma~\ref{cutJAL} this holds by the same argumentation as in the proof of \cite[Theorem 8.5.7]{Bible}.

(ii)$\to$(i). 
For this proof we treat $\cE G\simeq\Gq$ as $\cE G=\Gq$ (recall Theorem~\ref{ETopInvLimAndGq}).
Applying Lemmas~\ref{Lem858} and~\ref{ETopStandardTST} yields
$k$ edge-disjoint TSTs $\tilde{\Tau}_1,\hdots,\tilde{\Tau}_k$ of $\cE G$.
Let $\hat{\cA}$ be the set of all non-singleton auxiliary arc-components of $\GH$. For every $\cA\in\hat{\cA}$ we apply Theorem~\ref{ExternalSubTSTs} to find an \ATST{} $\Tau_{\cA}$ of $\cA$, and for each $i\in [k]$ we set
\begin{align*}
\hat{\Tau}_i=\overline{\mathring{E}\big(\tilde{\Tau}_i\big)}^{\;\GH}\cup\bigcup_{\cA\in\hat{\cA}}\Tau_{\cA}
\end{align*}
Then every $\hat{\Tau}_i$ is acirclic: Assume not for a contradiction, witnessed by some circle $C\subseteq\hat{\Tau}_i$. 

If $E(C)$ avoids all finite cuts of $G$, then $C$ being connected together with Lemma~\ref{cutJAL} yields $C\subseteq\Tau_{\cA}$ for some $\cA\in\hat{\cA}$, which is impossible.

Otherwise $E(C)$ meets some finite cut $F$ of $G$ in some edge $e$. Then $A:=C-\mathring{e}$ is an arc in $\hat{\Tau}_i-\mathring{e}$. 
By Lemma~\ref{cutJAL} we know that $(A\setminus\mathring{\AuxE})/{\sim}$ satisfies the premise of Lemma~\ref{ETopAconFinCut}, so $\tilde{\Tau}_i-\mathring{e}$ is still connected which is impossible. Thus $\Tau_i$ is acirclic.

Furthermore, every $\Tau_i$ is arc-connected: For this, let any two distinct points $x$ and $y$ of $\GH$ be given (without loss of generality none of $x$ and $y$ is in $\mathring{E}(\GG)$). If $x\sim y$, then by Theorem~\ref{AuxArcsChar} there is some $\cA\in\hat{\cA}$ containing $x$ and $y$, so we find an arc from $x$ to $y$ in $\Tau_{\cA}$ and hence in $\hat{\Tau}_i$. Otherwise $x\not\sim y$ holds, and we let $A$ be an arc in $\tilde{\Tau}_i$ from $[x]_\sim$ to $[y]_\sim$.
Then Lemma~\ref{GqArcExtension} lifts $A$ to an arc from $x$ to $y$ in $\hat{\Tau}_i$.
\end{proof}

\begin{corollary}
Every countable connected graph has an \ATST{}.\qed
\end{corollary}

\subsection{Outlook: limit thin sums}

While we managed to construct auxiliary arcs without imposing any cardinality bounds on $G$, we proved tree-packing only for countable connected $G$. On the one hand, this is needed for the Arc Lifting Lemma~(\ref{GqArcExtension}) to work. But on the other hand we used it to obtain \ATST{}s of auxiliary arc-components of $\GH$, and we would like to know whether these exist for arbitrary $G$, too.
Also, we would like to know whether there is a connection to dendrites.

There exist various obstructions to naive generalisations of thins sums of circles. For example, if $G$ is just a dominated ray embedded into the plane, and if our thin family consists of all the (inner) facial cycles of $G$, then their sum should yield the ray plus the sole auxiliary edge since the ray itself is an NST and hence induces an \ATST{}. Hence we could modify the definition of a thin sum in that we also add certain auxiliary edges. For example, if $G$ has no ultrafilter tangles, and if $(C_i\,|\,i\in I)$ is a thin family of circuits, then a candidate for such a thin sum could be
\begin{align*}
\sum_{i\in I}C_i+\bigg\{u\omega\in\AuxE(\Omega)\,\bigg\vert\,\omega\in\overline{\big(\bigcup_{i\in I} C_i\big)\cap E_{\GG}(u)}\;\Big\backslash\;\overline{\big(\sum_{i\in I}C_i\big)\cap E_{\GG}(u)}\bigg\}.
\end{align*}
On the other hand, investigation of $\LG$ from the next section should be prioritised.

\newpage
\part{Partial results}
\section{A tangle Hausdorff compactification}\label{LGsection}

The introduction of Chapter~\ref{AuxArcTP} motivated the auxiliary space $\GH$ after motivating the Hausdorff compactification $\LG$. 
Hence in this chapter, we immediately begin with the formal construction of $\LG$.

\subsection{Construction}\label{LGconstruction}

Recall section~\ref{tangleInvLim}.\index{$\LG$}
For every $X\in\cX$ and $\cC\subseteq\cC_X$ put
\begin{align*}
\LP(X,\cC)=\cO_{\TC}(X,\cC)\setminus\mathring{E},
\end{align*}
and for each $P\in\cP_X$ put
\begin{align*}
\hat{P}=\{\Lambda(X,\cC)\,|\,\cC\in P\}.
\end{align*}
Furthermore, for every $(X,P)\in\Gamma$ we denote by $\hat{p}(X,P)$ the finite partition of $V\cup\cU$ induced by $\hat{P}$ and the singleton subsets of $X$.
Let $\hat{G}$ be the graph on $V\cup\cU$ with edge set
\begin{align*}
E(G)\cup\{u\upsilon\,|\,\upsilon\in\Upsilon, u\in X_\upsilon\}\cup\{u\omega\,|\,\omega\in\Omega,u\in\Delta(\omega)\}.
\end{align*}
The edges in $E(\hat{G})\setminus E(G)$ are the \textit{limit edges}.
For every $\gamma=(X,P)\in\Gamma$ we let $\hat{G}/\hat{p}(X,P)$ be the multigraph on $\hat{p}(X,P)$ whose edges are precisely the cross-edges of $\hat{p}(X,P)$ with respect to $\hat{G}$. 
Vertices of $\hat{G}/\hat{p}(X,P)$ that are singleton subsets $\{x\}$ of $X$ we consider to be vertices of $G$ and refer to them as $x$; the other vertices of $\hat{G}/\hat{p}(X,P)$ are its \textit{dummy vertices}.
Now we let $\hat{G}_\gamma$ be the topological space obtained from the ground set of the 1-complex of $\hat{G}/\hat{p}(X,P)$ by endowing it with the topology generated by the following basis:

For every $x\in X$ and $\epsilon\in (0,1]$ we choose $\cO_{\hat{G}}(x,\epsilon)$.
For every inner edge point of an edge that is an edge of $G$ we choose the usual open neighbourhoods. For every dummy vertex $d=\Lambda(X,\cC)$ of $\hat{G}/\hat{p}(X,P)$ and every $\epsilon>0$ we choose as open the set
\begin{align*}
\{d\}\cup\bigcup\big\{t[0,\epsilon)x\,\big|\,x\in X\text{ and }xt\in E_{\hat{G}}(X,d)\big\}.
\end{align*}
If $i$ is an inner edge point of an edge $e\in E(\hat{G})\setminus E(G)$ with $e=x\upsilon$, i.e. there is some $\cC\in P$ and some $\upsilon\in\cO_{\cU}(X,\cC)$ with $e=x\upsilon$, then for
every $0\le \epsilon<\delta\le 1$ with $i\in x(\epsilon,\delta)\upsilon$ and every finite $F\subseteq E_{\hat{G}}(X,d)-\{e\}$ we declare as open the set
\begin{align*}
\bigcup\{x(\epsilon,\delta)\xi\,|\,x\in X\text{ and }x\xi\in E_{\hat{G}}(X,d)\setminus F\},
\end{align*}
also see Fig.~\ref{fig:LGbos}.
\begin{figure}[H]
    \centering
    \def\svgwidth{\columnwidth}
    \scalebox{.8}{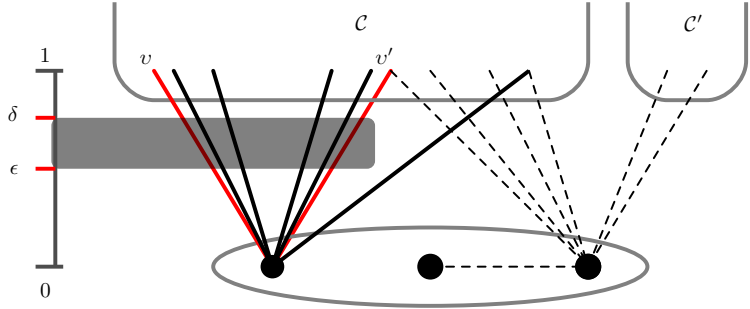}
    \caption{A basic open neighbourhood of an inner limit edge point.}
    \label{fig:LGbos}
\end{figure}
\begin{lemma}\label{LGgammesCptHD}
$\hat{G}_\gamma$ is a compact Hausdorff topological space for every $\gamma\in\Gamma$.
\end{lemma}
For every $\gamma=(X,P)\le (X',P')=\gamma'\in\Gamma$ we define a bonding map $\hat{f}_{\gamma',\gamma}\colon \hat{G}_{\gamma'}\to\hat{G}_\gamma$ which sends the vertices of $\hat{G}_{\gamma'}$ to the vertices of $\hat{G}_\gamma$ including them; which is the identity on the edges of $\hat{G}_{\gamma'}$  that are also edges of $\hat{G}_\gamma$; and which sends any other edge of $\hat{G}_{\gamma'}$ to the dummy vertex of $\hat{G}_\gamma$ that contains both its endvertices in $\hat{G}_{\gamma'}$.
\begin{lemma}
The $\hat{f}_{\gamma',\gamma}$ are continuous.
\end{lemma}
By Lemmas~\ref{LGgammesCptHD}~and~\ref{InvLimCptHD} the inverse limit
\begin{align*}
\LG=\invLim (\hat{G}_\gamma\,|\,\gamma\in\Gamma)
\end{align*}
is compact Hausdorff.

\subsection{Outlook}

Call a continuum $\Tau$ \textit{TST-like} if for every two distinct points $a$ and $b$ of $\Tau$ 
the set $\aS_{\Tau}(a,b)$ endowed with the subspace topology is a subcontinuum of $\Tau$. 
Recall that, by Theorem~\ref{cutPointOrdLin} the space $\aS_{\Tau}(a,b)$ is linearly ordered by its separation ordering, and by Theorem~\ref{sepOrderSubSame} the subspace topology on $\aS_{\Tau}(a,b)$ coincides with its order topology.
If $H=(V',E')$ is a subgraph of $\hat{G}$ where each $\upsilon\in V'\setminus V(G)$ is incident with a limit edge in $E'$, then we write $\overline{H}=\overline{V'\cup\mathring{E}'}$ for its closure in $\LG$ and call this closure a \textit{standard subspace}.
If $G$ is any graph and $\Xi$ is a standard subspace of $\LG$ with $V(G)\subseteq\Xi$, then we say that $\Xi$ is \textit{spanning}.
A \textit{TST} of $\LG$ is a spanning standard subspace of $\LG$ that is also a TST-like subcontinuum of $\LG$. A \textit{circle} $C$ of $\LG$ is a subcontinuum of $\LG$ such that $C-\{a,b\}$ is disconnected for every two distinct points $a$ and $b$ of $C$.\footnote{We should not need the additional requirement that $C-\{a\}$ be connected: an arc minus its two endpoints is still connected, and hence not a circle. Also see the claim of \cite[Theorem 28.14]{Willard}.}

\begin{conjecture}
If $G$ is any graph, then $\LG$ admits a TST.
\end{conjecture}
\begin{proof}[Idea]
By transfinite usage of Theorem~\ref{directContinua} we can find a subcontinuum $\Xi'$ of $\LG$ such that
\begin{enumerate}
\item $\LG\setminus\mathring{E}(G)\subseteq\Xi'$.
\item For every edge $e$ of $G$ the space $\Xi'-\mathring{e}$ is disconnected.
\end{enumerate}
Finally, use Theorem~\ref{irredSubcontinuumExists} to find a subcontinuum $\Xi$ of $\Xi'$ which is irreducible about $V(G)$. It remains to show that $\Xi$ is a TST-like standard subspace of $\LG$. (If $\Xi$ is not standard, a transfinite approach should work instead of Theorem~\ref{irredSubcontinuumExists}, but here we have to remove open neighbourhoods of inner edge points of limit edges with $\epsilon=0$ and $\delta=1$.)
\end{proof}

\newpage
\section{ETop and the maximal Hausdorff quotient}\label{sec:GqNatural}

\subsection{Introduction}

Since for connected graphs $G$ the quotient $\Gq=\TC/{\sim}$ is a Hausdorff quotient of $\TC$ with $\Gq\simeq\cE G$ by Theorem~\ref{ETopInvLimAndGq}, one might ask whether it is the maximal Hausdorff quotient of $\TC$. 
In this chapter, we study for which graphs $G$ the space $\Gq$ is the maximal Hausdorff quotient of the tangle compactification. 
In order to keep things to the point, we introduce some notation first:
Recall that the relation $\between$ is defined on $\TC\setminus\mathring{E}$ by letting $x\between y$ whenever there exist no two disjoint open neighbourhoods of $x$ and $y$ in $\TC$, and $\transcl$ denotes the transitive closure of $\between$. 
Since $\transcl$ is an equivalence relation, we write $\Gqq$ for the quotient space $\TC/{\transcl}$.\index{$\Gqq$}
The equivalence relation $\tsim$\index{$\tsim$} 
on $\TC$ satisfying $\TC/{\tsim}=\HTC$ is given by Theorem~\ref{minimalHDrelation}. Due to Corollary~\ref{HD}, the minimality of $\tsim$ and the definition of $\transcl$ we have
\begin{lemma}\label{3Rels}
$\between\subseteq\transcl\subseteq\tsim\subseteq\sim$ holds for every graph $G$.\qed
\end{lemma}
Now we are ready to start: Of course, the first question that comes to mind, is whether there even exist graphs for which $\Gq$ is not the maximal Hausdorff quotient, i.e. for which $\tsim$ is a proper subset of $\sim$. Surprisingly, a long known example graph positively answers this question:

\begin{figure}[h]
\includegraphics[clip,page=11,trim=40 170 40 285,width=\textwidth]{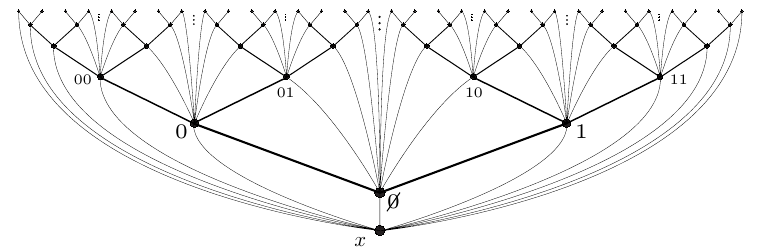}
\caption{A graph with $\transcl=\tsim\subsetneq\sim$ from \cite[Fig. 6]{TST}}
\label{fig:badG}
\end{figure}

The graph $G$ from \cite[Fig. 6]{TST}, also pictured in Fig.~\ref{fig:badG} here, originally served as an example of a non-finitely separable graph on which the relation used to obtain \textsc{ITop} as a quotient of \textsc{VTop} is an equivalence relation (i.e. for which \textsc{ITop} is well-defined even though $G$ is not finitely separable) and which has no topological spanning tree with respect to \textsc{ITop}. Even more strikingly, in \textsc{ITop} this graph admits an edgeless Hamilton circle (which witnesses the absence of topological spanning trees, see \cite[Proposition 3.4 and Corollary 3.5]{TST} for details).
In this graph, $\sim$ identifies every two points of $V\cup\Omega$, whereas  \textsc{ITop} identified every vertex with the two ends it dominates. Unsurprisingly, $\transcl$ identifies the same points as \textsc{ITop} did:

Clearly, it suffices to show that this graph has no ultrafilter tangles. For this, consider any $X\in\cX$. Then $\cC_X$ is finite: Indeed, let $X'$ be the union of $\{x\}$ and the first $n$ levels of the underlying binary tree of $G$, where $n\in\N$ is big enough that $X$ is included in $X'$. Now $\cC_{X'}$ is easily seen to be finite, and so $\cC_X$ must be finite. Hence every $\cC_X$ is finite, so $G$ has no ultrafilter tangles as claimed.

By arguments similar to an exercise from the lectures, it follows that $\Gqq$ is Hausdorff (we shifted the proof to the end of the introduction):
\begin{lemma}[$\L$]\label{badGisHD}
There is a connected graph with $\transcl=\tsim\subsetneq\sim$.
\end{lemma}
Furthermore, the edgeless Hamilton circle from \textsc{ITop} also is one in $\Gqq$, so $\Gqq$ has no TST either.
All in all, this graph shows a sufficiently rich structure of finite vertex separators branching in a binary tree like way (the separators of the form $\lceil u\rceil_T$ with $u$ a vertex of the NST $T$ of $G$ given by the underlying binary tree plus the edge $x\emptyset$) can force $\tsim\subsetneq\sim$ even though the graph does not have a single finite cut.

So $\sim$ and $\tsim$ in general do not coincide. As a $K_{2,\aleph_0}$ shows, $\between$ in general is not transitive, but maybe $\transcl$ coincides with $\tsim$ for every graph? An easy example shows that this is not the case:
\begin{figure}[H]
    \centering
    \def\svgwidth{\columnwidth}
    \scalebox{.75}{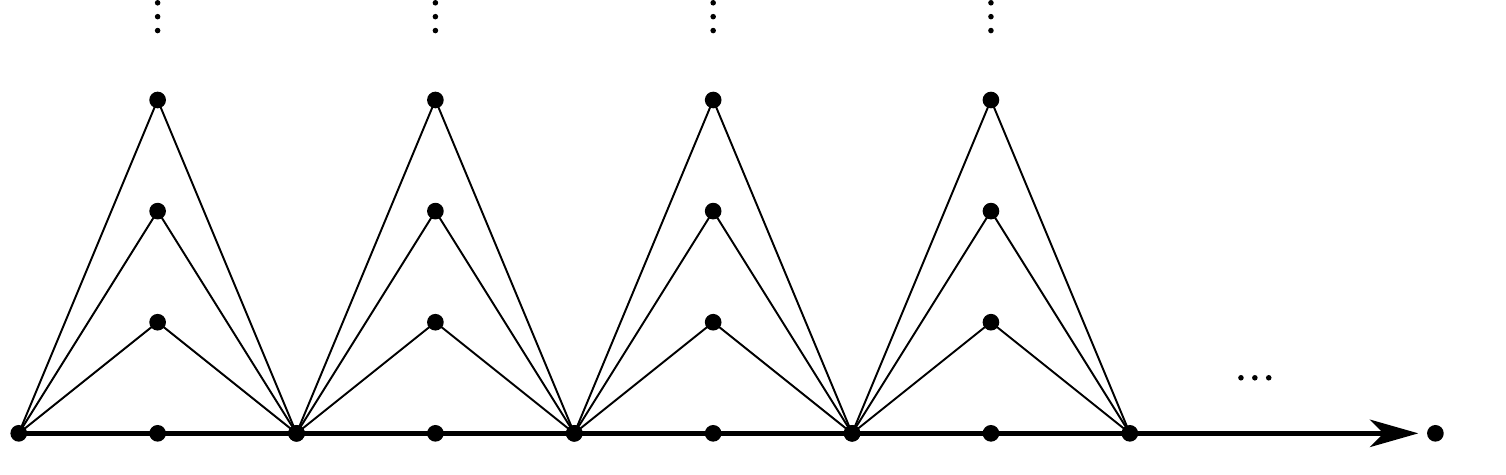}
    \caption{A graph with $\transcl\subsetneq\tsim=\sim$.}
    \label{fig:NonHD}
\end{figure}
If $\omega$ is an undominated end of an arbitrary graph, then clearly $[\omega]_{\utranscl}=\{\omega\}$ holds. 
The graph $G$ pictured above shows that $[\omega]_{\utranscl}\subsetneq [\omega]_\sim$ is possible:
Indeed, $[\omega]_\sim=\{u_n\,|\,n\in\N\}\cup\{\omega\}$ holds while $\sim$ and $\transcl$ agree on the vertex set of $G$. 
Thus $\Gq$ turns this graph into hawaiian earrings and $\Gqq$ does the same except for $\omega$ which does not get identified with any other point.
Taking a closer look reveals that $[u_0]_{\utranscl}=\{u_n\,|\,n\in\N\}$ is not closed in $\Gqq$, witnessed by $[\omega]_{\utranscl}=\{\omega\}$, so $\Gqq$ is not even T${}_1$ by Lemma~\ref{eqClassesClosed}. 
Meanwhile, $\tsim$ and $\sim$ coincide.

So if the $\transcl$-classes in general are not closed, then perhaps taking their closures might suffice? 
If (by abuse of notation) we define $\overline{\transcl}$\index{$\overline{\transcl}$} to be the relation on $\TC$ with $x\,\overline{\transcl}\,y$ whenever there is some $\transcl$-class whose closure in $\TC$ contains both $x$ and $y$, then the following example shows that neither of $\overline{\transcl}$ and its transitive closure do suffice:
\begin{figure}[H]
	\centering
	\def\svgwidth{\columnwidth}
	\scalebox{1}{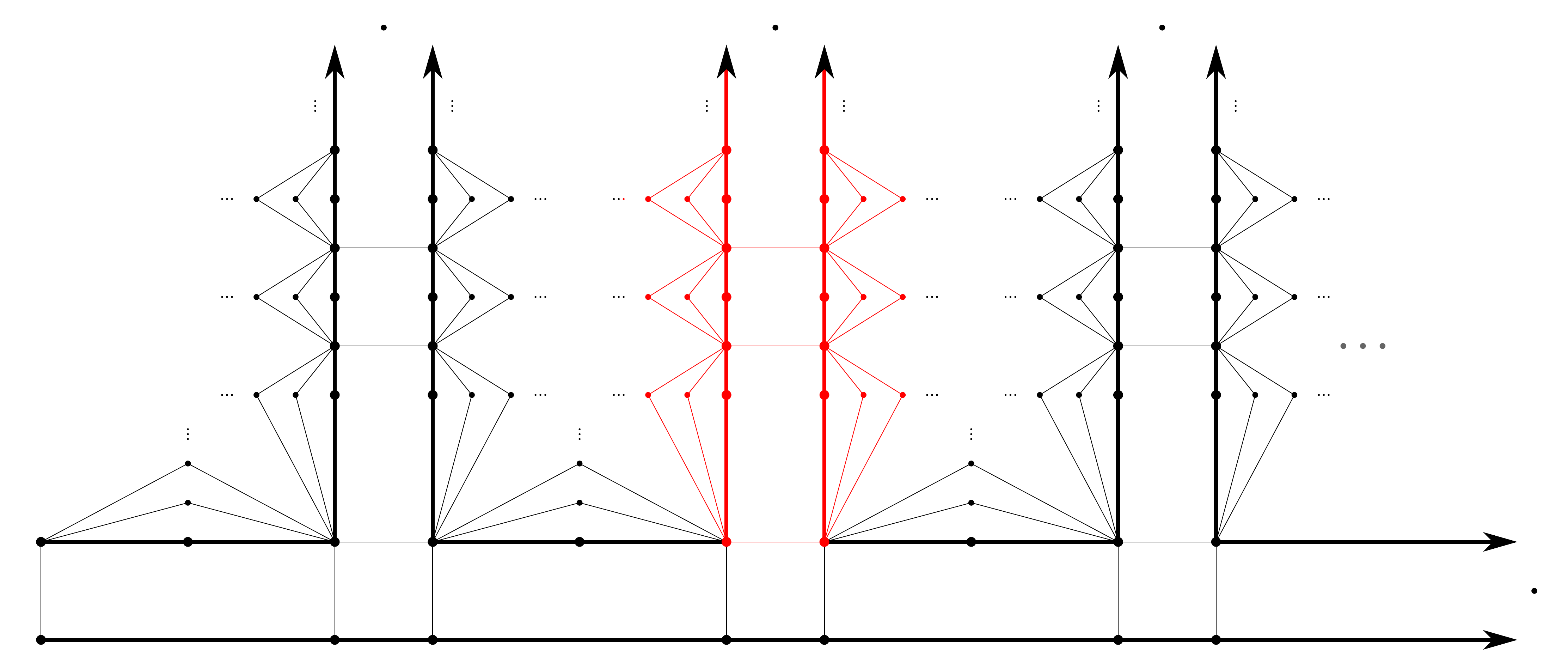}
	\caption{A graph with $\text{trcl}\big(\,\overline{\transcl}\,\big)\subsetneq\tsim=\sim$.}
	\label{fig:finHD}
\end{figure}
Indeed, if $G$ is the graph pictured in Fig.~\ref{fig:finHD}, 
then the $\overline{\transcl}$ is not an equivalence relation: 
For every double ray of heavy edges, its vertices of infinite degree (in $G$) form a $\transcl$-class, and these are the only non-trivial $\transcl$-classes, so every end $\omega_n$ with $n>0$ lies in the closure of two distinct $\transcl$-classes.
Thus, the only non-trivial equivalence class of the transitive closure of $\overline{\transcl}$ 
consist of all vertices of $G$ of infinite degree together with all ends $\omega_n$ ($n\in\N$).
Since the end $\omega$ is excluded from this class, the quotient $\TC/\text{trcl}\big(\,\overline{\transcl}\,\big)$ fails to be T${}_1$.

As described in the preliminaries of Theorem~\ref{minimalHDconstruction}, it is possible to obtain $\tsim$ from $\transcl$ via a transfinite construction.
The graph from Fig.~\ref{fig:finHD} can also be used to show that there exists no ordinal $\alpha<\omega$ such that $h^\alpha(\TC)=\HTC$ holds for every graph $G$ (where $h^\alpha(\TC)$ is defined in Section~\ref{subsec:GeneralTopology}).
Indeed, assume for a contradiction that there is such an $\alpha<\omega$, and let $G$ be the graph pictured in Fig.~\ref{fig:finHD}. Put $G_0=G$.
For every $n\in\N$ there is a subgraph $H_n$ corresponding to $\omega_n$ just like the red subgraph of $G$ corresponds to $\omega_1$.
To obtain the graph $G_1$ from $G$ we replace the red subgraph $H_1$ of $G$ with a copy of $G$ where the vertices $x$ and $y$ of the copy take over the roles of $u_1$ and $t_1$ of the original $G$, respectively, and we do the same for every other $n$.
Next we obtain $G_2$ from $G_1$ by replacing the $H_n$ of each copy of $G$ that was added in the previous construction step in the same way.
Proceeding inductively, we arrive at a graph $G_{\alpha+1}$ for which $h^\alpha(\vartheta (G_{\alpha+1}))$ is not Hausdorff, contradicting our assumption.

Now that all straightforward attempts starting `from below' with $\transcl$ failed and the graph from Fig.~\ref{fig:finHD} showed that a sufficiently rich structure can result in $\tsim\subsetneq\sim$, we will try a different approach.
Readers who attended a talk by Diestel about the discovery of topological infinite graph theory already noticed that every example graph from this introduction is based on a graph which occurred in his talks. In these talk, the examples led to the definition of a circle and the usage of arcs instead of (graph-)paths. 
So, perhaps some sort of special arcs might describe the $\tsim$-classes? 
\begin{figure}[H]
\centering
\includegraphics[scale=0.85]{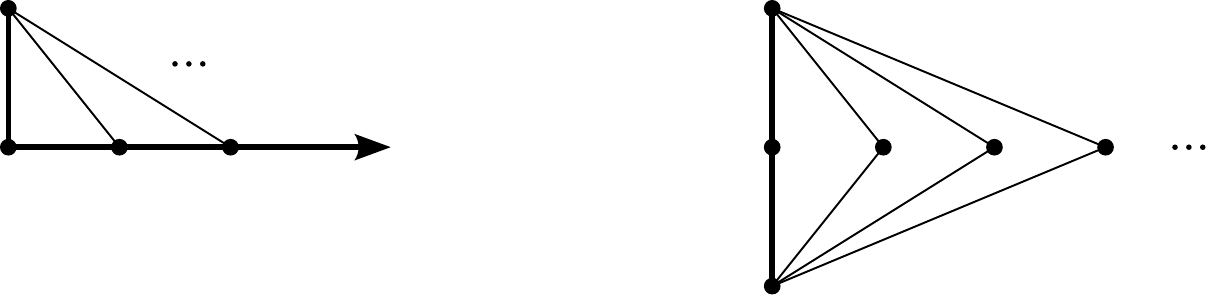}
\caption{A dominated ray and a $K_{2,\aleph_0}$.}
    \label{fig:Probs}
\end{figure}
For the example graphs from Fig.~\ref{fig:Probs} this does not work: If $G$ is the left graph, then every arc (in $\TC$) from the sole end of $G$ to its dominating vertex must visit some vertex of finite degree. Similarly, every arc between the two vertices of infinite degree of the right graph must visit some vertex of finite degree. Furthermore, arc-constructions in graphs that are not locally finite face various problems. But, what if we enrich $\TC$ with some auxiliary structure reflecting our intuition? Let us have a second look at the example graphs discussed so far, but this time we draw in auxiliary arcs (using grey):
\begin{figure}[H]
	\centering
	\def\svgwidth{\columnwidth}
	\scalebox{.75}{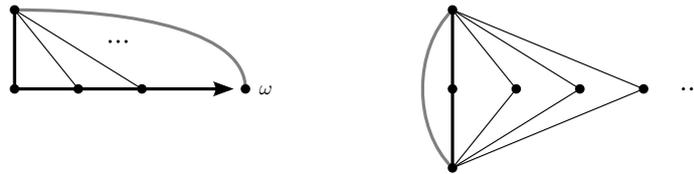}
	\caption{A dominated ray and a $K_{2,\aleph_0}$ with auxiliary arcs.}
\end{figure}
\begin{figure}[H]
	\centering
	\def\svgwidth{\columnwidth}
	\scalebox{.75}{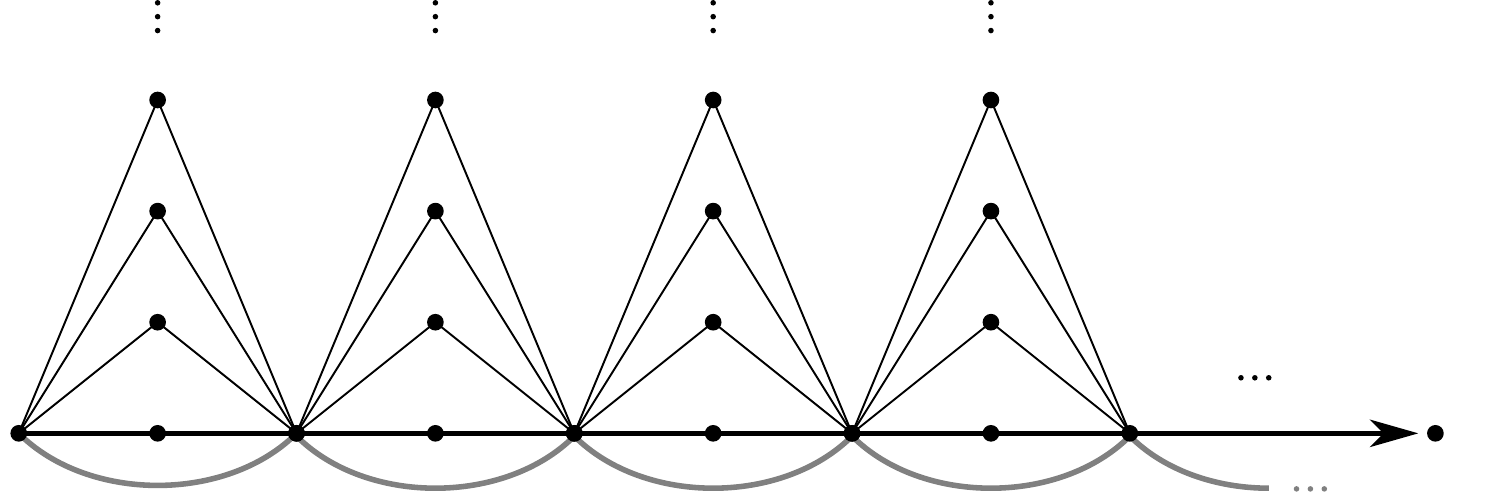}
	\caption{The graph from Fig.~\ref{fig:NonHD} with auxiliary arcs.}
\end{figure}
\begin{figure}[H]
	\centering
	\def\svgwidth{\columnwidth}
	\scalebox{1}{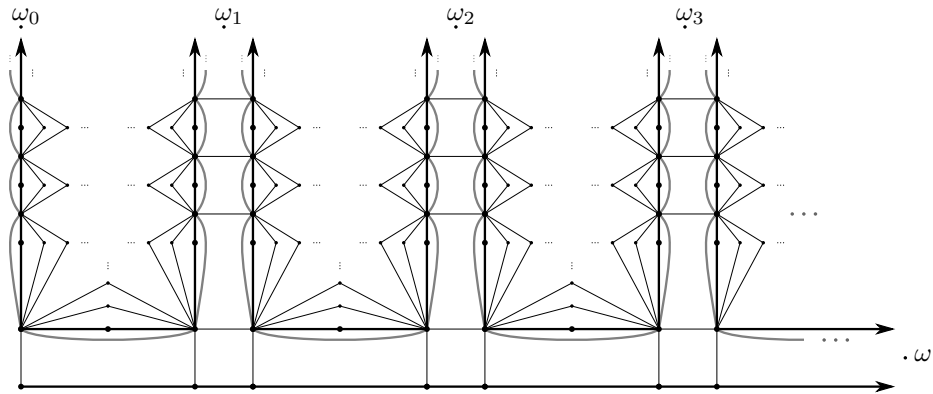}
	\caption{The graph from Fig.~\ref{fig:finHD} with auxiliary arcs.}
\end{figure}
For the simple examples things look fine, and after choosing a more convenient drawing, the edgeless Hamilton circle from Fig.~\ref{fig:badG} can be visualised by an auxiliary Hamilton circle:
\begin{figure}[H]
    \centering
    \def\svgwidth{\columnwidth}
    \scalebox{.6}{\input{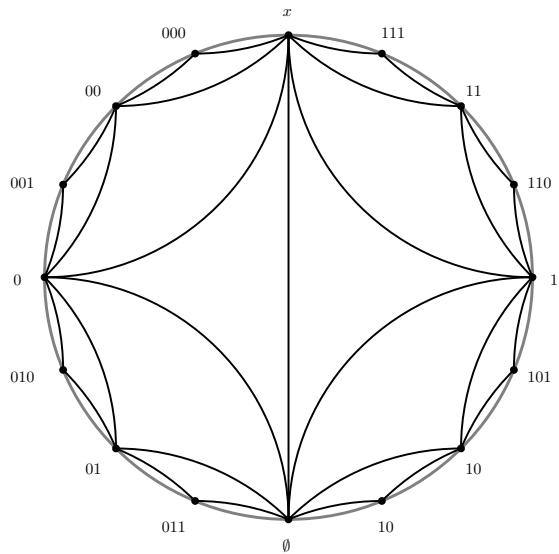}}
    \caption{The graph from Fig.~\ref{fig:badG} with auxiliary Hamilton circle.}
    \label{fig:badGnicelyDrawnETop}
\end{figure}
We have already seen that for the graph $G$ from Figures~\ref{fig:badG} and \ref{fig:badGnicelyDrawnETop} the quotient $\Gq$ is not the maximal Hausdorff quotient of $\TC$. In fact, we shall see later that auxiliary arc components describe precisely the $\sim$-classes. But not all is lost: The auxiliary Hamilton circle failing to describe the $\tsim$-classes is wild, while the auxiliary arcs from our other examples are all tame.\footnote{Recall that an arc is called \textit{wild} if it induces on some subset of its vertices the ordering of the rationals, and \textit{tame} otherwise.}
So perhaps we can use this in order to characterise the graphs with $\cE G=\HTC$?

I did not succeed in finding a combinatorial characterisation of the graphs $G$ satisfying $\cE G=\HTC$.
However, as our main contribution we at least present sufficient combinatorial conditions.
In Section~\ref{GqOutlook} we will discuss the main difficulty prohibiting me from giving a combinatorial characterisation. Basically, the difficulty lies in describing a limit property (here: wild and tame), which excludes naive inverse limit approaches right from the start.

Before we continue, we give a proof of a fact we used earlier:
\begin{proof}[Proof of Lemma~\ref{badGisHD}]
Let $G'$ be the graph from Fig.~\ref{fig:badG}, and consider $G:=G'-x$ for simplicity. Furthermore, let $T$ be the underlying binary tree of $G$, and note that it is an NST of $G$ (with root $\emptyset$). It suffices to verify that $\Gqq$ is Hausdorff. Therefore, let $x\neq y\in\TC$ be given. We have to find disjoint open neighbourhoods of $x$ and $y$ in $\TC$ which are $\transcl$-closed. 
Since the only non-trivial $\transcl$-classes are of the form $\{[u01^\omega]_\Omega,u,[u10^\omega]_\Omega\}$ with $u$ a vertex of $G$, without loss of generality we may assume that $x$ and $y$ are vertices of $G$. 
Let $\langle\,\cdot\,\rangle\colon \bigcup_{n\le\omega}2^n\to\I$ map each 0-1 sequence (finite or infinite) to its naturally corresponding value in $\I$, and let $\Phi\colon \Omega(G)\to\I$ map every end $\omega$ to the image of the binary sequence induced by $R_\omega^T$ under the map $\langle\,\cdot\,\rangle$.
We check two cases:

First suppose that $x$ and $y$ are incomparable with respect to $\le_T$. Then we define an open neighbourhood $W_x$ of $x$ in $\TC$ by taking the union of $\cO_G(x,1/2)$ with the following choices of sets: for every $n\ge 1$ choose
\begin{align*}
&\cO_G\big(x0^n,\,1/2\big)\cup\hat{C}_{\TC}\left(\lceil x0^n\rceil_T,[x0^n1^\omega]_{\Omega(G)}\right)\\
\cup\,&\cO_G\big(x1^n,\,1/2\big)\cup\hat{C}_{\TC}\left(\lceil x1^n\rceil_T,[x1^n0^\omega]_{\Omega(G)}\right).
\end{align*}
Similarly, we define $W_y$. These clearly are disjoint and $\transcl$-closed.

Now suppose that $x$ and $y$ are comparable with respect to $\le_T$ with $x<_Ty$, say. Without loss of generality suppose that the first digit of the 0-1 sequence $y$ after $x$ is 1 (the other case follows by symmetry).
Next, we pick some $\xi\in (\langle y0^\omega\rangle,\langle y01^\omega\rangle)\setminus\Q$ and let $R=u_1u_2\hdots$ be some normal ray in $T$ whose end in $G$ corresponds to $\xi$. In particular there are some $N(x)<N(y)\in\N$ such that $x=u_{N(x)}$ and $y=u_{N(y)}$. 
For every $n> N(x)$ we define $W_x^n$ as follows: If $u_{n+1}=u_n1$ we let
\begin{align*}
W_x^n:=\cO_G(u_n,1/2)\cup\hat{C}_{\TC}\left(\lceil u_n\rceil_T,[u_n0^\omega]_{\Omega(G)}\right)
\end{align*}
and $W_x^n:=\emptyset$ otherwise. Similarly, for every $n> N(y)$ we define $W_y^n$ as follows: If $u_{n+1}=u_n0$ we let
\begin{align*}
W_y^n:=\cO_G(u_n,1/2)\cup\hat{C}_{\TC}\left(\lceil u_n\rceil_T,[u_n1^\omega]_{\Omega(G)}\right)
\end{align*}
and $W_y^n:=\emptyset$ otherwise. Furthermore, for every $n>0$ we let $M_x^n$ denote
\begin{align*}
&\cO_G\big(x0^n,\,1/2\big)\cup\hat{C}_{\TC}\left(\lceil x0^n\rceil_T,[x0^n1^\omega]_{\Omega(G)}\right)
\end{align*}
and similarly for every $n>0$ we let $M_y^n$ denote
\begin{align*}
&\cO_G\big(y1^n,\,1/2\big)\cup\hat{C}_{\TC}\left(\lceil y1^n\rceil_T,[y1^n0^\omega]_{\Omega(G)}\right)
\end{align*}
Finally let 
\begin{align*}
W_z:=\cO_G(z,1/2)\cup\bigcup_{n> N(z)}W_z^n\cup\bigcup_{n>0} M_z^n
\end{align*}
for $z\in\{x,y\}$. Clearly, $W_x$ and $W_y$ are open and disjoint.
We show that $W_x$ is $\transcl$-closed (then $W_y$ will be $\transcl$-closed by an analogue argument). 

For this, we first show that if $\omega$ is an end in $W_x$ dominated by a vertex $u$ then $u$ is also in $W_x$.
If $\omega$ is in some $M_x^n$, then clearly $u\in W_x$ holds.
Otherwise $\omega$ is an end in some $W_x^n$. If $u\in W_x^n$ then we are done, so suppose not. 
Due to $\omega\in W_x^n$ the normal ray $R'=t_0t_1\hdots$ of $\omega$ starts with $u_n0$ and $u$ is in $\{u_0,\hdots, u_n\}$ since $T$ is normal. Pick $k\le n$ with $u=u_k$. By $u\notin W_x^n$ we have $k<n$. The only two possibilities for $R'$ are $u01^\omega$ and $u10^\omega$, but the first case is impossible due to $k<n$ and choice of $W_x^n$, so we have $R'=u10^\omega$. Since the sequence $u_n$ is an initial segment of $R'$ this yields $u=u_k\in W_x^k$ as desired. Thus we have $\Delta[W_x\cap \Omega(G)]\subseteq W_x$ as claimed.

Second, we show that if $u$ is a vertex in $W_x$ dominating an end $\omega$ then $\omega$ is also in $W_x$. Since $u$ and $R_\omega^T$ represent the same rational and by the construction of $W_x$, it suffices to show that the ends in $W_x$ correspond to $(\langle x0^\omega\rangle,\xi)$, i.e. that
\begin{align*}
\Phi[W_x\cap\Omega(G)]=(\langle x0^\omega\rangle,\xi).
\end{align*}
Since ``$\subseteq$'' is clear we show ``$\supseteq$''.
Given any $\lambda\in (\langle x0^\omega\rangle,\xi)$ pick $\omega\in\Omega(G)$ with $\Phi(\omega)=\lambda$ and let $N$ be maximal with $u_N\in R_\omega^T$. 
Then $\lambda<\xi$ implies $u_{N+1}=u_N1$.
If $N=N(x)$ then $\omega\in M_x^n$ for some $n\ge 1$ due to $\lambda>\langle x0^\omega\rangle$. Otherwise $N>N(x)$ and $\omega\in W_x^N$ holds. 

Since the only non-trivial $\transcl$-classes are of the form $\{[u01^\omega]_\Omega,u,[u10^\omega]_\Omega\}$ it follows from the two claims above that $W_x$ is $\transcl$-closed.
\end{proof}

\subsection{A first sufficient combinatorial condition}

Call a graph $G$ \textit{simply-branching} if $\sim$ and $\transcl$ agree on $V^2$.
This is a combinatorial definition since Corollary~\ref{ApproxOnV} combinatorially describes $\transcl$ on $V^2$.\index{simply-branching}

\begin{example}
The graph pictured in Fig.~\ref{fig:NonHD} is simply-branching, while the graphs from Figs.~\ref{fig:finHD}~and~\ref{fig:badG} are not.
\end{example}

\begin{proposition}\label{firstSufCombCond}
Let $G$ be a simply-branching graph with finitely many components. Then $\Gq=\HTC$.
\end{proposition}
\begin{proof} 
We show $\sim=\tsim$.
Recall that $\Gq$ is Hausdorff by Corollary~\ref{HD}, so $\tsim\subseteq\sim$ holds by minimality of $\tsim$.
Assume for a contradiction that $\tsim\subsetneq\sim$ holds, and pick any two distinct points $p,q\in\TC$ witnessing this, i.e. with $p\sim q$ and $p\not\tsim q$. Since $\transcl$ and $\sim$ agree on $V^2$, either $\{p,q\}$ meets $V$ and $\cU$, or both $p$ and $q$ are in $\cU$. We thus check two main cases:

For the first main case suppose $p\in V$ and $q\in\cU$, say. Write $u=p$ and $\upsilon=q$. We check two subcases:

First suppose $\upsilon\in\Upsilon$ and pick any $t\in X_\upsilon$ which exists since $G$ has only finitely many components. Thus $t\tsim\upsilon$ holds by Lemma~\ref{Xtau}, and furthermore $u\sim t$ implies $u\transcl t$. In particular $u\tsim t$, so $u\tsim\upsilon$ is a contradiction.

For the second subcase suppose $\upsilon\in\Omega$ and write $\omega=\upsilon$. If $\omega$ is dominated by any vertex $x$ then clearly $u\sim x$ as well as $x\tsim \omega$, yielding $u\tsim\omega$, a contradiction. So we may assume that $\omega$ is undominated. By Lemma~\ref{eqClassesClosed} it suffices to show that $\omega$ lies in the closure of $[u]_{\tsim}$ in $\TC$ to yield $\omega\in [u]_{\tsim}$. 
Applying Lemma~\ref{ctblNbhdBase} to the undominated end $\omega$ of $G$ yields a sequence $X_0,X_1,\hdots$ of non-empty elements of $\cX$ such that for all $n\in\N$ the component $C(X_n,\omega)$ includes both $X_{n+1}$ and $C(X_{n+1},\omega)$. Furthermore, the collection $\{\hat{C}_{\TC}(X_n,\omega)\,|\,n\in\N\}$ is a countable neighbourhood basis of $\omega$ in $\TC$. Since $\omega$ is undominated, we may assume that $u$ is not in $X_0\cup C(X_0,\omega)$. By Lemma~\ref{FinSepClassesMeetX} we know that every $X_n$ meets $[u]_\sim$ in some vertex $t_n$, so in particular every $X_n$ meets $[u]_{\tsim}$ since $t_n\sim u$ implies $t_n\tsim u$ by assumption.
Thus every $\hat{C}_{\TC}(X_n,\omega)$ meets $[u]_{\tsim}$ as desired, resulting in $\omega\in [u]_{\tsim}$, a contradiction.
This completes the second subcase and the first main case.

Now suppose both $p$ and $q$ are in $\cU$ and write $\upsilon_1=p,\upsilon_2=q$. We will check two subcases:

First suppose $\upsilon_1\in\Upsilon$ and $\upsilon_2\in\Omega$, say. As before, pick $t_1\in X_{\upsilon_1}$ and note $t_1\tsim\upsilon_1$ as well as $t_1\sim \upsilon_2$. Hence $t_1\not\tsim\upsilon_2$, since otherwise $\upsilon_1\tsim\upsilon_2$ is a contradiction. But then we derive a contradiction for $t_1$ and $\upsilon_2$ from the second subcase of the first main case.

For the second subcase suppose both $\upsilon_i$ are in $\Omega$. Since $\upsilon_1\sim\upsilon_2$ there is some vertex $u\in [\upsilon_1]_\sim$ by Lemma~\ref{Forms}. Then $\upsilon_1\tsim u\tsim\upsilon_2$ by the second subcase of the first main case, a contradiction.
\end{proof}




\subsection{Auxiliary arcs and the maximal Hausdorff quotient}

An arc is called \textit{wild}\index{wild} if it induces on some subset of $A\cap V(G)$ the ordering of the rationals, and \textit{tame}\index{tame} otherwise. 
Let $f_{\aA}$ be a function with domain $\TC\setminus\mathring{E}$ that is the identity on $V\cup\Omega$ and which assigns some element of $X_\upsilon$ to $\upsilon$ for every $\upsilon\in\Upsilon$.
We define the equivalence relation $\tcsim$ on $\TC\setminus\mathring{E}$ by letting $x\tcsim y$ whenever there is some tame auxiliary arc from $f_{\aA}(x)$ to $f_{\aA}(y)$ or $f_{\aA}(x)=f_{\aA}(y)$.\footnote{One special case which led to this definition is the following: If $\upsilon\in\Upsilon$ is such that $X_\upsilon=\{t\}$ then our definition must ensure $t\tcsim\upsilon$.}\index{$f_{\aA}$}\index{$\tcsim$}
By Theorem~\ref{InfIndPaths} the equivalence relation $\tcsim$ is well defined.

Next, we introduce a definition which helps us in that it describes `good' open subsets of $\GH$.
Suppose that $R$ is an equivalence relation on $\GH\setminus\mathring{E}(\GG)$.
An open subset $O$ of $\GH$ which is closed under $R$ is called $R$-\textit{standard} if it satisfies the following three conditions for every edge $e$ of $\GG$:\index{$R$-standard}
\begin{enumerate}
\item[(i)] If $O$ contains both endvertices of $e$, then also $\mathring{e}\subseteq O$.
\item[(ii)] If $O$ contains exactly one endvertex of $e$, say $x$, then $O\cap e=[x,m(e))$.
\item[(iii)] If $O$ avoids both endvertices of $e$, then $O$ also avoids $\mathring{e}$.
\end{enumerate}

Every open subset $O$ of $\TC$ which is $\tsim$-closed induces an ${\tsim}\cap (V\cup\Omega)^2$-standard subset $\check{O}$ of $\GH$, 
where $\check{O}$ is the union of the following choice of sets:\index{$\check{O}$}

For every $u\in O\cap V(G)$ choose $\cO_{\GG}(u,1/2)$.

For every $\omega\in O\cap\Omega(G)$ let $\hat{\aC}_\epsilon(X,\omega)$ be a basic open neighbourhood included in $O$, and choose $\hat{\aC}_{1/2}(X,\omega)$.

For every edge $e$ of $G$ we check three cases:
If both endvertices of $e$ are in $O$, then we choose $\mathring{e}$. Else if precisely one endvertex of $e$ is in $O$, say $x$, then we choose $[x,m(e))$. Otherwise no endvertex of $e$ is in $O$ and we choose the empty set.

For every auxiliary edge $e$ with one endvertex in $O$ we know that both endvertices must be in $O$, and hence we choose $\mathring{e}$.

This completes the definition of $\check{O}$.

\begin{obs}\label{standardHD}
If $O$ is an open subset of $\TC$ which is $\tsim$-closed, then
\begin{align*}
O\setminus\mathring{E}(G)=\check{O}\setminus\mathring{E}(\GG).
\end{align*}
If $O_1$ and $O_2$ are two disjoint open subsets of $\TC$ which are $\tsim$-closed, then $\check{O}_1$ and $\check{O}_2$ are disjoint.
\end{obs}

\begin{lemma}\label{BetweenTwoVerticesOnArc}
Let $x$ and $y$ be two vertices of $G$ with $x\not\tsim y$ and let $A$ be an auxiliary arc from $x$ to $y$. Then there is a vertex $u\in V(G)\cap A$ distinct from $x$ and $y$ with $x\not\tsim u\not\tsim y$.
\end{lemma}
\begin{proof}
Put $R={\tsim}\cap (V\cup\Omega)^2$. By Observation~\ref{standardHD} we find two disjoint $R$-standard neighbourhoods $O(x)$ of $x$ and $O(y)$ of $y$ in $\GH$.
Since $A$ is connected there is some point $\xi\in A\setminus (O(x)\cup O(y))$. In particular, $x\not\tsim \xi\not\tsim y$ holds.

If $\xi$ is a vertex of $G$ we are done. 

Else if $\xi$ is an inner edge point of some edge $e$ of $\GG$, then $e$ must be an auxiliary edge with endvertices $a$ and $b$, say. In particular, one of $a$ and $b$ is a vertex of $G$, say $a$. Since $O(x)$ and $O(y)$ are $R$-standard, $\xi\notin O(x)\cup O(y)$ implies that all of $e$ avoids $O(x)\cup O(y)$, so we are done with $u=a$.

Finally assume that $\xi$ is an end of $G$ and write $\omega=\xi$. 
Next, let $H$ be the complete graph on $\{x,y,\omega\}$. For every edge $e=ab$ of $H$ we use Observation~\ref{standardHD} to yield two disjoint $R$-standard neighbourhoods $O_e(a)$ of $a$ and $O_e(b)$ of $b$ in $\GH$. 
Then we let $O(a)$ be the intersection over all sets $O_e(a)$ where $e$ is an edge of $H$ at $a$. 
By Lemma~\ref{SigmaArcEdgesDense} we find some auxiliary edge $e'$ of $A$ together with some $j\in\mathring{e}'\cap O(\omega)$. Then $j$ is in both $O_{x\omega}(\omega)$ and $O_{y\omega}(\omega)$ by choice of $O(\omega)$. Since both of these are $R$-standard, both contain the endvertices of $e'$, and so does $O(\omega)$. Then one of the endvertices of $e'$ is a vertex of $G$ in $A\cap O(\omega)$, and this vertex can take on the role of $u$ by choice of $O(\omega)$.
\end{proof}

\begin{corollary}\label{IterateWild}
Let $x$ and $y$ be two distinct vertices of a graph $G$ with $x\not\tsim y$. Then every auxiliary arc from $x$ to $y$ is wild.
\end{corollary}
\begin{proof}
By Lemma~\ref{BetweenTwoVerticesOnArc} between every two distinct vertices $a$ and $b$ of $G$ on $A$ with $a\not\tsim b$ there is a vertex $u$ of $G$ on $A$ with $a\not\tsim u\not\tsim b$. Iterating this lemma, starting with $a=x$ and $b=y$, yields a subset of $A\cap V(G)$ on which $A$ induces the ordering of the rationals.
\end{proof}

\begin{lemma}\label{tcsimSubtsim}
For every graph $G$ we have $\tcsim\subseteq\tsim$ on $V(G)^2$.
\end{lemma}
\begin{proof}
Assume not for a contradiction, witnessed by two distinct vertices $x$ and $y$, i.e. with $x\tcsim y$ and $x\not\tsim y$. Pick a tame auxiliary arc $A$ from $x$ to $y$ witnessing $x\tcsim y$. Then the existence of $A$ contradicts Corollary~\ref{IterateWild}.
\end{proof}

\begin{lemma}\label{tsimSubtcsimSpecial}
If $G$ is a graph such that every auxiliary arc is tame, then $\tsim\subseteq\tcsim$ holds on $V(G)^2$.
\end{lemma}
\begin{proof}
Suppose that two distinct vertices $x$ and $y$ of $G$ are given with $x\tsim y$. Then $x\sim y$ holds by Lemma~\ref{3Rels}. Using Theorem~\ref{SigmaArcConstruction} we find an auxiliary arc from $x$ to $y$, which is tame by assumption, so we have $x\tcsim y$.
\end{proof}

\begin{lemma}\label{auxArcsTamThen3RelsEqual}
Let $G$ be a graph with finitely many components such that every auxiliary arc is tame.
Then $\tsim=\tcsim=\sim$ holds.
\end{lemma}
\begin{proof}
By Lemmas~\ref{tcsimSubtsim}~\&~\ref{tsimSubtcsimSpecial} together with Theorem~\ref{SigmaArcConstruction} we know that
\begin{enumerate}
\item[$(\ast)$] \textit{The three equivalence relations $\tsim$, $\tcsim$ and $\sim$ agree on $V(G)^2$.}
\end{enumerate}
We show ``$\tsim\subseteq\sim\subseteq\tcsim\subseteq\tsim$''.

``$\tsim\subseteq\sim$''. This follows from Corollary~\ref{HD} and minimality of $\tsim$.

``$\sim\subseteq\tcsim$''. Let $p,q\in\TC\setminus\mathring{E}$ be given with $p\sim q$. We check several cases:

If both $p$ and $q$ are in $V(G)\cup\Omega(G)$, then we are done by Theorem~\ref{AuxArcsChar} combined with the assumption that all auxiliary arcs are tame.

Next, suppose that $p$ is a vertex of $G$ and $q$ is an ultrafilter tangle, and write $u=p$ and $\upsilon=q$. Then we pick some $t\in X_\upsilon$. Recall that $\upsilon\between t$ holds by Lemma~\ref{Xtau}, so ``$\tsim\subseteq\sim$'' implies $\upsilon\sim t$. In particular, $u\sim\upsilon\sim t$ holds, which implies $u\sim t$, and hence $u\tcsim t$ by $(\ast)$. By definition of $\tcsim$, this yields $u\tcsim\upsilon$.

Now suppose that $p$ is an ultrafilter tangle and $q$ is an end of $G$, and write $\upsilon=p$ and $\omega=q$. Pick $t\in X_\upsilon$ and note that $t\sim\upsilon$ holds as before. Furthermore, $\upsilon\sim\omega$ implies $t\sim \omega$. Hence $t\tcsim\omega$ follows by a previous case. Together with $\upsilon\tcsim t$ this yields $\upsilon\tcsim\omega$.

Finally suppose both $p$ and $q$ are ends of $G$ and write $\omega=p$ and $\omega'=q$. By Lemma~\ref{Forms} there is some vertex $u$ of $G$ in $[\omega]_\sim$. Then $\omega\sim u\sim\omega'$ implies $\omega\tcsim u\tcsim\omega'$ by a previous case, and therefore $\omega\tcsim\omega'$.

``$\tcsim\subseteq\tsim$''. Let $p,q\in\TC\setminus\mathring{E}$ be given with $p\tcsim q$ and let $A$ be some tame auxiliary arc from $f_{\aA}(p)$ to $f_{\aA}(q)$. By $(\ast)$, not both $p$ and $q$ are vertices of $G$.
We check several cases:

First, suppose that $p$ is a vertex of $G$ and $q$ is an ultrafilter tangle, and write $u=p$ and $\upsilon=q$. 
Furthermore, write $t=f_{\aA}(\upsilon)$ and recall that $t$ is a vertex in $X_\upsilon$. In particular, we have $u\tcsim t$,
which implies $u\tsim t$ by $(\ast)$. Furthermore, $t\tsim\upsilon$ holds by Lemma~\ref{Xtau}. Together with $u\tsim t$, this yields $u\tsim\upsilon$.

Next, suppose that $p$ is a vertex of $G$ and $q$ is an end of $G$, and write $u=p$ and $\omega=q$.
Assume for a contradiction that $u\not\tsim \omega$ holds, witnessed by some disjoint open neighbourhoods $O(u)$ of $u$ and $O(\omega)$ of $\omega$ in $\TC$, respectively, which are both $\tsim$-closed.
By Lemma~\ref{SigmaArcEdgesDense} there is some auxiliary edge $e$ whose interior meets $\check{O}(\omega)$ and which $A$ traverses.
In particular, since $\check{O}(\omega)$ is $\tsim\cap (G\cup\Omega)^2$-standard, we know that the whole edge $e$ is included in $\check{O}(\omega)$. Let $t$ be an endvertex of $e$ which is also a vertex of $G$.
Then $u\not\tsim t$ holds, but also $u\tcsim t$ witnessed by a tame auxiliary arc $A'\subseteq A$. This contradicts $(\ast)$, so $u\tsim\omega$ must hold as desired.

Now suppose that $p$ is an ultrafilter tangle of $G$ and $q$ is an end of $G$, and write $\upsilon=p$ and $\omega=q$. Pick $t\in X_\upsilon$ and note $t\tcsim\upsilon$ as well as $t\tcsim \omega$ due to $t\tcsim\upsilon\tcsim\omega$. 
Hence $t\tsim\omega$ holds by the previous case. Together with $\upsilon\tsim t$ this yields $\upsilon\tsim\omega$.

Finally, suppose that both $p$ and $q$ are ends of $G$, and write $\omega=p$ and $\omega'=q$.
Pick some $X\in\cX$ witnessing $\omega\neq\omega'$.
Then $A$ meets $X$ in some vertex $u$ of $G$ due to Observation~\ref{vxJAL}.
Then $\omega\tcsim u\tcsim\omega'$ implies $\omega\tsim u\tsim\omega'$ by an earlier case, and hence $\omega\tsim\omega'$ follows.
\end{proof}

\subsection{A second sufficient combinatorial condition}

\newtheorem{claim_T2}{Claim}
\begin{lemma}\label{noT2allArcsTame}
If $G$ is a connected graph such that for every $x\in\Gq$ the graph $G$ has a normal spanning tree $T(x)$ whose subtree $\bigcup_{\omega\in x\cap\Omega}R_\omega^{T(x)}$ contains no subdivision of the (infinite) binary tree, then all auxiliary arcs (in $\GH$) are tame.
\end{lemma}
\begin{proof}
Let $A$ be any auxiliary arc, and assume for a contradiction that it is wild, witnessed by some infinite set $V_{\Q}\subseteq A$ of vertices of $G$ on which it induces the ordering of the rationals.
By Lemma~\ref{cutJAL} we know that there is some unique $x\in\Gq$ with $A\setminus\mathring{E}(\GG)\subseteq x$. For every $\omega\in x\cap\Omega$ write $R_\omega=R_\omega^{T(x)}$ and put $T=\bigcup_{\omega\in x\cap \Omega}R_\omega$, taking as root for $T$ the root of $T(x)$.
By \cite[Proposition 8.6.1]{Bible} 
the rooted tree $T$ is recursively prunable, so we may assume that every vertex of $T$ receives an ordinal label by some fixed recursive pruning of $T$. 
For every $\omega\in x\cap\Omega$ let $\alpha(\omega)$ be the minimal ordinal label of the vertices of $R_\omega$.
\begin{claim_T2}\label{AlphaOmegaUp}
If $t\in V(R_\omega)$ receives label $\alpha(\omega)$, then $tR_\omega=\lfloor t\rfloor_{T_{\alpha(\omega)}}$ holds.
\end{claim_T2}
\begin{proof}\renewcommand{\qedsymbol}{$\diamond$}
Every vertex $u$ of $tR_\omega$ receives label $\alpha(\omega)$ by minimality of $\alpha(\omega)$, so $tR_\omega\subseteq\lfloor t\rfloor_{T_{\alpha(\omega)}}$ holds. 
Since $tR_\omega\subsetneq\lfloor t\rfloor_{T_{\alpha(\omega)}}$ would contradict the fact that $\lfloor t\rfloor_{T_{\alpha(\omega)}}$ is a chain in $T_{\alpha(\omega)}$, equality must hold as claimed.
\end{proof}

Now choose some homeomorphism $\sigma\colon \I\to A$ and 
write $\I_{\Q}=\sigma^{-1}[V_{\Q}]$. Next, pick two points $\lambda^-<\lambda^+$ in $\I_{\Q}$
and let $\cA$ denote the set of all accumulation points of $\I_{\Q}$ in the open interval $(\lambda^-,\lambda^+)$. Since $\I_{\Q}$ is discrete in $\I$ we know that $\cA$ avoids $\I_{\Q}$ and $\sigma$ maps $\cA$ into the end space of $G$.
\begin{claim_T2}\label{AaccOfA}
Every point of $\cA$ is an accumulation point of $\cA$.
\end{claim_T2}
\begin{proof}\renewcommand{\qedsymbol}{$\diamond$}
Given any $\lambda\in\cA$ pick a sequence $(\lambda_n)_{n\in\N}$ in $\I_{\Q}$ with $\lambda_n\to\lambda$, and without loss of generality suppose that $\lambda_n<\lambda_{n+1}<\lambda$ holds for all $n\in\N$. 
Given an arbitrary $\epsilon>0$ pick $N\in\N$ big enough that $|\lambda-\lambda_n|<\epsilon$ and $\lambda_n\in (\lambda^-,\lambda^+)$ hold for all $n\ge N$. Using that $A$ induces the ordering of the rationals on $V_{\Q}$ (and hence on $\I_{\Q}$) it is easy to find some $\mu\in (\lambda_N,\lambda_{N+1})\cap\cA$. By choice of $N$ we have $|\lambda-\mu|<\epsilon$.
\end{proof}

In order to yield a contradiction, we inductively construct a sequence $(\omega_n)_{n\in\N}$ (where the $\omega_n$ are ends, not ordinals) in $\cA$ together with a sequence $(t_n)_{n\in\N}$ of vertices of $T$ such that for all $n\in\N$ we have
\begin{enumerate}
\item $t_n$ is a vertex of $R_{\omega_n}$,
\item $t_n$ receives label $\alpha(\omega_n)$,
\item $\alpha(\omega_{n+1})<\alpha(\omega_n)$.
\end{enumerate}
This suffices, since then by (iii) we have a strictly decreasing sequence $(\alpha(\omega_n))_{n\in\N}$ of ordinals, which is impossible.

We start the construction by picking an arbitrary $\omega_0\in\sigma[\cA]$ together with some vertex $t_0$ of $R_{\omega_0}$ receiving label $\alpha(\omega_0)$. Now suppose that we are at step $n>0$ of the construction, and $t_k$ and $\omega_k$ got constructed for $k<n$. Using Claim~\ref{AaccOfA} we find some $\omega_n$ in $\sigma[\cA]\cap\hat{\aC}(\lceil t_{n-1}\rceil_T,\omega_{n-1})$ other than $\omega_{n-1}$.
Let $u$ be the maximal vertex of $R_{\omega_n}\cap R_{\omega_{n-1}}$ with respect to the natural ordering $\le_T$ on $T$ induced by its root. 
Since $T(x)$ is an NST and $\lceil u\rceil_T=\lceil u\rceil_{T(x)}$ holds, we know that $\lceil u\rceil_T$ witnesses $\omega_n\neq\omega_{n-1}$. 
By (i) we have $t_{n-1}\in R_{\omega_{n-1}}$, so $u>_Tt_{n-1}$ holds since otherwise $u\le_Tt_{n-1}$ contradicts $C(\lceil t_{n-1}\rceil_T,\omega_n)=C(\lceil t_{n-1}\rceil_T,\omega_{n-1})$. 
By Claim~\ref{AlphaOmegaUp}, we find some vertex $t_n$ of $R_{\omega_n}\setminus R_{\omega_{n-1}}$ with $t_n>_Tu$ such that $t_n$ receives label $\alpha(\omega_n)$. 
In particular, $u>_Tt_{n-1}$ implies $t_n>_Tt_{n-1}$. Also recall that $t_{n-1}$ receives label $\alpha(\omega_{n-1})$ by (ii).
Together with $t_n\notin R_{\omega_{n-1}}$ and Claim~\ref{AlphaOmegaUp} applied to $t_{n-1}$ and $\omega_{n-1}$, the only possibility for $\alpha(\omega_n)$ is $\alpha(\omega_n)<\alpha(\omega_{n-1})$. Thus (i)--(iii) hold for $n$.
\end{proof}

\begin{proposition}\label{sufficientAllAuxArcsTame}
If $G$ is a connected graph such that for every $x\in\Gq$ the graph $G$ has an NST $T(x)$ whose subtree $\bigcup_{\omega\in x\cap\Omega}R_\omega^{T(x)}$ contains no subdivision of the (infinite) binary tree,
then $\cE G$ is the maximal Hausdorff quotient of $\TC$.
\end{proposition}
\begin{proof}
By Lemma~\ref{noT2allArcsTame} all auxiliary arcs (in $\GH$) are tame, so $\tsim$ and $\sim$ coincide by Lemma~\ref{auxArcsTamThen3RelsEqual}.
Furthermore, $\Gq\simeq\cE G$ holds by Theorem~\ref{ETopInvLimAndGq}.
\end{proof}

Proposition~\ref{sufficientAllAuxArcsTame} is not best possible:
\begin{lemma}
There exists a connected graph $G$ all whose auxiliary arcs are tame and which has an NST, but which does not satisfy the premise of Theorem~\ref{sufficientAllAuxArcsTame}.
\end{lemma}
\begin{proof}
Let $G$ be obtained from the infinite binary tree with root $r$ by adding an edge from $r$ to every other vertex. Then $\Upsilon$ is empty by the same argument we used on the graph from Fig.~\ref{fig:badG}, and $\sim$ only identifies all ends of $G$ with $r$ (since $r$ is the only vertex of infinite degree). Hence $\sim=\transcl$ implies $\Gq=\HTC$ by Lemma~\ref{3Rels}. Now Lemma~\ref{SigmaArcEdgesDense} together with the fact that the auxiliary edges form a star in $\GG$ (with center $r$ and the ends of $G$ as leaves) yields that every auxiliary arc traverses the closure of either one or two auxiliary edges, and hence must be tame.
\end{proof}

\subsection{Outlook}\label{GqOutlook}

Of course, we would like to know whether $\tsim$ and $\tcsim$ coincide for all $G$. Furthermore, if $x$ and $y$ are two vertices of $G$ and every auxiliary arc between them is wild, then the question arises whether there exists structures of the graph witnessing that. 
One difficulty are the auxiliary edges between vertices and ends which allow the auxiliary arcs to `jump' in an NST.

Let $\langle\,\cdot\,\rangle\colon \bigcup_{n<\omega}2^n\to\I$ map each sequence to its naturally corresponding value in $\I$.
\begin{conjecture}
For every graph $G$ and every two distinct vertices $x$ and $y$ of $G$ the following are equivalent:
\begin{enumerate}
\item $x\sim y$ and every auxiliary arc from $x$ to $y$ is wild.
\item There exists a countable subset $W$ of $\vec{S}(G-xy)$ such that the following hold:
\begin{enumerate}
\item[\normalfont(a)] Every $(A,B)\in W$ satisfies $x\in A\setminus B$ and $y\in B\setminus A$.
\item[\normalfont(b)] The partial ordering $\le$ of $\vec{S}(G-xy)$ induces the ordering of the rationals on $W$.
\item[\normalfont(c)] There exists a bijection $\varphi\colon T_2\to W$ satisfying $\varphi(a)\le\varphi(b)$ if and only if $\langle a\rangle\le\langle b\rangle$.
Let $\hat{\varphi}\colon T_2\to\cX$ map each vertex of the $T_2$ to the vertex separator of $\varphi(a)$. For every $a\in V(T_2)$ let $n(a)\in\N$ be the level of the $T_2$ containing $a$, and let $X_a$ be the set of all vertices of the first $n(a)$ levels of the $T_2$ minus $a$.
Then for every $a\in V(T_2)$ and every auxiliary arc $A$ from $x$ to $y$, the set $\hat{\varphi}(a)\setminus\bigcup_{b\in X_a}\hat{\varphi}(b)$ meets $A$.
\end{enumerate}
\item $x\not\tsim y$.
\end{enumerate}
\end{conjecture}




\newpage
\section*{Acknowledgement}

I am grateful to my family for all the years of support, and I am grateful to Prof. Dr. Reinhard Diestel for his inspiring lecture courses which led me here.
During the development of this work I received invaluable assistance from Dr. Max Pitz: He motivated me to investigate two questions of Prof. Dr. Reinhard Diestel which resulted in Chapter~\ref{ECOcompactification}, and he pointed out to me an abstract minimality proof.
When I thought about auxiliary edges, unsure about their usefulness, he showed me parallels to certain Stone-Čech compactifications. 
This motivated me to fight my way through the proof of Theorem~\ref{SigmaArcConstruction}.
And while I was still thinking about auxiliary arcs and the maximal Hausdorff quotient, he already asked about circles and topological spanning trees in the auxiliary space.\\ Thank you so much!

\newpage
\printindex

\clearpage
\phantomsection
\bibliography{MA_BIB}
\bibliographystyle{MYamsplain}

\end{document}